\documentclass[11pt,a4paper]{amsart}

\usepackage{amsmath, amsthm, amssymb, paralist, xspace, graphicx, url, amscd, euscript, mathrsfs, epic, eepic,color, verbatim, stmaryrd,bm}

\usepackage[shortlabels]{enumitem}

\usepackage{leftidx}

\usepackage[margin=.85in]{geometry}

\usepackage[normalem]{ulem}
\usepackage[colorlinks=true]{hyperref}
\usepackage{hyperref}
\usepackage[all]{xy}

\newcommand{\xinfty}{\infty_{\cX}} % orbifod infty

\newcommand{\lbspin}{\mathbf{L}^{\vee}} %line bundle for constructing spin structure

\newcommand{\lblog}{\mathbf{N}} %line bundle for constructing log structure

\newcommand{\uspin}{\mathcal{L}_R} %universal spin bundle

\newcommand{\tlog}{\mathcal{L}_S} %universal line bundle for R-target  log structure

\newcommand{\hodge}{\mathscr{H}} % the hodge bundle

\newcommand{\uomega}{\mathcal{L}_\omega} %universal \omega_{\log}

\newcommand{\rd}{R^{\bullet}} %right derived functor decoration

\newcommand{\tW}{\widetilde{W}} %twisted superpotential

\newcommand{\tomega}{\widetilde{\omega}} %twisted canonical bundle

\newcommand{\crit}{\operatorname{Crit}} % Critical locus

\newcommand{\cI}{\mathcal{I}_{\mu}} % inertia stack
\newcommand{\ocI}{\overline{\mathcal{I}}_{\mu}} % rigidified inertia stack
\DeclareMathOperator{\ev}{ev} %evaluation of the underlying structure

\newcommand{\Hom}{\operatorname{Hom}} % Hom stack

\newcommand{\BG}{\mathbf{BG}}
\newcommand{\BC}{\mathbf{BC}^*_\omega}
\newcommand{\LR}{\mathcal L_\omega}

\newcommand{\bE}{\mathbf{E}}
\newcommand{\bK}{\mathbf{K}}

\newcommand{\ul}{\underline}

\newcommand{\etwist}{a}
\newcommand{\ttwist}{\tilde{r}}

\newcommand{\ddata}{{\varsigma}} % discrete data along markings

\newcommand{\cH}{\mathcal{H}} % polarization of the target

\newcommand{\ocM}{\overline{\mathcal{M}}} % characteristic sheaf
\newcommand{\ocP}{\overline{\mathcal{P}}} % characteristic sheaf
\newcommand{\ocN}{\overline{\mathcal{N}}} % characteristic sheaf

\newcommand{\cM}{\mathcal M} % Log structure
\newcommand{\cN}{\mathcal N} % Log structure
\newcommand{\cP}{\mathcal P} % log structure

\newcommand{\cK}{\mathcal K} % log ideal

\newcommand{\lcm}{\operatorname{lcm}} % Least common multiple

\newcommand{\NN}{\mathbb{N}} % natural number
\newcommand{\QQ}{\mathbb{Q}} % rational number
\newcommand{\ZZ}{\mathbb{Z}} % integers
\newcommand{\CC}{\mathbb C}
\renewcommand{\AA}{\mathbb A}
\newcommand{\bk}{\mathbf{k}}

\newcommand{\fP}{\mathfrak{P}} % target of log/punctured R-maps

\newcommand{\punt}{\infty} % target of punctured R-maps

\newcommand{\ainfty}{\infty_{\cA}} % closed substack of \cA

\newcommand{\spint}{\mathfrak{X}} % target of spin structure

\newcommand{\ogamma}{\overline{\gamma}} % component of the rigidified cyclotomic gerbe

\newcommand{\fM}{\mathfrak{M}} % Artin stack
\newcommand{\fH}{\mathfrak{H}} % Artin stack
\newcommand{\fMtw}{\mathfrak{M}^{\mathrm{tw}}} % moduli of twisted curves

\newcommand{\fX}{\mathfrak{X}} % an r-spin Artin stack

\newcommand{\UM}{\fM^{\curlywedge}} % Artin stack of UMD log maps

\newcommand{\Um}{\fM^{{\curlyvee}}} % Artin stack of U min Degeneracy log maps

\newcommand{\UMm}{\fM^{{\diamondsuit}}} % Artin stack of U max-min D log maps

\newcommand{\ud}{\widehat{\diamond}}

\newcommand{\AMm}{\fM^{{\ud}}} % Artin stack of Aligned max min D log maps

\newcommand{\AM}{\fM^{{\shortuparrow}}} % Artin stack of Aligned D log maps

\newcommand{\TAM}{\fM^{{\tilde{\shortuparrow}}}} % Artin stack of twisted Aligned D log maps

\newcommand{\RMm}{\SH^{{\diamondsuit}}} % Artin stack of U max-min D log R maps

\newcommand{\RAMm}{\SH^{{\ud}}} % Artin stack of Aligned max min D log R maps

\newcommand{\RA}{\SH^{{\shortuparrow}}} %  stack of Aligned D  log R maps

\newcommand{\RTA}{\SH^{{\tilde{\shortuparrow}}}} % Artin stack of twisted Aligned D log R maps

\newcommand{\obA}{\overline{\mathbf{A}}} % aligned elements

\newcommand{\obD}{\overline{\mathbf{D}}} % sheaf of degeneracies

\newcommand{\ocK}{\overline{\mathcal{K}}} % characteristic ideal sheaf

\newcommand{\ali}{\operatorname{align}} % alignment morphism

\newcommand{\fC}{\mathfrak{C}} % universal curves over Artin stack
\newcommand{\zero}{\mathbf{0}}

\newcommand{\uC}{\underline{\mathcal{C}}} % underlying curve

\newcommand{\pC}{\mathcal{C}^{\circ}} % punctured curve

\newcommand{\uS}{\underline{S}} % base of underlying curve

\newcommand{\cL}{\mathcal L}
\newcommand{\Sp}{\mathcal{L}_{\spint}} % Spin bundle

\newcommand{\Log}{\operatorname{Log}} %Olsson's log stack
\newcommand{\cA}{\mathcal{A}} % Moduli of DF(1) log-structures

\newcommand{\cX}{\mathcal{X}} % some Artin stack

\newcommand{\cC}{\mathcal{C}} % log curve

\newcommand{\cZ}{\mathcal{Z}}

\newcommand\scrM{\mathscr{M}} % DM stack

\newcommand{\oM}{\overline{\mathscr M}} % DM moduli of curves

\newcommand{\SH}{\mathscr{R}}
\newcommand{\UH}{\mathscr{R}^{\curlywedge}} % the stack of stable hybrid maps with uniform max degeneracy

\newcommand{\TT}{\mathbb T} % tangent complex
\newcommand{\EE}{\mathbb{E}} % perfect obstruction theory
\newcommand{\LL}{\mathbb L} % cotangent complex

\newcommand{\KK}{\mathbb K} % boundary kernel complex

\newcommand{\FF}{\mathbb{F}} % some complex

\newcommand{\bF}{\mathbf F} % for forgetful morphism

\newcommand{\Fm}{\mathbf{F}_{\mathbf 1}} % for forgetful morphism

\newcommand{\bs}{\mathbf{s}} % underling sections

\newcommand{\fb}{\mathfrak{b}} % morphism induced by maximal degeneracy

\newcommand{\fp}{\mathfrak{p}} % stack bundle projection

\newcommand{\ft}{\mathfrak{t}} % arrow to the DM-stack

\newcommand{\fr}{\mathfrak{r}} % resolution morphism

\newcommand{\pun}{\mathrm{p}}

\newcommand{\Bl}{\operatorname{Bl}} % log blow-ups

\DeclareMathOperator{\Tw}{Tw} % twiste of alignments

\newcommand{\cO}{\mathcal O} %Structure sheaf

\newcommand{\cT}{\mathcal T} % C^*-torsor

\newcommand{\PP}{\mathbb{P}}

\newcommand{\A}{\mathbb{A}}  % affine space

\newcommand{\GG}{\mathbb{G}}

\newcommand{\Div}{\operatorname{Div}} % splitting morphism

\newcommand{\vir}{\mathrm{vir}}
\newcommand{\red}{\mathrm{red}}

\newcommand{\fine}{\mathrm{fine}}

\newcommand{\rig}{\mathrm{rig}}
\newcommand{\stab}{\mathrm{stab}}

\newcommand{\bw}{\mathbf w}

\DeclareMathOperator{\Aut}{Aut}

\DeclareMathOperator{\spec}{Spec}

\DeclareMathOperator{\diff}{d}

\DeclareMathOperator{\rk}{rk}
\DeclareMathOperator{\age}{age}
\DeclareMathOperator{\vb}{Vb}
\DeclareMathOperator{\DF}{DF}

\newcommand{\poleq}{\preccurlyeq} % partial ordering from log leq
\newtheorem{proposition}{Proposition}[section]
\newtheorem{corollary}[proposition]{Corollary}
\newtheorem{lemma}[proposition]{Lemma}
\newtheorem{conjecture}[proposition]{Conjecture}
\newtheorem{theorem}[proposition]{Theorem}
\newtheorem{question}[proposition]{Question}
\theoremstyle{definition}
\newtheorem{assumption}[proposition]{Assumption}
\newtheorem{definition}[proposition]{Definition}

\newtheorem{notation}[proposition]{Notation}
\newtheorem{example}[proposition]{Example}

\theoremstyle{remark}
\newtheorem{remark}[proposition]{Remark}

\numberwithin{equation}{section}

\newcommand{\ChFelix}[1]{{\color{blue} #1}}

\begin{document}

\title{Punctured logarithmic R-maps}
\author{Qile Chen \and Felix Janda \and Yongbin Ruan}

\address[Q. Chen]{Department of Mathematics\\
Boston College\\
Chestnut Hill, MA 02467\\
U.S.A.}
\email{qile.chen@bc.edu}

\address[F. Janda]{Department of Mathematics\\
University of Notre Dame\\
Notre Dame, IN 46556\\
U.S.A.}
\email{fjanda@nd.edu}

\address[Y. Ruan]{Institute for Advanced Study in Mathematics\\
Zhejiang University\\
Hangzhou, China}
\email{ruanyb@zju.edu.cn}

\date{\today}

\begin{abstract}
  In this paper, we develop the theory of punctured R-maps as a crucial ingredient
towards the foundations
  of logarithmic gauged linear sigma models (log GLSM). 
  A punctured R-map is a punctured map in the sense of Abramovich--Chen--Gross--Siebert which is further twisted by the sheaf of differentials on the domain curve. 
  They admit two different but closely related perfect obstruction theories --- the {\em canonical theory} and the {\em reduced theory}.
  While the canonical theory leads to generalized double ramification (DR) cycles with targets and spin structures, {\em without employing expansions}, the reduced theory naturally describes boundary contributions in log GLSM. 
  
  The unexpanded nature and the flexibility of punctured maps allow us
  to establish a sequence of useful formulas in both canonical and
  reduced theories, including:
  \begin{enumerate} 
  \item A product formula computing disconnected invariants in terms of a product of connected invariants,
  \item Fundamental class axioms, string equations and divisor equations analogous to those in Gromov-Witten theory.
  \end{enumerate}
  Even in the case of the canonical theory, the product formula in (1)
  was {\em unexpected} from the point of view using expansions and
  rubber maps, and is of independent interest.

  As an important application of these formulas in log GLSM, we study
  a class of invariants in the reduced theory, which we coin
  \emph{effective invariants}.
  They are at the heart of recent advances in Gromov--Witten theory,
  and will be shown to give rise to explicit correction terms to the
  quantum Lefschetz principle in higher genus Gromov-Witten theory for
  arbitrary smooth complete intersections in a forthcoming paper.

  In the famous case of quintic $3$-folds, we show that all effective
  invariants are determined by $\lfloor\frac{2g-2}{5}\rfloor + 1$ many
  basic effective invariants, in an explicit way via the above
  mentioned formulas. 
  This matches the number of free parameters of the famous BCOV
  B-model theory in physics.
  Similar results apply to other Calabi--Yau threefold complete
  intersections.
  This, together with the joint works of the last two authors and
  S. Guo on the genus two mirror theorem and the higher genus mirror
  conjectures for quintic threefolds, shows that log GLSM is an
  effective tool for proving BCOV-type conjectures in higher genus
  Gromov--Witten theory.

  Further applications of punctured R-maps include an LG/CY
  correspondence for effective invariants, and a relation to the locus
  of holomorphic differentials.
  The results of this paper form the foundation for major progress
  toward the higher genus LG/CY correspondence and other important
  conjectures in the subject.

\end{abstract}



\maketitle

\setcounter{tocdepth}{1}
\tableofcontents

\section{Introduction}

\subsection{Overview}

\emph{Gromov--Witten theory}, which virtually counts holomorphic curves
in a given projective variety with prescribed incidence conditions,
plays a central role in \emph{mirror symmetry}.
Despite several decades of fruitful developments, effective
calculation of Gromov--Witten invariants generally remains a central
but difficult problem in this area.

\subsubsection{Logarithmic gauged linear sigma models}

This is the third in a series of papers \cite{CJRS18P, CJR21} to build
the theory of \emph{the logarithmic Gauged Linear Sigma Model} (log
GLSM) and its applications, including computing Gromov--Witten
invariants of smooth complete intersections.
The mathematical theory of Gauged Linear Sigma Models (GLSM) has been
developed and studied intensively by many researchers in the past
decade, see for instance \cite{FJR18, CFGKS18P, FaKi20P, TiXu18P}.

Among others, GLSM was proved to (up to signs) recover Gromov-Witten
invariants of smooth complete intersections \cite{ChLi12, KiOh22,
  ChLi20, Pi20P, CJW21}.
However, two important problems remained unsolved --- (1) computing
GLSM invariants and (2) a (higher genus) quantum Lefschetz principle
relating the invariants with those of its ambient spaces.
The current paper is a crucial step of our project toward a solution
to both questions via a logarithmic compactification.%
\footnote{Some other recent approaches to the problem of extending
  quantum Lefschetz to higher genus are the theory of mixed spin
  $p$-fields \cite{CGL18Pb, CGL21, CGLL21, CLLL19, CLLL22}, the
  work of Fan--Lee \cite{FaLe19}, which are both of recursive nature.
  Another potential approach using resolutions is outlined in
  \cite{HuNi20P, LeLiOh22P}.}
Indeed, a major obstacle is the
non-properness of the relevant moduli stacks involved in defining GLSM
invariants.
In \cite{CJR21} we overcame the non-properness problem by first
providing a logarithmic compactification of the moduli stack using
stable log maps of \cite{AbCh14, Ch14, GrSi13}, and then constructing
a reduced perfect obstruction theory that recovers the usual GLSM
virtual cycles, and in particular the Gromov--Witten theory of
complete intersections in \emph{arbitrary} smooth, projective
varieties.
Our approach is not restricted to Gromov--Witten theory, and also
applies to other enumerative theories, such as FJRW theory.

\subsubsection{Punctured R-maps}

This paper develops a Gromov--Witten type theory of {\em punctured
  R-maps} using the punctured logarithmic maps of
Abramovich--Chen--Gross--Siebert \cite{ACGS20P}.
This new theory, as shown in the forthcoming paper \cite{CJR23P},
naturally describes the boundary contributions of log GLSM.

Punctured maps are analogs of the rubber maps that appear in relative
Gromov--Witten theory but defined \emph{without the expansions} of
\cite{GrVa05, Li02}.
The \emph{R-map} structure is a further twist of punctured maps using
(roots of) differentials on the domain curves.
Thus, punctured R-maps as studied in this paper are a generalization
of rubber maps by allowing both targets and twists by spin structures.

Fixing appropriate numerical data, stable punctured R-maps form a
proper moduli stack admitting a \emph{canonical perfect obstruction
  theory}, which leads to a version of double ramification cycles in
our general setting.%
\footnote{We refer to \cite{JPPZ17, JPPZ20, Gu16P, MaWi20, Ho21, BHPSS20P}
  for some of the developments around the double ramification cycle
  and its generalizations.}
However, the contribution of punctured R-maps to log GLSM is given by
a different but closely related \emph{reduced perfect obstruction
  theory}, which is constructed in this paper.

A \emph{major} part of this paper is to establish a sequence of
powerful axioms for punctured R-maps on the virtual cycle level which
will serve as a foundation for effective high genus computations. 
They include a logarithmic root comparison formula, a product formula
that computes invariants with disconnected domains via connected
invariants, and analogs of the fundamental class axiom, string and
divisor equations from Gromov--Witten theory in the setting of
punctured R-maps.
The unexpanded nature and the flexibility of punctured maps
\cite{ACGS20P} provide the key to these explicit formulas.
The following applications illustrate the versatility of the axioms
established in this paper.

\subsubsection{Applications to quantum Lefschetz in arbitrary genus}

The genus zero quantum Lefschetz principle states that Gromov--Witten
invariants of a convex complete intersection $Z$ agree with the
corresponding twisted Gromov--Witten invariant of the ambient space $X$.
This is of essential for many computations in genus zero since twisted
Gromov--Witten invariants of toric varieties are accessible to
localization techniques, and in particular, this is the geometric
input to the famous proof of the genus zero mirror conjecture for
quintic threefolds \cite{Gi96, LLY97}.

The naive generalization of the genus zero quantum Lefschetz principle
to higher genus fails, in the sense that higher genus twisted
Gromov--Witten invariants of $X$ in general do not agree with the
invariants of $Z$.
This is one of the main motivations for this project.
In the next paper \cite{CJR23P}, we will establish explicit structural
formulas computing GLSM invariants by studying the equivariant and
tropical geometry of log GLSM.
In particular, these formulas describe the difference between the
higher genus invariants of $Z$ and the corresponding twisted
Gromov--Witten invariants of $X$ in terms of an explicit correction
which is a graph sum of products of twisted Gromov--Witten invariants
of $X$ and a new ingredient, which we coin \emph{effective invariants}
(of the pair $(X, Z)$).
A major aspect of this work is to construct and study this new
ingredient as an integral over the reduced virtual cycle of punctured
R-maps.

In particular, in \S\ref{sec:effective}, we use axioms of punctured
R-maps to prove powerful vanishing and boundedness results for
effective invariants for ample complete intersections $Z \subset X$.
In this case, all genus $0$ effective invariants vanish, and this
means that there is no correction to the genus zero quantum Lefschetz
principle.
In \S\ref{ss:g=1-calculation}, we perform an explicit computation of
all genus one effective invariants, and thus log GLSM leads to a new
approach to computing genus one Gromov--Witten invariants of $Z$,
different from the other approaches we are aware of \cite{Zi08,
  KiLh18, CGLZ20}.
For each fixed genus $g \ge 2$, we prove that there is a finite number
of non-vanishing effective invariants.
If one is able to compute this finite number of effective invariants
for $g \le g_0$, one can compute the \emph{all} (ambient cohomology)
genus $g_0$ Gromov--Witten invariants of $Z$ in an efficient fashion.
All in all, log GLSM provides a unified approach to computation of
genus zero and higher genus Gromov--Witten invariants.

The efficiency of our approach is further supported by a growing list of
important applications --- higher genus mirror symmetry of quintic
threefolds \cite{GJR17P, GJR18P}, as well as the forthcoming works on
higher genus mirror symmetry for other complete intersections in
projective space, the all-genus LG/CY correspondence for quintic
threefolds \cite{GJR23P}, and to Gromov--Witten classes \cite{CJR23P}.

\subsubsection{Applications to strata of differentials}

In a different direction, the moduli space of punctured R-maps may be
applied to the study of strata of $k$-differentials, which
parameterize projective curves together with a $k$-differential (up to
scaling) with prescribed zeros and poles.
Strata of differentials play a central role in the study of the
geometry and dynamics of flat surfaces \cite{Ch17}.
There has been recent interest in compactifications \cite{BCGGM19P} of
the strata due to their applications to their intersection theory and
to computing geometrical and dynamical quantities such as the
Masur--Veech volumes, Siegel--Veech constants or Lyapunov exponents
\cite{CMSZ20}.
In the work \cite{MaWi20}, a new modular compactification equipped
with a canonical virtual fundamental cycle has been constructed, and
in \cite{CGHMS22P}, it was related to other compactifications.
The moduli of punctured R-maps gives a new modular compactification of
strata of $k$-differentials with both a canonical virtual
fundamental class, and, in the case of abelian differentials
(holomorphic $1$-differentials), a reduced virtual fundamental class.
It would be interesting to explore relations between the moduli of
punctured R-maps and other modular compactifications such as the
moduli of multi-scale differentials of \cite{BCGGM19P}, and the moduli
of logarithmic differentials of \cite{MaWi20}.

The work \cite{FaPa18} started a fuitful interplay between strata of
strictly meromorphic differentials and generalized double ramification
cycles.
In particular, \cite[Conjecture~A]{FaPa18} gives a Pixton-style
(conjectural) formula for the closure of the strata of strictly
meromorphic $1$-differentials in $\oM_{g, n}$, and this conjecture has
been further generalized to $k$-differentials in \cite{Sc18}.
The first proofs of these conjectures have been given in
\cite{BHPSS20P, HoSc21}.
As we will show in \cite{CJPPRSSZ22P}, the canonical virtual cycle of
punctured R-maps recovers the weighted fundamental class of twisted
$k$-canonical divisors of \cite{FaPa18}, and combined with the
structural formulae of log GLSM this yields a new proof of the
conjectures of \cite{FaPa18, Sc18}.

The case of abelian differentials is of major interest.
These strata have higher dimension than the strata of strictly
meromorphic differentials.
They are the subject of a different conjecture,
\cite[Conjecture~A]{PPZ19}, which connects the closure of strata of
abelian differentials to Witten's $r$-spin class.
The development of log GLSM was initiated by a conjectural proof
strategy for this conjecture, and the first foundational work
\cite{CJRS18P} constructs Witten's class as the reduced virtual class
of a moduli of stable log R-maps.
In \cite{CJPPRSSZ22P}, we will show that the reduced virtual cycle of
punctured R-maps recovers the closure of strata of abelian
differentials, and by employing the structural formulae of log GLSM,
complete the proof of \cite[Conjecture~A]{PPZ19}.

\subsection{The canonical theory}

\subsubsection{Punctured R-maps in general}

Just as for logarithmic R-maps, the targets of punctured R-maps are
defined over the classifying stack $\BC := \mathbf{BG}_m$, which has a
universal line bundle $\uomega$.
It is convenient to view $\BC$ as a log stack with the trivial log
structure.
Let $\punt \to \BC$ be a proper, DM-type morphism of log stacks.
A {\em punctured R-map} over a log scheme $S$ with target
$\punt \to \BC$ is a (2-)commutative diagram
\begin{equation}\label{eq:intro-punctured-R-map}
\xymatrix{
 && \punt \ar[d] \\
\pC \ar[rru]^{f} \ar[rr]_{\omega^{\log}_{\cC/S}}&& \BC
}
\end{equation}
where $f\colon \pC \to \punt$ is a punctured map over $S$ in the sense
of \cite{ACGS20P} (see \S \ref{ss:p-curve}), and the bottom arrow is defined by the pull-back $\uomega|_{\pC} = \omega^{\log}_{\cC/S}$, where $\omega^{\log}_{\cC/S}$ is the sheaf of relative log differentials of the log curves $\cC \to S$, or more precisely the dualizing sheaf of the underlying domain curve $\uC \to \uS$ twisted by markings.  {\em Pull-backs} of punctured R-maps are defined as the pull-back of the
corresponding punctured maps \cite[Definition 2.13]{ACGS20P}.

For a punctured R-map in \eqref{eq:intro-punctured-R-map}, the underlying family $\uC \to \uS$ obtained by removing log structures from the source $\pC \to S$ is a family of twisted pre-stable curves. Thus, the underlying $\ul{f} \colon \uC \to \ul{\punt}$ is a family of usual pre-stable maps.

\begin{remark}
  In GLSM, targets carry an additional $\CC_\omega^* := \CC^*$ action
  related to the {\em R-charge} in the physics literature.
  This {\em R-charge} structure translates to the commutative triangle
  \eqref{eq:intro-punctured-R-map}.
  The notion of punctured R-maps is a generalization of log R-maps in
  the sense of \cite[Definition~1.1]{CJR21} with images contained in the divisor
 $\infty$ of a log GLSM target $\fP$.
  This is explained further in Section~\ref{sec:logGLSM-targets}.
\end{remark}

\begin{remark}
  Punctured maps are variants of log maps which allow negative contact
  orders at marked points.
  They are developed in \cite{ACGS20P} to study boundary contributions
  in the log Gromov--Witten theory.
  In our case, punctured R-maps will naturally describe boundary
  contributions in log GLSM.
\end{remark}

\subsubsection{Targets from root constructions}\label{intro:target}

While the canonical and reduced theory of punctured R-maps are
constructed in a more general setting in \S\ref{s:P-map},
\S\ref{sec:red-theory} and \S\ref{sec:reduction-to-connected-inv}, we
will focus here on a more concrete set of targets covering a large
class of interesting applications Consider the following data:
\begin{enumerate}
 \item A smooth, proper DM-stack ${\xinfty}$ with projective coarse moduli.
 \item Two line bundles $\lbspin$ and $\lblog$ on $\xinfty$.
 \item Positive integers $r$, $d$ and $\ell$.
\end{enumerate}
We then form a Cartesian diagram of underlying stacks
\begin{equation}\label{eq:underlying-P-target}
\xymatrix{
\ul{\punt} \ar[rr] \ar[d]_{\uspin} && \xinfty\times \BC \ar[d]^{\lbspin\boxtimes \uomega} \\
\BG_m \ar[rr]_{\mbox{$r$-th power}} && \BG_m
 }
\end{equation}
where the two vertical arrows are induced by the line bundle $\lbspin\boxtimes \uomega$ and its universal $r$-th root $\uspin$, called the {\em  spin structure}.

Consider the log stack $\punt = (\ul{\punt}, \cM_{\punt})$ with the rank one Deligne--Faltings log structure $\cM_{\punt}$ given by the zero section of the line bundle $\cO(\punt) := \lblog\otimes \uspin^{-d}$, see \eqref{eq:generic-rank1-DF}.

Alternatively, consider the log stack
\begin{equation}\label{eq:A}
\cA = ([\A^1\big/\CC^*], \cM_{\cA})
\end{equation}
with the divisorial log structure $\cM_{\cA}$ given by the  origin of $[\A^1\big/\CC^*]$. Let $\ainfty  = (\BG_m, \cM_{\ainfty}) \subset \cA$ be the strict closed substack given by the unique closed point of $\cA$.  Then the  log structure $\cM_{\punt}$ is defined via the strict morphism
\begin{equation}\label{eq:l=1-log}
\xymatrix
{
\punt \ar[rrr]^{\cO(\punt) := \lblog\otimes \uspin^{-d}} &&& \ainfty,
}
\end{equation}
as in \eqref{eq:generic-rank1-DF}. We then form the Cartesian diagram
\begin{equation}\label{eq:g1-t-root}
\xymatrix{
\punt^{1/\ell} \ar[rr] \ar[d]_{\cO(\punt^{1/\ell})} && \punt \ar[d]^{\cO(\infty)} \\
\ainfty^{1/\ell} \ar[rr]_{\text{$\ell$-th power}} && \ainfty
}
\end{equation}
where $\ainfty^{1/\ell} \cong \ainfty$ with the superscript keeping
track of the roots, see \eqref{eq:ell-root-target}.
In particular, we have the universal $\ell$-th root
$\cO(\punt^{1/\ell})^{\ell} \cong \cO(\punt)$.
Our target is
\begin{equation}\label{eq:punt}
\punt^{1/\ell} \to \BC.
\end{equation}
\begin{remark}\label{rem:intro-ell=1}
  As in \cite[Remark~1.7]{CJR21}, different choices of input data can
  lead to the same target $\punt$.
  In fact, it will be shown in Proposition~\ref{prop:one-step-target}
  that for any choice of input data (1), (2), (3), the input data may
  be modified in such a way that $\ell = 1$ without changing the
  target $\infty^{1/\ell}$.
  On the other hand, the flexibility of different choices of $\ell$ is
  crucial for relating different theories such as for the LG/CY
  correspondence, as explained below in \S\ref{sss:intro-change-ell}.
\end{remark}
Because of this remark, unless further specified, we will assume
$\ell = 1$, so that $\punt = \punt^{1/\ell}$.

%\ChFelix{
%\begin{remark}
%  This setup includes an analog of rubber stable maps of \cite[Section 2.4]{GrVa05} \Qile{Cite Jun Li?}
%  and the resulting double ramification cycles of a target variety
%  $\infty_\cX$ as in \cite{JPPZ20}.
%  This is the special case when $r = 1$, $d = 0$, $\ell = 1$,
%  $\lbspin = \cO$, and where $\lblog$ is the line bundle on
%  $\infty_\cX$ defining the rubber $\PP^1$-bundle
%  $\PP_{\infty_\cX}(\lblog \oplus \cO)$.
%
%  For the ordinary double ramification cycles of \cite{JPPZ17}, we
%  further set $\infty_\cX$ to be a geometric point and $\lblog = \cO$.
%\end{remark}
%\begin{remark}
%  The special case of punctured R-maps that is closely related to
%  strata of $k$-differentials is when $\infty_\cX$ is a geometric
%  point, $r = 1$, $d = k$, $\ell = 1$ and $\lbspin$ and $\lblog$
%  are trivial.
%\end{remark}
%}

\begin{remark}\label{rem:relation-rubber-maps} 
  Consider the case $r = 1$, $d = 0$, $\ell = 1$, $\lbspin = \cO$, and
  that $\infty_\cX$ is a smooth projective variety.
  Then, punctured R-maps to the corresponding target $\infty$ are an
  \emph{unexpanded} analog of the rubber stable maps of
  \cite[Section~2.4]{GrVa05} to the $\PP^1$-bundle
  $\PP_{\infty_\cX}(\lblog \oplus \cO)$.
\end{remark}
\begin{remark}
  Consider the case $r = 1$, $d = k$, $\ell = 1$ and $\infty_\cX$ is a
  geometric point.
  Then $\lbspin$ and $\lblog$ are automatically trivial.  
  In this case, punctured R-maps are an analog of the $k$-canonical
  divisors and multi-scale differentials studied in \cite{FaPa18,
    Sc18, BCGGM19P}, but without expansions/level structures.
  An explicit example in the $g=1$ case is studied in
  Section~\ref{ss:g=1-calculation}.
\end{remark}

\subsubsection{The stability}
For a target $\punt \to \BC$ as in \S \ref{intro:target}, a punctured R-map $f \colon \pC \to \punt$ over $S$ is {\em stable} if
\begin{enumerate}
 \item The underlying morphism $\ul{f} \colon \uC \to \ul{\punt}$ of $f$ is representable,
 \item The coarse moduli of the composition $\uC \to \ul{\punt} \to \xinfty$ is stable in the usual sense.
\end{enumerate}
This stability condition is entirely determined by the underlying map
$\ul{f}$.

\subsubsection{The discrete data}
We will observe in \eqref{eq:restrict-to-marking} that
the restriction $\ul{f}|_{p_i}$ to the $i$-th marking $p_i \subset \uC$ factors through the projection \[
\ul{\punt_{\bk}} := \ul{\punt}\times_{\BC}\spec \bk  \to \ul{\punt}
\]
where $\spec \bk \to \BC$ is the  universal torsor. This defines the {\em $i$th evaluation morphism}
\begin{equation}\label{eq:intro-evaluation}
\ev_i \colon \ul{S} \to  \ocI\punt_\bk
\end{equation}
with $\ocI\punt_\bk$ the rigidified cyclotomic inertia stack of $\punt_\bk$, see \cite[Section~3]{AGV08}.

A punctured R-map $f \colon \pC \to \punt$ over $S$ is said to have discrete data
\begin{equation}\label{eq:intro-P-data}
(g, \beta, \ddata = \{(\ogamma_i, c_i)\}_{i=1}^n)
\end{equation}
if $\uC \to \ul{S}$ is a genus $g$ twisted curve with $n$ marked points, the composition
$\uC \to \ul{\punt} \to \xinfty$ has curve class
$\beta \in H_2(\xinfty)$, the $i$th evaluation $\ev_i $ factors
through a connected component (or \emph{sector})
$\ogamma_i \subset \ocI\punt_\bk$ and the $i$th marking has contact
order $c_i \in \ZZ$ for all $i$.

\subsubsection{The canonical virtual cycle}
Denote by $\SH_{g,\ddata}(\punt,\beta)$ the category of stable
punctured R-maps to $\punt \to \BC$ with the discrete data
\eqref{eq:intro-P-data}.
We prove in \S \ref{s:P-map}:

\begin{theorem}[Theorem \ref{thm:PR-moduli} and \S \ref{ss:canonical-theory}]\label{thm:intro-canonical-theory}
$\SH_{g,\ddata}(\punt,\beta)$ is a proper log Deligne--Mumford stack equipped with a {\em canonical perfect obstruction theory}, hence a {\em canonical virtual cycle} $[\SH_{g,\ddata}(\punt,\beta)]^{\vir}$.
\end{theorem}

\begin{remark}
  Pushing forward the canonical virtual cycle in the setup of
  Remark~\ref{rem:relation-rubber-maps} to the moduli space of stable
  maps to $\infty_\cX$ is expected to concide with the double
  ramification cycles with the target $\infty_\cX$ as in \cite{JPPZ20}.

  Further specializing $\infty_\cX$ to a geometric point, we obtain
  the original double ramification cycle \cite{JPPZ17} on the moduli
  of curves.
\end{remark}
\begin{remark}
  In general, we expect that pushing the canonical virtual cycle to
  the moduli of ordinary R-maps to $\ul\infty$ yields the pullback of
  an orbifold generalization of the twisted double ramification cycle
  of \cite{BHPSS20P}.
  In particular, we expect this to lead to a Pixton-type formula for
  the canonical virtual cycle, and such a formula might also have
  interesting applications to the Gromov--Witten theory of complete
  intersections.
\end{remark}

\subsection{The reduced theory}

Our main interest lies in the reduced theory introduced next, since in applications,
it determines the Gromov--Witten theory of complete
intersections (with ambient insertions).

\subsubsection{Further input}\label{sss:intro-condition-for-red}
To construct  the reduced theory, we further require the following:

\begin{enumerate}
 \item All contact orders in \eqref{eq:intro-P-data} are negative.
 \item A {\em transverse superpotential} $W$.
 \item A proper morphism $F \colon \UH_{g,\ddata}(\punt,\beta) \to \SH_{g,\ddata}(\punt,\beta)$ where $\UH_{g,\ddata}(\punt,\beta)$ is the stack of stable punctured R-maps with {\em uniform maximal degeneracy}.
\end{enumerate}
While (1) is clear, we explain (2) and (3) below.

\subsubsection{Superpotential}

A superpotential is a section $W$ of the line bundle
$\uomega \otimes \cO(\ttwist \infty)$ over $\punt$ for some positive integer
$\ttwist$ called the {\em order} of $W$.
It is said to be \emph{transverse (to the zero section)} if its
critical locus is empty, see \S\ref{sss:critical-loci-transverse}.

As shown in \S \ref{ss:GLSM-potential}, the notion of superpotential
is compatible with the superpotential of log GLSM
\cite[\S3.4.1]{CJR21} along logarithmic boundaries.

\subsubsection{Punctured R-maps with uniform extremal degeneracies}\label{sss:intro-extremal-degeneracy}
Let $f \colon \pC \to  \punt$ be a punctured R-map over a logarithmic point $S$. For each irreducible component $Z \subset \pC$, there is an element $e_{Z} \in \ocM_{S}$ in the characteristic sheaf of the log structure $\cM_S$,  which measures the ``speed'' of degeneration of $Z$ into the logarithmic boundary of the target, which is called the {\em degeneracy} of $Z$, see \S \ref{sss:degeneracies} for the precise definition.

The punctured R-map is said to have {\em uniform maximal degeneracy}
(resp. uniform minimal degeneracy) if the set of degeneracies has a
unique maximum (resp. unique minimum) under the natural monoid
ordering of $\ocM_S$.
Intuitively, this amounts to require all maximally (resp. minimally)
degenerated components have the same ``speed'' of degeneration.

On the category level, the tautological morphism
$F\colon \UH_{g,\ddata}(\punt,\beta) \to \SH_{g,\ddata}(\punt,\beta)$
exhibits $\UH_{g,\ddata}(\punt,\beta)$ as the subcategory of
$\SH_{g,\ddata}(\punt,\beta)$ parameterizing stable punctured R-maps
with {\em uniform maximal degeneracy}.
On the stack level, similarly to \cite{CJR21} this tautological
morphism can be viewed as a {\em logarithmic pricipalization} such
that the fiberwise maximal degeneracies glues to a global section of
the characteristic sheaf
$e_{\max} \in \Gamma(\ocM_{\UH_{g,\ddata}(\punt,\beta)})$.
As explained in \S\ref{sss:log-line-bundles}, this global section
induces a tautological line bundle $\cO(e_{\max})$ which, as shown
below, measures precisely the difference between the canonical and
reduced theory over $\UH_{g,\ddata}(\punt,\beta)$.%
\footnote{The line bundle in \cite{CJR21} analogous to $\cO(e_{\max})$ is $\mathbf L_{\max}^\vee$. }

\subsubsection{Reduced versus Canonical}

The stack $\UH_{g,\ddata}(\punt,\beta)$ admits a canonical perfect
obstruction theory by pulled back from $\SH_{g,\ddata}(\punt,\beta)$,
hence the canonical virtual cycle
$[\UH_{g,\ddata}(\punt,\beta)]^{\vir}$ with the push-forward property
\[
F_*[\UH_{g,\ddata}(\punt,\beta)]^{\vir} = [\SH_{g,\ddata}(\punt,\beta)]^{\vir}.
\]

In \S \ref{sec:red-theory}, we establish the reduced  theory and its relations to the canonical theory as follows.

\begin{theorem}[Theorem \ref{thm:red-POT}, Corollary \ref{cor:red-canonical-vir}]\label{thm:intro-red-theory}
  Further assuming the conditions in
  \S\ref{sss:intro-condition-for-red}, the stack
  $\UH_{g,\ddata}(\punt,\beta)$ admits a {\em reduced perfect
    obstruction theory} hence a {\em reduced virtual cycle}
  $[\UH_{g,\ddata}(\punt,\beta)]^{\red}$ satisfying
\[
[\UH_{g,\ddata}(\punt,\beta)]^{\vir} = \ttwist c_1 \Big(\cO(e_{\max}) \Big) \cdot [\UH_{g,\ddata}(\punt,\beta)]^{\red}.
\]
\end{theorem}
In particular, the reduced virtual dimension is one higher than the
canonical virtual dimension.
This is analogous to what happens for the Gromov--Witten theory of K3
surfaces, but is unlike the reduced virtual cycle considered in
\cite{CJR21} where the virtual dimension remains the same after reduction.

\subsection{Relating connected and disconnected theories}

Punctured R-maps with disconnected domains naturally appear in the
localization formula for log R-maps \cite{CJR23P}.
This is similar to the relative localization of \cite{GrVa05}.
{To obtain a more effective localization formula, it is necessary to
reduce disconnected invariants to connected ones.
In \S\ref{sec:reduction-to-connected-inv}, we establish such a
\emph{product formula} on the virtual cycle level, where the
flexibility of punctured maps of \cite{ACGS20P} plays a key role.
In the analogous setting of rubber stable maps, such a product formula
has been unexpected, as stated in \cite[Section~1.5]{MaPa06}:
``There is no product rule relating connected and disconnected rubber invariants.''}

\subsubsection{Further modifications}

Consider a collection of discrete data of punctured R-maps
\begin{equation}\label{eq:intro-n-P-data}
\{(g_j, \beta_j, \ddata_j) \}_{j \in [m]}
\end{equation}
where $[m] := \{1,2, \cdots, m\}$. We introduce the following variations of the stacks.

\smallskip

\noindent
$\bullet$ $\RMm_{g_j,\ddata_j}(\punt,\beta_j) \colon$ the stack of stable punctured R-maps with both uniform maximal and uniform minimal degeneracies.

\noindent
$\bullet$ $\RA_{[m]} \colon$ the stack of stable punctured R-maps such that
\begin{enumerate}
 \item Their domains have $m$ connected components labeled by $[m]$, each with discrete data \eqref{eq:intro-n-P-data}.
 \item The set of degeneracies are ordered, in which case the punctured R-maps are said to be {\em aligned}.
\end{enumerate}

In \S \ref{sec:reduction-to-connected-inv}, the stack $\RMm_{g_j,\ddata_j}(\punt,\beta_j)$ and $\RA_{[m]}$ are shown to be proper log Deligne--Mumford stacks with the tautological morphisms
\[
F^{\curlywedge} \colon \RMm_{g_j,\ddata_j}(\punt,\beta_j) \to \UH_{g_j,\ddata_j}(\punt,\beta_j) \ \ \mbox{and} \ \ F^{\diamondsuit}_{[m]} \colon \RA_{[m]} \to \prod_{j \in [m]} \RMm_{g_j,\ddata_j}(\punt,\beta_j).
\]
As in \S \ref{sss:intro-extremal-degeneracy}, we have global sections
\[
e_{\max,j} \in \Gamma\left(\ocM_{\RMm_{g_j,\ddata_j}(\punt,\beta_j)} \right) \ \ \mbox{and} \ \ e_{\max} \in \Gamma\left(\ocM_{\RA_{[m]} } \right)
\]
which are fiberwise the maximal degeneracies of the corresponding punctured R-maps.

Similarly, the condition of uniform minimal degeneracies is also a logarithmic principalization (see \S \ref{sss:p-UM-connected-map}), hence leading to global sections
\[
e_{\min,j} \in \Gamma\left(\ocM_{\RMm_{g_j,\ddata_j}(\punt,\beta_j)} \right) \ \ \mbox{and} \ \ e_{\min} \in \Gamma\left(\ocM_{\RA_{[m]} } \right)
\]
which are fiberwise the minimal degeneracies of the corresponding
punctured R-maps.
Their corresponding line bundles define the tautological classes (see
\eqref{eq:psi-min-class})
\[
\psi_{\min,j} := - c_1\left( \cO(e_{\min,j})\right) \ \ \mbox{and} \ \  \psi_{\min} := - c_1\left( \cO(e_{\min})\right).
\]

\begin{remark}
These  divisor classes $\psi_{\min,j}$ and $\psi_{\min}$ appear naturally in the log GLSM localization formula of  \cite{CJR23P}. In other  words, the subcategory of punctured maps with uniform minimal degeneracy is the suitable category in which the virtual localization formula holds.
\end{remark}

\begin{remark}\label{rem:intro-RA}
One may consider the stack $\UH_{[m]}$ of stable punctured R-maps with disconnected domains, which have uniform maximal degeneracies rather than aligned, see \S \ref{sss:curlywedge-R-map-moduli}. As aligned punctured R-maps automatically have uniform maximal degeneracies, there is a tautological morphism \eqref{diag:removing-orders}:
\[
F^{\curlywedge} \colon \RA_{[m]} \to \UH_{[m]}.
\]

In \S \ref{sec:red-theory}, we will construct a natural reduced perfect obstruction theory of $\UH_{[m]}$, which pulls back to the natural reduced perfect obstruction theory of $\RA_{[m]}$, hence defines the natural reduced virtual cycles $[\RA_{[m]}]^{\red}$. While $\UH_{[m]}$ is the natural category for constructing the reduced perfect obstruction theory, it is important to work with the virtual cycle $[\RA_{[m]}]^{\red}$ for two reasons.

First, there is no morphism from $\UH_{[m]}$ to $\RMm_{g_j,\ddata_j}(\punt,\beta_j)$ or  to $\UH_{g_j,\ddata_j}(\punt,\beta_j)$, as each connected component may not have uniform maximal degeneracy. This is crucially needed in the comparison of disconnected and connected invariants. Second, as shown in \cite{CJR23P} the stack $\RA_{[m]}$ naturally appears in the localization calculation of log GLSM.
\end{remark}

\subsubsection{Reduction to the connected theory}

Pulling back from $\UH_{g_j,\ddata_j}(\punt,\beta_j)$,
$\RMm_{g_j,\ddata_j}(\punt,\beta_j)$ admits the canonical perfect
obstruction theory, as well as the reduced one if we further assume
\S\ref{sss:intro-condition-for-red}.
We obtain the corresponding virtual cycles
$[\RMm_{g_j,\ddata_j}(\punt,\beta_j)]^{\vir}$ and
$[\RMm_{g_j,\ddata_j}(\punt,\beta_j)]^{\red}$ respectively.
These variants do not change the virtual geometry in the sense that
they satisfy the virtual push-forwards
(Proposition~\ref{prop:birational-canonical-vcycle} and
Proposition~\ref{prop:birational-red-vcycle})
\[
F^{\curlywedge}_*[\RMm_{g_j,\ddata_j}(\punt,\beta_j)]^{\vir} = [\UH_{g_j,\ddata_j}(\punt,\beta_j)]^{\vir} \ \ \mbox{and} \ \ F^{\curlywedge}_*[\RMm_{g_j,\ddata_j}(\punt,\beta_j)]^{\red} = [\UH_{g_j,\ddata_j}(\punt,\beta_j)]^{\red}.
\]

Let $[\RA_{[m]}]^{\vir}$ be the canonical virtual cycle of $\RA_{[m]}$ defined by pulling back the canonical perfect obstruction theory along $F^{\diamondsuit}_{[m]}$.  Let $[\RA_{[m]}]^{\red}$ be the reduced virtual cycle as in Remark \ref{rem:intro-RA}. The disconnected and connected theory are related beautifully as follows, see Theorem \ref{thm:virtual-split-disconnected-source} and \eqref{eq:push-forward-tw-align-pun}.

\begin{theorem}\label{thm:intro-disconnected-to-connected}
  Let $t$ be a formal parameter. We have the canonical virtual push-forward
  \begin{equation}
    \label{eq:intro-diconnected-to-connected-canonical}
    F^{\diamondsuit}_{[m],*}\left(\frac{[\RA_{[m]}]^{\vir}}{-t - \psi_{\min}} \right) = \prod_{j \in [m]} \frac{[\RMm_{g_j,\ddata_j}(\punt,\beta_j)]^\vir}{-t - \psi_{\min, j}}.
  \end{equation}
Further assuming \S \ref{sss:intro-condition-for-red}, we have the reduced virtual push-forward
\[
F^{\diamondsuit}_{[m],*}\left( \frac{\ttwist t \cdot [\RA_{[m]}]^{\red}}{-t - \psi_{\min}} \right) = \prod_{j \in [m]}\frac{\ttwist t\cdot [\RMm_{g_j,\ddata_j}(\punt,\beta_j)]^{\red}}{-t - \psi_{\min, j}}.
\]
\end{theorem}

The formal parameter $t$ and the denominator $-t - \psi_{\min}$ will naturally appear in the log GLSM localization formula of
\cite{CJR23P}.
Their analogs are also part of the relative virtual localization formula \cite{GrVa05} of Graber--Vakil.
Thanks to the above reduction, we can and will now focus on the connected theory
for the rest of this introduction.

\begin{remark}
  We expect that the analog of
  \eqref{eq:intro-diconnected-to-connected-canonical} also holds for
  the rubber maps in relative Gromov--Witten theory (with the
  $\psi$-class at one end of the universal expansion playing the role
  of $\psi_{\min}$).
  We are not aware that this property has been observed in the
  literature on rubber maps.
\end{remark}

\subsection{Changing roots}
\label{sss:intro-change-ell}

Punctured R-maps are relatively well-behaved when varying $\ell$ in \eqref{eq:g1-t-root}. Recall the notation $\punt = \punt^{1/1}$.
On the stack level, there is a canonical morphism
\[
\nu_{\ell} \colon \SH^{\bullet}_{g,\ddata}(\punt^{1/\ell}, \beta) \to \SH^{\bullet}_{g,\bar{\ddata}}(\punt,\beta)
\]
where $\SH^{\bullet}$ represents any of the configurations $\SH, \UH$
or $\RMm$ as in \eqref{eq:canonical-stack-change-ell},
\eqref{eq:reduced-stack-change-ell} and
\eqref{eq:reduced-stack-change-ell-with-m}.
The discrete data $(g, \beta, \ddata)$ and $(g, \beta, \bar{\ddata})$
uniquely determine each other as shown in \S
\ref{sss:discrete-data-change-ell}.

On the cycle level, they are related by the virtual push-forwards (see Proposition \ref{prop:can-vir-change-ell}, Proposition \ref{prop:reduced-change-ell} and Corollary \ref{cor:reduced-change-ell-with-min}):

\begin{theorem}
We have the canonical virtual push-forward
\[
\nu_{\ell, *}[\SH_{g,\ddata}(\punt^{1/\ell}, \beta)]^{\vir} = \ell^{-1} \cdot [\SH_{g,\bar{\ddata}}(\punt,\beta)]^{\vir}.
\]
Further assuming \S\ref{sss:intro-condition-for-red}, we have the
reduced virtual push-forward
\begin{equation}
  \label{eq:intro-reduced-changing-roots}
  \nu_{\ell,*}[\UH_{g,\ddata}(\punt^{1/\ell}, \beta)]^{\red} = \ell^{-1} \cdot [\UH_{g,\bar{\ddata}}(\punt,\beta)]^{\red}.
\end{equation}
Furthermore, we have
\[
\nu_{\ell,*}\left(\frac{t\cdot [\RMm_{g,\ddata}(\punt^{1/\ell}, \beta)]^{\red}}{-t-\psi_{\min, \ddata}}\right) = \ell^{-1} \cdot \frac{(t/\ell)\cdot [\RMm_{g,\bar{\ddata}}(\punt,\beta)]^{\red}}{-(t/\ell)-\psi_{\min, \bar{\ddata}}}
\]
where $t$ is a formal parameter.
\end{theorem}

\begin{remark}
  The LG/CY correspondence (see for example \cite{ChRu10}) relates the
  Gromov--Witten and FJRW theory of quintic threefolds.
  In \S\ref{sss:LGCY}, we observe that the corresponding punctured
  R-map targets differ by changing roots.
  In this setting, we refer to \eqref{eq:intro-reduced-changing-roots}
  as the \emph{geometric LG/CY correspondence}, and it plays an
  important role in the forthcoming work \cite{GJR23P}.
\end{remark}

\subsection{Axioms of removing units}

The forgetful morphism removing a marking plays an important role in
Gromov--Witten theory leading to a sequence of important properties
such as the string, dilaton and divisor equations, and the unit axiom.
We will show that punctured R-maps share similar properties.
As explained in Remark \ref{rem:intro-ell=1}, we may assume $\ell = 1$
for the discussion below without loss of generality.
These axioms have been used crucially in the genus two computation of
\cite{GJR17P}.

\subsubsection{The forgetful morphism}

The {\em unit sector}, defined in \S \ref{sss:unit-sector}, is a
special type of discrete data $\mathbf{1} = (\ogamma, c = -d)$ along markings,
which is the analogue of untwisted markings in orbifold Gromov--Witten
theory.

Denote by $\ddata + \mathbf{1}$ the collection of discrete data by
adding a unit sector to $\ddata$.
As shown in Proposition~\ref{prop:moduli-remove-unit}, there is a
natural forgetful morphism
\[
\bF_{\mathbf{1}} \colon \SH^{\bullet}_{g, \ddata + {\bf 1}}(\punt, \beta) \to \SH^{\bullet}_{g, \ddata}(\punt, \beta)
\]
for the configurations $\SH^{\bullet} = \SH$ or $\UH$ fitting (when
$2g - 2 + n > 0$) in a commutative diagram
\[
\xymatrix{
\SH^{\bullet}_{\ddata+\mathbf{1}} \ar[rr]^{\mathrm{p}_{\ddata+\mathbf{1}}} \ar[d]_{\bF_{\mathbf{1}}} && \oM_{g,n+1} \ar[d]^{\pi} \\
\SH^{\bullet}_{\ddata} \ar[rr]^{\mathrm{p}_{\ddata}} && \oM_{g,n}
}
\]
where $\oM_{g,n}$ is the Deligne--Mumford moduli of stable curves,
$\pi$ is the universal stable curve, and the horizontal arrows are the
tautological morphisms.
Indeed, Proposition \ref{prop:moduli-remove-unit} further exhibits
$\bF_{\mathbf{1}}$ as a composition
\[
\SH^{\bullet}_{g, \ddata + {\bf 1}}(\punt, \beta) \stackrel{\mathrm{S}}{\longrightarrow} \cC^{\bullet,\circ}_{g, \ddata}(\punt, \beta) \stackrel{\pi_{\ddata}}{\longrightarrow}  \SH^{\bullet}_{g, \ddata}(\punt, \beta)
\]
where $\pi_{\ddata}$ is the universal punctured curve over
$\SH^{\bullet}_{g, \ddata}(\punt, \beta)$, and $\mathrm{S}$ is the
{\em saturation morphism} \cite[Proposition 2.1.5]{Og18}.
We have the following analog of the {\em fundamental class axiom} and
the {\em unit axiom}, or alternatively, of the Invariance Property II
for the universal DR cycle of \cite[\S 0.6]{BHPSS20P}, in the setting
of punctured R-maps:

\begin{theorem}[Theorem \ref{thm:remove-unit}]
With the above assumptions, we have
\begin{enumerate}
\item $\mathrm{S}_{*}[\SH^{\bullet}_{g, \ddata + {\bf 1}}(\punt, \beta)]^{\star} = \pi_{\varsigma}^*[\SH^{\bullet}_{g, \ddata}(\punt, \beta)]^{\star}$.

\item $\bF_{\mathbf{1},*}\left(\ev_{n+1}^*D \cap [\SH^{\bullet}_{g, \ddata + {\bf 1}}(\punt, \beta)]^{\star} \right) = \left(\int_{\beta}D \right) \cdot [\SH^{\bullet}_{g, \ddata}(\punt, \beta)]^{\star}$ for a divisor $D \in CH^1(\xinfty)$.

\item if $2g - 2 + n > 0$, then $\mathrm{p}_{\ddata+\mathbf{1},*} [\SH^{\bullet}_{g, \ddata + {\bf 1}}(\punt, \beta)]^{\star} = \pi^{*}\mathrm{p}_{\ddata,*} [\SH^{\bullet}_{g, \ddata}(\punt, \beta)]^{\star}.$

\end{enumerate}
where $[-]^\star$ can be either $[-]^\vir$, or $[-]^\red$ when $\bullet = \curlywedge$ and \S \ref{sss:intro-condition-for-red} is assumed.
\end{theorem}

\subsubsection{Axioms with \texorpdfstring{$\psi_{\min}$}{psi-min}-classes}
The morphism $\bF_{\mathbf{1}}$ does not exist for $\SH^{\bullet} = \RMm$. Thus removing units  with $\psi_{\min}$-classes needs to be treated differently by choosing appropriate categories  of punctured maps. For applications we will  focus on the reduced classes:

\begin{theorem}[Corollary \ref{cor:string-divisor}, Theorem \ref{thm:unit-axiom-with-min}]\label{thm:intro-string-divisor}
Let $t$ be a formal parameter, and assume \S \ref{sss:intro-condition-for-red}. We have

  \begin{equation*}
  {\bf F}_{\mathbf{1},*} \circ F^{\curlywedge}_* \left( \frac{[\RMm_{g, \ddata + {\bf 1}}(\punt, \beta)]^{\red}}{t - \psi_{\min,\varsigma+\mathbf{1}}} \right) =
    \sum_{j=1}^n \frac{|c_j|}{r_j} F^{\curlywedge}_{*} \left( \frac{[\RMm_{g, \ddata}(\punt, \beta)]^{\red}}{(t - \ev_j^*\psi_{\DF})(t - \psi_{\min,\varsigma})} \right) ,
  \end{equation*}
  \begin{multline*}
 {\bf F}_{\mathbf{1},*} \circ F^{\curlywedge}_* \left( \ev_{n+1}^*\big(D\big) \cap \frac{[\RMm_{g, \ddata + {\bf 1}}(\punt, \beta)]^{\red}}{t - \psi_{\min,\varsigma+\mathbf{1}}} \right) = \\
    \big(\int_{\beta}D \big) F^{\curlywedge}_{*} \left( \frac{[\RMm_{g, \ddata}(\punt, \beta)]^{\red}}{t - \psi_{\min,\varsigma}} \right)
    + \sum_{j=1}^n \frac{|c_j|}{r_j} F^{\curlywedge}_{*} \left( \ev_j^*\big(D_{\ocI}\big) \cap \frac{[\RMm_{g, \ddata}(\punt, \beta)]^{\red}}{(t - \ev_j^*\psi_{\DF})(t - \psi_{\min,\varsigma})} \right) ,
  \end{multline*}

  \begin{multline*}
 \mathrm{p}_{\ddata+\mathbf{1},*} \left( \psi_{\min,\varsigma+\mathbf{1}}^{k} \cap [\RMm_{g, \ddata + {\bf 1}}(\punt, \beta)]^{\red} \right) =   \pi^* \mathrm{p}_{\ddata,*} \left(\psi_{\min,\varsigma}^k \cap [\RMm_{g, \ddata}(\punt, \beta)]^{\red} \right)  \\
 +  \sum_{j=1}^n \frac{|c_j|}{r_j} \delta_{j,n+1} \cap \pi^* \mathrm{p}_{\ddata,*} \left( \sum_{k' = 0}^{k-1}  \ev_j^*\big( \psi_{\DF}^{k'} \big) \cdot   \psi_{\min,\varsigma}^{k-1-k'} \cap [\RMm_{g, \ddata}(\punt, \beta)]^{\red} \right)
\end{multline*}
\end{theorem}
We refer to \S \ref{sss:nota-string-divisor} for various notations
involved in the above formulas.
The first and second equations can be viewed as the analogue of string
and divisor equations in the usual Gromov--Witten theory respectively.

\begin{remark}
  One may also consider more complicated situations that
  involve cotangent line classes at the first $n$ markings.
  Since these cases are not involved in the applications we have in mind,
  we will defer these cases to a future work.
\end{remark}

\begin{remark}
  In this paper, we restrict ourselves to axioms of removing unit
  sectors.
  In the future, it will be interesting to see whether the virtual
  cycles of punctured R-maps satisfy other axioms analogous to
  Gromov--Witten theory, such as a splitting formula.
\end{remark}

These formulas above combined with other properties of punctured R-maps are
very powerful in determining punctured R-invariants that are involved
in calculations of Gromov--Witten invariants of complete intersections.
We finish the introduction by discussing applications in the setting
of complete intersections.

\subsection{Applications to complete intersections}\label{ss:intro-application}

\subsubsection{The setup}

Consider the triple $(\cX, \bE, s_{\cX})$ of a smooth projective
variety $\cX$, a vector bundle $\bE$ over $\cX$, and a section
$s_{\cX} \in H^0(\bE)$.
Set
\[
\xinfty = \PP(\bE^{\vee}), \  \lblog = \cO_{\xinfty}(1), \  \lbspin \cong \cO, \  r =d = \ell = 1,
\]
as in \S \ref{intro:target}.
Assume that $\cZ := (s_{\cX} = 0) \subset \cX$ is smooth of
codimension $\rk \bE$.
Then $s_{\cX}$ induces a transverse superpotential.
Further assume \S\ref{sss:intro-condition-for-red} for the reduced theory.

In this situation, the \emph{effective cycles} in the homology
(resp. Chow) group of $\oM_{g, n}$ are cycles of the form
\begin{equation}\label{eq:intro-effective-cycle}
  p_{\varsigma, *} \left( \psi_{\min}^k \cdot \prod \ev_i^*\alpha_i \cap [\RMm_{g, \ddata}(\punt, \beta)]^{\red} \right),
\end{equation}
where we may assume that $\alpha_i \in H^*(\xinfty)$ (resp.
$\alpha_i \in A^*(\xinfty)$) are homogeneous for all $i$.
If they are of dimension zero, we may define the corresponding
\emph{effective invariants}:
\begin{equation}\label{eq:intro-effective-inv}
  \int_{[\RMm_{\varsigma}]^{\red}} \psi_{\min}^k \cdot \prod \ev_i^*\alpha_i := \deg \left(\psi_{\min}^k \cdot \prod \ev_i^*\alpha_i \cap [\RMm_{g, \ddata}(\punt, \beta)]^{\red}\right),
\end{equation}
These effective cycles and invariants are of great importance as they
determine the Gromov--Witten theory (with ambient insertions) of $\cZ$
via an explicit localization formula in \cite{CJR23P}.
The terminology ``effective'' is inspired by Polishchuk's effective
spin structures \cite{Po06}, see \S\ref{sss:r-spin}.

\subsubsection{A naive vanishing}
Let $\beta_{\cX}$ be the push-forward of $\beta$ along $\xinfty \to \cX$, and $\oM_{g,n}\Big(\cZ, \beta_\cX\Big)$ be  the moduli of stable maps to $\cZ$ with the corresponding discrete data. On the virtual dimension level punctured R-maps and Gromov--Witten theory of $\cZ$ are related  by \eqref{eq:red-vir-dim-X}:
\[
\red \dim\RMm_{g, \ddata}(\punt, \beta)  = \vir\dim \oM_{g,n}\Big(\cZ, \beta_\cX\Big) + \rk \bE \cdot \sum_i (c_i + 1).
\]
In particular, the effective cycles and invariants are efficient in the sense that they contain no obviously redundant information to determine the Gromov--Witten invariants of $\cZ$:

\begin{corollary}[Corollary \ref{cor:general-type-vanishing}]
  If $\vir\dim \oM_{g,n}\Big(\cZ, \beta_\cX\Big) < 0$, then
  $ [\RMm_{g, \ddata}(\punt, \beta)]^{\red} = 0$.
\end{corollary}

\subsubsection{Determination in genus zero and one}

In \S\ref{ss:determine-g=0-1} and \S\ref{ss:g=1-calculation}, we determine all $g=0,1$ effective cycles and invariants.

\begin{proposition}
  \begin{enumerate}
  \item If $\bE$ is nef, then all $g=0$ effective cycles and invariants vanish.
  \item If $\bE$ is ample, then all $g=1$ effective cycles and invariants are explicitly determined.
  \end{enumerate}
\end{proposition}

The first statement is a direct consequence of the fact that the
moduli of stable punctured R-maps for $g=0$ are all empty for nef
$\bE$.
This fact is closely related to the genus zero quantum Lefschetz
property.

When $\bE$ is ample, as observed in Lemma~\ref{lem:g=1-effectiveness}
that the moduli of stable punctured R-maps in $g=1$ are all empty
unless $\beta = 0$ and all markings are unit.
Repeatedly removing
unit markings using Theorem~\ref{thm:intro-string-divisor}, the $g=1$ effective
theory is then calculated explicitly by reducing to the base case of $n=1$, which is computed in
Proposition~\ref{eq:g1-reduced-virtual-cycle} .

\subsubsection{Vanishing for \texorpdfstring{$g \geq 2$}{g >= 2}}

A careful virtual dimension analysis combined with
Theorem~\ref{thm:intro-string-divisor} leads to a powerful vanishing
lemma:

\begin{lemma}[Lemma \ref{lem:general-non-vanishing}]
  Suppose that $g \geq 2$ and the following holds
  \begin{equation}
    \label{eq:intro-general-bound-0}
    (3- \dim \cX + \rk \bE)(g-1) - 2\int_{\beta_{\cX}} c_1(K_{\cX}\otimes \det \bE) < 0.
  \end{equation}
  Then, the effective cycles \eqref{eq:intro-effective-cycle} and
  invariants \eqref{eq:intro-effective-inv} vanish unless there is an
  $i$ such that $c_i = -1$ and $\alpha_i \in H^1(\xinfty)$.
\end{lemma}

This vanishing lemma applies to a broad spectrum of interesting
examples, see \S \ref{ss:vanishing}.
In the introduction, we exhibit its efficiency by stating a vanishing
result of effective cycles and invariants for most higher dimensional
Fano complete intersections in projective spaces:

\begin{corollary}[Example \ref{ex:projective-hypersurface}, \ref{ex:projective-complete-intersection}]\label{cor:intro-Fano-complete-intersections}
Let $\cZ \subset \PP^N$ be a Fano complete intersection of type $(d_1, d_2, \cdots, d_R)$ with $\dim \cZ = N - R \geq 4$. Without loss of generality, we assume that $2 \leq d_1 \leq d_2 \leq \cdots \leq d_R$.  Then any effective cycle \eqref{eq:intro-effective-cycle} (in homology or Chow whichever applies) and invariant \eqref{eq:intro-effective-inv} vanish unless one of the following situations holds
\begin{enumerate}
\item When $R = 1$, $(d = d_1, N)$ is from the following list
\[
(d = 3, N = 3, 4, 5, 6,  7, 8), (d=4, N = 4, 5).
\]

\item When $R \geq 2$, $(d_1, d_2, \cdots, d_R)$ is from the following list:
\[
(2,2), \ (2,3), \ (2,4), \ (2,2,2), \ (2,2,3), \ (2,2,2,2).
\]
\end{enumerate}
\end{corollary}

Consequently, the above vanishing together with \cite{CJR21} will
imply that the Gromov--Witten theory with ambient insertions of Fano
complete intersections \emph{outside} of the above lists are {\em entirely
determined} by the Gromov--Witten theory of the ambient projective
spaces for any $g \geq 2$.

\subsubsection{Calabi--Yau threefolds}

We next consider the interesting case where
\[
3- \dim \cX + \rk \bE = 0 \ \ \mbox{and} \ \  c_1(K_{\cX}\otimes \det \bE) = 0.
\]
In particular, $\cZ$ is a Calabi--Yau threefold.
While effective invariants \eqref{eq:intro-effective-inv} may be
non-vanishing in this case, they are determined via explicit equations
in Theorem \ref{thm:intro-string-divisor} by a small set of rational
numbers, called {\em basic effective invariants} (see
\S\ref{ss:basic-effective}).
In the prominent example of quintic threefolds, this specializes to:

\begin{corollary}[Example \ref{ex:quintic-effective-inv}]
  Let $\cZ \subset \PP^4$ be a quintic threefold.
  For each $g \geq 2$, the set of all effective invariants
  \eqref{eq:intro-effective-inv} are determined by the set of
  $\lfloor \frac{2g-2}{5}\rfloor + 1$ numbers
\[
\{ \deg [\RMm_{g, \ddata}(\punt, \beta)]^{\red}  \ \Big| \ \int_{\beta} c_1(\lblog) \leq  2g-2 \  \mbox{and}  \ c_i = -2 \ \mbox{for all} \ i\},
\]
via the equations in Theorem \ref{thm:intro-string-divisor}.
\end{corollary}

It is interesting to observe that the number
$\lfloor \frac{2g-2}{5}\rfloor + 1$ is compatible with physicists'
predictions \cite[(3.80)]{HKQ09}.
Indeed, the same method applies to four other families of Calabi--Yau
threefold complete intersections in projective spaces, and the numbers
of basic effective invariants in these cases do match with physicists'
prediction, see Example \ref{ex:CY3-complete-intersection}.
We refer to \S \ref{ss:basic-effective} for more details regarding various
examples and the general situation.

\subsection{Plan of the paper}

The moduli space and its virtual cycles are constructed in
Section~\ref{s:P-map} and \ref{sec:red-theory}.
The decomposition theorem of moduli space with disconnected domain is
proved in Section~\ref{sec:reduction-to-connected-inv}.
Section~\ref{sec:log-root} is devoted to the study of the logarithmic
root construction.
In Section~\ref{sec:logGLSM-targets}, we study in more detail the
targets $\infty$ of punctured R-maps, and the relation to the targets
$\fP$ from \cite{CJR21}.
In Section~\ref{sec:remove-unit} and \ref{sec:psimin-axioms}, we give
an explicit description of the universal curve of punctured R-maps,
and use this to prove the axioms for punctured R-maps.
The final Section~\ref{sec:effective} is devoted to the study of
effective invariants in many examples.

\subsection{Acknowledgments}

The authors would like to thank Dan Abramovich, Jarod Alper, Mark
Gross, Shuai Guo, Rahul Pandharipande, Adrien Sauvaget, Yefeng Shen,
Bernd Siebert, Nawaz Sultani, Rachel Webb, Dimitri Zvonkine for many
useful discussions during various stages of this work.

The first author was partially supported by NSF grant DMS-2001089.
The second author was partially supported by NSF grants DMS-1901748
and DMS-2054830.
The last author acknowledges support by the University of Michigan and
the Institute for Advanced Study in Mathematics at Zhejiang
University.

\subsection{Notations}
In this paper, we work over an algebraically closed field of characteristic zero, denoted by $\bk$.
All log structures are assumed to be \emph{fine and saturated} or simply \emph{fs} \cite{Ka88} unless they are along punctured markings or otherwise specified. We refer to \cite{ACGS20P} for the foundational work of punctured logarithmic maps.

A list of notations is provided below:

\begin{description}[labelwidth=2cm, align=right]
\item[$\BC$] the universal stack of $\CC^*_{\omega}$-torsors
\item[$\punt \to \BC$] the target of punctured $R$-maps
\item[$\fP \to \BC$] the target of log $R$-maps
\item[$\punt_\bk$] the base-change $\punt \times_{\BC} \spec\bk$

\item[$\infty_\cX$] a proper Deligne--Mumford stack with a projective coarse moduli
\item[$r$, $d$, $\ell$] positive integers
\item[$\lbspin$, $\lblog$] line bundles on $\xinfty$

\item[$\uC \to \uS$] a family of underlying pre-stable curves over $\uS$
\item[$\cC \to S$] a family of log curves over $S$
\item[$\cC^\circ \to \cC \to S$] a family of punctured curves over $S$

\item[$f\colon \pC \to \punt$] a punctured $R$-map
\item[$f\colon \cC \to \fP$] a log $R$-map

\item[$\beta$] a curve class in $H_2(\xinfty)$.
\item[$n$] the number of markings
\item[$m$] the number of connected components of $\cC$
\item[$\ddata$] collection of discrete data at all markings

\item[$\SH_{g,\ddata}(\infty,\beta)$] the moduli stack of stable punctured $R$-maps
\item[$\UH_{g,\ddata}(\infty,\beta)$] the moduli of stable punctured $R$-maps with uniform maximal degeneracy
\item[$\RMm_{g,\ddata}(\infty, \beta)$] the moduli of stable punctured $R$-maps with uniform extremal degeneracies
\item[$\RA_{g,\ddata}(\infty, \beta)$] the moduli of stable punctured $R$-maps with aligned degeneracies

\item[$\psi_{\min}$] universal min-class
\item[$\psi_{\DF}$] universal DF-class

\end{description}

% Punctured R-maps

\section{Punctured R-maps and their moduli}\label{s:P-map}

In this section, we first recall some basics of punctured maps from
\cite{ACGS20P}, which extends the theory of stable log maps to allow
contact orders to be negative.
For our purposes, we will need to extend the theory of \cite{ACGS20P}
slightly, in order to be able to consider punctured curves with
possible orbifold structures along markings and nodes.
However, the proofs in \cite{ACGS20P} apply with little change in the
situation needed in this paper.
We will point out the precise references for the readers' convenience.

In the second part of this section, we introduce the notion of
\emph{punctured R-maps}, construct the moduli stack of stable
punctured R-maps, and study its canonical obstruction theory.

%% Log curves
\subsection{Log curves}\label{s:curves}

%--- underlying curves
\subsubsection{Prestable curves}
Recall from \cite{AbVi02} that a \emph{twisted $n$-pointed curve} over a scheme $\ul{S}$ consists of the following data
\begin{equation}\label{eq:pre-stable-curve}
(\ul{\cC} \to \ul{C} \to \ul{S}, \{p_i\}_{i=1}^n)
\end{equation}
where
\begin{enumerate}
 \item $\ul{\cC}$ is a proper Deligne--Mumford stack, and is \'etale locally a nodal curve over $\ul{S}$.
 \item $p_i \subset \ul{\cC}$ are disjoint closed substacks in the smooth locus of $\ul{\cC} \to \ul{S}$.
 \item $p_i \to \ul{S}$ are \'etale gerbes banded by the multiplicative group $\mu_{r_i}$ for some positive integer $r_i$.
 \item the morphism $\ul{\cC} \to \ul{C}$ is the coarse moduli morphism.
 \item along each stacky locus of $\ul{\cC} \to \ul{S}$, the group action of $\mu_{r_i}$ is balanced.
 \item $\ul{\cC} \to \ul{C}$ is an isomorphism over $\ul{\cC}_{gen}$,
   where $\ul{\cC}_{gen}$ is the complement of the markings and the
   stacky locus of $\ul{\cC} \to \ul{S}$.
\end{enumerate}

Given a twisted curve as above, by \cite[4.11]{AbVi02} the coarse space
$\ul{C} \to \ul{S}$ is a family of $n$-pointed usual pre-stable curves
over $\ul{S}$ with the markings determined by the images of
$\{p_i\}$.
The \emph{genus} of the twisted curve $\ul{\cC}$ is defined as the genus of the
corresponding coarse pre-stable curve $\ul{C}$.

When there is no danger of confusion, we will simply write
$\ul{\cC} \to \ul{S}$, and the terminologies twisted curves and
pre-stable curves are interchangeable in this paper.

\subsubsection{Log curves}
An \emph{$n$-pointed log curve} over a fs log scheme
$S$ in the sense of \cite{Ol07} consists of a pair
\begin{equation}\label{eq:log-curve}
(\pi\colon \cC \to S, \{p_i\}_{i=1}^n),
\end{equation}
such that
\begin{enumerate}
 \item The underlying data $(\ul{\cC} \to \ul{C} \to \ul{S}, \{p_i\}_{i=1}^n)$ is a pre-stable curve as in \eqref{eq:pre-stable-curve}.
 \item $\pi$ is a proper, logarithmically smooth and integral morphism of fine and saturated logarithmic stacks.
 \item If $\ul{U} \subset \ul{\cC}$ is the non-singular locus of $\ul{\pi}$, then $\ocM_{\cC}|_{\ul{U}} \cong \pi^*\ocM_{S}\oplus\bigoplus_{i=1}^{n}\NN_{p_i}$ where $\NN_{p_i}$ is the constant sheaf over $p_i$ with fiber $\NN$.
\end{enumerate}

For simplicity, we may refer to $\pi\colon \cC \to S$ as a log curve
when there is no danger of confusion.
The \emph{pull-back} of a log curve $\pi\colon \cC \to S$ along an
arbitrary morphism of fs log schemes $T \to S$ is the
log curve $\pi_T\colon \cC_T:= \cC\times_S T \to T$ where the fiber
product is taken in the fs category.

\smallskip

We associate to a log curve \eqref{eq:log-curve}  the \emph{log cotangent bundle} $\omega^{\log}_{\cC/S} := \omega_{\uC/\uS}(\sum_i p_i)$ where $\omega_{\uC/\uS}$ is the relative dualizing line bundle of $\uC \to \uS$.

\subsubsection{The stack of log curves}

Denote by $\fM_{g, \{r_i\}}$ the stack of $n$-pointed, genus $g$ twisted curves such that the $i$-th marking is a $\mu_{r_i}$-gerbe. The universal curve $\fC_{g, \vec{r}} \to \fM_{g, \{r_i\}}$ can be viewed as a morphism of log stacks with the {\em canonical log structures} induced by their underlying structure. Then the log stack $\fM_{g, \{r_i\}}$ viewed as a category fibered over fs log schemes, is the category of log curves with discrete data $(g,\{r_i\})$, see  \cite{Ol07}.

%----------------------------------------------------------------------------Punctured curves
\subsection{Punctured curves}\label{ss:p-curve}

%--- Definition

\subsubsection{The definition}

Given a log curve \eqref{eq:log-curve}, denote by $\cP_{\cC} \subset \cM_{\cC}$ the sub-log structure given by the pre-image of $\ocP_{\cC} = \bigoplus_{i=1}^{n}\NN_{p_i}$ via $\cM_{\cC} \to \ocM_{\cC}$.

Consider the sub-sheaf of monoids $\ocN_{\cC} \subset \ocM_{\cC}$ such
that $\ocM_{\cC} = \ocN_{\cC}\oplus\ocP$.
Denote by $\cN_{\cC} \subset \cM_{\cC}$ the sub-log structure given by
the pre-image of $\ocN_{\cC}$ via $\cM_{\cC} \to \ocM_{\cC}$.
Thus, we have a decomposition
$\cM_{\cC} = \cN_{\cC}\oplus_{\cO^\times} \cP$.

A {\em punctured curve} over a fine and saturated log scheme $S$ in the sense of \cite[\S 2.1.4]{ACGS20P} consists of the following data
\begin{equation}\label{eq:punctured-curve}
(\pC \stackrel{\pun}{\longrightarrow} \cC \stackrel{\pi}{\longrightarrow} S, \{p_i\}_{i=1}^n)
\end{equation}
such that $(\cC \stackrel{\pi}{\longrightarrow} S, \{p_i\}_{i=1}^n)$
is a log curve as in \eqref{eq:log-curve}, and
$\pun\colon \pC \to \cC$ is a morphism of fine log stacks ($\cM_{\pC}$
is not necessarily saturated) satisfying
\begin{enumerate}
 \item The underlying map $\ul{\pun}\colon \ul{\pC} \to \ul{\cC}$ is an isomorphism.

 \item $\pun^{\flat}\colon \cM_{\cC} \to \cM_{\pC}$ is an isomorphism away from the collection of markings.

 \item $\pun^{\flat}$ induces a sequence of inclusions of sheaves of fine monoids
 \begin{equation}\label{eq:puncture-constraint}
 \cM_{\cC} \stackrel{\pun^{\flat}}{\hookrightarrow} \cM_{\pC} \subset \cN\otimes_{\cO^*}\cP^{gp}.
 \end{equation}
\end{enumerate}

$\pun$ is called the {\em puncturing of $\cC$ along $\cP$ or along the markings}. Markings of a punctured curve are called {\em punctured points,} or simply {\em punctures}. If $\pun|_{p_i}$ is an isomorphism, we say that the puncturing is {\em trivial} along $p_i$. For simplicity, we may refer to $\pC \to S$ as a punctured curve when there is no danger of confusion.

We will see in \S \ref{sss:contact-order} that non-trivial punctures are precisely the markings with negative contact orders.

%---- pullbacks of punctured curves
\subsubsection{Pull-backs of punctured curves}
The {\em pull-back} of a punctured curve \eqref{eq:punctured-curve}
along a morphism of fine and saturated log schemes $T \to S$ consists
of the following data
\begin{equation}\label{eq:pullback-punctured-curve}
(\pC_T \stackrel{\pun}{\longrightarrow} \cC_T \stackrel{\pi}{\longrightarrow} T, \{p_{i,T}\}_{i=1}^n),
\end{equation}
where $(\cC_T \stackrel{\pi}{\longrightarrow} T, \{p_{i,T}\}_{i=1}^n)$ is the pull-back of the corresponding log curve, and $\pC_T = \cC_T\times_{\cC}\pC$ with the fiber product taken in the fine category.

It is a non-trivial fact that the pull-back of punctured curves defined as above is again a punctured curve, see \cite[Proposition 2.7]{ACGS20P}.

%---Moduli of punctured curves
\subsubsection{Moduli of punctured curves}

Let $\pi\colon \cC \to S$ be a family of log curves.
Denote by $S^{\pun}$ the fibered category over fine and saturated
$S$-log schemes such that for any $T \to S$, $S^{\pun}(T)$ is the
category of punctured curves over $T$ whose corresponding log curve is
the pull-back $\pi_T\colon \cC_T \to T$.
There is the tautological morphism
\begin{equation}\label{eq:remove-punctures}
S^{\pun} \to S
\end{equation}
which takes the corresponding log curves.

It has been shown in \cite[Proposition 3.4]{ACGS20P} that there is a natural log-ideal $\cK_{S^{\pun}} \subset \cM_{S^{\pun}}$ such that the morphism \eqref{eq:remove-punctures} is an idealized log \'etale, strict, local immersion defined by $\cK_{S^{\pun}}$. In particular, if $S$ is a log algebraic stack, so is $S^{\pun}$. Roughly speaking, this log-ideal $\cK_{S^{\pun}}$ is precisely the obstruction for deforming punctures \cite[Proposition 2.45]{ACGS20P}.

%-----------------------DF log structures

\subsection{Deligne--Faltings log structures of rank one}\label{ss:universal-DF}

In this paper, the notion of Deligne--Faltings log structures will play
a crucial role.
We will recall them below.

%For the purposes of this paper, we will focus on targets of Deligne--Faltings type of rank $1$. These targets can be conveniently described as follows.

\subsubsection{The stack of Deligne--Faltings log structures of rank one}\label{ss:universal-target}
Recall the log stack $\cA$ defined in \eqref{eq:A}. There is a natural surjection $\NN_{\cA} \to \ocM_{\cA}$ given by $\NN = \Gamma(\cA, \ocM_{\cA})$, where $\NN_{\cA}$ is the constant sheaf on $\cA$.
A log algebraic stack $X$ is said  {\em of Deligne--Faltings type of rank $1$}, if there is a surjection $\NN_X \to \ocM_{\cX}$, which is equivalent to the existence of a unique strict morphism $X \to \cA$.

Consider the generator $1_{\cA} \in \NN_{\cA}$ with image $1_X \in \Gamma(X, \ocM_X)$.
The pre-image $\cT := 1_{X} \times_{\ocM_X}\cM_X \subset \cM_X$ is an
$\cO^*$-torsor.
The restriction of the structural morphism
$\alpha_{\cM_X}\colon \cM_X \to \cO_X$ to $\cT$ completes to a morphism
$s^{\vee}$ of line bundles on $X$:
\[
\cT \subset \cO(-1_X) \stackrel{s^{\vee}}{\longrightarrow} \cO_X,
\]
where $s^{\vee}$ vanishes precisely along the pre-image of
$\ainfty \subset \cA$.
Conversely, by \cite[Complement~1]{Ka88} the section
$s \in \Gamma(X,\cO(1_X) := \cO(-1_X)^{\vee})$ determines the log
structure $\cM_X$.

\subsubsection{Line bundles associated to global sections of characteristic sheaves}\label{sss:log-line-bundles}

Let $\fM$ be a log stack and $e \in \Gamma(\fM,\ocM_{\fM})$ a global section.
Similar to the above construction, the section $e$ induces
\begin{equation}\label{eq:compare-torsor}
\cM_{\fM} \times_{\ocM_{\fM}} \{e\} \subset \cO_{\fM}(-e) \stackrel{\bs^{\vee}_e}{\longrightarrow} \cO_{\fM}.
\end{equation}
with an $\cO^*$-torsor on the left completing to a line bundle $\cO_\fM(-e)$ with a morphism $\bs^{\vee}_e$.

The section $s_e \in \Gamma(\fM, \cO_\fM(e) := \cO_{\fM}(-e)^{\vee})$ defines a rank one Deligne--Faltings log structure, denoted by $\cM$. Consequently, we obtain a strict morphism
\begin{equation}\label{eq:rank1-DF}
\xymatrix{
(\ul{\fM}, \cM) \ar[rr]^-{(\cO_\fM(e), \bs_e)} && \cA.
}
\end{equation}
with $\cO_\fM(e) \cong \cO(\ainfty)|_{\fM}$.
The vanishing locus of $\bs_e$ is
\begin{equation}\label{eq:global-log-divisor}
\Div(e) = \ul{\fM}\times_{\cA}\ainfty \subset \ul{\fM}
\end{equation}
which is also set-theoretically the non-vanishing locus of $e$.
Locally on a strict smooth chart $U \to \fM$, the locus $\Div(e)$
can be described as follows.
Choose a lift $\tilde{e}_U \in \Gamma(U, \cM_U)$ of $e|_U$.
Then $\Div(e)\times_{\ul{\fM}}\ul{U}$ is the substack defined by the
vanishing of $\alpha(\tilde{e}_U)$.
Note that $\Div(e)\times_{\ul{\fM}}\ul{U}$ hence $\Div(e)$ does
not depend on the choices of local liftings $\tilde{e}_U$.
In case $e|_{\fM^{\circ}} = 0$ over an open dense substack
$\fM^{\circ}\subset \fM$, the locus $\Div(e)$ is a {\em Cartier
  divisor} in $\ul{\fM}$.

When the section $\bs_e$ vanishes identically, we simply write
\begin{equation}\label{eq:generic-rank1-DF}
\xymatrix{
(\ul{\fM}, \cM) \ar[rr]^-{\cO_\fM(e)} && \cA & \mbox{or} & (\ul{\fM}, \cM) \ar[rr]^-{\cO_\fM(e)} && \ainfty
}
\end{equation}
as the left morphism factors through the strict closed substack $\ainfty \subset \cA$.

Further observe that $\cM$ is the sub-log structure of $\cM_{\fM}$ generated by $\cM_{\fM}\times_{\ocM_{\fM}}\{e\} \subset \cM$. Hence we obtain the composition
\begin{equation}\label{eq:sub-rank1-DF}
\xymatrix{
\fM \ar[rr] && (\ul{\fM}, \cM) \ar[rr]^-{(\cO_\fM(e), \bs_e)} && \cA.
}
\end{equation}
with the left arrow is given by the inclusion $\cM \subset \cM_{\fM}$.

%% The punctured maps

\subsection{Punctured maps}
We next review punctured maps  needed in the setup of this paper.

\subsubsection{The definition}\label{def:punctured-map}

Recall from \cite[Definition 2.13]{ACGS20P} that a {\em punctured map}
to a log algebraic stack $X$ consists of the following data
\begin{equation}\label{eq:punctured-map}
(\pC \to \cC \to S, \{p_i\}, f\colon \pC \to X )
\end{equation}
such that $(\pC \to \cC \to S, \{p_i\})$ is a (possibly disconnected) punctured curve, and $f$ is a morphism of log stacks.
It is called {\em pre-stable} if $\cM_{\pC}$ is
generated by $\cM_{\cC}$ and $f^{\flat}(f^*\cM_{X})$.
Note that the pre-stability condition is an open condition, and is
stable under arbitrary base change, see
\cite[Proposition~2.15]{ACGS20P}.
The {\em pull-backs} of punctured maps are defined via pull-backs of punctured curves as usual.

For simplicity, we write $f\colon \pC \to X$ for a punctured map when there is no danger of confusion. Furthermore, all punctured maps are assumed to be pre-stable unless otherwise specified.

For the purpose of this paper, we will focus on the case that $X$
is of Deligne--Faltings type of rank $1$.
Thus, given a punctured map $f\colon \pC \to X$ we obtain a punctured map
$\mathfrak{f} \colon \pC \to \cA$ by composing $f$ with $X \to \cA$.
By the strictness of $X \to \cA$,
$f$ is pre-stable iff $\mathfrak{f}$
is pre-stable.

\subsubsection{Contact orders}\label{sss:contact-order}

Consider a punctured map $f\colon \pC \to \cA$ over a geometric point $S$.
Along each marking $p_i \subset \pC$, we obtain a composition
\begin{equation}\label{eq:define-contact}
\NN_{\cA}|_{p_i} \longrightarrow \ocM_{\cA}|_{p_i} \stackrel{\bar{f}}{\longrightarrow} \ocM_{\pC}|_{p_i} \hookrightarrow \ocM_{S}|_{p_i}\oplus \ZZ \longrightarrow \ZZ
\end{equation}
where the second-to-last inclusion is given by \eqref{eq:puncture-constraint}. The image of $1$ via the above composition is an integer $c_i$ called the {\em contact order} at the marking $p_i$. Contact orders remain constant along a connected family of punctured maps \cite[Proposition~2.23]{ACGS20P}.

For a target $X$ of Deligne--Faltings type of rank $1$,
the contact order of a punctured map $f\colon \pC \to X$ along
$p_i$ is defined to be the contact order of the corresponding
punctured map $\mathfrak{f} \colon\pC \to \cA$.

If $c_i \geq 0$, then the pre-stability forces $\pun$ being an
isomorphism around $p_i$.
In this case, the contact order is defined in the same way as for log
maps.

If $c_i < 0$, then $\pun$ is necessarily a non-trivial puncturing
along $p_i$.
This is a new feature of punctured maps compared to log maps.

\subsubsection{Degeneracies of components and their partial ordering}\label{sss:degeneracies}
Consider a punctured map $f \colon \pC \to \cA$ over a log point $S$ with possibly disconnected domain curve.  Let $Z \subset \pC$ be an irreducible component with the generic point $\eta_{Z}$. The image $e_{Z} = f^{\flat}(1)$ of
\[
f^{\flat}|_{\eta_Z}\colon (f^*\ocM_{\cA})|_{\eta_Z} \cong \NN \to \ocM_{\pC}|_{\eta_{Z}} \cong \ocM_S
\]
is called the {\em degeneracy} of the component $Z$.
Note that $e_Z = 0$ iff $f(\eta_Z)$ avoids $\ainfty$. Denote by
$\obD(f) \subset \ocM_S$ the {\em set of degeneracies} of $f$.
The set of degeneracies admits a natural partial ordering inherited
from the partial ordering of monoids.

\begin{definition}
  For an fs sharp monoid $P$, there is a natural {\em partial ordering} on
  the underlying set $P$ such that $e_1 \poleq_P e_2$ for any
  $e_1, e_2 \in P$ with $e_2 - e_1 \in P$.

Given a fs log scheme $S$, the characteristic sheaf $\ocM_S$ is naturally
a {\em sheaf of partially ordered sets} with the partial order
$\poleq_{\ocM_S}$ by gluing $\poleq_{\ocM_{S}|_{s}}$ at each geometric
point $s \in S$.
\end{definition}

Thus $\obD(f)$ is a partially ordered set with respect to
$\poleq_{\ocM_S}$.
A degeneracy $e_Z \in \obD(f)$ is called \emph{maximal}
(resp. \emph{minimal}) if $e_Z$ is maximal (resp. \emph{minimal}) in
$\obD(f)$.
The corresponding irreducible component $Z \subset \pC$ is called a
\emph{maximal (resp. minimal) component} of $f$.
Denote by $\obD^m(f) \subset \obD(f)$ the subset consisting of minimal
degeneracies.

\subsubsection{Extremal degeneracies}\label{sss:extremal-degeneracy}
We introduce the following configurations of degeneracies that will play an important role in this paper.

\begin{definition}\label{def:extremal-degeneracies}
  A punctured map over a geometric point is said to have \emph{uniform
    maximal degeneracy} or the \emph{$\curlywedge$-configuration}
  (resp.\ \emph{uniform minimal degeneracy} or the
  \emph{$\curlyvee$-configuration}) if the set of degeneracies has a
  unique maximum (resp.\ minimum) with respect to $\poleq_{\ocM_S}$.

A family of punctured maps is said to have \emph{uniform maximal
  degeneracy} or \emph{$\curlyvee$-configuration} (resp.\ \emph{uniform minimal degeneracy} or \emph{$\curlywedge$-configuration}) if each
geometric fiber has uniform maximal (resp.\ minimal) degeneracy. A punctured map is said to have {\em uniform extremal degeneracies} or the {\em $\diamondsuit$-configuration} if it admits both the $\curlywedge$-configuration and the $\curlyvee$-configuration.
\end{definition}

The $\curlywedge$-configuration is necessary for the construction of
the reduced theory in \S \ref{sec:red-theory}.
The $\curlyvee$-configuration is necessary to describe the virtual
normal bundle in the localization formula for log GLSM \cite{CJR23P}.
These two configurations and their combinations will be studied in
more detail in later sections.

\begin{corollary}
The locus $S_u \subset S$ parameterizing fibers with the $\curlywedge$-configuration (resp. $\curlyvee$-configuration), is a strict open subscheme. %\Qile{Check if we need this.}
\end{corollary}
\begin{proof}
This is identically to the proof of \cite[Proposition 3.3]{CJRS18P}, since punctured points and disconnectedness of the domains play no role in the proof.
\end{proof}

\subsubsection{Degeneracies in families}

Consider a family of punctured maps $f \colon \pC \to \cA$ over a fs log scheme $S$ with possibly disconnected domain curves. For each geometric point $s \in S$, let $f_s$ be the fiber of $f$ over $s$, hence partially ordered subsets $\obD^m(f_s) \subset \obD(f_s)\subset \ocM_{S}|_{s}$.

\begin{proposition}\label{prop:sheaf-of-degeneracies}
  The collection $\{\obD^m(f_s) \subset \obD(f_s)\}_{s \in S}$ glues
  to a subsheaf $\obD^m(f) \subset \obD(f)$ of partially ordered
  subsets of $\ocM_S$.
  We call $\obD(f)$ the \emph{sheaf of degeneracies} of $f$, and
  $\obD^m(f)$ the \emph{sheaf of minimal degeneracies}.
\end{proposition}

\begin{proof}

  The fact that the collection $\{\obD(f_s)\}_{s\in S}$ forms a
  subsheaf $\obD(f) \subset \ocM_S$ of sets is a restatement of
  \cite[Lemma 3.4]{CJRS18P}.
  By \cite[Lemma 3.6]{CJRS18P}, the fiberwise partial order
  $\poleq_{\ocM_{S}|_{s}}$ is stable under generization, which makes
  $\obD(f) \subset \ocM_S$ a subsheaf of partially ordered sets.
  Thus the collection of minimal elements $\{\obD^m(f_s)\}_{s \in S}$
  glues to a subsheaf of partially ordered sets
  $\obD^m(f) \subset \obD(f)$.
\end{proof}

\begin{notation}\label{not:extrem-degeneracy}
  Under the above assumptions, suppose $f$ admits the
  $\curlywedge$-configuration.
  In this case, there is a global section
  $e_{\max}(f) \in \Gamma(\ocM_S, S)$ whose fiber
  $e_{\max}(f|_s) \in \obD(f_s)$ is the maximum for each $s$.
  We call $e_{\max}(f)$ the {\em maximal degeneracy} of $f$.

  Similarly, for $f$ with the $\curlyvee$-configuration, there is a
  global section $e_{\min}(f) \in \Gamma(\ocM_S, S)$ whose fiber
  $e_{\min}(f|_s) \in \obD(f_s)$ is the minimum for each $s$.
  We call $e_{\min}(f)$ the {\em minimal degeneracy} of $f$.
  In this case, the sheaf $\obD^m(f)$ is a constant sheaf given by the
  section $e_{\min}(f)$.

  For simplicity, we may write $e_{\max}$ or $e_{\min}$ when there is
  no confusion which $f$ is meant.
\end{notation}

\subsubsection{The moduli of punctured maps to $\cA$}\label{sss:universal-puncture-moduli}

Denote by $\fM_{g,\ddata'}(\cA)$ the category of $n$-pointed, genus
$g$ punctured maps to $\cA$ with connected domains, fibered over the
category of fs log schemes such that the $i$-th markings are
$\mu_{r_i}$-gerbes with contact order $c_i$.
Here $\ddata' = \{(r_i, c_i)\}_{i=1}^{n}$ is the \emph{reduced
  discrete data} at the markings.

By \cite[Proposition 3.3]{ACGS20P}, the fibered category
$\fM_{g,\ddata'}(\cA)$ is represented by a fs log algebraic stack.
Indeed, its underlying stack $\ul{\fM}_{g,\ddata'}(\cA)$ is the stack
of \emph{basic punctured maps} with the following {\em universal
  property} (\cite[Proposition~2.26]{ACGS20P}):
For any punctured map
$f\colon \pC \to \cA$ over $S$, there is a basic object
$f_{bas}\colon \pC_{bas} \to \cA$ over $S_{bas}$ and a morphism
$S \to S_{bas}$ such that $f$ is the pull-back of $f_{bas}$.
Furthermore, the pair $(f_{bas}, S \to S_{bas})$ is unique up a unique
isomorphism. Furthermore, the basicness can be described
combinatorially similar to the case of log maps, see \cite[\S
2.3.1]{ACGS20P}.
In particular, the universal punctured map over $\fM_{g,\ddata'}(\cA)$
is basic.

\subsubsection{The moduli of punctured maps to \texorpdfstring{$\infty_{\cA}$}{the Artin fan}}

Consider the closed substack
\begin{equation}\label{eq:main-component}
\fM_{g,\ddata'}(\infty_{\cA}) \subset \fM_{g,\ddata'}(\cA)
\end{equation}
parameterizing  punctured maps factorizing through the strict closed substack $\infty_{\cA} \subset \cA$.  This will be a key player in the construction of this paper.

\begin{remark}\label{rem:universal-stack-of-punctured-maps}
In the general construction of \cite{ACGS20P}, the stack $\fM_{g,\ddata'}(\infty_{\cA})$ is formulated  as follows. Recall the real positive half axis $\Sigma(\cA) = \mathbf{R}_{\geq 0}$ as the cone complex of $\cA$. Let $\tau = (G, \mathbf{g}, \mathbf{\sigma}, \overline{\mathbf{u}})$ be the {\em global type} of punctured maps (\cite[Definition 3.4 (1)]{ACGS20P}) such that
\begin{enumerate}
 \item $G$ is a connected graph with a unique vertex $v$, a set $L(G)$ of $n$ legs corresponding to the $n$-markings, and no edges.
 \item $\mathbf{g}$ assigns genus $g$ to the vertex $v$.
 \item $\mathbf{\sigma}$ assigns the interior $\mathbf{R}_{>0} \subset \Sigma(\cA)$ to the vertex and legs of $G$.
 \item $\overline{\mathbf{u}}$ assigns the contact order $c_i$ to the leg corresponding to the $i$-th marking.
\end{enumerate}
By \cite[Definition 3.8]{ACGS20P}, $\fM_{g,\ddata'}(\infty_{\cA}) \cong \fM_{g,\ddata'}(\cA, \tau)$
is the stack of punctured maps in $\fM_{g,\ddata'}(\cA)$ marked by $\tau$. It is equipped with a {\em canonical log-ideal} $\cK_{\fM_{g,\ddata'}(\infty_{\cA})}$ such that the forgetful morphism
\begin{equation}\label{eq:universal-stack-ideal-etale}
\fM_{g,\ddata'}(\infty_{\cA}) \longrightarrow \fM_{g, \{r_i\}}
\end{equation}
is idealized log \'etale \cite[Theorem 3.24]{ACGS20P}, where
$\fM_{g, \{r_i\}}$ is equipped with the canonical log structure and
the trivial log ideal.
\end{remark}

\begin{lemma}\label{lem:universal-punctured-map-moduli}
\begin{enumerate}
\item $\fM_{g,\ddata'}(\infty_{\cA})$ is non-empty, reduced and
  pure-dimensional with
  \begin{equation}\label{eq:base-dimension}
    \dim \fM_{g,\ddata'}(\infty_{\cA}) = 3g - 4 + n.
  \end{equation}
\item The locus
  $\fM^{\circ}_{g,\ddata'}(\infty_{\cA}) \subset
  \fM_{g,\ddata'}(\infty_{\cA})$ with smooth domains is open and dense
  in $\fM_{g,\ddata'}(\infty_{\cA})$ such that
  $\ocM_{\fM^{\circ}_{g,\ddata'}(\infty_{\cA})} \cong
  \NN_{\fM^{\circ}_{g,\ddata'}(\infty_{\cA})}$ is generated by the
  degeneracy of the smooth component, and with the canonical log-ideal
  $\cK_{\fM^{\circ}_{g,\ddata'}(\infty_{\cA})} =
  \cM_{\fM^{\circ}_{g,\ddata'}(\infty_{\cA})}\setminus \cO^*$.
\end{enumerate}
\end{lemma}

\begin{proof}
  These statements follow directly from \cite{ACGS20P}.
  For the readers' convenience, we provide precise references below.

By \cite[Proposition 3.27]{ACGS20P}, it is straightforward to check that $\tau$ is realizable, with the corresponding basic monoid $Q_{\tau} \cong \NN$ generated by the degeneracy of the smooth component. (\cite[(2.16)]{ACGS20P}). Then (1) follows from \cite[Proposition 3.28]{ACGS20P} with $B = \spec \bk$ with the trivial log structure.

The open locus  $\fM^{\circ}_{g,\ddata'}(\infty_{\cA})$ consists of punctured maps whose tropicalizations are of type $\tau$ (\cite[Definition 2.22]{ACGS20P}).  Thus the characteristic sheaf is given by the basic monoid
\[
\ocM_{\fM^{\circ}_{g,\ddata'}(\infty_{\cA})} \cong Q_{\tau,{\fM^{\circ}_{g,\ddata'}(\infty_{\cA})}}  = \NN_{{\fM^{\circ}_{g,\ddata'}(\infty_{\cA})}}.
\]
Furthermore  $\cK_{\fM^{\circ}_{g,\ddata'}(\infty_{\cA})} = \cM_{\fM^{\circ}_{g,\ddata'}(\infty_{\cA})}\setminus \cO^*$ follows from \cite[Proposition 3.23]{ACGS20P}.

Finally, the statement that $\fM^{\circ}_{g,\ddata'}(\infty_{\cA}) \subset \fM_{g,\ddata'}(\infty_{\cA})$ is dense follows from the local description in \cite[Remark 3.25]{ACGS20P}. This is indeed a consequence of the idealized log \'etaleness of \eqref{eq:universal-stack-ideal-etale}.
\end{proof}

\subsubsection{The balancing condition}

Punctured maps admits a simple but important numerical constraint called the {\em balancing condition}, which plays an important role in the applications in \S \ref{sec:effective}.

\begin{lemma}[The balancing condition]\label{lem:balancing}
Let $f \colon \pC \to \ainfty$ be a punctured map over a log  point $S$ with discrete data $\ddata' = \{(r_i, c_i)\}_{i=1}^n$ along markings. Then we have
\[
\deg f^*\cO(\ainfty) = \sum_{i=1}^n \frac{c_i}{r_i}.
\]
\end{lemma}
\begin{proof}
  The statement can be viewed as a special case of \cite[Proposition 2.27]{ACGS20P} with the only change that the degree of the $i$-th marking is $\frac{1}{r_i}$ due to the orbifold structure.
\end{proof}

%P-definition
\subsection{Punctured R-maps}

Next, we introduce the central notion of \emph{punctured R-maps}.
They are a combination of punctured maps and R-maps, as an analogue of
the log R-maps in \cite{CJR21} with possibly negative contact orders.

%\subsubsection{A general situation}
\subsubsection{The targets and the stability condition}\label{sss:P-target}

Recall from \eqref{eq:intro-punctured-R-map} that a {\em punctured R-map} over a log scheme $S$ with target $\punt \to \BC$ is a commutative triangle
\begin{equation}\label{eq:punctured-R-map}
\xymatrix{
 && \punt \ar[d] \\
\pC \ar[rru]^{f} \ar[rr]_{\omega^{\log}_{\cC/S}}&& \BC
}
\end{equation}
where $\pC \to S$ is a punctured curve and the bottom arrow is defined such that the pull-back of $\LR$ is $\omega^{\log}_{\cC/S}$. It is {\em pre-stable} if the punctured map $f$ is pre-stable. {\em Pull-backs} of punctured R-maps are defined using pull-backs of punctured curves as usual.

In this paper, we further impose the following conditions on $\punt \to \BC$:

\begin{enumerate}

\item There exists a strict (hence unique) morphism
  $\punt \stackrel{\cO(\punt)}{\longrightarrow}\infty_{\cA}$ defined as in \eqref{eq:generic-rank1-DF}.

\item The underlying DM-stack
  $\ul{\punt_{\bk}} := \ul{\punt}\times_{\BC}\spec \bk$ has a
  projective coarse moduli $\ul{\punt_{\bk}}'$.
\item There is a morphism $\ul{\punt} \to \ul{\punt_{\bk}}'$ such that
  the composition
  $\ul{\punt_{\bk}} \to \ul{\punt} \to \ul{\punt_{\bk}}' $ is the
  coarsification.
\end{enumerate}
We fix an ample line bundle $\cH$ on $\ul{\punt_{\bk}}'$, and
denote its pull-back to $\punt$ again by $\cH$ when there is
no confusion.

Note that the target \eqref{eq:punt} from the root construction in \S
\ref{intro:target} fits in the above setup.
Indeed, in this case $\ul{\punt_{\bk}}'$ is the coarse moduli of
${\xinfty}$.

\begin{remark}
  Condition (3) may be reformulated in terms of
  $\mathbf{C}^*_\omega$-equivariant geometry assuming that the
  underlying DM-stack
  $\ul{\punt_{\bk}} := \ul{\punt}\times_{\BC}\spec \bk$ has a reduced
  projective coarse moduli $\ul{\punt_{\bk}}'$.

  Note that by properties of coarse moduli, the morphism
  $\ul{\punt} \to \BC$ induces compatible
  $\mathbf{C}^*_\omega$-actions on $\ul{\punt_\bk}$ and
  $\ul{\punt_\bk}'$.
  %\Qile{Why there's necessary $\CC^*$-action on $\ul{\punt_\bk}'$? By the functoriality of coarse moduli?}
  We then set
  \begin{itemize}
  \item[(3')] The $\mathbf{C}^*_\omega$-action on $\ul{\punt_\bk}'$ is
    trivial.
  \item[(3'')] The $\mathbf{C}^*_\omega$-action on $\ul{\punt_\bk}$ is
    trivial after reparameterization of $\mathbf{C}^*_\omega$.
  \end{itemize}
  By \cite[Proposition~5.32]{AHR20}, these two conditions are
  equivalent.

  On the other hand, assuming Condition (3) we know that
  $\ul{\punt_\bk} \to \ul{\punt_\bk}'$ is
  $\mathbf{C}^*_\omega$-equivariant for the trivial action on
  $\ul{\punt_\bk}'$.
  Hence, \cite[Proposition~5.32]{AHR20} implies Condition (3').
  In the other direction, assuming Condition (3'), we know that
  $\ul{\punt_\bk}' \to [\ul{\punt_\bk}' / \CC^*_\omega]$ admits a
  section, and we may define a desired morphism
  $\ul{\punt} \to \ul{\punt_\bk}'$ as the composition
  \begin{equation*}
    \ul{\punt} \cong [\ul{\punt_\bk} / \CC^*_\omega] \to [\ul{\punt_\bk}' / \CC^*_\omega] \to \ul{\punt_\bk}'.
  \end{equation*}
\end{remark}

\begin{definition}\label{def:P-curve-class}
Given a punctured R-map $f\colon \pC \to \punt$, we say that $f$ has curve class $\beta \in H_2(\ul{\punt_{\bk}}')$ if the curve class of the following composition is $\beta$:
\begin{equation}\label{eq:coarse-map}
\pC \to \punt \to \punt_{\bk}'.
\end{equation}
\end{definition}

\begin{definition}\label{def:P-stability}
A pre-stable punctured R-map \eqref{eq:punctured-R-map} with the target described above is called {\em stable} if $f$ is representable and $\omega_{\cC/S}^{\log}\otimes f^*\cH^{\otimes k} > 0$ for $k \gg 0$.
\end{definition}

Note that the positivity condition above means that the coarse moduli
of the composition \eqref{eq:coarse-map} is stable in the usual sense.
Furthermore, the stability condition depends only on the underlying
R-maps.
Thus, we may call an underlying R-map {\em stable}, if it satisfies
the above stability condition.

\subsubsection{The discrete data along markings}

Consider a punctured R-map $f\colon \pC \to \punt$ over $S$ and its
$i$-th marking $p_i \subset \pC$.
Since $\omega^{\log}_{\cC/S}|_{p_i} \cong \cO_{p_i}$, there is a
factorization
\begin{equation}\label{eq:restrict-to-marking}
\xymatrix{
p_i \ar@/^1.2pc/[rr]^{f|_{p_i}} \ar[r] & \punt_{\bk} \ar[r] & \punt.
}
\end{equation}

Denote by $\ocI\punt_\bk$ the rigidified cyclotomic inertia
\cite[Section~3]{AGV08} with the trivial log structure. There is a decomposition
$\ocI\punt_\bk = \cup_j \ogamma_j$ into connected
components.
By \cite[Section~3]{AGV08}, the factorization
\eqref{eq:restrict-to-marking} induces the {\em $i$-th evaluation
  morphism}
\begin{equation}\label{eq:underlying-evaluation}
\ev_i \colon S \to \ocI\punt_\bk.
\end{equation}

\begin{definition}\label{def:marking-discrete-data}
  The {\em discrete data} at the $i$-th marking is a pair
  $(\ogamma_i, c_i)$ consisting of a connected component
  $\ogamma_i \subset \ocI\punt_\bk$ and a contact order $c_i \in \ZZ$.

  Throughout this paper, denote by $r_i$ the order of the cyclotomic group of
  the gerbe parameterized by $\ogamma_i$.
\end{definition}

\subsubsection{The moduli functor}
Consider the collection of discrete data
\begin{equation}\label{eq:P-data}
(g, \beta, \ddata = \{(\ogamma_i, c_i)\}_{i=1}^n)
\end{equation}
where $g$ denotes the genus, $\beta \in H_2(\ul{\infty_{\bk}}')$ is a curve class, and $\ddata$ is a collection of the discrete data at the marked points.

Denote by $\SH_{g,\ddata}(\punt,\beta)$ the category of $n$-pointed genus $g$ stable punctured R-maps to $\punt \to \BC$ with the discrete data \eqref{eq:P-data}.  The following result will be established in the next two sections \S\ref{ss:PR-representability} and \ref{ss:PR-proper}:

\begin{theorem}\label{thm:PR-moduli}
$\SH_{g,\ddata}(\punt,\beta)$ is represented by a proper, log Deligne--Mumford stack.
\end{theorem}

\subsection{\texorpdfstring{$\SH_{g,\ddata}(\punt,\beta)$}{R} is a log Deligne--Mumford stack}\label{ss:PR-representability}

\subsubsection{Tautological morphism to \texorpdfstring{$\fM_{g,\ddata'}(\infty_\cA)$}{M}}

Let $\ddata' = \{(r_i, c_i)\}$ be the reduced discrete data
at the markings.
There is a canonical morphism
\begin{equation}\label{eq:take-log}
\SH_{g,\ddata}(\punt,\beta) \to \fM_{g,\ddata'}(\ainfty)
\end{equation}
obtained by composing the punctured R-maps with the strict morphism
$\punt \to \ainfty$.

\subsubsection{The moduli of underlying stable R-maps}

Denote by
$\SH_{g,\{\ogamma_i\}}(\ul{\punt},\beta)$ the stack of stable R-maps to the
target $\ul{\punt} \to \BC$ with discrete data
$(g, \beta, \{\ogamma_i\})$.
Since $\ul{\punt}$ has the trivial log structure, contact orders are
irrelevant in this case.
We equip $\SH_{g,\{\ogamma_i\}}(\ul{\punt},\beta)$ with the canonical log
structure from its source curves.
We obtain the tautological morphism
\begin{equation}\label{eq:P-forget-log}
\SH_{g,\ddata}(\punt,\beta) \to \SH_{g,\{\ogamma_i\}}(\ul{\punt},\beta).
\end{equation}

On the other hand, composing stable R-maps in
$\SH_{g,\{\ogamma_i\}}(\ul{\punt},\beta)$ with the underlying morphism of
$\punt\to \ainfty$, we obtain a canonical strict morphism
\begin{equation}\label{eq:forget-R}
\SH_{g,\{\ogamma_i\}}(\ul{\punt},\beta) \to \fM_{g,\{r_i\}}(\ul{\ainfty}).
\end{equation}
Here $\fM_{g,\{r_i\}}(\ul{\ainfty})$ is algebraic, and is equipped with the canonical log structure from its source curve.

The algebraicity of the moduli of (not necessarily stable) R-maps has
been established in \cite[\S 5.1]{CJR21}.
Since the stability condition in Definition \ref{def:P-stability} is
an open condition, $\SH_{g,\{\ogamma_i\}}(\ul{\punt},\beta)$ is a log
algebraic stack.

\subsubsection{The algebraicity}

Putting together \eqref{eq:take-log}, \eqref{eq:P-forget-log} and \eqref{eq:forget-R}, we obtain a tautological morphism of fibered categories
\[
\SH_{g,\ddata}(\punt,\beta) \to \fM_{g,\ddata'}(\ainfty)\times_{\fM_{g,\{r_i\}}(\ul{\ainfty})}\SH_{g,\{\ogamma_i\}}(\ul{\punt}.\beta)
\]
The same proof as in \cite[Lemma 3.2]{ACGS20P} shows that the above
morphism is an isomorphism.
Thus, the algebraicity of $\SH_{g,\ddata}(\punt,\beta)$ follows from
the algebraicity of the right hand side.

\subsubsection{Finiteness of automorphisms}

By \cite[Proposition 2.27]{ACGS20P}, the morphism
\eqref{eq:P-forget-log} is of DM-type.
It suffices to show that objects of $\SH_{g,\{\ogamma_i\}}(\ul{\punt},\beta)$
have at most finitely many automorphisms.

Consider a stable R-map $f\colon \uC \to \ul{\punt}$ over a geometric
point $\ul{S}$.
Let $\uC \to \ul{C}$ be the coarse curve.
Then $f$ induces a pre-stable map
$f'\colon \ul{C} \to \ul{\punt_{\bk}}'$ which leads to a morphism
\[
\Aut(f) \to \Aut(f')
\]
of the automorphism groups. The stability of $f$ implies that $f'$ is a stable map, thus $\Aut(f')$ is a finite group.  It remains to show that $H = \ker\big(\Aut(f)\to \Aut(f')\big)$ is finite.

Since $\uC \to \BC$ factors through $\ul{C} \to \BC$, $f$ is
equivalent to a map $\rho \colon \uC \to \ul{\punt}\times_{\BC}\ul{C}$
whose composition $\uC \to \ul{\punt}\times_{\BC}\ul{C} \to \ul{C}$ is
the coarse moduli.
Since $H$ fixes $\ul{C}$, hence fixes $\ul{\punt}\times_{\BC}\ul{C}$, it
is the automorphism group of $\rho$.
However, the stability of $f$ implies that $\rho$ is also stable.
Therefore, $H$ is a finite subgroup.

\subsection{Properness}\label{ss:PR-proper}

\subsubsection{Removing the punctured R-structure}
Consider the tautological morphism
\begin{equation}\label{eq:remove-PR}
\SH_{g,\ddata}(\punt,\beta) \to \overline\scrM_{g,n}(\ul{\punt_{\bk}}',\beta)
\end{equation}
by composing punctured R-maps with $\punt \to \ul{\punt_{\bk}}'$,
removing log structures, and taking the corresponding coarse maps.
Here $\overline\scrM_{g,n}(\ul{\punt_{\bk}}',\beta)$ is the moduli of stable
maps to $\ul{\punt_{\bk}}'$, hence is proper.
No contraction of the source curve is needed thanks to Definition
\ref{def:P-stability}.

\subsubsection{Boundedness}
Consider the morphism
\[
\SH_{g,\ddata}(\punt,\beta) \to \fM_{g,n}
\]
defined by taking the coarse pre-stable curve.
It factors through \eqref{eq:remove-PR}, hence has image covered by a
finite type scheme, say $T \to \fM_{g,n}$.
It suffices to show that
$\SH_T := \SH_{g,\ddata}(\punt,\beta)\times_{\fM_{g,n}}T$ is of finite
type.

Denote by $\ul{C} \to T$ the universal curve over $T$, equipped with
the trivial log structures.
Consider $\infty_{T} := \infty \times_{\BC}\ul{C} \to T$.
Let $f_{\SH_T}\colon \pC \to \infty$ be the universal stable punctured
map over $\SH_T$. Then there is a commutative diagram
\[
\xymatrix{
\pC \ar[r] \ar@/^1pc/[rr]^{f_{\SH_{T}}} \ar[rd] \ar[d] & \infty_T  \ar[r]  \ar[d] & \infty \ar[d] \\
 \SH_T  \ar[rd] & \ul{C} \ar[d] \ar[r] & \BC \\
 & T &
}
\]
Denote by $\beta'$ the curve class of the family $\infty_T \to T$
given by the image of $\pC \to \infty_T$ above.
It is determined by $\beta$.
Consider the moduli $\overline\scrM_{g, \ddata}(\infty_T/T,\beta')$ of stable punctured maps with target $\infty_T \to T$.
The same proof of \cite[Proposition 3.10 and 3.11]{ACGS20P} implies
that $\overline\scrM_{g, \ddata}(\infty_T/T,\beta')$ is of finite type over
$T$.
Then $\SH_T$ is the open substack of
$\overline\scrM_{g, \ddata}(\infty_T/T,\beta')$ where the composition of the
punctured map $\pC \to \infty_T$ with $\infty_T \to \ul{C}$ is the
coarse moduli $\pC \to \ul C$, hence contracts no components.
In particular, $\SH_T$ is of finite type.

\subsubsection{The weak valuative criterion}
To show that $\SH_{g,\ddata}(\punt,\beta)$ is proper, it remains to
show that \eqref{eq:remove-PR} is universally closed.
Let $R$ be a DVR and $K$ be its fraction field.
Consider the following commutative diagram of solid arrows
\[
\xymatrix{
\ul{\eta} \ar[rr]^{[f_{\eta}]} \ar[dd] && \ul{\SH_{g,\ddata}(\punt,\beta)} \ar[d] \\
&& \SH_{g,\{\ogamma_i\}}(\ul{\punt},\beta) \ar[d] \\
\ul{S} \ar[rr]^{[f']} \ar@/^1pc/@{-->}[rruu]^{[f]} \ar@/^0.5pc/@{-->}[rru]^{[\ul{f}]} &&  \overline\scrM_{g,n}(\ul{\punt_{\bk}}',\beta)
}
\]
where $\ul{S} = \spec R$, $\ul{\eta} = \spec K$,
$f'\colon \ul{C} \to \ul{\punt_{\bk}}'$ is a stable map over $\ul{S}$, and
$f_{\eta}\colon \pC_{\eta} \to \punt$ is a stable punctured map over
$\ul{\eta}$.
Possibly after a finite extension of DVRs, we would like to show that
there is a unique dashed arrow $[f]$ making the above diagram
commutative.

By \cite[Theorem 3.12]{ACGS20P}, the morphism
$ \SH_{g,\ddata}(\punt,\beta) \to
\SH_{g,\{\ogamma_i\}}(\ul{\punt},\beta) $ satisfies the week valuative
criterion.
Thus, it suffices to show that, possibly after a finite extension of
DVRs, there is a unique dashed arrow $[\ul{f}]$ making the above
diagram commutative.
Here $\ul{f}\colon \uC \to \ul{\punt}$ denotes an underlying R-map over
$\ul{S}$ with the restriction $\ul{f}|_{\ul{\eta}} = \ul{f_{\eta}}$.

Denote by $\ul{\infty_{\ul{S}}} := \ul{C} \times_{\BC} \ul{\infty}$.
The underlying $\ul{f_{\eta}}$ induces a stable map
$\uC_{\eta} \to \ul{\infty_{\ul{S}}}$.
Possibly after a finite extension of DVRs, we may extend this to a
stable map $\uC \to \ul{\infty_{\ul{S}}}$.
We arrive at a commutative diagram
\[
\xymatrix{
\uC \ar[r] \ar[rd] & \ul{\infty_{\ul{S}}} \ar[r] \ar[d] & \ul{\punt_{\bk}}'. \\
& \ul{C} \ar[ru]_{f'}&
}
\]
We will show that the morphism $\uC \to \ul{C}$ is the coarse moduli, hence the composition $\uC \to \ul{\infty_{\ul{S}}} \to  \ul{\infty}$ gives the unique $\ul{f}$ as needed. It suffices to show that $\uC \to \ul{C}$ contracts no components over the central fiber.

Suppose $\uC \to \ul{C}$ contracts some components.
Then there is a rational component $\cZ \subset \uC$ with at most two
special points contracted in $\ul{C}$.
Thus $\cZ$ maps to a fiber of $\ul{\infty_{\ul{S}}} \to \ul{C}$.
Since the composition
$\cZ \to \ul{\infty_{\ul{S}}} \to \ul{\punt_{\bk}}'$ coincides with
$\cZ \to \ul{C} \to \ul{\punt_{\bk}}'$, the component $\cZ$ is
contracted in $\ul{\infty_{\ul{S}}}$.
This contradicts the construction of $\uC \to \ul{\infty_{\ul{S}}}$ as
a stable map.

\bigskip

This completes the proof of Theorem \ref{thm:PR-moduli}. \qed

\subsection{The canonical perfect obstruction theory}\label{ss:canonical-theory}

Now assume that the strict morphism
\begin{equation}\label{eq:PT-smooth}
\punt \to \BC \times \ainfty
\end{equation}
is smooth. Note that it is also of DM-type.
We will use $\LL$ to denote the log cotangent complex in the sense of
\cite{LogCot}. By \cite[1.1 (ii), (iii)]{LogCot}, we have
\[
\Omega_{\punt/\BC}:= \LL_{\punt/\BC\times\ainfty} \cong \LL_{\ul{\punt}/\BC\times\ul{\ainfty}} \cong \Omega_{\ul{\punt}/\BC\times\ul{\ainfty}}
\]
where $\Omega_{\ul{\punt}/\BC\times\ul{\ainfty}}$ is the usual cotangent bundle.

Using the short-hand notations $\SH := \SH_{g,\ddata}(\punt,\beta)$
and $\fM := \fM_{g,\ddata'}(\ainfty)$, consider the universal
punctured maps
\[
f_{\SH}\colon \pC_{\SH} \to \punt \ \ \ \mbox{and} \ \ \  f_{\fM}\colon \pC_{\fM} \to \ainfty
\]
over $\SH$ and $\fM$ respectively. They fit in a commutative diagram
\[
\xymatrix{
& \punt \ar[r] & \BC\times\ainfty \\
\pC_{\SH} \ar[r] \ar[d]_{\pi} \ar@/^.5pc/[ru]^{f_{\SH}} & \pC_{\fM} \ar@/^.5pc/[ru]_{f_{\fM}} \ar[d]^{\pi} \\
\SH \ar[r] & \fM
}
\]
where the bottom square is Cartesian. For simplicity, let $\pi$ be the projection for both families. The top square leads to two distinguished triangles
\[
f^*_{\SH} \Omega_{\punt/\BC} \to \LL_{\pC_{\SH}/\BC\times\ainfty} \to \LL_{f_{\SH}} \stackrel{[1]}{\to} \ \ \mbox{and} \ \
\LL_{f_{\fM}}|_{\pC_{\SH}} \to \LL_{\pC_{\SH}/\BC\times\ainfty} \to \LL_{\pC_{\SH}/\pC_{\fM}} \stackrel{[1]}{\to}
\]
hence a morphism
\[
f^*_{\SH} \Omega_{\punt/\BC} \to  \LL_{\pC_{\SH}/\pC_{\fM}} \cong \pi^*\LL_{\SH/\fM}.
\]
Let $\omega^{\bullet}_{\pi} := \omega_{\pi}[1]$ be the dualizing
complex of $\ul{\pi}$.
Then a standard argument gives the morphism
\begin{equation}\label{eq:POT}
\varphi_{\SH/\fM}^{\vee} \colon \EE_{\SH/\fM}^{\vee} := \rd\pi_*\big(f^*_{\SH} \Omega_{\punt/\BC}\otimes \omega^{\bullet}_{\pi}\big) \to \LL_{\SH/\fM},
\end{equation}
and equivalently its dual
\begin{equation}\label{eq:tan-POT}
\varphi_{\SH/\fM} \colon \TT_{\SH/\fM} \to \EE_{\SH/\fM} \cong \rd\pi_*\big(f^*_{\SH}\Omega^{\vee}_{\punt/\BC} \big).
\end{equation}
Since $\SH \to \fM$ is strict, $\LL_{\SH/\fM}$ agrees with the usual
cotangent complex of the underlying morphism.
By \cite[Proposition 4.2]{ACGS20P}, \eqref{eq:POT} defines a perfect
obstruction theory of $\SH \to \fM$ in the sense of \cite{BeFa97}.
Since $\fM$ is equidimensional by Lemma
\ref{lem:universal-punctured-map-moduli}, we may define the {\em
  canonical virtual cycle} of $\SH$, denoted by $[\SH]^{\vir}$, via
virtual pull-back with respect using \eqref{eq:POT}.

%%%%%%%%%%%%%%%%%%%%%%%%%%%%%%%%%%%%%%%%%%%%%%%%%%%%%%%%%%%
% Configurations of log structures and the reduced theories
%%%%%%%%%%%%%%%%%%%%%%%%%%%%%%%%%%%%%%%%%%%%%%%%%%%%%%%%%%%

\section{The reduced perfect obstruction theory of punctured R-maps}\label{sec:red-theory}

In this section, we will explain how a superpotential introduced below
can be used to construct a reduced perfect obstruction theory.
The compatibility of the reduced theory in \cite{CJR21} and the
current paper will be made explicit in \S\ref{ss:GLSM-potential}.

\begin{definition}
  \label{def:superpotential2}
  A \emph{superpotential} of a target $\infty$ is a section $W$ of the
  line bundle
  \begin{equation*}
    \uomega \otimes \cO(\ttwist \infty)
  \end{equation*}
  over $\punt$, for some positive integer $\ttwist$ called the {\em order} of $W$.
\end{definition}

\subsection{The category of punctured maps for the reduced theory}\label{ss:category-for-reduced-theory}

It was shown in \cite{CJRS18P, CJR21} that log maps to $\cA$ with
uniform maximal degeneracy form are the right subcategory of log maps
for constructing the reduced theory of log R-maps.
Such a configuration of degeneracies is also necessary for
constructing reduced theory of punctured R-maps, which is the topic of
this section.

\subsubsection{The moduli stacks with universal target}\label{sss:universal-UM-stack-disconnected}

In the following, we will consider maps from possibly disconnected curves, hence
we will consider a sequence of discrete data
\begin{equation}\label{eq:[m]-punctured-data}
(g_i, \ddata_i'), \ \ \mbox{for } i \in [m] = \{1, 2, \cdots, m\},
\end{equation}
for punctured maps to $\ainfty$ as in \S \ref{sss:universal-puncture-moduli},
 and write for simplicity
$
\fM_i := \fM_{g_i,\ddata_i'}(\ainfty)
$
with the universal punctured  map $\mathfrak{f}_i\colon \pC_i \to \ainfty$. The product $\prod_{i \in [m]}\fM_i$  is the moduli of pre-stable punctured maps to $\cA$ with domains given by $m$ connected components labeled by $[m]$ with discrete data \eqref{eq:[m]-punctured-data}.

Denote by $\UM_{[m]}$ the subcategory of punctured maps in
$\prod_{i \in [m]}\fM_i$ of uniform maximal degeneracy, or the $\curlywedge$-configuration.
We obtain a tautological morphism
\begin{equation}\label{eq:forget-curlywedge}
\UM_{[m]} \to \prod_{i \in [m]}\fM_i
\end{equation}
The construction in \cite[\S 3]{CJRS18P} applies to the above setting with little change:

\begin{proposition}\label{prop:umd-punctured-stack}
\begin{enumerate}
\item The tautological morphism \eqref{eq:forget-curlywedge} is proper, log \'etale, and represented by log algebraic spaces locally of finite type.

\item If $m=1$ and $(g, \ddata')=(g_1, \ddata_1')$, then
  $\fM^{\circ}_{g,\ddata'}(\ainfty)$ is open and dense in
  $\UM_{g,\ddata'}(\ainfty) := \UM_{[1]}$, over which
  \eqref{eq:forget-curlywedge} is the identity.
  In particular, when $m=1$, \eqref{eq:forget-curlywedge} is
  birational, and $\UM_{g,\ddata'}(\ainfty)$ is represented by an fs
  log algebraic stack locally of finite type of pure dimension
  \eqref{eq:base-dimension}.
\end{enumerate}
\end{proposition}

\begin{proof}
  By Lemma \ref{lem:universal-punctured-map-moduli} (2), the universal punctured
  maps over $\fM^{\circ}_{g,\ddata'}(\ainfty)$ have uniform maximal
  degeneracy.
  Thus when $m=1$, \eqref{eq:forget-curlywedge} restricts to the
  identity over $\fM^{\circ}_{g,\ddata'}(\ainfty)$.
  The rest of the statement follows from an identical proof as in
  \cite[Theorem~3.15]{CJRS18P}.
  Note that the $\curlywedge$-configuration of degeneracies only
  concerns the generic points of irreducible components.
  Thus, punctured points play no role in the proof.
\end{proof}

\subsubsection{The moduli of punctured R-maps with the \texorpdfstring{$\curlywedge$}{wedge}-configuration}\label{sss:curlywedge-R-map-moduli}
Consider a collection of discrete data \eqref{eq:P-data} for punctured R-maps to a target $\punt \to \BC$
\begin{equation}\label{eq:[m]-punctured-R-data}
\left(g_i, \beta_i, \ddata_i = \{(\ogamma_{ij}, c_{ij})\}_{j=1}^{n_i}\right), \ \ \mbox{for } i \in [m].
\end{equation}
compatible with \eqref{eq:[m]-punctured-data}.
Write for simplicity $\SH_i = \SH_{g_i,\ddata_i}(\punt,\beta_i)$. We obtain a Cartesian diagram with strict vertical arrows
\begin{equation}\label{diag:UM-R-map}
\xymatrix{
\UH_{[m]} \ar[rr] \ar[d] &&  \prod_{i\in [m]} \SH_i \ar[d] \\
 \UM_{[m]} \ar[rr] && \prod_{i\in [m]} \fM_i
}
\end{equation}
with the left vertical arrow given by the tautological morphism \eqref{eq:take-log}.

The product $ \prod_{i\in [m]} \SH_i$ is the moduli of stable punctured R-maps to $\punt \to \BC$ with domains given by $m$ connected components labeled by $[m]$ with discrete data \eqref{eq:[m]-punctured-R-data}. By construction $\UH_{[m]}$ is the moduli of stable punctured maps in $\prod_{i\in [m]} \SH_i$ with $\curlywedge$-configurations.

For later use, consider the universal $\curlywedge$-configured punctured maps:
\begin{equation}\label{eq:universal-curlywedge-map}
f:= \sqcup_{i \in [m]} f_i \colon \pC = \sqcup_{i \in [m]}\pC_i \to \punt \ \ \ \mbox{and} \ \ \ \mathfrak{f}:= \sqcup_{i \in [m]} \mathfrak{f}_i\colon \mathfrak{C}^{\circ} := \sqcup_{i \in [m]}\mathfrak{C}^{\circ}_i \to \ainfty
\end{equation}
over $\UH$ and $\UM$ respectively. Here $f_i := f|_{\pC_i}$ and $\mathfrak{f}_i := \mathfrak{f}|_{\mathfrak{C}_i}$ denote the restrictions to the corresponding $i$-th components. Note that the maximal degeneracy of $e_{\max}(\mathfrak{f}) \in \Gamma(\ocM_{\UM_{[m]}},\UM_{[m]})$ pulls back to the maximal degeneracy $e_{\max}(f)$, see Notation \ref{not:extrem-degeneracy}.

For the purpose of constructing perfect obstruction theories, we
assume that \eqref{eq:PT-smooth} is smooth for the rest of this
section.
By \eqref{eq:tan-POT}, the right vertical arrow of
\eqref{diag:UM-R-map} admits a canonical perfect obstruction theory.
This pulls back to the {\em canonical perfect obstruction theory}
\begin{equation}\label{eq:product-POT}
\varphi_{\UH_{[m]}/\UM_{[m]}} \colon  \TT_{\UH_{[m]}/\UM_{[m]}} \to \bigoplus_{i\in [m]} \EE_i =: \EE_{\UH_{[m]}/\UM_{[m]}}
\end{equation}
of the left vertical arrow, where $\EE_{i} := \pi_{i,*} f^*_i \Omega^{\vee}_{\punt/\BC}$ with $\pi_{i} \colon \pC_i \to \UH_{[m]}$ the projection.

\subsection{The twisted superpotential}\label{ss:general-twisted-superpotential}

To construct the reduced perfect obstruction theory by modifying the
canonical theory \eqref{eq:product-POT}, we introduce an appropriate
{\em twisted superpotential} related to the superpotential defined
above.
This is similar to the case of log R-maps \cite[Section~3.4]{CJR21}.

\subsubsection{Expanded targets}\label{sss:target-exp}
Let $\cA_{\max} := \cA$ with the closed substack $\Delta_{\max} \subset \cA_{\max}$.
We obtain the following commutative diagram with Cartesian squares
\begin{equation}\label{diag:change-univ-target}
\xymatrix{
\infty_{\cA^{e,\circ}} \ar[d] \ar[r] & \infty_{\cA^{e}} \ar[d] \ar[r]^-{\fb} & \ainfty \times\Delta_{\max} \ar[d] \\
\cA^{e,\circ} \ar[r] & \cA^{e} \ar[r] & \cA\times\cA_{\max}.
}
\end{equation}
where the vertical arrows are strict closed embeddings, the bottom
right arrow is the log blow-up along the center
$\ainfty\times\Delta_{\max}$, and the bottom left arrow is the open
sub-stack obtained by removing the proper transform of
$\ainfty\times \cA_{\max}$.
In particular, all horizontal arrows are log \'etale. Let
$\cO(\ainfty)$ be the line bundle defining the log structure of
$\ainfty$ as in \eqref{eq:generic-rank1-DF}.
A straightforward local calculation shows that:

\begin{lemma}\label{lem:vb-target}
  The underlying morphism of
  $\infty_{\cA^{e,\circ}} \to \ainfty\times\Delta_{\max}$ is represented by
  the total space of the line bundle
  $\cO(\ainfty)^{\vee}\boxtimes\cO(\Delta_{\max})$ over
  $\ainfty\times\Delta_{\max}$.
\end{lemma}

Consider a target $\punt$ with the log structure induced by a morphism
$ \cO(\punt) \colon \punt {\longrightarrow} \ainfty $ as in \S
\ref{sss:P-target}.
The construction in \eqref{diag:change-univ-target} can be pulled back to the targets of punctured R-maps via the following Cartesian square:

\begin{equation}\label{diag:expand-R-target}
\xymatrix{
\punt_{e,\circ} \ar[rr] \ar[d] && \punt\times\Delta_{\max} \ar[rr] \ar[d] && \BC\times\Delta_{\max} \\
\infty_{\cA^{e,\circ}} \ar[rr] && \ainfty\times\Delta_{\max} &&
}
\end{equation}

Lemma \ref{lem:vb-target} implies:

\begin{corollary}\label{cor:vb-R-target}
The underlying morphism of $\infty_{e,\circ} \to \punt\times\Delta_{\max}$ is represented by the total space of the line  bundle $\cO(\punt)^{\vee}\boxtimes\cO(\Delta_{\max})$ over $\punt\times\Delta_{\max}$.
\end{corollary}

\subsubsection{The twisted superpotential}

\begin{definition}\label{def:superpotential}
A {\em twisted superpotential} of a target $\punt$ is a morphism of underlying stacks
\[
\tW \colon \cO(\punt)^{\vee}\boxtimes\cO(\Delta_{\max})  \to \uomega\boxtimes\cO(\ttwist\cdot \Delta_{\max})
\]
over $\BC\times\ul{\Delta}_{\max}$, for some positive integer $\ttwist$, called the {\em order} of $\tW$.
\end{definition}

Thanks to Corollary~\ref{cor:vb-R-target}, in what follows we may view a twisted superpotential $\tW$ as a
morphism of log stacks
\begin{equation}\label{eq:twisted-potential}
\tW \colon \punt_{e,\circ}  \to \uomega\boxtimes\cO(\ttwist\cdot \Delta_{\max})
\end{equation}
over $\BC\times{\Delta}_{\max}$, with the target equipped with the log  structure pulled back from $\BC\times\Delta_{\max}$.

Alternatively, a twisted superpotential can be viewed as a $\CC^*_{\omega}$-equivariant function as follows. Consider the pull-backs
\[
\CC_{\omega} \cong \uomega\times_{\BC}\spec \bk \ \ \mbox{and} \ \ \CC_{\max} \cong \cO(\Delta_{\max})\times_{\Delta_{\max}}\spec \bk.
\]
Denote by  $\cO(\punt_{\bk})^{\vee} := \cO(\punt)^{\vee} \times_{\BC}\spec \bk$ and consider the pull-back
 \[
 \punt_{e,\circ,\bk} := \punt_{e,\circ}\times_{\BC\times\Delta_{\max}}\spec\big(\NN \to \bk\big)
 \]
 along the strict smooth cover from a standard log point $\spec \big(\NN \to \bk\big) \to \BC\times\Delta_{\max}$. Note that by Corollary \ref{cor:vb-R-target}, we have a line bundle structure with the $\CC^*_{\omega}$-action
 \begin{equation}\label{eq:base-change-punt-e}
 \ul{\punt}_{e,\circ,\bk} = \cO(\punt_{\bk})^{\vee}\otimes\CC_{\max} \to \punt_{\bk} := \punt\times_{\BC}\spec \bk.
 \end{equation}
The twisted superpotential $\tW$ then pulls back to a $\CC^*_{\omega}$-equivariant function
 \begin{equation}\label{eq:rigid-potential}
 \tW_{\bk} \colon \punt_{e,\circ,\bk} \to \CC_{\omega}\otimes\CC_{\max}^{\otimes\ttwist}
 \end{equation}
 where the right hand side has the log structure pulled back from $\spec \big(\NN \to \bk\big)$.

\subsubsection{Twisted superpotentials versus superpotentials}
Given a superpotential $W$, there is an {\em associated twisted
superpotential} $\tW$ constructed as follows.
The diagonal
$p \colon \punt_{e, \circ} \to \punt_{e, \circ} \times_{\BC \times
  \ul\Delta_{\max}} \punt_{e, \circ}$ is naturally viewed as a
section of the line bundle
$\left(\cO(\infty)^\vee \boxtimes \cO(\Delta_{\max})
\right)|_{\punt_{e,\circ}}$.
It is also known as the \emph{tautological section} of
$\left(\cO(\infty)^\vee \boxtimes \cO(\Delta_{\max})
\right)|_{\punt_{e,\circ}}$, and vanishes precisely along the zero
section of the total space of
$\cO(\infty)^\vee \boxtimes \cO(\Delta_{\max}) \to \infty \times
\Delta_{\max}$.
Then the product
\begin{equation*}
  W|_{\punt_{e, \circ}} \cdot p^{\ttwist},
\end{equation*}
is a section of the line bundle
\begin{equation*}
  \left(\uomega \otimes \cO(\ttwist \infty) \right)|_{\punt_{e,\circ}} \otimes \left(\cO(\infty)^\vee \boxtimes \cO(\Delta_{\max})\right)|_{\punt_{e,\circ}}^{\otimes \ttwist}
  \cong \left( \uomega \boxtimes \cO(\ttwist \cdot \Delta_{\max}) \right)|_{\punt_{e,\circ}},
\end{equation*}
and hence corresponds to a morphism $\tW$ as in
Definition~\ref{def:superpotential}. Indeed, any twisted superpotential can be obtained this way:

\begin{lemma}\label{lem:super-potential-equivalence}
  For every twisted superpotential $\tW$ as in
  Definition~\ref{def:superpotential}, there is a superpotential $W$,
  such that $\tW$ is obtained from $W$ by the above construction.
\end{lemma}

\begin{proof}
  Base-changing along $\spec \bk \to \ul\Delta_{\max}$, we see
  that $\tW$ is equivalent to a $\CC^*$-equivariant morphism
  \begin{equation*}
    \cO(\infty)^\vee \otimes \CC_{\max} \to \uomega \otimes \CC_{\max}^{\otimes\ttwist}
  \end{equation*}
  over $\BC \cong \BC \times \spec \bk$, with the $\CC^*$-action given by the scaling of $\CC_{\max}$.
  This data is clearly equivalent to a $\CC^*$-equivariant morphism

  \begin{equation*}
  \tW' \colon  \cO(\infty)^\vee \to \uomega
  \end{equation*}
  over $\BC$ with the analogous weights.

  Pulling back first along $\spec \bk \to \BC$, then along an \'etale
  cover $U = \spec(A) \to \punt_{\bk}$, we obtain trivializations of
  both $\cO(\infty)^\vee$ and $\uomega$ over $U$.
  Now $\tW'|_U$ may be viewed as a morphism $U \times \A^1 \to \A^1$,
  or equivalently a polynomial $w' \in A[y]$.
  The homogeneity implies that $w'$ is of the form
  $w' = y^{\ttwist} \cdot w''$, where $w'' \in A$.
  The restriction $p|_U$ of the tautological section $p$ may similarly
  be identified with a section $p' \in A[y]$ of the form
  $p' = y \cdot p''$, where $p'' \in A^{\times}$ (noting that the
  tautological section only vanishes along the zero section).
  Thus, we can write $w' = w''' \cdot (p')^{\ttwist}$ for
  $w''' = w'' \cdot (p'')^{-\ttwist}$.
  The descent datum for $p$ then yields the descent datum for gluing
  $w'''$ to a global section $W$ of
  $\uomega \otimes \cO(\ttwist \infty)$, and we have the desired
  factorization $\tW = W \cdot p^{\ttwist}$.
\end{proof}

\subsubsection{Critical loci and transversality}
\label{sss:critical-loci-transverse}

In this section, we introduce the critical locus of a (twisted)
superpotential, as well as the notion of a transverse superpotential.
For this purpose, it is important to view a twisted superpotential $\tW$
as a morphism of log stacks \eqref{eq:twisted-potential}.
This leads to the differential of log tangent bundles:
\[
\diff \tW \colon T_{\punt_{e,\circ}/\BC\times\Delta_{\max}} \to \tW^*T_{\uomega\boxtimes\cO(\ttwist\cdot \Delta_{\max})/\BC\times\Delta_{\max}} \cong \big( \uomega\boxtimes\cO(\ttwist\cdot \Delta_{\max}) \big)|_{\punt_{e,\circ}}
\]
Note that the left horizontal arrow in \eqref{diag:expand-R-target} is log \'etale. Thus, the following
\[
T_{\punt_{e,\circ}/\BC\times\Delta_{\max}} \cong \Omega^{\vee}_{\punt\times\Delta_{\max}/\BC\times\Delta_{\max}}|_{\punt_{e,\circ}} \cong \Omega^{\vee}_{\punt/\BC}|_{\punt_{e,\circ}}
\]
is a vector bundle over $\punt_{e,\circ}$. Hence we have
\begin{equation}\label{eq:diff-potential}
\diff \tW \colon \Omega^{\vee}_{\punt/\BC}|_{\punt_{e,\circ}} \to \big(\uomega\boxtimes\cO(\ttwist\cdot \Delta_{\max}) \big)|_{\punt_{e,\circ}}.
\end{equation}

\begin{definition}\label{def:critical-locus}
The {\em critical locus} of $\tW$, denoted by $\crit(\tW)$ is the
strict closed substack of $\punt_{e,\circ}$ along which $\diff \tW$
vanishes.
The superpotential $\tW$ is said to have {\em proper critical locus}
if $\crit(\tW) \to \BC\times\Delta_{\max}$ is proper.
\end{definition}

From the $\CC^*_{\omega}$-equivariant point of view, the product
\begin{equation}\label{eq:rigid-critical-loci}
  \crit(\tW_{\bk}) = \crit(\tW)\times_{\BC\times\Delta_{\max}}\spec \big(\NN \to \bk\big),
\end{equation}
is the critical locus of the $\CC^*_{\omega}$-equivariant function
$\tW_{\bk}$ as in \eqref{eq:rigid-potential}. In particular $\tW$ has
proper critical locus iff $\crit(\tW_{\bk})$ is proper.
Next, we relate the properness of twisted superpotentials to the
following transversality of the corresponding superpotential $W$.

\begin{definition}\label{def:transverse-superpotential}
  The \emph{critical locus} of a superpotential $W$ is the
  closed substack $\crit(W)$ of $\ul\punt$ such that for any smooth
  cover $U \to \ul\punt$ and
  $(\uomega \otimes \cO(\ttwist \infty))|_U \cong \cO_U$, the critical
  locus of the regular function $W|_U\colon U \to \AA^1$ is
  $\crit(W)|_U$.
  We say that a superpotential $W$ is \emph{transverse (to the zero
    section)} if the critical locus $\crit(W)$ is empty.
\end{definition}

\begin{proposition}\label{prop:proper-critical-locus}
  Let $W$ be a superpotential, and let $\tW$ be the corresponding
  twisted superpotential.
  Then the following are equivalent:
  \begin{enumerate}
  \item The superpotential $W$ is transverse.
  \item $\tW$ has proper critical locus.
  \item Set-theoretically, $\crit(\tW)$ is supported along the zero section of $\cO(\infty)^\vee \boxtimes \cO(\Delta_{\max})$.
  \end{enumerate}
\end{proposition}
\begin{proof}
  We start by making the differential $d\tW$ more explicit.
  For this, we introduce local coordinates.
  By \eqref{diag:change-univ-target} and \eqref{diag:expand-R-target},
  the characteristic sheaf $\ocM_{\punt_{e,\circ,\bk}}$ has global
  generators
  \[
    \delta_1, \delta_2 \in \Gamma(\punt_{e,\circ,\bk},\ocM_{\punt_{e,\circ,\bk}})\cong \NN^2
  \]
  such that $\delta_1$ is the pull-back of of the unique generator of
  $\Gamma(\punt,\ocM_{\punt}) = \NN$, and $\delta_2$ is trivial away
  from the zero section of \eqref{eq:base-change-punt-e}.
  Denote by $e$ the generator of
  $\Gamma(\Delta_{\max}, \ocM_{\Delta_{\max}}) \cong \NN$.
  Then we have $e = \delta_1 + \delta_2$, where $e$ also denotes its
  pull-back over $\punt_{e,\circ,\bk}$.

  Selecting a local chart of $\cM_{\punt_{e,\circ,\bk}}$, we may lift
  $e, \delta_1$ and $\delta_2$ to the corresponding local sections in
  $\cM_{\punt_{e,\circ,\bk}}$.
  Recall the structure morphism
  $\alpha \colon \cM_{\punt_{e,\circ,\bk}} \to \cO$.
  Then $\alpha(\delta_2)$ is identified with a local section of the
  line bundle $\cO(\punt_{\bk})^{\vee}\otimes\CC_{\max}$, and is up to units
  the pullback of the tautological section $p$.

  Let $t$ be a coordinate of $\CC_\omega \otimes \CC_{\max}^{\ttwist}$.
  Recall that by Lemma~\ref{lem:super-potential-equivalence}
  \begin{equation}
    \label{eq:local-potential}
    \tW_{\bk}^*(t) = w \cdot \alpha(\delta_2)^{\ttwist},
  \end{equation}
  where $w$ is a local section corresponding to $W$.
  Differentiating $\tW_{\bk}$ as in \eqref{eq:rigid-potential}, we compute
  \[
    \big(\diff \tW_{\bk}\big)^{\vee}(\diff t) = \diff \big(w \cdot \alpha(\delta_2)^{\ttwist} \big) = \alpha(\delta_2)^{\ttwist} \diff w  + \ttwist w \cdot \alpha(\delta_2)^{\ttwist} \frac{\diff \alpha(\delta_2)}{\alpha(\delta_2)} = \alpha(\delta_2)^{\ttwist} \diff w  + \ttwist w \cdot \alpha(\delta_2)^{\ttwist} \diff \log \delta_2.
  \]
  Note that $\diff \log \delta_2$ is a local generator of the sheaf of
  log differentials $\Omega_{\punt/\BC}|_{\punt_{e,\circ,\bk}}$.

  Hence by \eqref{eq:rigid-critical-loci}, the critical locus $\crit(\tW)$ contains the zero
  section $\Div(\delta_2)$.
  Furthermore, at each geometric point of the critical locus, at least
  one of $\alpha(\delta_2)$ and $(w, \diff w)$ must vanish.
  This proves that $(1)$ and $(3)$ are equivalent.

  Note that $(3)$ implies $(2)$ since $\crit(\tW)$ is a closed
  substack of the $\ttwist$th infinitesimal neighborhood of the zero
  section.

  Finally, suppose $\crit(\tW)$ contains a point
  away from $\Div(\delta_2)$. By \eqref{eq:rigid-critical-loci},  $\crit(\tW_{\bk})$
  contins a point away from the zero section of \eqref{eq:base-change-punt-e}.
  Hence it contains some fiber of the line bundle \eqref{eq:base-change-punt-e}
  by the $\CC^*_{\omega}$-equivariance of
  $\tW_{\bk}$. This implies the non-properness of   $\crit(\tW_{\bk})$, hence non-properness of  $\crit(\tW)$.
  Therefore, $(2)$ implies $(3)$.
\end{proof}

\subsection{Cosection of the canonical theory}\label{ss:relative-cosection}

We now pull back the (differential of the) twisted superpotential to
the moduli of punctured R-maps to obtain a cosection of the canonical
theory \eqref{eq:product-POT}, leading to the definition of the
reduced theory in the following section.
The key is the $\curlywedge$-configuration defined in
\S\ref{ss:category-for-reduced-theory}.

\subsubsection{Expansions along uniform maximal degeneracy}

The following condition can be viewed as the analog of the compact
type loci of \cite{CJR21}, and is necessary in the construction of the
reduced theory below.

\begin{assumption}\label{as:reduced-punctured-cycle}
The contact orders at all marked points are negative.
\end{assumption}

Consider a family of punctured maps
$\mathfrak{f}\colon \pC \to \ainfty$ over $S$ with uniform maximal
degeneracy.
Here $\pC$ is possibly disconnected.
Let $e_{\max} \in \Gamma(S, \ocM_S)$ be the maximal degeneracy.
It induces a canonical morphism $S \to \Delta_{\max}$ as in
\eqref{eq:sub-rank1-DF}, such that
$\bar{\mathfrak{f}}^{\flat}(e) = e_{\max}$ for the generator
$e \in \Gamma(\ocM_{\Delta_{\max}})$.
Thus, we have
\[
\cO(\Delta_{\max})|_{S} \cong \cO_S(e_{\max}).
\]
Because of the above isomorphism, in what follows we are allowed to use either $\cO(\Delta_{\max})$ or $\cO(e_{\max})$.

\begin{lemma}\label{lem:exp-map}
  Suppose that $\mathfrak{f}$ satisfies
  Assumption~\ref{as:reduced-punctured-cycle}, and has uniform maximal
  degeneracy.
  Then there is a canonical morphism $\mathfrak{f}_{e_{\max}}$ making
  the following diagram commutative
  \[
    \xymatrix{
      && \pC \ar@{-->}[lld]_-{\mathfrak{f}_{e_{\max}}} \ar[d]^{(\mathfrak{f},e_{\max})} \ar[rr] && S \ar[d]^{e_{\max} = e|_S} \\
      \infty_{\cA^{e_{\max},\circ}} \ar[rr]_-{\fb} && \ainfty\times\Delta_{\max} \ar[rr] &&\Delta_{\max}
    }
  \]
  where the bottom is given by \eqref{diag:change-univ-target}, and
  $\infty_{\cA^{e_{\max},\circ}} := \infty_{\cA^{e,\circ}}$ with
  $e_{\max}$ in the superscript to emphasize
  $\bar{\mathfrak{f}}^{\flat}(e) = e_{\max}$.
  Furthermore, identify the underlying of $\infty_{\cA^{e,\circ}}$
  with the line bundle
  $\cO(\ainfty)^{\vee}\boxtimes\cO(\Delta_{\max})$ by Lemma
  \ref{lem:vb-target}.
  Over any geometric fiber, for each irreducible component
  $\cZ \subset \pC$ with associated degeneracy $e_{\cZ}$, we have:
  \begin{enumerate}
  \item $\mathfrak{f}_{e_{\max}}(\cZ)$ lies in the zero section of
    $\cO(\ainfty)^{\vee}\boxtimes\cO(\Delta_{\max})$ if
    $e_{\cZ} \poleq e_{\max}$ and $e_{\cZ} \neq e_{\max}$.
  \item $\mathfrak{f}_{e_{\max}}(\cZ)$ does not entirely lie in the
    zero section of $\cO(\ainfty)^{\vee}\boxtimes\cO(\Delta_{\max})$
    if $e_{\cZ} = e_{\max}$.
  \item All markings are mapped to the zero section of
    $\cO(\ainfty)^{\vee}\boxtimes\cO(\Delta_{\max})$ via
    $\mathfrak{f}_{e_{\max}}$.
\end{enumerate}
\end{lemma}
\begin{proof}
  The proof is identical to \cite[Lemma 3.19]{CJRS18P}.
  For the reader's convenience, we briefly recall some key points.
  Let $\delta$ be the generator of
  $\Gamma(\ainfty,\cM_{\ainfty}) \cong \NN$.
  Then, the arrow $\pC \to \ainfty\times\Delta_{\max}$ is defined by
  the pair of elements $(\mathfrak{f}^{\flat}(\delta), e_{\max})$ in
  $\ocM_{\pC}$.
  By the blow-up construction in \eqref{diag:change-univ-target}, the
  lift $\mathfrak{f}_{e_{\max}}$ if it exists, is defined by
  $\mathfrak{f}_{e_{\max}}^{\flat} \colon (e_{\max} - \delta) \mapsto
  \big(e_{\max} - \mathfrak{f}^{\flat}(\delta)\big) \in
  \ocM_{\pC}^{gp}$ provided that
  \begin{equation}\label{eq:expand-condition}
    \big(e_{\max} - \mathfrak{f}^{\flat}(\delta)\big) \in \ocM_{\pC}.
  \end{equation}
  Now the same proof as in \cite[Lemma 3.19]{CJRS18P} shows that
  Assumption \ref{as:reduced-punctured-cycle} together with the
  maximal degeneracy $e_{\max}$ imply that
  $\mathfrak{f}^{\flat}(\delta) \poleq e_{\max}$ in
  $\Gamma(\pC, \cM_{\pC})$, hence \eqref{eq:expand-condition}.

  Finally, note that the zero section of
  $\cO(\ainfty)^{\vee}\boxtimes\cO(\Delta_{\max})$ is the locus where
  $(e - \delta) \neq 0$.
  We check (1) and (2) over a generic point of $\cZ$ using the
  observation that $0 \neq e_{\max} - e_{\cZ}$ unless $\cZ$ is
  maximally degenerated.
  A similar analysis using the local description at each marking,
  gives (3).
\end{proof}

The above lemma pulls back to punctured R-maps via \eqref{diag:expand-R-target} as follows:

\begin{corollary}\label{cor:exp-R-map}
  Suppose $f\colon \pC \to \punt$ is a punctured R-map over $S$ with
  uniform maximal degeneracy $e_{\max}$ and satisfying Assumption
  \ref{as:reduced-punctured-cycle}.
  We then obtain a commutative diagram
\[
\xymatrix{
&& && \pC \ar@/_1.2pc/[lllld]_-{f_{e_{\max}}} \ar@/_.5pc/[lld]^{(f,e)} \ar[rr] \ar[d]^-{(\omega_{\log},e_{\max})} && S \ar[d]^{e_{\max}} \\
\infty_{e_{\max},\circ} \ar[rr] && \punt\times\Delta_{\max} \ar[rr]  &&  \BC\times\Delta_{\max}  \ar[rr] &&\Delta_{\max}
}
\]
where $\infty_{e_{\max},\circ} := \infty_{e,\circ}$ with the subscript
$e_{\max}$ to emphasize the pull-back $e|_S = e_{\max}$.
Furthermore, by identifying the underlying of
$\infty_{e_{\max},\circ}$ with the line bundle
$\cL_{S}\boxtimes\cO(\Delta_{\max})$ via Corollary
\ref{cor:vb-R-target}, over any geometric fiber, for each irreducible
component $\cZ \subset \pC$ with associated degeneracy $e_{\cZ}$, we
have:
\begin{enumerate}
\item $f_{e_{\max}}(\cZ)$ lies in the zero section of
  $\cO(\punt)^{\vee}\boxtimes\cO(\Delta_{\max})$ if
  $e_{\cZ} \poleq e_{\max}$ and $e_{\cZ} \neq e_{\max}$.
\item $f_{e_{\max}}(\cZ)$ does not lie in the zero section of
  $\cO(\punt)^{\vee}\boxtimes\cO(\Delta_{\max})$ if
  $e_{\cZ} = e_{\max}$.
\item All markings are mapped to the zero section of
  $\cO(\punt)^{\vee}\boxtimes\cO(\Delta_{\max})$ via $f_{e_{\max}}$.
\end{enumerate}
\end{corollary}
\begin{proof}
Let $\mathfrak{f}$ be the composition $\pC \to \punt\to\ainfty$. We have the following commutative diagram of solid arrows with the square Cartesian:
\begin{equation}\label{diag:pull-back-expansion}
\xymatrix{
\pC \ar@/^1pc/[rrrrd]^{(f, e_{\max})} \ar@/_1pc/[rrdd]_-{\mathfrak{f}_{e_{\max}}} \ar@{-->}[rrd]^{f_{e_{\max}}} && && && \\
&& \infty_{e_{\max},\circ} \ar[rr] \ar[d] && \punt\times\Delta_{\max} \ar[d] \ar[rr] && \BC\times\Delta_{\max} \\
&& \infty_{\cA^{e_{\max},\circ}} \ar[rr] && \ainfty\times\Delta_{\max} &&
}
\end{equation}
This defines $f_{e_{\max}}$. The rest of the statement follows from Lemma \ref{lem:exp-map}.
\end{proof}

\subsubsection{The relative cosection}

We will employ the notations of \S \ref{sss:curlywedge-R-map-moduli},
assume Assumption \ref{as:reduced-punctured-cycle}, and use $e_{\max}$
for both $e_{\max}(f)$ and $e_{\max}(\mathfrak{f})$.
Consider a twisted superpotential $\tW$ as in
\eqref{eq:twisted-potential}, fitting in the commutative diagram
\[
\xymatrix{
\punt_{e_{\max}, \circ} \ar[rr]^-{\tW} \ar[rd] && \uomega\boxtimes\cO(\ttwist\Delta_{\max}) \ar[ld] \\
&\BC\times\Delta_{\max}&
}
\]
Then \eqref{diag:pull-back-expansion} gives a commutative diagram
\begin{equation}\label{eq:bullet-universal-maps}
\xymatrix{
&&&& \uomega\boxtimes\cO(\ttwist\Delta_{\max}) \ar[d] \\
\UH_{[m]} \ar[d] & \pC \ar[l] \ar[r]^-{f_{e_{\max}}} \ar[d] & \punt_{e_{\max},\circ} \ar[d] \ar[r] \ar@/^1pc/[rru]^-{\tW} & \punt\times\Delta_{\max} \ar[r] \ar[d] & \BC\times\Delta_{\max} \ar[d] \\
\UM_{[m]} \ar[r] & \mathfrak{C}^{\circ} \ar[r]_-{\mathfrak{f}_{e_{\max}}} & \infty_{\cA^{e_{\max},\circ}} \ar[r] & \ainfty\times\Delta_{\max} \ar[r]  & \Delta_{\max}
}
\end{equation}
where $f_{e_{\max}}$ and $\mathfrak{f}_{e_{\max}}$ are obtained by applying Lemma \ref{lem:exp-map} and Corollary \ref{cor:exp-R-map} to the universal maps in  \eqref{eq:universal-curlywedge-map} respectively.

Pulling back $\diff \tW$ as in \eqref{eq:diff-potential} along $f_{e_{\max}}$, we obtain
\begin{equation}\label{eq:dW-over-curve}
f_{e_{\max}}^* \diff \tW \colon  f^*\Omega^{\vee}_{\punt/\BC} \cong f_{e_{\max}}^* \Omega^{\vee}_{\punt/\BC}|_{\punt_{e_{\max},\circ}} \to \omega^{\log}_{\pC/\UH_{[m]}}\otimes f_{e_{\max}}^*\cO(\ttwist\Delta_{\max}).
\end{equation}

The following lemma, relying on Assumption \ref{as:reduced-punctured-cycle}, is crucial for constructing the reduced theory.
\begin{lemma}\label{lem:factor-through-omega}
There is a canonical factorization
\[
\xymatrix{
  f^* \Omega^{\vee}_{\punt/\BC} \ar@/_.5pc/[rd] \ar[rr]^-{f_{e_{\max}}^* \diff \tW} && \omega^{\log}_{\pC/\UH_{[m]}}\otimes f_{e_{\max}}^*\cO(\ttwist\Delta_{\max}) \\
&\omega_{\pC/\UH_{[m]}}\otimes f_{e_{\max}}^*\cO(\ttwist\Delta_{\max}) \ar@/_.5pc/[ru]&
}
\]
We will also denote the left skewed arrow by $f_{e_{\max}}^* \diff \tW$ for simplicity.
\end{lemma}
\begin{proof}
  By Corollary
  \ref{cor:exp-R-map}, $f_{e_{\max}}^* \diff \tW$ degenerates along
  all markings, which implies the factorization as in the statement.
\end{proof}

Further pushing forward along the projection $\pi \colon \pC \to \UH_{[m]}$ and applying the above lemma, we obtain
\begin{equation}\label{eq:complex-cosection-component}
\pi_{*} \big( f_{e_{\max}}^* \diff \tW \big) \colon \EE_{\UH_{[m]}/\UM_{[m]}} \to  \pi_{*}\omega_{\pC/\UH_{[m]}}\otimes \cO(\ttwist\Delta_{\max})|_{\UH_{[m]}},
\end{equation}
where $\EE_{\UH_{[m]}/\UM_{[m]}}$ is as in \eqref{eq:product-POT}. Taking $H^1$, we obtain the {\em relative cosection}:
\begin{equation}\label{eq:relative-cosection}
\sigma_{\UH_{[m]}/\UM_{[m]}} \colon H^1(\EE_{\UH_{[m]}/\UM_{[m]}}) \to \oplus_{i \in [m]} \cO(\ttwist\Delta_{\max})|_{\UH_{[m]}} \to \cO(\ttwist\Delta_{\max})|_{\UH_{[m]}},
\end{equation}
where the first arrow is $R^1\pi_{*} \big( f_{e_{\max}}^* \diff \tW \big)$ and the second arrow is taking the sum of each component. For later use, define the {\em boundary complex}
\begin{equation}\label{eq:boundary-complex}
\FF_{\UH_{[m]}/\UM_{[m]}} :=  \cO(\ttwist\Delta_{\max})|_{\UH_{[m]}}[-1]
\end{equation}
Then the relative cosection \eqref{eq:relative-cosection} takes the form
\begin{equation}\label{eq:relative-cosection-complex}
\sigma_{\UH_{[m]}/\UM_{[m]}} \colon H^1(\EE_{\UH_{[m]}/\UM_{[m]}}) \to \FF_{\UH_{[m]}/\UM_{[m]}}[1].
\end{equation}

The following property is crucial in the construction of the reduced theory:

\begin{proposition}\label{prop:cosection-surjective}
Suppose $\tW$ has proper critical locus. Then $\sigma_{\UH_{[m]}/\UM_{[m]}}$ is surjective.
\end{proposition}

\begin{proof}
  It suffices to check the statement over each geometric point $s \to \UH_{[m]}$. The fiber of \eqref{eq:relative-cosection} over $s$ is
  \begin{equation}\label{eq:cosection-surjective}
  \oplus_{i \in [m]}H^1(f^*_{i,s} \Omega^{\vee}_{\punt/\BC}) \to \oplus_{i \in [m]}\cO(\ttwist\Delta_{\max})|_{s} \to \cO(\ttwist\Delta_{\max})|_{s},
  \end{equation}
  where $f_{i,s}$ is the fiber of $f_i$ as in \eqref{eq:universal-curlywedge-map}.
  Note that at least one of the connected component of $\pC_{s}$, say $\pC_{k,s}$, has a maximally degenerate component. We will show that the $k$-th component   of the left arrow in \eqref{eq:cosection-surjective}
  \begin{equation}\label{eq:cosection-component-surjective}
  \sigma_{k} \colon H^1(f^*_{k,s} \Omega^{\vee}_{\punt/\BC}) \to \cO(\ttwist\Delta_{\max})|_{s}
  \end{equation}
 is surjective, hence the desired surjectivity as in the statement.
  Below, we follow the same line of argument as in \cite[Proposition 3.18]{CJR21}.

  Applying Serre duality to \eqref{eq:cosection-component-surjective}
  and taking duals, we have
\[
\sigma^{\vee}_{k,s} \colon H^0\big(\cO(-\ttwist\Delta_{\max})|_{\pC_{k,s}}\big) \to H^0(\omega_{\pC_{k,s}}\otimes f^*_{k,s} \Omega_{\punt/\BC}).
\]
Thus $\sigma_k$ is surjective iff $\sigma^{\vee}_{k}$ is injective. Note that $\cO(-\ttwist\Delta_{\max})|_{\pC_{k,s}} \cong \cO_{\pC_{k,s}}$. Thus it remains to show that $\sigma^{\vee}_{k,s} \in H^0(\omega_{\pC_{k,s}}\otimes f^*_{k,s} \Omega_{\punt/\BC})$ is non-zero.

Taking duals, we may write $\sigma^{\vee}_{k,s}$ as a morphism:
 \[
 \omega^{\vee}_{\pC_{k,s}}\otimes f^*_{k,s} \Omega^{\vee}_{\punt/\BC} \to \cO(\ttwist\Delta_{\max})|_{\pC_{k,s}},
 \]
 which is equivalent to
 \begin{equation}\label{eq:cosection-surject-over-curve}
 \big(f_{e_{\max}}^* \diff \tW \big)|_{\pC_{k,s}} \colon f^*_{k,s} \Omega^{\vee}_{\punt/\BC} \to \omega_{\pC_{k,s}}\otimes\cO(\ttwist\Delta_{\max})|_{\pC_{k,s}}
 \end{equation}
given by the left skewed arrow in Lemma \ref{lem:factor-through-omega}.

 By Proposition \ref{prop:proper-critical-locus} and Corollary \ref{cor:exp-R-map} (2), the image of the maximally degenerate component of $\pC_{k,s}$ via $f_{e_{\max}}$ is not contained in the critical locus $\crit(\tW)$. Thus
$ \big(f_{e_{\max}}^* \diff \tW_{\max} \big)|_{\pC_{k,s}}$ is non-trivial along the maximally degenerate component. This implies that $\sigma^{\vee}_{k,s} \neq 0$, hence the surjectivity of $\sigma_{k,s}$.
\end{proof}

%%----------------
\subsection{The reduced perfect obstruction theory}\label{ss:red-pot}

\subsubsection{The statement}
Consider the composition of complexes
\begin{equation}\label{eq:R-complex-cosection}
\xymatrix{
\EE_{\UH_{[m]}/\UM_{[m]}} \ar[rr] && H^1(\EE_{\UH_{[m]}/\UM_{[m]}})[-1] \ar[rr]^-{\sigma_{\UH_{[m]}/\UM_{[m]}}} && \FF_{\UH_{[m]}/\UM_{[m]}},
}
\end{equation}
induced by \eqref{eq:relative-cosection-complex}. This leads to a triangle
\begin{equation}\label{eq:red-POT}
\EE^{\red}_{\UH_{[m]}/\UM_{[m]}} \longrightarrow \EE_{\UH_{[m]}/\UM_{[m]}} \longrightarrow \FF_{\UH_{[m]}/\UM_{[m]}} \stackrel{[1]}{\longrightarrow}.
\end{equation}

\begin{theorem}\label{thm:red-POT}
There is a canonical factorization of $\varphi_{\UH_{[m]}/\UM_{[m]}}$ through a perfect obstruction theory $\varphi^{\red}_{\UH_{[m]}/\UM_{[m]}}$ of $\UH_{[m]} \to \UM_{[m]}$ as follows:
\begin{equation}\label{diag:POT-factorization}
\xymatrix{
\TT_{\UH_{[m]}/\UM_{[m]}} \ar[rr]^{\varphi_{\UH_{[m]}/\UM_{[m]}}} \ar[rd]_{\varphi^{\red}_{\UH_{[m]}/\UM_{[m]}}} && \EE_{\UH_{[m]}/\UM_{[m]}} \\
&\EE^{\red}_{\UH_{[m]}/\UM_{[m]}} \ar[ru]&
}
\end{equation}
We call $\varphi^{\red}_{\UH_{[m]}/\UM_{[m]}}$ the {\em reduced perfect obstruction theory} of $\UH_{[m]} \to \UM_{[m]}$.
\end{theorem}

When $m=1$, by the pure dimensionality of Proposition \ref{prop:umd-punctured-stack} (2) we obtain the {\em reduced} and canonical virtual cycles associated to $\varphi^{\red}_{\UH_{[1]}/\UM_{[1]}}$ and $\varphi_{\UH_{[1]}/\UM_{[1]}}$, denoted by $[\UH_{[1]}]^{\red}$ and $[\UH_{[1]}]^{\vir}$ respectively. By \eqref{eq:red-POT}, these two cycles are related as follows:

\begin{corollary}\label{cor:red-canonical-vir}
$[\UH_{[1]}]^{\vir} = c_1\left( \cO(\ttwist\Delta_{\max})|_{\SH} \right) \cap [\UH_{[1]}]^{\red}$.
\end{corollary}
\begin{proof}
  This follows from \cite[Corollary 4.9]{Ma12}.
  Indeed, \eqref{eq:red-POT} gives a compatible triple of perfect
  obstruction theories where the zero map
  $0 = \TT_{\UH_{[1]}/\UH_{[1]}} \to \FF_{\UH_{[1]}/\UM_{[1]}}$ is the
  perfect obstruction theory of the identity
  $\UH_{[1]} \to \UH_{[1]}$.
\end{proof}

The reduced virtual cycles in the $m>1$ case will be studied in
\S\ref{sec:reduction-to-connected-inv} in the context of further
modifications of $\UH_{[m]}$, to aid the reduction of the disconnected
to the connected theory
(Theorem~\ref{thm:intro-disconnected-to-connected}).

Next, we prove Theorem \ref{thm:red-POT}, following the method of
\cite[\S 6]{CJR21}.

\subsubsection{The twisted Hodge bundle}
Consider the universal curve $\pi_{\UM_{[m]}} \colon \mathfrak{C}^{\circ} \to \UM_{[m]}$, and the following line bundle over $\mathfrak{C}^{\circ}$
\[
\tomega := \omega_{\mathfrak{C}^{\circ}/\UM_{[m]}}\otimes\cO(\ttwist\Delta_{\max})|_{\mathfrak{C}^{\circ}}.
\]
Pushing forward along $\pi_{\UM_{[m]}}$ and applying the projection formula, we obtain
\begin{equation}
\fH := R^0\pi_{\UM_{[m]},*}\big( \tomega \big) \cong R^0\pi_{\UM_{[m]},*}(\omega_{\mathfrak{C}^{\circ}/\UM_{[m]}})\otimes\cO(\ttwist\Delta_{\max})|_{\UM_{[m]}},
\end{equation}
called the {\em twisted Hodge bundle} over $\UM_{[m]}$.
We may view $\fH$ as the log stack with the strict morphism to $\UM_{[m]}$. Recall the puncturing $\mathfrak{C}^{\circ} \to \fC$. Consider the projection
\[
\pi_{\fH}\colon \mathfrak{C}_{\fH} := \mathfrak{C}\times_{\UM_{[m]}}\fH \to \fH.
\]
Note that $\fH$ is the direct image cone over $\UM_{[m]}$ as in \cite[Definition 2.1]{ChLi12}, which parameterizes sections of $\tomega \to \mathfrak{C}$. Let $\rho_{\fH}$ be the universal section over $\fH$ fitting in the following commutative triangle
\begin{equation}\label{eq:h-object}
\xymatrix{
&& \uomega\boxtimes\cO(\ttwist\Delta_{\max}) \ar[d] \\
\mathfrak{C}_{\fH} \ar[rr] \ar@/^1pc/[rru]^-{\rho_{\fH}} && \BC\times\Delta_{\max}
}
\end{equation}

Note that \eqref{eq:bullet-universal-maps} defines a commutative triangle
\begin{equation}\label{eq:R-section}
\xymatrix{
&& \uomega\boxtimes\cO(\ttwist\Delta_{\max}) \ar[d] \\
\pC \ar[rr] \ar@/^1pc/[rru]^-{\rho_{\UH_{[m]}} := \tW\circ f_{e_{\max}}} && \BC\times\Delta_{\max}
}
\end{equation}
This leads to a canonical morphism over $\UM_{[m]}$
\begin{equation}\label{eq:R-to-H}
\UH_{[m]} \to \fH
\end{equation}
along which $\rho_{\fH}$ pulls back to $\rho_{\UH_{[m]}}$.

By \cite[Proposition 2.5]{ChLi12}, $\fH \to \UM_{[m]}$ has a perfect obstruction theory by deforming $\rho_{\fH}$:
\begin{equation}\label{equ:Hodge-perfect-obs}
\varphi_{\fH/\UM_{[m]}} \colon \TT_{\fH/\UM_{[m]}} \to \EE_{\fH/\UM_{[m]}} := \pi_{\fH,*}\big(\tomega|_{\mathfrak{C}_{\fH}} \big) .
\end{equation}
Note that $\TT_{\fH/\UM_{[m]}} = T_{\fH/\UM_{[m]}} \cong \fH|_{\fH}$, since $\fH \to \UM_{[m]}$ is a vector bundle.
Define a complex on $\fH$:
\begin{equation}\label{eq:H-boundary-complex}
\FF_{\fH} := \cO(\ttwist\Delta_{\max})|_{\fH}[-1].
\end{equation}

Consider the following surjection
\[
\sigma_{\fH/\fM} \colon H^1(\EE_{\fH/\UM_{[m]}}) = R^1\pi_{\fH,*}\big(\tomega|_{\mathfrak{C}_{\fH}} \big) = \oplus_{i\in[m]} \cO(\ttwist\Delta_{\max})|_{\fH} \longrightarrow  \FF_{\fH}[1]
\]
where the last arrow is induced by taking summation.
This defines a composition
\begin{equation}\label{eq:H-cosection}
\EE_{\fH/\UM_{[m]}} \longrightarrow H^1(\EE_{\fH/\UM_{[m]}})[-1] \longrightarrow \FF_{\fH}.
\end{equation}

\begin{lemma}\label{lem:H-POT-vanishing}
The following composition is the zero morphism
\begin{equation}\label{eq:H-composition-vanishing}
\TT_{\fH/\UM_{[m]}}  \to \EE_{\fH/\UM_{[m]}} \to \FF_{\fH}.
\end{equation}
\end{lemma}
\begin{proof}
Indeed, the composition factors through
\[
\TT_{\fH/\UM_{[m]}}  \cong H^0(\EE_{\fH/\UM_{[m]}}) \to  \EE_{\fH/\UM_{[m]}},
\]
hence is the zero morphism.
\end{proof}

\begin{lemma}\label{lem:POT-R-H}
There is a canonical commutative diagram
\[
\xymatrix{
\TT_{\UH_{[m]}/\UM_{[m]}} \ar[rr] \ar[d]_{\varphi_{\UH_{[m]}/\UM_{[m]}}} && \TT_{\fH/\UM_{[m]}}|_{\UH_{[m]}} \ar[d]^{\varphi_{\fH/\UM_{[m]}}|_{\UH_{[m]}}} \\
\EE_{\UH_{[m]}/\UM_{[m]}} \ar[rr] \ar[d] && \EE_{\fH/\UM_{[m]}}|_{\UH_{[m]}} \ar[d] \\
\FF_{\UH_{[m]}/\UM_{[m]}} \ar[rr]^{\cong} && \FF_{\fH}|_{\UH_{[m]}}
}
\]
where the top horizontal arrow is given by \eqref{eq:R-to-H},  the lower left vertical arrow is the composition \eqref{eq:R-complex-cosection}, and the lower right vertical arrow is given by \eqref{eq:H-cosection}.
\end{lemma}
\begin{proof}
We first construct the upper square. Consider the commutative diagram
\[
\xymatrix{
\UH_{[m]} \ar[d] && \pC \ar[ll]_{\pi} \ar[rr]^-{f_{e_{\max}}} \ar[d] && \punt_{e_{\max},\circ} \ar[d]^{\tW} \\
\fH \ar[d] && \mathfrak{C}_{\fH} \ar[ll]_{\pi_{\fH}} \ar[rr]^-{\rho_{\fH}} \ar[d] &&  \uomega\boxtimes\cO(\ttwist\Delta_{\max}) \ar[d] \\
\UM_{[m]} && \mathfrak{C} \ar[ll]_{\pi_{\UM_{[m]}}} \ar[rr] && \BC\times\Delta_{\max}.
}
\]
where the left column is Cartesian, and
$\uomega\boxtimes\cO(\ttwist\Delta_{\max})$ is equipped with
the log structure pulled back from $\BC\times\Delta_{\max}$.
This leads to a commutative diagram of complexes
\[
\xymatrix{
\TT_{\UH_{[m]}/\UM_{[m]}}|_{\pC} \cong \TT_{\pC/\mathfrak{C}} \ar[d] \ar[rr] && f_{e_{\max}}^*\Omega^{\vee}_{\punt/\BC}|_{\punt_{e, \circ}} \ar[d]^{f_{e_{\max}}^*\diff\tW} \\
\TT_{\fH/\UM_{[m]}}|_{\pC} \cong \TT_{\mathfrak{C}_{\fH}/\mathfrak{C}}|_{\pC} \ar[rr] && \omega_{\pC/\UH_{[m]}}\otimes f_{e_{\max}}^*\cO(\ttwist\Delta_{\max})
}
\]
where the right vertical arrow is given by Lemma \ref{lem:factor-through-omega}. Pushing forward along $\pi$ and applying adjunction, we obtain the top square in the statement
\[
\xymatrix{
\TT_{\UH_{[m]}/\UM_{[m]}} \ar[d] \ar[rr] && \TT_{\fH/\UM_{[m]}}|_{\UH_{[m]}}  \ar[d]\\
\EE_{\UH_{[m]}/\UM_{[m]}}  \ar[rr]^{\pi_{*}f_{e_{\max}}^*\diff\tW}  && \EE_{\fH/\UM_{[m]}}|_{\UH_{[m]}}
}
\]
where the right vertical arrow is \eqref{eq:complex-cosection-component}.

For the bottom square, note that the bottom isomorphism follows immediately from \eqref{eq:boundary-complex} and \eqref{eq:H-boundary-complex}. The commutativity follows from the observation that the composition $\EE_{\UH_{[m]}/\UM_{[m]}} \to \EE_{\fH/\UM_{[m]}}|_{\UH_{[m]}} \to \FF_{\fH}|_{\UH_{[m]}}$ is indeed  \eqref{eq:R-complex-cosection}.
\end{proof}

\subsubsection{Proof of Theorem \ref{thm:red-POT}}

Lemma \ref{lem:POT-R-H} provides a commutative diagram of solid arrows
\begin{equation}\label{diag:compare-POTs}
\xymatrix{
\TT_{\UH_{[m]}/\UM_{[m]}} \ar@/^1pc/[rrd] \ar@{-->}[rd]_{\varphi^{\red}_{\UH_{[m]}/\UM_{[m]}}}  \ar@/_4pc/[rdd] &&&& \\
 & \EE^{\red}_{\UH_{[m]}/\UM_{[m]}}  \ar[r] & \EE_{\UH_{[m]}/\UM_{[m]}} \ar[d] \ar[r] & \FF_{\UH_{[m]}/\UM_{[m]}} \ar[r]^{[1]} \ar@{=}[d] & \\
 & \TT_{\fH/\UM_{[m]}}|_{\UH_{[m]}} \ar[r] & \EE_{\fH/\UM_{[m]}}|_{\UH_{[m]}} \ar[r] & \FF_{\fH}|_{\UH_{[m]}}  &
}
\end{equation}

First, note that the composition $\TT_{\UH_{[m]}/\UM_{[m]}} \to \EE_{\UH_{[m]}/\UM_{[m]}} \to \FF_{\UH_{[m]}/\UM_{[m]}}$ factors through the zero morphism \eqref{eq:H-composition-vanishing} by Lemma \ref{lem:H-POT-vanishing}, hence is also a zero morphism. This gives the desired dashed arrow $\varphi^{\red}_{\UH_{[m]}/\UM_{[m]}}$ making \eqref{diag:compare-POTs} commutative.

Next, we verify that $\EE_{\UH_{[m]}/\UM_{[m]}}^{\red}$ is perfect in $[0,1]$. Since $\FF_{\UH_{[m]}/\UM_{[m]}}$ is a vector bundle in degree $1$, and  $\EE_{\UH_{[m]}/\UM_{[m]}}^{\red}$ is perfect in $[0,1]$, we conclude that $\EE_{\UH_{[m]}/\UM_{[m]}}^{\red}$ is perfect in $[0,2]$. It remains to show that $H^2(\EE^{\red}_{\UH_{[m]}/\UM_{[m]}}) = 0$. The long exact sequence of \eqref{eq:red-POT} gives
\[
H^1(\EE_{\UH_{[m]}/\UM_{[m]}}) \stackrel{\sigma_{\UH_{[m]}/\UM_{[m]}}}{\longrightarrow} H^1(\FF_{\UH_{[m]}/\UM_{[m]}}) \longrightarrow H^2(\EE^{\red}_{\UH_{[m]}/\UM_{[m]}}) \longrightarrow H^2(\EE_{\UH_{[m]}/\UM_{[m]}}) = 0.
\]
Thus, $H^2(\EE^{\red}_{\UH_{[m]}/\UM_{[m]}}) = 0$ follows from the surjectivity of $\sigma_{\UH_{[m]}/\UM_{[m]}}$ in Proposition \ref{prop:cosection-surjective}.

Now taking $H^0$, we obtain a commutative diagram
\[
\xymatrix{
&& H^0(\TT_{\UH_{[m]}/\UM_{[m]}}) \ar[d]^{H^0(\varphi_{\UH_{[m]}/\UM_{[m]}})} \ar[ld]_{H^0(\varphi^{\red}_{\UH_{[m]}/\UM_{[m]}})} & & \\
0 \ar[r] & H^0(\EE^{\red}_{\UH_{[m]}/\UM_{[m]}}) \ar[r]^{\cong} & H^0(\EE_{\UH_{[m]}/\UM_{[m]}}) \ar[r] & H^0(\FF_{\UH_{[m]}/\UM_{[m]}}) = 0
}
\]
The isomorphism $H^0(\varphi_{\UH_{[m]}/\UM_{[m]}})$ implies that
$H^0(\varphi^{\red}_{\UH_{[m]}/\UM_{[m]}})$ is also an isomorphism.

Finally taking $H^1$, we obtain a commutative diagram
\begin{equation}\label{diag:lift-red-obstruction}
\xymatrix{
&& H^1(\TT_{\UH_{[m]}/\UM_{[m]}}) \ar[d]^{H^1(\varphi_{\UH_{[m]}/\UM_{[m]}})} \ar[ld]_{H^1(\varphi^{\red}_{\UH_{[m]}/\UM_{[m]}})} & & & \\
0  \ar[r] & H^1(\EE^{\red}_{\UH_{[m]}/\UM_{[m]}}) \ar[r] & H^1(\EE_{\UH_{[m]}/\UM_{[m]}}) \ar[r] & H^1(\FF_{\UH_{[m]}/\UM_{[m]}}) = \FF_{\UH_{[m]}/\UM_{[m]}}[1] \ar[r] & 0
}
\end{equation}
Thus the injectivity of $H^1(\varphi_{\UH_{[m]}/\UM_{[m]}})$ implies that $H^1(\varphi^{\red}_{\UH_{[m]}/\UM_{[m]}})$ is injective.

We conclude that $\varphi^{\red}_{\UH_{[m]}/\UM_{[m]}}$ is a perfect
obstruction theory of $\UH_{[m]} \to \UM_{[m]}$ in the sense of
\cite{BeFa97}.

%%%%%%%%%%%%%%%%%%%%%%%%%%%%%%%%%%%%
% Configurations of the degeneracies
%%%%%%%%%%%%%%%%%%%%%%%%%%%%%%%%%%%%

\section{Reduction to the connected theory}\label{sec:reduction-to-connected-inv}

We follow the notations in \S \ref{sss:universal-UM-stack-disconnected}, and write for simplicity
\begin{equation}\label{eq:disconnected-universal-stacks}
\fM_i := \fM_{g_i,\ddata_i'}(\ainfty) \ \ \ \mbox{and} \ \ \ \UM_i := \UM_{g_i,\ddata_i'}(\ainfty)
\end{equation}
with the universal punctured maps $\mathfrak{f}_i\colon \pC_i \to \ainfty$
and
$\mathfrak{f}^{\curlywedge}_i\colon \cC^{\curlywedge,\circ}_i \to \ainfty$
respectively.
It was shown in \S \ref{sec:red-theory} that the stacks $\UM_i$ and
$\UM_{[m]}$ provide the right categories for defining the reduced
theory.
The goal of this section to compute disconnected invariants of both
reduced and canonical theories in terms of connected ones.
As explained in Remark \ref{rem:intro-RA}, this requires further
modifications of $\UM_{[m]}$ for two reasons that we recall.
First, there is no morphism from $\UM_{[m]}$ to
$\prod_{i \in [m]}\UM_i$, since each connected component of the domain
might not have the uniform maximal degeneracy.
Second, the localization formulas of log GLSM in \cite{CJR23P} involve
$\psi_{\min}$ classes (see \S \ref{ss:psi-min}), which require the
$\curlyvee$-configuration \S \ref{sss:extremal-degeneracy}.
These two motivations lead to further combinations of the $\curlywedge$-
and $\curlyvee$-configurations as discussed below.

\subsection{Maps with extremal configurations}

\subsubsection{The \texorpdfstring{$\curlyvee$}{v}-configuration with connected domains}\label{sss:p-UM-connected-map}

Given discrete data $(g, \ddata')$ for punctured maps to $\ainfty$,
consider the sheaf $\obD(\mathfrak{f})$ of degeneracies over
$\fM_{g,\ddata'}(\ainfty)$, and the subsheaf
$\obD^m(\mathfrak{f}) \subset \obD(\mathfrak{f})$ of minimal elements.
As observed in \S \ref{sss:degeneracies}, sections in
$\obD^m(\mathfrak{f})$ are nowhere zero.
Thus we obtain a sheaf of nowhere trivial toric ideals
$\ocK^m \subset \ocM_{\fM_{g,\ddata'}(\ainfty)}$ generated by
$\obD^m(\mathfrak{f})$.
Denote by $\cK^m \subset \cM_{\fM_{g,\ddata'}(\ainfty)}$ the
corresponding log-ideal given by the pre-image of $\ocK^m$.
Then $\cK^m$ is everywhere non-trivial.
Consider the log blow-up along $\cK^m$ in the fs category
\cite[III~2.6]{Og18}:
\begin{equation}\label{eq:minimization}
\Bl_{\cK^m} \colon \Um_{g,\ddata'}(\ainfty) \to \fM_{g,\ddata'}(\ainfty).
\end{equation}

Note that $\Bl_{\cK^m}$ is projective and log \'etale \cite[III 2.6.4, 2.6.5]{Og18}. Furthermore, it is the universal principalization of $\cK^m$ \cite[III 2.6.1 (1)]{Og18}. Thus the everywhere non-trivial log-ideal $\cK^{\curlyvee} := \cK^m|_{\Um_{g,\ddata'}(\ainfty)}$ is principal. By construction, there is a global section $e_{\min} \in \Gamma\left(\cK^{\curlyvee}\right)$ locally generates $\cK^{\curlyvee}$, whose fiber over each geometric point is the unique minimal degeneracy. Thus the universal punctured map over $\Um_{g,\ddata'}(\ainfty)$ has uniform minimal degeneracy. Furthermore, it is universal:

\begin{proposition}\label{prop:minimize-componentwise}
\begin{enumerate}
  \item  $\fM^{\circ}_{g,\ddata'}(\ainfty)$ is an open dense substack of $\Um_{g,\ddata'}(\ainfty)$. In particular, $\Bl_{\cK^m}$ is birational.

 \item $\Um_{g,\ddata'}(\ainfty)$ is the stack parameterizing punctured maps with uniform minimal degeneracy and the discrete data $(g_i, \ddata_i')$.
\end{enumerate}
\end{proposition}

\begin{proof}
  The first statement follows from Lemma
  \ref{lem:universal-punctured-map-moduli} (2).
  For (2), consider a punctured map $f\colon \pC \to \ainfty$ over $S$
  given by $S \to \fM_{g,\ddata'}(\ainfty)$.
  If $f$ has uniform minimal degeneracy, then the pull-back
  $\cK^m|_{S}$ is generated by the minimal degeneracy, hence is
  principal.
  The universal property of $\Bl_{\cK^m}$ implies a canonical map
  $S \to \Um_{g,\ddata'}(\ainfty)$ with $f$ the pull-back of the
  universal punctured map over $\Um_{g,\ddata'}(\ainfty)$.
\end{proof}

\subsubsection{The \texorpdfstring{$\curlyvee$}{v}-configuration with disconnected domains}
Write for simplicity
\[
 \Um_i := \Um_{g_i,\ddata'_i}(\ainfty)
\]
with the universal punctured maps
$\mathfrak{f}^{\curlyvee}_i\colon \cC^{\curlyvee,\circ}_i \to
\ainfty$.
Let $\mathfrak{f} = \sqcup \mathfrak{f}^{\curlyvee}_i$ be the
universal punctured map over $\prod_{i \in [m]} \Um_i$.
Denote by $e_{\min,i}$ the minimal degeneracy over $\Um_i$.
Then the sheaf $\obD^m(\mathfrak{f})$ of minimal degeneracies is
globally constant given by $\{e_{\min,i}\}_{i \in [m]}$.
Let $\ocK \subset \ocM_{\prod_{i \in [m]} \Um_i}$ be the ideal
generated by $\obD^m(\mathfrak{f})$, and
$\cK \subset \cM_{\prod_{i \in [m]} \Um_i}$ the log-ideal given by the
pre-image of $\ocK$.
Consider the log blow-up along $\cK$ in the fs category:
\begin{equation}\label{eq:blowup-disconnected-min}
\Bl_{\cK} \colon \Um_{[m]} \to \prod_{i \in [m]} \Um_i.
\end{equation}
The log blow-up $\Bl_{\cK}$ is projective and log \'etale, with the following universal property:

\begin{proposition}\label{prop:minimize-multi-component}
  $\Um_{[m]}$ parameterizes punctured maps with uniform minimal
  degeneracy such that the following hold fiberwise:
  \begin{enumerate}
  \item The domain curves have $m$ connected components labeled by $[m]$.
  \item For each $i \in [m]$, the $i$th connected component has uniform minimal degeneracy with discrete data $(g_i, \ddata_i')$.
  \end{enumerate}
\end{proposition}
\begin{proof}
The proof is similar to the case of Proposition \ref{prop:minimize-componentwise} (2).
\end{proof}

For later use, we consider the structure of $\Bl_{\cK}$.

\begin{lemma}\label{lem:min-stack-composition}
The morphism $\Bl_{\cK}$ is given by a composition
\begin{equation}\label{eq:min-log-blowup}
\Um_{[m]} \stackrel{\mathrm{S}}{\longrightarrow} \fM_{[m]}^\fine \stackrel{\Bl^f_{\cK}}{\longrightarrow} \prod_{i \in [m]} \Um_i
\end{equation}
such that
\begin{enumerate}
\item $\mathrm{S}$ is the saturation which is an isomorphism over the
  open dense substack $\fM^{\curlyvee, \circ}_{[m]} \subset \Um_{[m]}$
  with smooth source curves.
\item There is an open dense substack
  $\fM^{\curlyvee, \circ,\circ}_{[m]} \subset \fM^{\curlyvee,
    \circ}_{[m]}$ such that the degeneracies of connected components
  are all identical.
  In particular,
  $\ocM_{\fM^{\curlyvee, \circ,\circ}_{[m]}} \cong
  \NN_{\fM^{\curlyvee, \circ,\circ}_{[m]}}$ is the globally constant
  sheaf.

\item $\Bl^f_{\cK}$ is the log blow-up along $\cK$ in the fine category. In particular, the underlying of $\Bl^f_{\cK}$ is the projection
\[
\ul{\fM}_{[m]}^\fine \cong \PP(\bigoplus_i \cO(e_{\min,i})) \to \prod_{i \in [m]} \Um_i,
\]
where $\cO(e_{\min,i})$ is the line bundle associated to $e_{\min,i}$ as in \S \ref{sss:log-line-bundles}.
\end{enumerate}
\end{lemma}
\begin{proof}
The factorization $\Bl_{\cK} = \Bl^f_{\cK} \circ \mathrm{S}$ follows from the construction of log blow-ups in \cite[III 2.6.3]{Og18}. Since $\ocK_{[m]}$ is generated by the set of global sections $\{e_{\min,i}\}_{i\in[m]}$, (3) follows from the local description \cite[III 2.6.4]{Og18} of log blow-ups.

Note that the log structure $\ocM_{\prod_{i \in [m]} \fM_i^{\circ}}$
is locally free with generators given by $\{e_{\min,i}\}_{i\in[m]}$.
A straightforward calculation shows that the saturation is trivial
over $\fM^{\curlyvee,\circ}_{[m]}$.
This is (1).
Finally,
$\fM^{\curlyvee, \circ,\circ}_{[m]}\subset \fM^{\curlyvee,
  \circ}_{[m]}$ is the open locus along which the sections
$\{e_{\min,i}|_{\fM^{\curlyvee, \circ,\circ}_{[m]}}\}$ are identical.
This implies (2).
\end{proof}

\subsubsection{Universal punctured maps with \texorpdfstring{$\diamondsuit$}{diamond}-configurations}\label{sss:pun-extreme-configuration}
Consider the fiber products
\begin{equation}\label{eq:pun-extreme-configuration}
\UMm_i := \UM_i\times_{\fM_i}\Um_i \ \ \ \mbox{and} \ \ \ \UMm_{[m]} := \Um_{[m]}\times_{\prod_{i\in[m]}\Um_i}(\prod_{i\in [m]}\UMm_i),
\end{equation}
both taken in the fs category, with projections
\begin{equation}\label{eq:remove-diamond}
\UMm_i \to \UM_i \to \fM_i,
\end{equation}
\begin{equation}\label{eq:split-min}
\Bl_{\cK} \colon \UMm_{[m]} \to \prod_{i\in [m]}\UMm_i.
\end{equation}

The morphism \eqref{eq:remove-diamond} exhibits $\UMm_i$ as the stack
parameterizing punctured maps in $\fM_i$ with
$\diamondsuit$-configurations.
By Proposition \ref{prop:umd-punctured-stack} and
\ref{prop:minimize-componentwise}, the morphisms in
\eqref{eq:remove-diamond} are proper, representable, log \'etale and
birational, and restrict to isomorphisms over the open dense substack
$\fM^{\circ}_i \subset \UMm_i$.
In particular, $\UMm_i$ is equidimensional with
$\dim \UMm_i = \dim \fM_i$.

By the base change property of log blow-ups
\cite[III~2.6.3~(1)]{Og18}, the morphism \eqref{eq:split-min} is
obtained by pulling back \eqref{eq:blowup-disconnected-min}, and is
the log blow-up along the pull-back $\cK^{\diamondsuit}$ of $\cK$.
Furthermore, it exhibits $\UMm_{[m]}$ as the stack parameterizing
punctured maps in $\prod_{i\in [m]}\UMm_i$ with the
$\curlyvee$-configuration.
More precisely, denote by
\[
\mathfrak{f}^{\diamondsuit} = \sqcup_{i\in [m]} \mathfrak{f}^{\diamondsuit}_i \colon \mathfrak{C}^{\circ,\diamondsuit}_{[m]} = \sqcup_{i\in[m]} \mathfrak{C}^{\circ,\diamondsuit}_i \to \ainfty
\]
the universal punctured maps over $\UMm_{[m]}$, where $\mathfrak{f}^{\diamondsuit}_i \colon \mathfrak{C}^{\circ,\diamondsuit}_i \to \ainfty$ is the pull-back of the universal punctured  map over $\UMm_i$. Then $\mathfrak{f}^{\diamondsuit}$ has the $\curlyvee$-configuration.

By \eqref{eq:base-dimension} and Lemma \ref{lem:min-stack-composition}, the stack $\UMm_{[m]}$ is equidimensional with
\begin{equation}\label{eq:equal-dimension-disconnected-domain}
\dim \UMm_{[m]} = \sum_{i\in[m]}\dim \fM_i + m-1 = \sum_{i \in [m]} (3g_i + \ell(\ddata'_i)) -3m-1,
\end{equation}
where $\ell(\ddata'_i)$ denotes the number of markings labeled by $\ddata'_i$.

For later use,  let $e_{\max,i} \in \Gamma\left(\UM_i,\ocM_{\UM_i}\right)$ be the maximal degeneracy, see Notation \ref{not:extrem-degeneracy}. It pulls back to the maximal degeneracy of the $i$-th component of the universal punctured map over $\UMm_{[m]}$, denoted again by $e_{\max,i}$.

Let $\obD^M$ be the sheaf of maximal degeneracies over $\UMm_{[m]}$. Indeed,  $\obD^M$ is globally constant given by the set of global sections
\begin{equation}\label{eq:component-max-degeneracy}
\left\{e_{\max,i} \in \Gamma(\UMm_{[m]},\ocM_{\UMm_{[m]}}) \ | \ i \in [m]\right\},
\end{equation}
which fiberwisely do not necessarily have a maximum with respect to
the canonical ordering.

\subsubsection{Punctured  R-maps with \texorpdfstring{$\diamondsuit$}{diamond}-configurations}\label{sss:diamond-PR-maps}
Consider  the moduli of stable punctured  R-maps $\SH_{g,\ddata}(\beta,\infty)$ with connected domain curves. We obtain a Cartesian diagram with strict vertical arrows
\begin{equation}\label{diag:main-player}
\xymatrix{
\RMm_{g,\ddata}(\beta,\infty) \ar[r]^{F^{\curlywedge}} \ar[d] & \UH_{g,\ddata}(\beta,\infty) \ar[r]^{F} \ar[d] & \SH_{g,\ddata}(\beta,\infty) \ar[d] \\
\UMm_{g,\ddata'}(\ainfty) \ar[r] & \UM_{g,\ddata'}(\ainfty) \ar[r] & \fM_{g,\ddata'}(\ainfty)
}
\end{equation}
Then $\RMm_{g,\ddata}(\beta,\infty)$ is the moduli of stable punctured R-maps with uniform extremal degeneracies or the {\em $\diamondsuit$-configuration} (Definition \ref{def:extremal-degeneracies}), and with the discrete data  $(g, \beta, \ddata)$.  The universal punctured R-map  over $\RMm_{g,\ddata}(\beta,\infty)$ is simply the pullback of the universal punctured R-map of $\SH_{g,\ddata}(\beta,\infty)$.

Assume \eqref{eq:PT-smooth} is smooth.
We obtain canonical virtual cycles
$[\RMm_{g,\ddata}(\beta,\infty)]^{\vir}$,
$[\UH_{g,\ddata}(\beta,\infty)]^{\vir}$ and
$[\SH_{g,\ddata}(\beta,\infty)]^{\vir}$ by pulling back the canonical
perfect obstruction theory \eqref{eq:tan-POT} of the right vertical
arrow.
Since the bottom arrows in \eqref{diag:main-player} are proper and
birational, the virtual cycles are related by push-forwards.

\begin{proposition}\label{prop:birational-canonical-vcycle}
Notations and assumptions as above, we have
\[
F^{\curlywedge}_*[\RMm_{g,\ddata}(\beta,\infty)]^{\vir} = [\UH_{g,\ddata}(\beta,\infty)]^{\vir} \  \  \mbox{and} \ \ F_*[\UH_{g,\ddata}(\beta,\infty)]^{\vir} = [\SH_{g,\ddata}(\beta,\infty)]^{\vir}.
\]
\end{proposition}

Further, given a transverse superpotential
(Definition~\ref{def:superpotential2} and
Definition~\ref{def:transverse-superpotential}), and assuming negative
contact orders at all markings, we obtain reduced virtual cycles
$[\RMm_{g,\ddata}(\beta,\infty)]^{\red}$ and
$[\UH_{g,\ddata}(\beta,\infty)]^{\red}$ by pulling back the reduced
perfect obstruction theory \eqref{diag:POT-factorization} of the
middle vertical arrow.
Similarly, we have:

\begin{proposition}\label{prop:birational-red-vcycle}
Notations and assumptions as above, we have
\[
F^{\curlywedge}_*[\RMm_{g,\ddata}(\beta,\infty)]^{\red} = [\UH_{g,\ddata}(\beta,\infty)]^{\red}.
\]
\end{proposition}

\subsection{Alignments}

Since the sheaf of monoids $\obD^M$ as in
\eqref{eq:component-max-degeneracy} does not have uniform maximal
degeneracies fiberwise, $\UMm_{[m]}$ does not admit the
$\curlywedge$-configuration, a property required for the reduced
theory as in \S \ref{sec:red-theory}.
This requires a further modification of $\UMm_{[m]}$ by aligning
$\{e_{\max,i} \}$ as follows.

\subsubsection{The stack of aligned log structures}\label{sss:aligned-log}

Recall from \cite[\S 8.1]{ACFW13} that an {\em aligned log structure}
on a scheme $\ul{S}$ is a locally free log structure $\cM_{S}$,
together with a sheaf of (possibly empty) finite subsets
$\obA_S \subset \ocM_{S}$ such that for each geometric point $s \in S$
the fiber $\ocM_{S}|_s \cong \NN^m$ has a basis
$\{e_1, e_2, \cdots, e_m\}$ and
\begin{equation}\label{eq:alignment-elements}
\obA_S|_{s} = \{ 0, e_1, e_1 + e_2, e_1 + e_2 + e_3, \cdots, e_1 + \cdots + e_m\} .
\end{equation}

By \cite[Proposition 8.1.2]{ACFW13}, the stack of aligned log structures $\cT_{\log}$ is algebraic. We view $\cT_{\log}$ as a log stack with the universal aligned log structure $\cM_{\log}$ and the universal subset $\obA_{\log} \subset \ocM_{\log}$. By \cite[Proposition 8.2.2 and 8.3.1]{ACFW13}, $\cT_{\log}$ is smooth and log smooth.

\subsubsection{The universal alignment}

Let $\ocM^a_{\log} \subset \ocM_{\log}$ be the free submonoid generated by $\obA_{\log}$. Denote by $\cM^a_{\log} = \ocM^a_{\log} \times_{\ocM_{\log}} \cM_{\log} \subset \cM_{\log}$ the corresponding sub-log structure, and by $\cT^a_{\log} = (\ul{\cT}_{\log}, \cM^a_{\log})$ the log stack.

Let $\Log^{fr} \subset \Log$ be the open substack parameterizing locally free log structures. Then we obtain a composition, called the {\em universal alignment}:
\begin{equation}\label{eq:alignment-morphism}
\xymatrix{
\ali \colon \cT_{\log} \ar[rr]^-{\cM^a_{\cT_{\log}} \subset \cM_{\cT_{\log}}} && \cT^a_{\log} \ar[rr]^-{\mbox{strict}} && \Log^{fr}.
}
\end{equation}

Let $S \to \cT_{\log}$ be a strict morphism given by an aligned log structure $(\cM_S, \obA_S)$ as in \S \ref{sss:aligned-log}. Consider the pull-back log structure $\cM^a_S = \cM^a_{\log}$. Then  fiberwise the non-zero elements in $\obA_S$ form a set of generators of $\ocM^a_S$.

\begin{proposition}\label{prop:alignment}
The morphism \eqref{eq:alignment-morphism} is proper, representable, birational, and log \'etale.
\end{proposition}
\begin{proof}
The statement that \eqref{eq:alignment-morphism} is log \'etale can be checked using \eqref{eq:alignment-elements} and the local criterion  \cite[Theorem (3.5)]{Ka88}.  Since $\ali_{\leq 1}$ is an isomorphism, the birationality follows. The representability and properness will follow from Proposition \ref{prop:configuration-projective} below.
\end{proof}

\subsubsection{Twisted aligned log structure}\label{sss:twisted-align-log}
A {\em twisted aligned log structure} on a scheme $\ul{S}$ is a locally free log structure $\cM_S$ together with a sheaf of subsets  $\obA_S \subset \ocM_{S}$, such that over each geometric point $s \in \ul{S}$ the fiber  $\ocM_S|_{s} \cong \NN^m$ has a basis  $\{e_1, \cdots, e_m\}$ and
\begin{equation}\label{eq:tw-alignment-elements}
\obA_S|_{s} = \{ 0, r_1 e_1, r_1 e_1 + r_2 e_2, r_1 e_1 + r_2 e_2 + r_3 e_3, \cdots, r_1 e_1 + \cdots + r_m e_m \} ,
\end{equation}
where the collection of $r_i \in \NN_{> 0}$ is called the {\em twisting choice}.
Let $\cT^{tw}_{\log}$ be the stack of twisted aligned log structures.  By \cite[\S 7.2]{ACFW13}, $\cT^{tw}_{\log}$ is a smooth and log
smooth algebraic stack with a tautological morphism
\[
\cT^{tw}_{\log} \to \cT_{\log}
\]
which is isomorphic over the open dense substack with the trivial log structure.
On the level of characteristic sheaves,
$\ocM_{\cT^{tw}_{\log}} \leftarrow \ocM_{\cT_{\log}}|_{\cT^{tw}_{\log}} $ is defined by
$r_i e_r \mapsfrom e_i$ using the notations in
\eqref{eq:alignment-elements} and \eqref{eq:tw-alignment-elements}.
Thus smooth-locally, this morphism takes the $r_i$-th root
stack along the irreducible boundary divisor $\Div(e_i) \subset \cT_{\log}$
corresponding to the local section $e_i$ of $\ocM_{\cT_{\log}}$, see \eqref{eq:global-log-divisor}.

One can also fix a choice of twisting.
To do this, take a discrete set $\Lambda$, and define the stack
$\cT^{\Lambda}_{\log}$ with a strict morphism
$\cT^{\Lambda}_{\log} \to \cT_{\log}$ which parameterizes aligned log
structures such that locally the generators of
$\ocM_{\cT^{\Lambda}_{\log}}$ are labeled by $\Lambda$, see \cite[\S
7.1,~\S 8.1]{ACFW13}.
A {\em twisting choice} is a map $\fr \colon \Lambda \to \NN_{>0}$.
Thus, for each strict geometric point $S \to \cT^{\Lambda}_{\log}$,
$\fr$ assigns integer $r_i = \fr(e_i)$ to each generator $e_i$ as in
\eqref{eq:tw-alignment-elements}.
Denote by $\cT^{\fr}_{\log}$ with the twisting choice specified by
$(\Lambda, \fr)$.
There is a strict tautological morphism
\[
\cT^{\fr}_{\log} \to \cT^{tw}_{\log}\times_{\cT_{\log}}\cT^{\Lambda}_{\log}
\]
isomorphic onto its image, see \cite[Lemma 7.2.4 (2)]{ACFW13}.

\subsubsection{The truncated alignment}

Let $\Log^{fr}_{\leq m} \subset \Log^{fr}$ be the open dense substack
over which the fibers of $\ocM_{\Log^{fr}_{\leq m}}$ have rank
$\leq m$. Let $\Log^{fr}_{=m} \subset \Log^{fr}_{\leq m}$ be the
strict closed substack over which the fibers of
$\ocM_{\Log^{fr}_{\leq m}}$ have rank precisely $m$.
The morphism \eqref{eq:alignment-morphism} restricts to the {\em
  truncated alignment}:
\begin{equation}\label{eq:trancated-alignment-morphism}
\ali_{\bullet} := \ali|_{\cT_{\bullet}} \colon\cT_{\bullet} := \cT_{\log}\times_{\Log^{fr}}\Log^{fr}_{\bullet} \longrightarrow \Log^{fr}_{\bullet}
\end{equation}
where $\bullet$ represents $\leq m$ or $=m$. For later use, denote by
$\cM^a_{\bullet} := \ali_{\bullet}^*\cM_{\Log^{fr}_{\bullet}}$.

We call $\cT_{\leq m}$ {resp. $\cT_{= m}$} the {\em stack of aligned
  log structures of length $\leq m$ (resp. $=m$)}. We also obtain the
universal subsets $\obA_{\leq m} = \obA_{\log}|_{\cT_{\leq m}}$ and
$\obA_{= m} = \obA_{\log}|_{\cT_{= m}}$.

\subsubsection{The labeled alignment}

We are also interested in the case where all elements in
\eqref{eq:alignment-elements} are labeled.
Consider $\cA^m \cong \prod_{i \in [m]}\cA_{i}$ with the $m$ copies of
$\cA_i \cong \cA$ labeled by the set $[m] = \{1, 2, \cdots, m\}$.
Taking the base change along the strict morphism
$\cA^m \to \Log^{fr}_{\leq m}$, we obtain the {\em $[m]$-labeled
  alignment}:
\[
\ali_{[\leq m]} := \ali_{\leq m}\times_{\Log^{fr}_{\leq m}}\cA^m \colon \cT_{[\leq m]} := \cT_{\leq m}\times_{\Log^{fr}_{\leq m}}\cA^m \to \cA^m.
\]
For any strict morphism $S \to \cT_{[\leq m]}$ induced by an aligned log structure $(\cM_S, \obA_S)$, non-zero sections in  $\obA_S$ are labeled by $[m]$.

For later use, we also introduce the {\em $[m]$-labeled alignment of length $m$}:
\begin{equation}\label{eq:labeled-n-alignment}
  \ali_{[m]} \colon \cT_{[m]} := \cT_{[\leq m]}\times_{\cA^m} \ainfty^m \to \ainfty^{m}
\end{equation}
obtained by pulling back $\ali_{[\leq m]}$ along $\ainfty^{m} \to \cA^m$.

\begin{proposition}\label{prop:configuration-projective}
  $\ali_{[\leq m]}$ is birational and projective.
\end{proposition}
\begin{proof}
  Note that the restriction
  $\ali_{[\leq m]}|_{\cT_{[\leq 1]}} = \ali_{[\leq 1]}$ over the dense
  open $\cT_{[\leq 1]} \subset \cT_{[\leq m]}$ is the identity.
  It remains to show the projectivity.
  In the following, we will show that $\cT_{[\leq m]}$ is the moduli
  $X^{[m]}$ of degree $m$ stable configurations of points in the pair
  $X = (\cA, \infty)$ with $\ali_{[\leq m]}$ the evaluation given by
  the $m$ sections labeled by $[m]$, see \cite[Definition
  1.5.2]{AbFa17}.
  Then the projectivity follows from \cite[Proposition 1.5.4]{AbFa17}.

  We view $X^{[m]}$ as a log stack with the canonical log structure
  given by its universal expansions.
  For any strict morphism $S \to X^{[m]}$, denote by $S[m] \to S$ the
  family of expansions.
  Over each geometric point $s \in S$, we have the fiber of length
  $k \leq m$ (see \cite[Convention 1.4.1]{AbFa17})
  \[
    S[m]_s = X \cup_{\infty = \infty_{-}} P_1  \cup_{\infty_{+} = \infty_{-}} \cdots \cup_{\infty_{+} = \infty_{-}} P_k.
  \]
  The fiber $\ocM_{S}|_{s} \cong \NN^k$ is a free monoid with
  generators $e_1, \cdots, e_k$ such that $e_i$ corresponds to the
  smoothing parameter of the node given by $\infty_{-} \subset P_i$.
  Each $P_i$ has a unique open dense point, which contains at least
  one section by stability.
  For the $k$th section $\epsilon_k \colon S \to S[m]$, we define an
  element $\delta_{k,s} = e_1 + e_2 + \cdots + e_{k_i}$ if the fiber
  $\epsilon_k|_{s}$ lands in $P_{k_i}$. This defines a set
  $\obA_{S,s} := \{\delta_{k,s} \ | \ k \in [m]\} \cup \{0\}$ of
  length $\leq m$ since sections can intersect. Observe that the
  fiberwise defined $\obA_{S,s}$ glues to a sheaf of totally ordered
  subsets $\obA_S \subset \ocM_S$ labeled by $[m]$. This defines a
  morphism $S \to \cT_{[\leq m]}$, hence $X^{[m]} \to \cT_{[\leq m]}$.

  To obtain the inverse $\cT_{[\leq m]} \to X^{[m]}$, note that
  $\cT_{[\leq m]}$ carries a natural family of expansions of length
  $\leq m$ by \cite[Proposition 8.1.2]{ACFW13}.
  The log structure $\cM_{[\leq m]}$ is the canonical log structure of
  the expansions.
  For a strict morphism $S \to \cT_{[\leq m]}$, let $S[m] \to S$ be
  the family of expansions pulled back from $\cT_{[\leq m]}$.
  The section $\epsilon_k \colon S \to S[m]$ is constructed using the
  section in $\obA_{[\leq m]}$ labeled by $k\in [m]$ by reversing the
  above construction.

  This completes the proof.
\end{proof}

\begin{proof}[Proof of Proposition \ref{prop:alignment}]
To finish the proof of Proposition \ref{prop:alignment}, note that $\Log^{fr}$ is covered by charts $\{\cA^m \to \Log^{fr}\}_{m \in \NN}$. Thus the representability and properness of \eqref{eq:alignment-morphism} can be checked over each chart, which is precisely Proposition \ref{prop:configuration-projective}.
\end{proof}

Putting Proposition \ref{prop:alignment} and \ref{prop:configuration-projective} together, we have

\begin{corollary}\label{cor:labeled-alignment}
 $\ali_{[m]}$ is projective and log \'etale.
\end{corollary}
Note that $\ali_{[m]}$ is birational iff $m=1$.
Indeed, $\ali_{[1]}$ is an isomorphism.

\subsubsection{The universality}\label{sss:align-universality}

Let $S$ be a log stack with a sheaf of finite subsets
$\obA \subset \ocM_S$.
A morphism $h\colon T \to S$ is called an {\em alignment} of $\obA$ if
the image of
\[
h^{-1}\obA \to f^{-1}\ocM_S \to \ocM_T
\]
is totally ordered with respect to $\poleq_{\ocM_T}$.

Let $\ocM^{\obA}_{S}$ be the sheaf of free monoids over $\ul{S}$ whose
fiberwise generators are labeled by elements of $\obA$.
Thus the inclusion $\obA \subset \ocM_S$ induces a morphism
$\ocM^{\obA}_S \to \ocM_S$ hence a log structure
$\cM^{\obA}_{S} = \cM_{S}\times_{\ocM_S}\ocM^{\obA}_S$ over $\ul{S}$
with the structure arrow defined by the composition
$\cM^{\obA}_{S} \to \cM_{S} \to \cO_{S}$.
We obtain a composition
\[
\xymatrix{
S \ar[rr]^-{\cM_S \leftarrow \cM^{\obA}_S} && (\ul{S}, \cM^{\obA}_S) \ar[rr]^-{\cM^{\obA}_S} && \Log^{fr}
}
\]
By Proposition \ref{prop:alignment}, the projection is a proper, representable and log \'etale:
\begin{equation}\label{eq:align-A}
\ali_{\obA} \colon \cT_{\obA} := \cT_{\log}\times_{\Log^{fr}}S \to S
\end{equation}

Denote by $\cM^a_{\cT_{\obA}}$ the pull-back
$\cM^a_{{\log}}|_{\cT_{\obA}} \cong \cM^{\obA}_{S}|_{\cT_{\obA}}$.
The sheaf $\obA_{\log}|_{\cT_{\obA}}$ is fiberwise given by the set of
generators of $\ocM^a_{\cT_{\obA}}$, hence we have an identification
$\obA_{\log}|_{\cT_{\obA}} = \obA|_{\cT_{\obA}}$.
Therefore, the image of the composition
\[
\obA|_{\cT_{\obA}}  = \obA_{\log}|_{\cT_{\obA}}\to \ocM^a_{\cT_{\obA}}  \to \ocM_{\cT_{\obA}}
\]
is totally ordered with respect to $\poleq_{ \ocM_{\cT_{\obA}}}$. This shows that the projection $\cT_{\obA} \to S$ is an alignment of $\obA$. Indeed, this is universal:

\begin{proposition}\label{prop:universal-alignment}
The category of alignments of $\obA$ over fs log schemes is represented by the log algebraic stack $\cT_{\obA}$.
\end{proposition}
\begin{proof}
Let $h \colon T \to S$ be an alignment of $\obA$. Consider the sequence of morphisms of log structures
$
\cM_{S}^{\obA}|_{T} \longrightarrow \cM_{S}|_{T} \stackrel{h^{\flat}}{\longrightarrow} \cM_T.
$
We construct a sub-sheaf of monoids $\ocM'_T \subset (\ocM_{S}^{\obA}|_{T})^{gp}$ as follows. For a geometric point $t \to \ul{T}$, let
$\{ \delta_1, \cdots, \delta_m\}$ be the set of generators of $\ocM_{S}^{\obA}|_{t}$. Denote by $\delta'_i \in \ocM_T|_{t}$ the image of $\delta_i$. After reordering, we may assume that
\[
\delta'_1 \poleq_{\ocM_T} \delta'_2 \poleq_{\ocM_T} \cdots \poleq_{\ocM_T} \delta'_m.
\]
Fiberwisely define $\ocM'_T|_t$ to be the free monoid generated by
$
\delta_1, \delta_2 - \delta_1, \cdots, \delta_m - \delta_{m-1}.
$

The fiberwise construction glues to a subsheaf of monoids $\ocM'_T \subset (\ocM_{S}^{\obA}|_{T})^{gp}$.
The fact that $h$ is an alignment of $\obA$ implies that
$\delta'_{i+1} - \delta'_i \in \ocM_T$.
Thus the composition
$\ocM'_T \subset (\ocM_{S}^{\obA}|_{T})^{gp} \to \ocM_T^{gp}$ factors
through $\ocM'_T \to \ocM_T$.
This defines a log structure $\cM'_T := \cM_T \times_{\ocM_T}\ocM'_T$.
By construction, the pair $(\obA_T, \cM'_T)$ defines an aligned
log structure on $\ul{T}$, inducing a strict morphism $(\ul{T}, \cM'_T) \to \cT_{\log}$. This leads to a commutative diagram
\[
\xymatrix{
T \ar[rr] \ar[d] && S \ar[d] \\
\cT_{\log} \ar[rr] && \Log^{ft}
}
\]
with the left arrow given by the composition $ T \to (\ul{T}, \cM'_T) \to \cT_{\log}$. Thus we obtain an $S$-morphism $T \to \cT_{\obA}$ such that the sheaf of ordered sets $(\obA_T, \poleq_{\ocM_T})$ is the pull-back of $(\obA|_{\cT_{\obA}}, \poleq_{\ocM_{\cT_{\obA}}})$.
\end{proof}

\subsubsection{\texorpdfstring{$\ali_{\obA}$}{ali} in the labeled case}\label{sss:labeled-alignment}

Consider a sheaf of finite subsets $\obA \subset \ocM_S$ over $S$ as
in \S \ref{sss:align-universality}.
An {\em $[m]$-labeling} of $\obA$ is a strict morphism
$(\uS, \cM^{\obA}_{S}) \to \cA^m \cong \prod_{i\in [m]}\cA_{i}$ with
the $m$ copies of $\cA_i \cong \cA$ labeled by $[m]$.
Thus, each non-zero (local) section of $\obA$ is uniquely labeled by
an element of $[m]$.

In case $(\uS, \cM^{\obA}_{S}) \to \cA^m$ factors through $\ainfty^m$,
note that $\obA \setminus \{0\}$ as the sheaf of set of generators of
$\ocM_{\ainfty^m}|_{S}$ is a constant sheaf with $m$ elements.

\begin{proposition}\label{prop:labeled-alignment}
Suppose $\obA$ is labeled. Then $\ali_{\obA}$ in \eqref{eq:align-A} is projective and log \'etale.
\end{proposition}
\begin{proof}
This follows from Corollary \ref{cor:labeled-alignment} and the following Cartesian squares in the fs category
\[
\xymatrix{
\cT_{\obA} \ar[rr] \ar[d]_{\ali_{\obA}} && \cT_{[\leq m]} \ar[rr] \ar[d]_{\ali_{[\leq m]}} && \cT_{\log} \ar[d]^{\ali} \\
S \ar[rr] && \cA^m \ar[rr] && \Log^{fr}
}
\]
\end{proof}

\subsubsection{Order removing}

Suppose the sheaf of subsets $\obA \subset \ocM_S$ over $S$ is
$[m]$-labeled.
Consider the subset $[m'] \subset [m]$ for some
$m' \leq m$.
This induces a subsheaf $\obA' \subset \obA$ by taking
elements labeled by $[m']$.
Since any alignment of $\obA$ is
automatically an alignment of $\obA'$, this leads to a natural
morphism over $S$:
\begin{equation}\label{eq:order-removing}
  \ali_{\obA \supset \obA'} \colon \cT_{\obA} \to \cT_{\obA'}.
\end{equation}
By Proposition \ref{prop:labeled-alignment}, the morphism $\ali_{\obA' \subset \obA}$ is projective over $S$.

\subsection{Maps with aligned maximal degeneracies and their virtual cycles}

\subsubsection{Universal punctured  maps}
Consider the maximal degeneracies \eqref{eq:component-max-degeneracy}
over $\UMm_{[m]}$.
Define
$\obD^M_{\setminus [i]} = \{e_{\max, i+1}, e_{\max,i+2}, \cdots,
e_{\max,m}\}$ for $i = 0,\cdots, m-1$, hence a decreasing filtration
of constant sheaves
\[
\obD^M =: \obD^M_{\setminus [0]}  \supset \cdots \supset \obD^{M}_{\setminus [m-2]} \supset   \obD^{M}_{\setminus [m-1]}.
\]

Denote by $\AMm_{\setminus[i]} = \cT_{\obD^M_{\setminus [i]}}$ the
stack of alignments of $\obD^M_{\setminus [i]}$ as in
\S\ref{sss:labeled-alignment}.
It parameterizes punctured maps in $\UMm_{[m]}$ together with an
alignment of $\obD^M_{\setminus [i]}$.
By \eqref{eq:order-removing}, the decreasing filtration induces a
sequence of projective and log \'etale morphisms removing alignments:
\begin{equation}\label{eq:remove-order-max}
  \AMm_{[m]} := \AMm_{\setminus[0]} \stackrel{\ali_1}{\longrightarrow} \AMm_{\setminus[1]} \stackrel{\ali_2}{\longrightarrow} \cdots \stackrel{\ali_{m-2}}{\longrightarrow} \AMm_{\setminus[m-2]} \stackrel{\ali_{m-1}}{\longrightarrow}\AMm_{\setminus[m-1]} = \UMm_{[m]} .
\end{equation}
Note that this is a sequence of morphisms over $\Um_{[m]}$.
By Lemma \ref{lem:min-stack-composition} (2), there is an open dense
substack
$\fM^{\curlyvee, \circ,\circ}_{[m]} \subset \AMm_{\setminus[i]}$
contained in each one of the stacks in \eqref{eq:remove-order-max}.
Thus for each $i$, $\ali_i$ is birational and $\AMm_{\setminus[i]}$ is
equidimensional with the dimension computed in
\eqref{eq:equal-dimension-disconnected-domain}:
\begin{equation}\label{equal-dim-partial-align}
  \dim \AMm_{\setminus[i]} = \dim \UMm_{[m]} = \sum_{i \in [m]} (3g_i + \ell(\ddata'_i)) -3m -1.
\end{equation}
Note that the universal punctured map over $\AMm_{[m]}$ has the
$\curlywedge$-configuration as required for the reduced theory as in
\S\ref{ss:red-pot} with disconnected domains.

For later use, denote by
\[
e_{\max, \geq i+1} \in \Gamma\left(\ocM_{\AMm_{\setminus[i]}}, \AMm_{\setminus[i]} \right)
\]
the maximal degeneracy of the components labeled by $\{i+1, i+2,  \cdots, m\}$.
Consider
\[
\UM_{\setminus [i]} :=\left(\prod_{j = 1}^{i} \UM_j \right) \times \UM_{\{i+1, \cdots , m\}}
\]
where $\UM_{\{i+1,\cdots, m\}}$ parameterizes punctured maps in
$\prod_{j = i+1}^{m} \fM_{j}$ with $\curlywedge$-configurations, see
\S\ref{sss:universal-UM-stack-disconnected}.
We obtain a commutative diagram
\begin{equation}\label{diag:partial-Umax-configuration}
\xymatrix{
\AMm_{\setminus[i]} \ar[rr]^{\ali_{i+1}} \ar[d] && \AMm_{\setminus[i+1]} \ar[d] \\
\UM_{\setminus [i]} \ar[rr] && \UM_{\setminus [i+1]}
}
\end{equation}
where the vertical arrows are the tautological ones.

\subsubsection{Punctured R-maps}
Recall the finite collection of discrete data
\[
(g_i, \beta_i, \ddata_i = \{(\ogamma_{ij}, c_{ij}))\}_{j=1}^{n_i}), \ \ \mbox{for } i \in [m]
\]
as in \eqref{eq:[m]-punctured-R-data} compatible with \eqref{eq:[m]-punctured-data}.
Write for simplicity $\SH_i = \SH_{g_i,\ddata_i}(\punt,\beta_i)$. We obtain a Cartesian diagram with strict vertical arrows:
\begin{equation}\label{diag:R-map-partial-align}
\xymatrix{
\RAMm_{\setminus [i]} \ar[r] \ar[d] & \UH_{\setminus [i]} \ar[d] \ar[r] & \prod_{j \in [m]} \SH_j \ar[d] \\
\AMm_{\setminus [i]} \ar[r] & \UM_{\setminus [i]} \ar[r] & \prod_{j\in [m]} \fM_j
}
\end{equation}
where the right vertical arrow is given by \eqref{eq:take-log}.
The left vertical arrow in \eqref{diag:R-map-partial-align} admits a canonical perfect obstruction theory
\begin{equation}\label{eq:can-POT-partial-alignment}
\varphi_{\RAMm_{\setminus [i]}/\AMm_{\setminus [i]}} \colon \TT_{\RAMm_{\setminus [i]}/\AMm_{\setminus [i]}} \to \EE_{\RAMm_{\setminus [i]}/\AMm_{\setminus [i]}}.
\end{equation}
by pulling back the canonical ones \eqref{eq:tan-POT} from the right vertical arrow. This defines the canonical virtual cycle $[\RAMm_{\setminus [i]}]^{\vir}$ of $\RAMm_{\setminus [i]}$.

The case of the reduced theories is more subtle. The middle vertical  arrow of \eqref{diag:R-map-partial-align} is
\[
\left(\prod_{j=1}^i \UH_{j} \right) \times \UH_{\{i+1, \cdots, m\}} \longrightarrow \left(\prod_{j=1}^i \UM_{j} \right)\times \UM_{\{i+1,\cdots, m\}}
\]
hence admits a reduced perfect obstruction theory by taking the direct sum
\[
\left(\bigoplus_{j=1}^{i}\varphi^{\red}_{\UH_{i}/\UM_{j}}\right) \oplus \varphi^{\red}_{\UH_{\{i+1, \cdots, m\}}/\UM_{\{i+1,\cdots, m\}}} \colon \TT_{\UH_{\setminus [i]}/\UM_{\setminus [i]}} \to \EE^{\red}_{\UH_{\setminus [i]}/\UM_{\setminus [i]}},
\]
whenever the assumptions in Theorem~\ref{thm:red-POT} are fulfilled.
This pulls back to the {\em reduced perfect obstruction theory} of the
left vertical arrow:
\begin{equation}\label{eq:red-POT-partial-alignment}
\varphi^{\red}_{\RAMm_{\setminus [i]}/\AMm_{\setminus [i]}} \colon \TT_{\RAMm_{\setminus [i]}/\AMm_{\setminus [i]}} \to \EE^{\red}_{\RAMm_{\setminus [i]}/\AMm_{\setminus [i]}} := \EE^{\red}_{\UH_{\setminus [i]}/\UM_{\setminus [i]}}|_{\RAMm_{\setminus [i]}}
\end{equation}
defining the {\em reduced virtual cycle} $[\RAMm_{\setminus [i]}]^{\red}$ of $\RAMm_{\setminus [i]}$.

For $i<m$, the stack $\RAMm_{\setminus [i]}$ admits other reduced
perfect obstruction theories.
Indeed, we have a Cartesian diagram with strict vertical arrows
\begin{equation}\label{eq:R-map-partial-order-remove}
\xymatrix{
\RAMm_{\setminus [i]} \ar[rr]^{\ali_{i+1}} \ar[d] && \RAMm_{\setminus [i+1]} \ar[d] \\
\AMm_{\setminus[i]} \ar[rr]^{\ali_{i+1}}  && \AMm_{\setminus[i+1]}
}
\end{equation}
where the top is again denoted by $\ali_{i+1}$ when there is no
confusion.
Pulling back the reduced perferct obstruction theory
$\varphi^{\red}_{\RAMm_{\setminus [i+1]}/\AMm_{\setminus [i+1]}}$ of
the right vertical arrow, we obtain the {\em $+$reduced perfect
  obstruction theory} of the left vertical arrow
\begin{equation}\label{eq:red-POT-partial-alignment+}
\varphi^{\red}_{\RAMm_{\setminus [i]}/\AMm_{\setminus [i]},+} \colon \TT_{\RAMm_{\setminus [i]}/\AMm_{\setminus [i]}} \to \EE^{\red}_{\RAMm_{\setminus [i]}/\AMm_{\setminus [i]},+} := \EE^{\red}_{\RAMm_{\setminus [i+1]}/\AMm_{\setminus [i+1]}}|_{\RAMm_{\setminus [i]}},
\end{equation}
defining the {\em $+$reduced virtual cycle} $[\RAMm_{\setminus [i]}]_{+}^{\red}$ of $\RAMm_{\setminus [i]}$. It follows from the virtual push-forward that
\begin{lemma}\label{lem:v-cycle-partial-order-remove}
\begin{enumerate}
 \item $\ali_{i+1,*}[\RAMm_{\setminus [i]}]^{\vir} = [\RAMm_{\setminus [i+1]}]^{\vir}$.
 \item $\ali_{i+1,*}[\RAMm_{\setminus [i]}]_{+}^{\red} = [\RAMm_{\setminus [i+1]}]^{\red}$, when the corresponding reduced theories are defined.
\end{enumerate}
\end{lemma}

\subsubsection{\texorpdfstring{$[\RAMm_{\setminus [i]}]^{\red}$}{[R]} versus \texorpdfstring{$[\RAMm_{\setminus [i]}]_{+}^{\red}$}{[R]_+}}

To relate the two perfect obstruction theories
\eqref{eq:red-POT-partial-alignment} and
\eqref{eq:red-POT-partial-alignment+}, we consider the global sections
\[
(e_{\max, \geq i+1} - e) \in \Gamma(\ocM_{\AMm_{\setminus[i]}},\AMm_{\setminus[i]}).
\]
for $e = e_{\max,i+1} \mbox{ or } e_{\max,\geq i+2}$ over $\AMm_{\setminus[i]}$.  These global sections lead to line bundles as in \eqref{eq:compare-torsor}, satisfying
\begin{equation}\label{eq:log-bundle-product}
  \cO(e - e_{\max, \geq i+1}) \cong \cO(-e_{\max, \geq i+1}) \otimes \cO(e)|_{\AMm_{\setminus[i]}}.
\end{equation}
Replacing $e$ by $e_{\max,i+1}$ or $e_{\max, \geq i+2}$, we obtain canonical morphisms respectively
\begin{equation}\label{eq:compare-lb}
  \cO(-e_{\max, \geq i+1}) \longrightarrow \cO(-e_{\max,i+1})|_{\AMm_{\setminus[i]}} \ \ \mbox{and} \ \ \cO(-e_{\max, \geq i+1}) \longrightarrow \cO(-e_{\max,\ge i+2})|_{\AMm_{\setminus[i]}}.
\end{equation}
For a positive integer $\ttwist$, this induces a morphism
\begin{equation}\label{eq:boundary-bundle-sujectivity}
  \cO(-\ttwist e_{\max, \geq i+1}) \longrightarrow  \big( \cO(-\ttwist e_{\max,i+1}) \oplus \cO(-\ttwist e_{\max,\ge i+2}) \big)|_{\AMm_{\setminus[i]}}
\end{equation}
with cokernel $\bK^{\vee}_{i}$.

\begin{lemma}\label{lem:compare-lb}
  The morphism \eqref{eq:boundary-bundle-sujectivity} is an injection of a sub-bundle.
  In particular, $\bK^{\vee}_{i}$ is a line bundle.
\end{lemma}

\begin{proof}
It suffices to show that fiberwise over $\AMm_{\setminus[i]}$ the two morphisms in \eqref{eq:compare-lb} cannot vanish simultaneously.  To see this, consider a geometric point $s \to \AMm_{\setminus[i]}$. Then over $s$ one has at least one of the following
\[
e_{\max, \geq i+1} - e_{\max, i+1} = 0 \ \ \ \mbox{or} \ \ \ e_{\max, \geq i+1} - e_{\max, \geq i+2} = 0,
\]
which induces an isomorphism of the trivial line bundle  as in  \eqref{eq:compare-torsor}.
This proves the statement.
\end{proof}

Next we construct the following commutative diagram of distinguished
triangles
\begin{equation}\label{diag:compare-reduced-POT}
\xymatrix{
\EE^{\red}_{\RAMm_{\setminus [i]}/\AMm_{\setminus [i]},+}  \ar[rr] \ar[d] && \EE_{\RAMm_{\setminus [i]}/\AMm_{\setminus [i]}}  \ar[rr] \ar@{=}[d] && \FF_{\RAMm_{\setminus [i]}/\AMm_{\setminus [i]},+}  \ar[rr]^{[1]} \ar[d] && \\
\EE^{\red}_{\RAMm_{\setminus [i]}/\AMm_{\setminus [i]}} \ar[rr] \ar[d] && \EE_{\RAMm_{\setminus [i]}/\AMm_{\setminus [i]}} \ar[rr] \ar@{=}[d] && \FF_{\RAMm_{\setminus [i]}/\AMm_{\setminus [i]}} \ar[rr]^{[1]} \ar[d] && \\
\KK_{i} \ar[rr] \ar[d]^{[1]} && 0 \ar[rr] \ar[d]^{[1]} && \KK_{i}[1] \ar[rr]^{[1]} \ar[d]^{[1]} && \\
&& && &&
}
\end{equation}
as follows. The top and middle row are obtained by pulling back  \eqref{eq:red-POT}. By \eqref{eq:boundary-complex} we have
\begin{align*}
\FF_{\RAMm_{\setminus [i]}/\AMm_{\setminus [i]},+}[1] &= \left(\bigoplus_{j=1}^{i+1} \cO(\ttwist e_{\max,j}) \right) \oplus \cO(\ttwist e_{\max, \geq i+2}),\\
\FF_{\RAMm_{\setminus [i]}/\AMm_{\setminus [i]}}[1] &= \left(\bigoplus_{j=1}^{i} \cO(\ttwist e_{\max,j}) \right) \oplus \cO(\ttwist e_{\max, \geq i+1})
\end{align*}
Taking the dual of \eqref{eq:boundary-bundle-sujectivity}, we obtain
\[
\cO(\ttwist e_{\max,i + 1}) \oplus \cO(\ttwist e_{\max, \geq i+2}) \to \cO(\ttwist e_{\max, \geq i+1}).
\]
This defines the arrow $\FF_{\RAMm_{\setminus [i]}/\AMm_{\setminus [i]},+} \to \FF_{\RAMm_{\setminus [i]}/\AMm_{\setminus [i]}}$ in \eqref{diag:compare-reduced-POT}, which completes to a triangle
\begin{equation}\label{eq:boundary-compare-complex}
\KK_{i} \longrightarrow \FF_{\RAMm_{\setminus [i]}/\AMm_{\setminus [i]},+} \longrightarrow \FF_{\RAMm_{\setminus [i]}/\AMm_{\setminus [i]}}   \stackrel{[1]}{\longrightarrow}
\end{equation}
with $\KK_{i} := \bK_{i}[-1]$. The diagram \eqref{diag:compare-reduced-POT} is thus constructed.

Taking the long exact sequence of the left column in \eqref{diag:compare-reduced-POT}, we obtain a short exact sequence relating the obstructions of the two reduced theories:
\[
0 \to H^1(\EE^{\red}_{\RAMm_{\setminus [i]}/\AMm_{\setminus [i]},+}) \to H^1(\EE^{\red}_{\RAMm_{\setminus [i]}/\AMm_{\setminus [i]},+}) \to \bK_{i} \to 0.
\]
This proves:

\begin{lemma}\label{lem:+red-vs-red}
 $[\RAMm_{\setminus [i]}]^{\red} = c_1(\bK_{i}) \cdot [\RAMm_{\setminus [i]}]_{+}^{\red}$.
\end{lemma}

\subsection{\texorpdfstring{$\psi_{\min}$}{psi-min}-classes}\label{ss:psi-min}
To further relate the virtual cycles $[\RAMm_{\setminus [i]}]^{\red}$
and $[\RAMm_{\setminus [i+1]}]^{\red}$, we introduce the
$\psi_{\min}$-classes, which appear in the virtual normal bundle in
the localization calculation of \cite{CJR23P}.
Consider the global sections
\[
e_{\min} \in \Gamma\left(\ocM_{\UMm_{[m]}}, \UMm_{[m]} \right)  \ \ \mbox{and} \ \ e_{\min,i} \in \Gamma\left(\ocM_{\UMm_i}, \UMm_i \right)
\]
which are the corresponding minimal degeneracies. We introduce the tautological classes
\begin{equation}\label{eq:psi-min-class}
\psi_{\min} := c_1\left( \cO(-e_{\min})\right) \ \ \mbox{and} \ \ \psi_{\min,i} := c_1\left( \cO(-e_{\min,i})\right).
\end{equation}

\smallskip

The following will be used to compute the push forward of
$[\RAMm_{\setminus [i]}]^{\red}$ in
Theorem~\ref{thm:virtual-split-disconnected-source} below.
\begin{lemma}\label{lem:push-forward-K}
$\ali_{i+1,*}c_1(\bK_{i}) = - \ttwist \psi_{\min}.$
\end{lemma}

\begin{proof}
Lemma \ref{lem:compare-lb} implies a relation between Chern classes
\[
1/\ttwist \cdot c_1(\bK_{i}) =  c_1(\cO(e_{\max, i+1})) + c_1(\cO(e_{\max, \ge i+2})|_{\AMm_{\setminus[i]}}) - c_1(\cO(e_{\max, \ge i+1})).
\]
Consider $\mathbf{D} := \frac 1\ttwist c_1(\bK_{i}) - c_1(\cO(e_{\min}))$ over $\AMm_{\setminus[i]}$. It remains to show $\ali_{i+1,*} \mathbf{D} = 0$.

On the other hand, write
\begin{align*}
\mathbf{D} &= c_1\big(\cO(e_{\max, i+1} - e_{\min})\big) - c_1\big(\cO(e_{\max, \ge i+1} - e_{\max, \ge i+2}) \big) \\
&=: \mathbf{D}_1 - \mathbf{D}_2.
\end{align*}
We observe that $\mathbf{D}_1$ and $\mathbf{D}_2$ can be represented
by effective Cartier divisors.
Indeed, as global sections over $\AMm_{\setminus[i]}$, the two
sections $(e_{\max,i+1} - e_{\min})$ and
$(e_{\max,\geq i+1} - e_{\max, \geq i+2})$ vanish over the open dense
substack of $\AMm_{\setminus[i]}$ defined by
\[
e_{\max,i+1} = e_{\max,\geq i+1} = e_{\max, \geq i+2} = e_{\min}.
\]
Hence, $\mathbf{D}_1$ and $\mathbf{D}_2$ can be identified with
effective Cartier divisors as in \eqref{eq:global-log-divisor},
supported along $(e_{\max, i+1} - e_{\min}) > 0$ and
$(e_{\max, \ge i+1} - e_{\max, \ge i+2}) > 0$ respectively.

By the local description of the Cartier divisors in \S
\ref{sss:log-line-bundles}, along $\mathbf{D}_2$ one has
$e_{\max, \geq i+1} = e_{\max, i+1}$.
This implies that the components in both $\mathbf{D}_1$ and
$\mathbf{D}_2$ along which $e_{\max,\geq i+2} = e_{\min}$, are
identical.
Hence $\mathbf{D} = \mathbf{D}_1 - \mathbf{D}_2$ is represented by an
effective divisor in $\AMm_{\setminus[i]}$ supported along the locus
with $e_{\min} \poleq e_{\max, \geq i+2} \poleq e_{\max,i+1}$ but
$e_{\min} \neq e_{\max, \geq i+2} \neq e_{\max,i+1}$.
Finally observe that the image of this locus in $\AMm_{\setminus[i+1]}$
is of codimension at least 2, hence $\ali_{i+1,*} \mathbf{D} = 0$ as desired.
\end{proof}

We next compute the push forward of the following expression along
$\Bl_{\cK}$ in \eqref{eq:split-min}:
\begin{equation}\label{eq:psi-min-series}
\frac{1}{-t - \psi_{\min}} = -\frac{1}{t} \Big(1 + \frac{\psi_{\min}}{-t} + \big(\frac{\psi_{\min}}{-t}\big)^2 + \cdots \Big).
\end{equation}
Here $t$ is a formal parameter, which will play the role of the
equivariant parameter in \cite{CJR23P}.

\begin{proposition}\label{prop:push-forward-min}
$\Bl_{\cK,*} \big( \frac{1}{-t - \psi_{\min}} \big) = \prod_{i \in [n]}\frac{1}{-t - \psi_{\min, i}}$.
\end{proposition}

\begin{proof}
  Pulling back \eqref{eq:min-log-blowup} along the second product in
  \eqref{eq:pun-extreme-configuration}, $\Bl_{\cK}$ decomposes into
  \begin{equation}\label{eq:push-forward-min}
    \UMm_{[m]} \stackrel{\mathrm{S}}{\longrightarrow} (\UMm_{[m]})^\fine \stackrel{\Bl^f_{\cK}}{\longrightarrow} \prod_{i\in [m]} \UMm_i,
  \end{equation}
where $\mathrm{S}$ is the saturation and $\Bl^f_{\cK}$ is the log blow-up in the fine category. On the underlying level, we have
\[
\Bl^f_{\cK} \colon \ul{(\UMm_{[m]})}^\fine = \PP(\bigoplus_{i \in [m]} \cO(e_{\min, i})) \to \prod_{i\in [m]} \ul{\UMm_i}.
\]
Note that the minimal degeneracy $e_{\min}$, the line bundle
$\cO(-e_{\min})$ and the class $\psi_{\min}$ exist over
$(\UMm_{[m]})^\fine$, which pull back to the corresponding $e_{\min}$,
$\cO(-e_{\min})$ and $\psi_{\min}$ over $\UMm_{[m]}$.
Thus, we have
$\Bl_{\cK,*} \frac{1}{-t - \psi_{\min}} = \Bl^f_{\cK,*}
\frac{1}{-t - \psi_{\min}}$.

To prove the statement, set $s_{1/t}(\bigoplus_i \cO(e_{\min, i}))$
and $c_{1/t}(\bigoplus_i \cO(e_{\min, i}))$ be the Segre and Chern
polynomial of $\bigoplus_i \cO(e_{\min, i})$ with variable $1/t$.
Observe that $\cO(-e_{\min})$ is the twisting sheaf $\cO(1)$ of the
projective bundle $\Bl^f_{\cK}$.
Pushing forward \eqref{eq:psi-min-series}, we have
\begin{align*}
\Bl^f_{\cK,*} \frac{1}{-t - \psi_{\min}}
&= -\frac{1}{t} \Bl^f_{\cK,*} \big( (\frac{\psi_{\min}}{-t})^{m-1} + (\frac{\psi_{\min}}{-t})^m + \cdots \big) \\
&= \frac{(-1)^{m}}{t^m} \Bl^f_{\cK,*} \big( \psi_{\min}^{m-1} + \frac{\psi_{\min}^m}{-t} + \frac{\psi_{\min}^{m+1}}{(-t)^2} + \cdots \big) \\
&= \frac{(-1)^{m}}{t^m} s_{-1/t}\Big(\bigoplus_i \cO(e_{\min, i})\Big) \\
&= \frac{(-1)^{m}}{t^m} c_{-1/t}^{-1}\Big(\bigoplus_i \cO(e_{\min, i})\Big) \\
&= \frac{(-1)^{m}}{t^m} \prod_{i \in [m]} \frac{1}{1 + \psi_{\min, i}/t}\\
&= \prod_{i \in [m]}\frac{1}{-t - \psi_{\min, i}}.
\end{align*}
This proves the statement.
\end{proof}

\subsection{Virtual push-forward along partial alignments}
Consider the Cartesian squares
\begin{equation}\label{diag:remove-max-order}
\xymatrix{
\RAMm_{[m]} \ar[r]^{\ali_1} \ar[d] & \RAMm_{\setminus[1]} \ar[r]^{\ali_2} \ar[d] & \cdots \ar[r] & \RMm_{\setminus[m-2]} \ar[r]^-{\ali_{m-1}} \ar[d]& \RMm_{\setminus[m-1]} = \RMm_{[m]} \ar[d] \ar[r]^-{\Bl_{\cK}}& \prod_{i\in [m]} \RMm_i \ar[d] \\
\AMm_{[m]} \ar[r]^{\ali_1} &  \AMm_{\setminus[1]} \ar[r]^{\ali_2} & \cdots \ar[r] & \AMm_{\setminus[m-2]} \ar[r]^-{\ali_{m-1}} &  \AMm_{\setminus[m-1]} = \UMm_{[m]} \ar[r]^-{\Bl_{\cK}} &   \prod_{i\in [m]} \UMm_i .
}
\end{equation}
where the bottom arrows are given by \eqref{eq:remove-order-max} and \eqref{eq:split-min}, and all vertical arrows are strict. By abuse of notations, we use $\ali_i$ and $\Bl_{\cK}$ to denote the corresponding top and bottom arrows. Consider the composition
\[
\ali_{[m]} := \ali_{m-1}\circ\cdots \circ \ali_1 \colon \RAMm_{[m]} \longrightarrow \RMm_{[m]} .
\]
Similarly, we write $\ali_{[m]} \colon \AMm_{[m]} \to \UMm_{[m]}$ for
the corresponding composition.
Our next goal is to establish the following virtual push-forward
formulas.

\begin{theorem}\label{thm:virtual-split-disconnected-source}
\begin{enumerate}
 \item For canonical virtual cycles, we have
\begin{equation}\label{eq:can-v-split-disconnected-source}
\Bl_{\cK,* }\ali_{[m],*}\left(\frac{[\RAMm_{[m]}]^{\vir}}{-t - \psi_{\min}} \right) = \prod_{i \in [m]} \frac{[\RMm_i]^\vir}{-t - \psi_{\min, i}}.
\end{equation}

\item Further assume Assumption \ref{as:reduced-punctured-cycle}, and
  fix a transverse superpotential $W$. Then, we have
 \begin{equation}\label{eq:red-v-split-disconnected-source}
\Bl_{\cK,* }\ali_{[m],*}\left( \frac{\ttwist t \cdot [\RAMm_{[m]}]^{\red}}{-t - \psi_{\min}} \right) = \prod_{i \in [m]}\frac{\ttwist t\cdot [\RMm_i]^{\red}}{-t - \psi_{\min, i}}.
 \end{equation}
\end{enumerate}
\end{theorem}

\begin{proof}
  We first establish \eqref{eq:can-v-split-disconnected-source}.
  Since $\psi_{\min}$ is defined over
  $\fM^{\diamondsuit,\fine}_{[m]}$, by
  Lemma~\ref{lem:v-cycle-partial-order-remove} (1) and the projection
  formula, we have
\begin{equation}\label{eq:can-vir-remove-align}
\ali_{[m], *}\left(\frac{[\RAMm_{[m]}]^{\vir}}{-t - \psi_{\min}} \right) = \left(\frac{[\RMm_{[m]}]^{\vir}}{-t - \psi_{\min}} \right)
\end{equation}

Consider the Cartesian squares
\begin{equation}\label{eq:split-min-vir-pushout}
\xymatrix{
\RMm_{[m]} \ar[rr] \ar[d] \ar@/^1.5pc/[rrrr]^{\Bl_{\cK}} && \SH^{\diamondsuit,\fine}_{[m]} \ar[d] \ar[rr] && \prod_{i\in [m]} \RMm_i \ar[d]\\
\UMm_{[m]} \ar[rr]^{\mathrm{S}} \ar@/_1.5pc/[rrrr]_{\Bl_{\cK}} && \fM^{\diamondsuit,\fine}_{[m]} \ar[rr]^{\Bl^f_{\cK}} && \prod_{i\in [m]} \UMm_i
}
\end{equation}
where the bottom decomposition is obtained by pulling back
\eqref{eq:min-log-blowup}, and the vertical arrows are all strict.
The canonical perfect obstruction theory of the right vertical arrow
pulls back to the canonical perfect obstruction theory of the middle
vertical arrow, hence defines the canonical virtual cycle
$[\SH^{\diamondsuit,\fine}_{[m]}]^{\vir}$.
Since the saturation $\mathrm{S}$ is projective and birational, the
virtual push-forward of \cite{Co06, Ma12} implies that
$\mathrm{S}_{*}[\RMm_{[m]}]^{\vir} =
[\SH^{\diamondsuit,\fine}_{[m]}]^{\vir}$.
Since $\psi_{\min}$ is defined over $\fM^{\diamondsuit,\fine}_{[m]}$,
the projection formula implies that
\begin{equation}\label{eq:can-vir-push-to-fine}
\mathrm{S}_{*} \left(\frac{[\RMm_{[m]}]^{\vir}}{-t - \psi_{\min}} \right) = \frac{[\SH^{\diamondsuit,\fine}_{[m]}]^{\vir}}{-t - \psi_{\min}}.
\end{equation}

Since $\Bl^{f}_{\cK}$ is flat, we obtain that
$\Bl^{f,*}_{\cK}\prod_{i\in [m]} [\RMm_i]^{\vir} =
[\SH^{\diamondsuit,\fine}_{[m]}]^{\vir}$.
Applying Proposition~\ref{prop:push-forward-min} and the projection
formula, we obtain
\begin{equation}\label{eq:can-vir-push-to-split}
\Bl^{f}_{\cK,*} \left( \frac{[\SH^{\diamondsuit,\fine}_{[m]}]^{\vir}}{-t - \psi_{\min}} \right) = \prod_{i \in [m]} \frac{[\RMm_i]^{\vir}}{-t - \psi_{\min, i}}.
\end{equation}
The canonical push-forward \eqref{eq:can-v-split-disconnected-source} then follows from \eqref{eq:can-vir-remove-align}, \eqref{eq:can-vir-push-to-fine} and \eqref{eq:can-vir-push-to-split}.

\smallskip

For the reduced push-forward \eqref{eq:red-v-split-disconnected-source}, using Lemma \ref{lem:v-cycle-partial-order-remove} (2), \ref{lem:+red-vs-red} and \ref{lem:push-forward-K} and the projection formula, we obtain
\[
\ali_{i+1,*}[\RAMm_{\setminus [i]}]^{\red} = -\ttwist \psi_{\min} \cdot [\RAMm_{\setminus [i+1]}]^{\red},
\]
hence
\begin{equation}\label{eq:ali-push-forward-red}
  \ali_{[m],*}[\RAMm_{[m]}]^\red = \ttwist^{m-1} (-\psi_{\min})^{m-1} \cdot [\RMm_{[m]}]^{\red}.
\end{equation}

To  compute the push-forward along  $\Bl_{\cK}$, observe that
\[
\frac{t (-\psi_{\min})^{m-1}}{-t - \psi_{\min}} = \big( t (-\psi_{\min})^{m-2} + t^2 (-\psi_{\min})^{m-3} + \cdots + t^{m-1} \big) + \frac{t^m}{-t - \psi_{\min}}.
\]
Furthermore, the factorization of $\Bl_{\cK}$ in \eqref{eq:split-min-vir-pushout} implies that $\Bl_{\cK,*} \psi_{\min}^{k} = 0$ for $k \leq m-1$.
Thus, the push-forward \eqref{eq:red-v-split-disconnected-source} follows from \eqref{eq:ali-push-forward-red} and the following
\[
\Bl_{\cK,*}\left( \ttwist^m \frac{t^{m} \cdot [\RMm_{[m]}]^{\red}}{-t - \psi_{\min}} \right) = \prod_{i \in [m]}\frac{\ttwist t\cdot [\RMm_i]^{\red}}{-t - \psi_{\min,i}},
\]
whose proof is identically to \eqref{eq:can-vir-push-to-fine} and \eqref{eq:can-vir-push-to-split}.

This completes the proof of Theorem \ref{thm:virtual-split-disconnected-source}.
\end{proof}

\subsection{A resolution using twisted alignments}

In order to apply the virtual localization of \cite{GrPa99} to
log GLSM, one needs to work with absolute perfect obstruction theories
of log R-maps.
However, the perfect obstruction theories of log R-maps are naturally
relative to the stack $\fM_{g, \ddata'}(\cA)$ of log maps
\cite{CJR21}, which is log smooth but in general has a singular
underlying stack.
This technical assumption can be addressed via a resolution using the
twisted alignments as in \S \ref{sss:twisted-align-log}.
While not used in the rest of this paper, we discuss this resolution
below.
It is an analogue of the expanded picture as in \cite{ACFW13,Kim} but
with no expansions.

\subsubsection{Aligned maps}\label{sss:aligned-map}
A punctured/log map $f\colon \pC \to \cA$ is {\em aligned} if the sheaf of degeneracies $\obD(f)$ is totally ordered. We consider a stack $\fM$ of aligned punctured/log maps which is a finite type open substack of either one of the following
\begin{enumerate}
\item $\prod_{i \in [m]}\fM_i$ where $\fM_i$ is the stack of punctured maps as in \eqref{eq:disconnected-universal-stacks}.

\item $\fM_{g, \ddata'}(\cA)$ the stack of log maps to $\cA$ with discrete data $(g, \ddata')$.
\end{enumerate}

Let $\obD_{\fM}$ be the sheaf of degeneracies of the universal punctured/log map over $\fM$. By Proposition \ref{prop:universal-alignment}, we obtain the the stack of aligned punctured/log maps $\AM := \cT_{\obD_{\fM}\cup\{0\}}$ with the projection $\AM \longrightarrow \fM$ given by \eqref{eq:align-A}.

In the case that $\fM$ parameterizes log maps as in (2), since $\AM$
admits maximal degeneracies, the morphism $\AM \longrightarrow \fM$ factors as
\begin{equation}\label{eq:aligned-map-stack}
\AM \longrightarrow \UM  \longrightarrow \fM.
\end{equation}
By Proposition \ref{prop:alignment} and \cite[Theorem 3.17]{CJRS18P},
the arrows in \eqref{eq:aligned-map-stack} are log \'etale, proper,
representable and birational.

In the case that $\fM$ parameterizes punctured maps as in (1), we
obtain a tautological morphism
\begin{equation}\label{eq:to-align-max}
\ali^{\ud} \colon \AM \to \AMm,
\end{equation}
which is proper and representable over $\fM$ by Proposition \ref{prop:universal-alignment}. Furthermore, \eqref{eq:to-align-max} restricts to an isomorphism over the open dense stack parameterizing smooth source curves, hence it is also birational.

\subsubsection{Twisted aligned maps}\label{sss:twisted-alignment}

In general, $\cM_{\AM}$ is not necessarily locally free, which causes
the underlying of $\AM$ to be singular.
To achieve smoothness, we introduce further twists as follows.

A {\em twisted aligned log/punctured map} $f\colon \pC \to \cA$ over
$S$ is an aligned log/punctured map such that fiberwise the set
$\obD(f) \cup \{0\}$ is of the form as in
\eqref{eq:tw-alignment-elements}:
\[
 0, r_1 e_1, r_1 e_1 + r_2 e_2, r_1 e_1 + r_2 e_2 + r_3 e_3, \cdots, r_1 e_1 + \cdots + r_k e_k.
\]
By \S \ref{sss:twisted-align-log}, the moduli of twisted aligned maps in $\fM$ is $\cT^{tw}\times_{\cT_{\log}}\AM$.

To study the structure of $\ocM_S$, suppose that $f$ is given by a strict geometric log point $S \to \cT^{tw}\times_{\cT_{\log}}\AM$. Let $G$ be the (not necessarily connected) dual graph of the underlying curve of $\pC$ with the the set of vertices $V(G)$ corresponding to the the set irreducible components, and edges $E(G)$ corresponding to the set of nodes with the additional data:
\begin{enumerate}
 \item A partition $V(G) = \cup_{i=0}^{k}V_i$ with $V_i$ representing the set of components with the degeneracies $\delta_0 = 0$ and $\delta_i := r_1 e_1 + \cdots + r_i e_i$ for $i \geq 1$.

 \item Each $E \in E(G)$ is decorated by its contact order $c_E \in \NN$.
\end{enumerate}
For each $i < j$, denote by $E_{ij}(G) \subset E(G)$ the subset
consisting of $E \in E(G)$ joining vertices in $V_{i'}$ and $V_{j'}$
such that $i' \leq i < j \leq j'$.
Let $E^{0}(G) \subset E(G)$ be the subset consisting of edges with the
zero contact order.
The key to the smoothness of $\TAM$ in the case of log maps is:

\begin{proposition}\label{prop:freeness}
Notations and assumptions as above, $\ocM_S$ is free provided that $c_E | r_{i+1}$ for any $E \in E_{i(i+1)}(G)$ and any $i \in \{0, 1, \cdots, k-1\}$.
\end{proposition}

\begin{proof}
To describe the monoid $\ocM_S$, we introduce:
\begin{enumerate}
\item[$\rho_E\colon$] the nodal smoothing parameter corresponding to $E$ as in \cite[(2.3)]{ACGS20P};
\item[$\delta_v\colon$] the degeneracy of the vertex $v \in V(G)$.
\end{enumerate}
These parameters viewed as elements of $\ocM_S$ satisfy:
\begin{enumerate}
\item[(A)$\colon$] $\delta_v = \delta_i$ if $v \in V_i$.
\item[(N)$\colon$] $\delta_v = \delta_{v'} + c_E \cdot \rho_E$ for $E$ joining $v \in V_i$ and $v' \in V_j$ with $i \leq j$.
\end{enumerate}
Denote by $\ocN_S$ the fine sharp monoid given by
\[
\langle e_i,  \rho_E \ | \  1 \leq i \leq k \ \mbox{and} \ E \in E(G) \rangle \big/ \sim
\]
with the relations $\sim$ given by (A) and (N). By the fiber product construction of $\cT^{tw}\times_{\cT_{\log}}\AM$ and the definition of basic monoids in \cite[\S 2.3.1]{ACGS20P}, the monoid $\ocM_S$ is the saturation of $\ocN_S$ in the category of sharp monoids. Now the proposition follows from the next lemma.
\end{proof}

\begin{lemma}\label{lem:twisted-aligned-monoid}
In the above situation, $\ocM_S \cong \ocN_S$ is a free monoid with the set of generators
\begin{equation}\label{eq:twisted-aligned-monoid}
\{e_i,  \rho_E \ | \  1 \leq i \leq k \ \mbox{and} \ E \in E(G) \ \mbox{with} \ c_E = 0 \}.
\end{equation}
\end{lemma}

\begin{proof}
It suffices to show that for any $E \in E(G)$ with $c_E \neq 0$, the element $\rho_E$ can be expressed as a integral combination of $e_i$'s with positive coefficients.

Suppose that $E$ connects two vertices $v \in V_i$ and $v' \in V_j$ with $i < j$. We have
\[
c_E \cdot \rho_E = \delta_{v'} - \delta_v = \delta_{v_j} - \delta_{v_i} = \sum_{l = i+1}^j r_l e_l.
\]
Since $c_E | r_l$ by assumption, the statement follows from
\[
\rho_E =  \sum_{l = i+1}^j \frac{r_l}{c_E}e_l.
\]
\end{proof}

\subsubsection{Twisted aligned log R-maps}\label{sss:twist-choice-log}
Suppose that $\fM \subset \fM_{g, \ddata'}(\cA)$ is a finite type open substack of the moduli of log maps, and $\AM \to \fM$ is the corresponding alignment. For any twisting choice $\fr$, we obtain a log smooth stack $\TAM_{\fr} := \cT^{\fr}_{\log}\times_{\cT}\AM$ with the projection
\begin{equation}\label{eq:twist-aligned-degeneracy}
\Tw \colon \TAM_{\fr}  \to \AM,
\end{equation}
which is birational as it restricts to an isomorphism over the open
dense substack with the trivial log structure.

Note that a log smooth stack is smooth if the log structure is locally
free.
Given the freeness criterion in Proposition \ref{prop:freeness}, we
observe that one can choose $\fr$ such that $\TAM_{\fr}$ has free log
structures, hence is both log smooth and smooth, and $\Tw$ is proper.

Indeed, since $\fM$ is of finite type, the possible contact orders at
the nodes are also finite.
A {\em global twisting index} is a positive number $\lambda$ divisible
by the contact order of each node over each geometric point of $\fM$.
Fix such a $\lambda$.
Let $\fr$ be the twisting choice that assigns $r_i = \lambda$ to each
$r_i$ as in \eqref{eq:tw-alignment-elements}.
Locally, the morphism $\cT^{\fr} \to \cT$ takes the $\lambda$-th root
stack construction along each irreducible component of the boundary of
$\cT$.
In particular, $\cT^{\fr} \to \cT$ is proper which pulls back to the
proper $\Tw$.
Furthermore by Proposition \ref{prop:freeness}, $\cM_{\TAM_{\fr}}$ is
locally free, hence $\TAM_{\fr}$ is both log smooth and smooth.

\begin{remark}
  The choice $\fr$, and hence the resolution $\TAM_{\fr}$ are not
  unique.
  One could also take the minimal choice similar to \cite{AbFa17}.
  In fact, these choices are to ensure the technical assumptions of
  the localization formula, and play no role in the calculation of
  virtual cycles.
\end{remark}

Consider the discrete data for the stable log R-maps of \cite{CJR21}:
\begin{equation}\label{eq:[n]-log-R-data}
(g, \beta, \ddata = \{(\ogamma_j, c_j))\}_{j=1}^{n})
\end{equation}
compatible with the discrete data $(g, \ddata')$ of log maps to $\cA$. Write $\SH = \SH_{g,n}(\fP, \beta)$  for simplicity. We have a tautological morphism $\SH \to \fM_{g, \ddata'}(\cA)$ induced by the strict morphism $\fP \to \cA$. Since $\SH$ is of finite type, let $\fM \subset \fM_{g,\ddata'}(\cA)$ be a finite type open substack containing the image of $\SH$. We form the following commutative diagram of Cartesian squares in fs category with strict vertical arrows
\begin{equation}\label{eq:resolution-log-R-map}
\xymatrix{
\RTA \ar[r] \ar[d] & \RA \ar[r] \ar[d] & \UH \ar[r] \ar[d] & \SH \ar[d] \\
\TAM_{\fr} \ar[r]^{\Tw} & \AM \ar[r]^{\ali^{\curlywedge}} & \UM \ar[r]^{F^{\curlywedge}} & \fM
}
\end{equation}
where the middle and right arrows of the bottom are given by \eqref{eq:aligned-map-stack}. Pulling back the canonical perfect obstruction theory of $\SH \to \fM$ as defined in \cite{CJR21}, we obtain the canonical virtual cycles
$
[\RTA]^{\vir}, [\RA]^{\vir}, [\UH]^{\vir}$ and  $[\SH]^{\vir}$.
Pulling back the reduced perfect obstruction theory of $\UH \to \UM$ as defined in \cite{CJR21}, we obtain the reduced virtual cycles
$[\RTA]^{\red}, [\RA]^{\red}$, and  $[\UH]^{\red}$.

Choosing a global twisting choice $\lambda$ as above, we may assume that $\TAM_{\fr}$ is a smooth stack, and the bottom arrows are all proper birational. In this case, virtual push-forward of  canonical  and reduced virtual cycles holds, namely
\begin{equation}\label{eq:push-forward-tw-align-log}
\Tw_* \ali^{\curlywedge}_*F^{\curlywedge}_*[\RTA]^{\vir} = [\SH]^{\vir} \ \ \ \mbox{and} \ \ \ \Tw_* \ali^{\curlywedge}_*F^{\curlywedge}_*[\RTA]^{\red} = [\SH]^{\red}.
\end{equation}

\subsubsection{Twisted aligned punctured R-maps}\label{sss:twist-choice-pun}
Suppose that $\fM \subset \prod_{i\in [m]} \fM_i$ is a finite type
open substack of punctured maps, and $\AM \to \fM$ is the
corresponding stack of aligned maps.
The morphism $\AM \to \cT_{\log}$ factors through the closed substack
$\cT_{\Delta}\subset \cT_{\log}$ of non-trivial log structures.
For a twisting choice $\fr$, let
$\cT^{\fr}_{\Delta} \subset \cT^{\fr}_{\log}$ be the closed substack
with non-trivial log structures, and write
$\TAM_{\fr} := \cT^{\fr}_{\Delta}\times_{\cT_{\Delta}}\AM$ for the
twisted alignements with the projection
\begin{equation}\label{eq:twist-aligned-pun}
  \Tw \colon \TAM_{\fr} \longrightarrow \AM.
\end{equation}

As in the case of log maps, a {\em global twisting index} is a
positive number $\lambda$ divisible by the contact order of each node
over each geometric point of $\fM$.
Again, fix such a $\lambda$.
We choose $\fr$ which assigns $r_i = \lambda$ to each generator
fiberwise as in \eqref{eq:tw-alignment-elements}.
Thus, the morphism $\cT^{\fr}_{\Delta} \to \cT_{\Delta}$, and hence
$\Tw$ is proper and generically \'etale of degree $1/\lambda$.

For punctured R-maps, consider a finite collection of discrete data
\[
(g_i, \beta_i, \ddata_i = \{(\ogamma_{ij}, c_{ij}))\}_{j=1}^{n_i}), \ \ \mbox{for } i \in [m].
\]
as in \eqref{eq:[m]-punctured-R-data} compatible with
\eqref{eq:[m]-punctured-data}.
Write for simplicity $\SH_i = \SH_{g_i,\ddata_i}(\punt,\beta_i)$.
Then for each $i \in [m]$, we obtain a morphism $\SH_i \to \fM_i$ from
\eqref{eq:take-log}.
As $\SH_i$ is of finite type, we may replace $\fM_i$ by a finite type
open substack containing the image of $\SH_i$, which by abuse of
notation, we again denote by $\fM_i$.
We obtain a commutative diagram with all squares Cartesian in the fs
category:
\begin{equation}\label{diag:removing-orders}
\xymatrix{
\RTA_{[m],\fr} \ar[r] \ar[d] & \RA_{[m]} \ar[r] \ar[d] & \RAMm_{[m]} \ar[r] \ar[d] & \RMm_{[m]} \ar[r] \ar[d] & \prod_{i\in [m]} \RMm_i \ar[r] \ar[d] &  \prod_{i \in [m]} \SH_i \ar[d]\\
\TAM_{[m],\fr} \ar[r]^{\Tw}  & \AM_{[m]} \ar[r]^{\ali^{\ud}}  & \AMm_{[m]} \ar[r]^{\ali_{[m]}}  & \UMm_{[m]} \ar[r]^-{\Bl_{\cK}}  & \prod_{i\in [m]} \UMm_i \ar[r]^{(\ref{eq:remove-diamond})}  & \prod_{i \in [m]} \fM_i
}
\end{equation}
Each stack of punctured R-maps on the top carries the configuration of
degeneracies as the corresponding bottom stack of punctured maps.
Since all the vertical arrows are strict, squares in both diagrams are
Cartesian on the level of underlying stacks as well.
Thus, whenever the canonical or reduced theory of
$\RMm_{[m]} \to \AMm_{[m]}$ is defined in this paper, their pull-backs
define the canonical and reduced virtual cycles of $\RTA_{[m],\fr}$,
denoted by $[\RTA_{[m],\fr}]^{\vir}$ and $[\RTA_{[m],\fr}]^{\red}$,
respectively.
Choosing a global twisting index $\lambda$ as above, we further have
the virtual push-forwards
\begin{equation}\label{eq:push-forward-tw-align-pun}
\Tw_{*}\ali^{\ud}_*[\RTA_{[m],\fr}]^{\vir} = (1/\lambda) \cdot [\RAMm_{[m]}]^{\vir}  \ \mbox{and}  \ \Tw_{*}\ali^{\ud}_*[\RTA_{[m],\fr}]^{\red} = (1/\lambda) \cdot [\RAMm_{[m]}]^{\red}.
\end{equation}

Theorem \ref{thm:virtual-split-disconnected-source} together with \eqref{eq:push-forward-tw-align-log} and \eqref{eq:push-forward-tw-align-pun} will be used in \cite{CJR23P} to compute the virtual cycles of log GLSM of \cite{CJR21}  via torus localization and boundary decompositions.

%%%%%%%%%%%%%Changing ell

\section{Logarithmic root constructions}\label{sec:log-root}

In this section, we fix a target $\punt \to \BC$ as in \S \ref{sss:P-target}, and a positive integer $\ell$. Consider the following Cartesian diagram with strict vertical arrows:
\begin{equation}\label{eq:ell-root-target}
\xymatrix{
\punt^{1/\ell} \ar[rr] \ar[d] && \punt \ar[d] \\
\ainfty^{1/\ell} \ar[rr]^{\nu_{\ell}} && \ainfty.
}
\end{equation}
where  $\ainfty^{1/\ell} := \ainfty$ and $\nu_{\ell}$ is the {\em $\ell$th root morphism} defined by $\bar{\nu}_{\ell}^{\flat}(1) = \ell$ on the level of characteristic sheaves.
Then the composition $\punt^{1/\ell} \to \punt \to \BC$
defines another target satisfying the conditions in
\S\ref{sss:P-target}.

We will compare punctured R-maps and their virtual cycles with targets
$\punt$ and $\punt^{1/\ell}$.
This comparison is the geometric foundation for the LG/CY
correspondence as explained in Section~\ref{sss:LGCY} below.

%%%%%%%%%%%Log root
\subsection{Comparison of canonical virtual cycles}

\subsubsection{Comparison of discrete data}\label{sss:discrete-data-change-ell}

Let $(g, \beta, \ddata = \{(\ogamma_i, c_i)\}_{i=1}^n)$ be the
discrete data of punctured R-maps to $\punt^{1/\ell}$ with connected
domains.
By \cite[Lemma~1.1.11]{AbFa17}, the sector
$\ogamma_i \subset \ocI\punt^{1/\ell}_{\bk}$ uniquely determines a
sector $\ogamma'_i \subset \ocI\punt_{\bk}$ with a natural
$\mu_{\ell/\rho_i}$-gerbe $\ogamma_i \to \ogamma'_i$ induced by
$\punt^{1/\ell}_{\bk} \to \punt_{\bk}$, where $\rho_i := r_i / r'_i$
with $r_i$ and $r_i'$ the indices of the gerbes parameterized by
$\ogamma_i$ and $\ogamma'_i$ respectively.
In particular, we have $\rho_i | \ell$.
We then introduce  the {\em compatible discrete data} of punctured  R-maps to $\punt$
\begin{equation}\label{eq:change-ell-discrete-data}
(g, \beta, \bar{\ddata} = \{(\ogamma'_i, c'_i = \frac{\ell }{\rho_i}\cdot c_i)\}_{i=1}^n).
\end{equation}

\begin{lemma}\label{lem:change-target}
  Let $f_{\ell} \colon \pC_{\ell} \to \punt^{1/\ell}$ be a stable
  punctured R-map over a log scheme $S$ with the discrete data
  $(g, \beta, \ddata = \{(\ogamma_i, c_i)\}_{i=1}^n)$.
  Then the composition $\pC_{\ell} \to \punt^{1/\ell} \to \punt$
  factors through a unique stable punctured R-map
  $f_1 \colon \pC_{1} \to \punt$ with discrete data
  \eqref{eq:change-ell-discrete-data} such that the underlying
  $\uC_{\ell} \to \uC_1$ induces an isomorphism between the coarse
  curves.
\end{lemma}

Thus, we obtain a canonical morphism of log stacks
\begin{equation}\label{eq:canonical-stack-change-ell}
\nu_{\ell} \colon \SH_{g,\ddata}(\punt^{1/\ell}, \beta) \to \SH_{g,\bar{\ddata}}(\punt,\beta).
\end{equation}

\begin{proof}
  The proof is similar to \cite[Proposition 2.15]{CJR21} with minor
  modifications taking care of the punctured points.
  Let $\ul{C}_{\ell}$ be the underlying coarse curve of $\pC_{\ell}$.
  Consider the punctured map
  $h_{\ell} \colon \pC_{\ell} \to {\punt}^{1/\ell} \times_{\BC}
  \ul{C}_{\ell}$ induced by $f_{\ell}$.
  The stability of $f_{\ell}$ implies that the underlying
  $\ul{h}_{\ell}$ is stable in the usual sense.
  Thus by \cite[Proposition 9.1.1]{AbVi02} the composition
  $\uC_{\ell} \to \ul{\punt}^{1/\ell} \times_{\BC} \ul{C}_{\ell} \to
  \ul{\punt} \times_{\BC} \ul{C}_{\ell}$ factors through a usual
  stable map
  $\ul{h}_1 \colon \uC_{1} \to \ul{\punt} \times_{\BC}
  \ul{C}_{\ell}$ over $\ul{S}$.
  Since $\ul{\punt}^{1/\ell} \to \ul{\punt}$ is \'etale, the composition
  $\uC_{1} \to \ul{\punt} \times_{\BC} \ul{C}_{\ell} \to
  \ul{C}_{\ell}$ is the coarse morphism.
  Denote by $\ul{f}_1$ the R-map given by the composition
  $\uC_{1} \to \ul{\punt} \times_{\BC} \ul{C}_{\ell} \to
  \ul{\punt}$. The sector associated to the $i$th marking of  the underlying R-map $\ul{f}_1$ is $\ogamma_i'$ as in \eqref{eq:change-ell-discrete-data}.

  It remains to lift $\ul{f}_1$ to a punctured map.
  Denote by
  $\mathfrak{f}_{\ell} \colon \pC_{\ell} \to \ainfty^{1/\ell}$ the
  punctured map induced by $f_{\ell}$, and by
  $\ul{\mathfrak{f}}_{1} \colon \uC_{1} \to \ul{\ainfty}$ the map
  induced by $\ul{f}_1$.
  It suffices to construct a punctured map
  ${\mathfrak{f}}_{1} \colon \pC_{1} \to \ainfty$ over $S$ with
  underlying $\ul{\mathfrak{f}}_{1}$ making the following diagram
  commutative:
  \begin{equation}\label{diag:universal-map-change-ell}
    \xymatrix{
      \pC_{\ell} \ar[rr]^{\mathfrak{f}_{\ell}} \ar@{-->}[d] && \ainfty^{1/\ell} \ar[d] \\
      \pC_{1} \ar@{-->}[rr]^{\mathfrak{f}_{1}}  && \ainfty
    }
  \end{equation}

First, by the same argument as in the second paragraph in the proof of
\cite[Proposition 2.15]{CJR21}, we obtain a log curve $\cC_{1} \to S$
over $\uC_1$ and a morphism of log curves $\cC_{\ell} \to \cC_1$ over
$S$.
Since both $\ainfty^{1/\ell}$ and $\ainfty$ are closed substacks of
the Artin fan $\cA$, to define the dashed arrows in above commutative
diagram, it suffices to define the punctured curve and the dashed
morphisms on the level of characteristic sheaves as in the following
commutative diagram:
\begin{equation}\label{diag:characteristic-map-change-ell}
\xymatrix{
\ocM_{\cC_{\ell}} \ar[rr]^{\subset} && \ocM_{\pC_{\ell}}  && \ocM_{\ainfty^{1/\ell}}|_{\pC_{\ell}}  \ar[ll]_{\mathfrak{\bar f}^{\flat}_{\ell}}\\
\ocM_{\cC_1}|_{\pC_{\ell}}  \ar[u] \ar[rr]^{\subset} && \ocM_{\pC_1}|_{\pC_{\ell}}  \ar@{-->}[u] && \ocM_{\ainfty}|_{\pC_{\ell}} \ar@{-->}[ll]_{\mathfrak{\bar f}^{\flat}_1}  \ar[u]_{\bar{\nu}^{\flat}_{\ell} \colon 1 \mapsto \ell}
}
\end{equation}
Since both
$\ocM_{\ainfty^{1/\ell}}|_{\pC_{\ell}} \cong \NN|_{\pC_{\ell}}$ and
$\ocM_{\ainfty}|_{\pC_{\ell}} \cong \NN|_{\pC_{\ell}}$ are globally
constant, the commutativity of
\eqref{diag:characteristic-map-change-ell} forces the following
equality in the monoid $\Gamma(\pC_{\ell}, \ocM_{\pC_{\ell}})$:
\begin{equation}\label{eq:characteristic-map-change-ell}
\mathfrak{\bar f}^{\flat}_1(1) :=  \mathfrak{\bar f}^{\flat}_{\ell}(\ell) = \ell \cdot \mathfrak{\bar f}^{\flat}_{\ell}(1).
\end{equation}

Since the case of nodes and smooth points is completely identical to
\cite[Proposition 2.15]{CJR21}, we will only construct the punctured
maps around markings.
Unwinding \eqref{eq:characteristic-map-change-ell}, consider a
neighborhood $U_1 \to \cC_{1}$ of the $i$th marking $p$, and let
$U_{\ell} \to \pC_{\ell}$ be an \'etale neighborhood of the marking
$p'$ above $p$ such that the underlying $\ul{U}_{\ell} \to \ul{U}_1$
is the $\rho_i$th root stack along the $i$th marking.
Let $\sigma \in \ocM_{U_1}$ and $\sigma' \in \ocM_{U_{\ell}}$ be the
local sections that lift to the local equation of the $i$th marking.
This implies that $\rho_i \sigma' = \sigma$.

Using the above local sections, the equation \eqref{eq:characteristic-map-change-ell} takes the shape
\[
\mathfrak{\bar f}^{\flat}_1(1) = \ell \cdot \mathfrak{\bar f}^{\flat}_{\ell}(1) = \ell (e + c_i \sigma) =  \ell e + \ell c_i \sigma
\]
where $e$ is the degeneracy of $U_{\ell}$.
On the other hand, if $\mathfrak{\bar f}^{\flat}_1$ lifts to a
punctured map $f_1$, then locally we have
\begin{equation}\label{eq:log-map-remove-ell}
\mathfrak{\bar f}^{\flat}_1(1) = e' + c'_i \sigma'
\end{equation}
where $e'$ is the corresponding degeneracy and $c'_i$ is the contact
order of the to be constructed $f_1$ along the $i$th marking.
Thus \eqref{eq:characteristic-map-change-ell} amounts to
\begin{equation}\label{eq:unwinding-characteristic-change-ell}
e' = \ell e \ \ \mbox{and} \ \ \  \ell c_i = c'_i \rho_i.
\end{equation}
Since $\rho_i | \ell$, we obtain $c'_i = \frac{\ell c_i}{\rho_i} \in \ZZ$ as needed.
This means that once we have constructed a punctured curve  $\pC_1 \to S$
with the local section $e' + c'_i \sigma'$ around the $i$th marking, then $\mathfrak{f}_1$ is uniquely defined using \eqref{eq:log-map-remove-ell}.

Finally we construct the punctured curve $\pC_1 \to S$ as follows.
We only need to consider the situation $c_i < 0$, i.e. there is a
puncture at $p$.
In this case since $\ell > 0, \rho_i > 0$, we have $c'_i < 0$.
Then $e + c_i \sigma$ is a local section in $\ocM_{\pC_\ell}$.
This implies that any lifting of $e$, hence $e'$ in the log structure
constantly vanishes in the structure sheaf.
Thus we may define the punctured curve $\pC_1$ by adding
$e' + c'_i \sigma'$ to the characteristic sheaf as in
\cite[Definition~2.1, Remark~2.2]{ACGS20P}.

This finishes the proof.
\end{proof}

For later use, we record the following result on degeneracies:

\begin{corollary}\label{cor:change-ell-degeneracy}
Suppose we are in the situation of Lemma \ref{lem:change-target} with $S$ a log point. Let $Z \subset \pC_{\ell}$ be an irreducible component with degeneracy $e$. Let $Z' \subset \pC_{1}$ be the image of $Z$ with degeneracy $e'$. Then $e' = \ell \cdot e$.
\end{corollary}
\begin{proof}
This is the first equality in \eqref{eq:unwinding-characteristic-change-ell}. Though \eqref{eq:unwinding-characteristic-change-ell} is for a neighborhood of a marking, the same argument applies to any smooth point locally on the curve $\pC_{\ell}$.
\end{proof}

Conversely we have:

\begin{lemma}\label{lem:map-taking-ell-root}

Fix discrete data $(g, \beta, \bar{\ddata} = \{(\ogamma'_i, c'_i)\}_{i=1}^n)$ of stable punctured maps to $\punt$, and write $\rho_i := \lcm(\ell, c_i')/c_i'$ for all $i$. Then there is discrete data
\[
(g, \beta, {\ddata} = \{(\ogamma_i, c_i = \lcm(\ell, c_i')/\ell =  \frac{\rho_i\cdot c'_i}{\ell})\}_{i=1}^n)
\]
of stable punctured R-maps to $\punt^{1/\ell}$ unique in the following sense:

Let $f_{\ell} \colon \pC_{\ell} \to \punt^{1/\ell}$ be any stable
punctured R-map such that the composition
$\pC_{\ell} \to \punt^{1/\ell} \to \punt$ factors through a stable
punctured R-map $f_1 \colon \pC_1 \to \punt$ as in Lemma
\ref{lem:change-target} with the discrete data
$(g, \beta, \bar{\ddata})$.
Suppose the underlying of $\pC_{\ell} \to \pC_{1}$ is the $\rho_i$th
root stack along the $i$th marking for each $i$.
Then $f_{\ell}$ has the discrete data $(g, \beta, {\ddata})$.
\end{lemma}
\begin{proof}
Note that $c_i = \frac{\rho_i\cdot c'_i}{\ell}$  follows from \eqref{eq:unwinding-characteristic-change-ell}. Now suppose we are given a punctured R-map $f_{\ell} \colon \pC_{\ell} \to \punt^{1/\ell}$ as in the statement.
Following the notations in the proof  of Lemma \ref{lem:change-target}, consider the corresponding $i$th markings  $p' \subset \pC_1$ and $p \subset \pC_{\ell}$. The factorization implies a commutative diagram
\begin{equation}\label{diag:sector-change-ell}
\xymatrix{
p \ar[rr] \ar[d] && \punt^{1/\ell}_{\bk} \ar[d] \ar[rr] &&  \punt^{1/\ell} \ar[d]  \\
p' \ar[rr] && \punt_{\bk} \ar[rr] &&  \punt
}
\end{equation}
where the square on the right is Cartesian with strict horizontal arrows, and the composition of the top and bottom arrows are given by $f_{\ell}$ and $f_1$ respectively. Let $\ogamma_i \subset \ocI\punt^{1/\ell}_{\bk}$ be the sector corresponding to $p \to \punt^{1/\ell}_{\bk}$. It remains to determine this sector $\ogamma_i$.

Since the sector $\ogamma_i$ can be identified locally, let $[X/G] \to  \punt_{\bk}$ be a  local chart where $X$ is a variety and G is a finite group. The action of $G$ on $\cO(\punt)|_{X}$ is given by a character $\chi \colon G  \to \GG_m$. A similar discussion as in \cite[\S 1.1.10]{AbFa17} implies a local  model $[X/\tilde{G}] \to \punt^{1/\ell}_{\bk}$ over  $[X/G]$ with the group extension $\tilde{G} = G \times_{\chi, \GG_m, \nu_{\ell}}\GG_m$. Recall that $p$ is an $\mu_{r_{i}}$-gerbe. To determine $\ogamma_i$ it suffices to exhibit an injection $\mu_{r_i} \to \tilde{G}$ of inertia groups associated to  $p \to \punt^{1/\ell}_{\bk}$.

Note that the middle and right vertical arrows of \eqref{diag:sector-change-ell} are given by the $\ell$th root $\cO(\punt)|_{\punt^{1/\ell}} \cong \cO(\ell \cdot \punt^{1/\ell})$. Recall from the proof of Lemma  \ref{lem:change-target} that
\[
\cO(\punt)|_p \cong \cO(\bar{f}_1^{\flat}(1))|_p  \cong \cO(\ell \bar{f}_{\ell}^{\flat}(1))|_{p}.
\]
Thus the underlying of $p \to \punt^{1/\ell}_{\bk}$ is defined by the pull-back $\cO(\punt^{1/\ell})|_{p} \cong \cO(\bar{f}_{\ell}^{\flat}(1))|_{p}$. Using $\bar{f}_{\ell}^{\flat}(1) = e + c_i \sigma$ around $p$, we have $\cO(\punt^{1/\ell})|_{p} = \cO(e)\otimes \cO(c_i \sigma)$. Since $p$ is a $\mathbf{B}\mu_{r_i}$-gerbe, for any point $\spec \bk  \to p$ the $\mu_{r_i}$-action on the vector space $\cO(\punt^{1/\ell})|_{p}\times_{p}\spec \bk \cong \bk$ is described by $z \cdot v = z^{c_i} v$ for any $z \in \mu_{r_i}$ and $v \in \bk$.

Similarly, since $p'$ is a $\mathbf{B}\mu_{r_i'}$-gerbe, for any point $\spec \bk \to p'$, the $\mu_{r_i'}$-action on the vector space $\cO(\punt)|_{p'}\times_{p'}\spec \bk \cong \bk$ is described by $\zeta \cdot w = \zeta^{c_i'}w$ for any $\zeta \in \mu_{r_i'}$ and $w \in \bk$. We have arrived at a commutative diagram of solid arrows of groups
\begin{equation}\label{diag:change-ell-inertia}
\xymatrix{
\mu_{r_i} \ar[rr]^{z \mapsto z^{\rho_i}} \ar@{-->}[d] \ar@/_2pc/[dd]_{z \mapsto z^{c_i}} && \mu_{r_i'} \ar[d] \ar@/^2pc/[dd]^{\zeta \mapsto \zeta^{c_i'}} \\
\tilde{G} \ar[rr] \ar[d] && G \ar[d]^{\chi} \\
\GG_m \ar[rr]^{\nu_{\ell}} && \GG_m
}
\end{equation}
where $\mu_{r_i'} \to G$ is the injection given by $\ogamma_i'$ such that the composition $\mu_{r_i'} \to G \to \GG_m$ is $\zeta \to \zeta^{c_i'}$. The commutativity of the outer square follows from $c_i \ell = c'_i \rho_i$. Since the bottom square is Cartesian, we obtain a unique group morphism $\mu_{r_i} \to \tilde{G}$ fitting in the above diagram. Note that $\ker(\mu_{r_i} \to \mu_{r_i'}) = \mu_{\rho_i}$. To see that this determines $\ogamma_i$ uniquely, it suffices to observe that $\gcd(\rho_i, c_i) = 1$ hence the morphism $\mu_{r_i} \to \tilde{G}$ is injective.

Finally we observe that the diagram \eqref{diag:change-ell-inertia} is completely determined by $\ogamma_i', \ell, c_i'$, independent of the existence of $f_{\ell}$. This completes the proof.
\end{proof}

\begin{remark}\label{rem:marking-data-change-ell}
  The proof of Lemma~\ref{lem:map-taking-ell-root} shows that
  $\ogamma_i$ is uniquely determined by the data
  $(\ogamma_i', \ell, c_i')$.
  For any positive rational $x \in \QQ$, denote by $\{x\}$ the
  fractional part of $x$.
  The proof also shows that
  \begin{equation}
    \label{eq:ell-lift-age}
    \age \cO(\punt^{1/\ell})|_{\gamma_i} = \{\frac{c_i}{r_i}\} = \{\frac{c'_i}{r'_i \ell}\}.
  \end{equation}
  In fact, the discussion in \cite[\S 1.1.10]{AbFa17} implies that
  there are exactly $r$ sectors $\ogamma$ over $\ogamma_i'$,
  distiguished by $\age \cO(\punt^{1/\ell})|_{\ogamma}$, and hence
  \eqref{eq:ell-lift-age} determines $\ogamma_i$ in terms of
  $\ogamma_i'$.
\end{remark}
\begin{remark}
  Although not stated in \cite{CJR21}, there is an analog of
  Lemma~\ref{lem:map-taking-ell-root} in the setting of log R-maps by a  similar argument.
\end{remark}

\subsubsection{Comparison of canonical virtual cycles}

\begin{proposition}\label{prop:can-vir-change-ell}
Suppose \eqref{eq:PT-smooth} is smooth. Then we have the virtual push-forward
\[
\nu_{\ell, *}[\SH_{g,\ddata}(\punt^{1/\ell}, \beta)]^{\vir} = \ell^{-1} \cdot [\SH_{g,\bar{\ddata}}(\punt,\beta)]^{\vir}.
\]
\end{proposition}
\begin{remark}
  The analogous result in orbifold Gromov--Witten theory involves a
  factor of $\ell^{2g} \cdot \ell^{-1}$ instead of $\ell^{-1}$ (see
  for instance \cite{AJT11P}).
  To give some intuition for the different prefactors, note that given
  a smooth, unmarked curve $C$ of genus $g$, there are $\ell^{2g}$
  choices of $\ell$-torsion line bundles, or in other words,
  $\ell^{2g}$ choices of maps $C \to B\mu_{\ell} \to \spec(\bk)$.
  On the other hand, if we consider taking $\ell$th roots on the level
  of logarithmic maps, we only parameterize such line bundles that
  admit a trivialization, and so only one out of the $\ell^{2g}$
  choices is allowed.
  In both cases, the automorphisms of $\ell$th roots lead to the
  additional factor $\ell^{-1}$.
\end{remark}

\begin{proof}
  The proof  is similar to \cite[Proposition 3.2]{CJR21} following the strategy of \cite{AbWi18}. Write $\SH = \SH_{g,\ddata}(\punt^{1/\ell}, \beta)$ and $\SH' = \SH_{g,\bar{\ddata}}(\punt,\beta)$ for simplicity. Denote by
  \[
    \SH \to \fM \ \ \mbox{and} \ \  \SH' \to \fM'
  \]
  the morphisms  defined in \eqref{eq:take-log}, where $\fM$ and $\fM'$ parameterize the corresponding punctured maps to $\ainfty^{1/\ell}$ and $\ainfty$ respectively.

  Consider the auxiliary stack $\fM^a := \fM(\ainfty^{1/\ell} \to \ainfty)$ which associates to a fs log scheme $S$ the category of commutative squares
  \begin{equation}\label{diag:auxiliary-object-change-ell}
    \xymatrix{
      \pC_{\ell} \ar[rr]^{\mathfrak{f}_{\ell}} \ar[d] && \ainfty^{1/\ell} \ar[d]^{\nu_{\ell}} \\
      \pC_{1} \ar[rr]^{\mathfrak{f}_{1}} && \ainfty
    }
  \end{equation}
such that
  \begin{enumerate}
  \item the left vertical arrow induces a log \'etale morphism of log curves $\cC_{\ell} \to \cC_1$ over $S$, which induces the identity on the level of coarse curves.
  \item $[\mathfrak{f}_{\ell}] \in \fM(S)$ and $[\mathfrak{f}_{1}] \in \fM'(S)$
  \item The induced morphism $ \pC_{\ell} \to \pC_{1}\times_{\ainfty} \ainfty^{1/\ell}$ is representable.
  \end{enumerate}
  Thus we have two tautological morphisms
  \[
    F_{h} \colon \fM^a \to \fM \ \ \ \mbox{and} \ \ \ F_{t} \colon \fM^a \to \fM'
  \]
  by $F_h \colon \mbox{\eqref{diag:auxiliary-object-change-ell}} \mapsto \mathfrak{f}_{\ell}$ and $F_t \colon \mbox{\eqref{diag:auxiliary-object-change-ell}} \mapsto \mathfrak{f}_{1}$  respectively. By Lemma \ref{lem:change-target}, the tautological morphism $\SH \to \fM$ factors through $\SH \to \fM^a$ via $[f_{\ell}] \mapsto \mbox{\eqref{diag:universal-map-change-ell}}$. These arrows fit in the following commutative diagram
  \begin{equation}\label{diag:stack-change-ell}
    \xymatrix{
      \SH \ar[d]^{\nu_{\ell}} \ar[rr] \ar@/^1pc/[rrrr]&& \fM^a    \ar[d]^{F_{t}} \ar[rr]_{F_{h}} &&  \fM \\
      \SH' \ar[rr]  && \fM'  &&
    }
  \end{equation}
  where the top curly arrow and the bottom arrow are the tautological morphisms \eqref{eq:take-log}. The square is Cartesian by Lemma \ref{lem:map-taking-ell-root}.

  Since the horizontal arrows in \eqref{eq:ell-root-target} are strict and \'etale, pulling back \eqref{eq:tan-POT} we obtain the perfect obstruction theory
  \begin{equation}\label{eq:pullback-pot-change-ell}
    \varphi_{\SH/\fM} = \varphi_{\SH'/\fM'}|_{\SH}.
  \end{equation}
  The same proof as in \cite[Lemma 4.1, Section 5.1]{AbWi18} implies that $F_{h}$ is strict and \'etale. Thus, $\varphi_{\SH/\fM}$ defines a perfect obstruction theory of $\SH \to \fM^a$. We next check that $F_t$ is of Deligne--Mumford type and proper of pure degree $1/\ell$. The statement then follows from Costello's virtual push-forward \cite[Theorem 1.1]{HeWi21}, see also \cite[Theorem 5.0.1]{Co06} and \cite[Proposition 5.29]{Ma12}.

  For the properness, consider any $S \to \fM'$ corresponding to a punctured map $\mathfrak{f}_1 \colon \pC_1 \to \ainfty$ over $S$. Then fiber product $\fM^a\times_{\fM'}S$ is the moduli of stable punctured maps to the family of targets $\pC_{1}\times_{\ainfty} \ainfty^{1/\ell} \to S$ with the induced discrete data, whose underlying morphsisms are usual twisted stable maps \cite{AbVi02}. Thus the same proof as in \cite[Theorem 3.17]{ACGS20P} implies that the projection $\fM^a\times_{\fM'}S \to S$, and hence also $F_t$ is proper of Deligne--Mumford type.

  Since $F_h$ is strict and \'etale, Lemma \ref{lem:universal-punctured-map-moduli} (2) implies that $\fM^a$ has an open dense substack with smooth source curves, denoted by $\mathring{\fM}^{a} \subset \fM^a$ with image $\mathring{\fM}' \subset \fM'$. By Corollary \ref{cor:change-ell-degeneracy}, we have a commutative diagram
  \begin{equation}\label{diag:1/ell-degeneracy}
    \xymatrix{
      \mathring{\fM}^{a} \ar[rr] \ar[d]_{\cO(e)} && \mathring{\fM}' \ar[d]_{\cO(e')} \\
      \ainfty^{1/\ell} \ar[rr]^{\nu_{\ell}} && \ainfty
    }
  \end{equation}
  where  $e$ and $e'$ are the degeneracies of the corresponding source curves as in Lemma \ref{lem:universal-punctured-map-moduli} (2), and the vertical arrows are strict and defined as in \eqref{eq:generic-rank1-DF}. We will show that \eqref{diag:1/ell-degeneracy} is indeed Cartesian, hence $F_t$ is of pure degree $1/\ell$.

  For any strict $S \to \mathring{\fM}'$ corresponding to a punctured map $\mathfrak{f}_1 \colon \pC_1 \to \ainfty$, consider the commutative diagram of solid arrows
  \begin{equation}\label{diag:lift-root}
    \xymatrix{
      \ainfty^{1/\ell} \ar[d] & S^{1/\ell} \ar[d] \ar[l]_{\cO(e)}  & \cC_{\ell} \ar[d] \ar[l] & \pC_{\ell} \ar@{-->}[l] \ar@{-->}[r]^{\mathfrak{f}_{\ell}} \ar@{-->}[d]& \ainfty^{1/\ell} \ar[d] \\
      \ainfty & S \ar[l]^{\cO(e')} & \cC_1 \ar[l] & \pC_1 \ar[r]_{\mathfrak{f}_1}  \ar[l] & \ainfty
    }
  \end{equation}
  where the left-most square is Cartesian with strict horizontal
  arrows defined as in \eqref{eq:generic-rank1-DF}, $\cC_1 \to S$ is
  the log curve of $\pC_1 \to S$, and $\cC_\ell \to S^{1/\ell}$ is the
  log curve obtained by taking the $\rho_i$th root along the $i$th
  marking of $\cC_1\times_{S}S^{1/\ell} \to S^{1/\ell}$.
  It suffices to exhibit a unique puncturing
  $\pC_{\ell} \to \cC_{\ell}$ and a punctured map
  $\mathfrak{f}_{\ell}$ making the above diagram commutative.

Away from markings, $\pC_{\ell}$ is isomorphic to $\cC_{\ell}$, and $\mathfrak{f}_{\ell}$ is just the composition $\pC_{\ell} \to \cC_{\ell} \to S^{1/\ell} \to \ainfty^{1/\ell}$ of the top arrows in \eqref{diag:lift-root}. Along the $i$th marking $p_i \subset \cC_{\ell}$, we have on the characteristic level a commutative diagram
  \[
    \xymatrix{
      (\ocM_{S^{1/\ell}}\oplus\NN)|_{p_i} \ar@{-->}[rr]^{\subset} && \ocM_{\pC_{\ell}}|_{p_i}  && \NN \ar@{-->}[ll]_{\mathfrak{\bar f}_{\ell}^{\flat}|_{p_i}} \\
      (\ocM_{S}\oplus\NN)|_{p_i} \ar[rr]^{\subset} \ar[u]^{(e',1) \mapsto (\ell e, \rho_i)}&& \ocM_{\pC_{\ell}}|_{p_i} \ar@{-->}[u] && \NN \ar[ll]_{\mathfrak{\bar f}_{1}^{\flat}|_{p_i}} \ar[u]_{\bar{\nu}^{\flat}_{\ell} \colon 1\mapsto \ell}
    }
  \]
  Since $\mathfrak{\bar f}_{1}^{\flat}|_{p_i} (1) = e' + c_i (0,1)$, the commutativity of the above diagram forces
  \[
    \mathfrak{\bar f}_{\ell}^{\flat}|_{p_i}(1) = e + \frac{c_i \rho_i}{\ell}(0,1) = e + c_i' (0,1),
  \]
  which is compatible with the discrete data
  \eqref{eq:change-ell-discrete-data}.
  Note that any lift of $e$ in the log structure corresponding to the
  zero section in the structure sheaf.
  Thus we define the punctured structure by adding the local section
  $e + c_i' (0,1)$ to $\ocM_{\pC_{\ell}}$.
  This defines the unique puncturing and the punctured map
  $\mathfrak{f}_{\ell}$ as needed.
\end{proof}

\subsection{Comparison of reduced virtual cycles}\label{ss:red-change-ell}

\subsubsection{Comparison of superpotentials}
\label{sss:change-superpotential}

Recall that a superpotential $W$ of order $\ttwist$ for $\punt$ is a
section of the line bundle $\uomega \otimes \cO(\ttwist\infty)$.
Note the isomorphism of line bundles
\begin{equation*}
  \uomega \otimes \cO(\ttwist\infty)|_{\punt^{1/\ell}}
  \cong \uomega \otimes \cO(\ell\ttwist\infty^{1/\ell}).
\end{equation*}
We may thus define the \emph{induced superpotential }$W_\ell$ of
$\punt^{1/\ell}$ as $W_\ell := W|_{\punt^{1/\ell}}$.
The order of $W_\ell$ is $\ell\ttwist$, and since
${\punt^{1/\ell}} \to {\punt}$ is \'etale, the superpotential $W_\ell$ is
transverse if and only if so is $W$.

\subsubsection{Comparison of twisted superpotentials}
Write $\cA^{1/\ell}_{\max} \cong \cA$ and  $\cA^{1/\ell} \cong \cA$ with the corresponding closed substacks
$\Delta^{1/\ell}_{\max} \subset \cA^{1/\ell}_{\max}$ and $\ainfty^{1/\ell} \subset \cA^{1/\ell}$. We may rewrite \eqref{diag:change-univ-target} as follows
\begin{equation}\label{diag:change-univ-target-ell}
\xymatrix{
\infty^{1/\ell}_{\cA^{e,\circ}} \ar[d] \ar[r] & \infty^{1/\ell}_{\cA^{e}} \ar[d] \ar[r]^-{\fb} & \ainfty^{1/\ell} \times\Delta^{1/\ell}_{\max} \ar[d] \\
\cA^{1/\ell, e,\circ} \ar[r] & \cA^{1/\ell,e} \ar[r] & \cA^{1/\ell}\times\cA^{1/\ell}_{\max}.
}
\end{equation}

Replacing $\punt$ by $\punt^{1/\ell}$, we may rewrite \eqref{diag:expand-R-target} using \eqref{diag:change-univ-target-ell} as follows
\begin{equation}\label{diag:expand-R-target-ell}
\xymatrix{
\punt^{1/\ell}_{e,\circ} \ar[rr] \ar[d] && \punt^{1/\ell}\times\Delta^{1/\ell}_{\max} \ar[rr] \ar[d] && \BC\times\Delta^{1/\ell}_{\max} \\
\infty^{1/\ell}_{\cA^{e,\circ}} \ar[rr] && \ainfty^{1/\ell}\times\Delta^{1/\ell}_{\max} &&
}
\end{equation}

Consider the $\ell$th root morphisms $\cA^{1/\ell} \to \cA$, $\cA^{1/\ell}_{\max} \to \cA_{\max}$ and their restrictions to the closed substacks $\ainfty^{1/\ell} \to \ainfty$ and $\Delta^{1/\ell}_{\max} \to \Delta_{\max}$.
These $\ell$th root morphisms fit in a commutative diagram of log stacks
\begin{equation}\label{diag:universal-exp-change-ell}
\xymatrix{
\punt^{1/\ell}_{e,\circ} \ar[rr] \ar[d] && \infty_{\cA^{e,\circ}}^{1/\ell} \ar[rr] \ar[d]  &&  \ainfty^{1/\ell}\times\Delta^{1/\ell}_{\max} \ar[d] \ar[rr] && \BC\times\Delta^{1/\ell}_{\max}  \ar[d]  \\
\punt_{e,\circ} \ar[rr] && \infty_{\cA^{e,\circ}}  \ar[rr] && \ainfty\times\Delta_{\max} \ar[rr] && \BC\times\Delta_{\max}
}
\end{equation}
where the square on the left is Cartesian. Observe that all vertical arrows are log \'etale as they are defined by the root constructions.

Suppose we are given a twisted superpotential $\tW$ of $\punt$ as in
\eqref{eq:twisted-potential} of order $\ttwist$.
We define the {\em induced twisted superpotential} $\tW_{\ell}$ of
$\punt^{1/\ell}$ to be the following composition over
$\BC\times\Delta^{1/\ell}_{\max}$:

\begin{equation}\label{eq:induced-superpotential}
\punt^{1/\ell}_{e,\circ} \longrightarrow \punt_{e,\circ} \longrightarrow \uomega\boxtimes\cO(\ttwist \Delta_{\max}) \cong \uomega\boxtimes\cO(\ttwist \cdot \ell \Delta_{\max}^{1/\ell})
\end{equation}
where the left arrow is the left vertical arrow  in \eqref{diag:universal-exp-change-ell}, and the total space of the line bundles on the right are equipped with the log structures pulled back from $\BC\times\Delta^{1/\ell}_{\max}$. Readers can check directly that the constructions of the induced twisted superpotential is compatible with  the induced superpotential with respect to the correspondence of
  Lemma~\ref{lem:super-potential-equivalence}.

\begin{corollary}\label{lem:change-l-proper-cricial-locus}
Notations and assumptions as above, we have
\begin{enumerate}
 \item The order of $\tW_{\ell}$ is $\ttwist \cdot \ell$.
 \item $\tW_{\ell}$ has proper critical locus iff $\tW$ has proper critical locus.
\end{enumerate}
\end{corollary}

\subsubsection{Comparison of reduced virtual cycles}
By Corollary~\ref{cor:change-ell-degeneracy}, the morphism
\eqref{eq:canonical-stack-change-ell} preserves the order of
degeneracies, and hence induces a morphism
\begin{equation}\label{eq:reduced-stack-change-ell}
\nu_{\ell} \colon \UH_{g,\ddata}(\punt^{1/\ell}, \beta) \to \UH_{g,\bar{\ddata}}(\punt,\beta).
\end{equation}
Fix a transverse superpotential $W$ of $\punt$, inducing a transverse
superpotential $W_{\ell}$ of $\punt^{1/\ell}$.
We further assume $c_i, c_i' < 0$ for all $i$ so that the reduced
theories with respect to $W$ and $W_{\ell}$ are defined.

\begin{proposition}\label{prop:reduced-change-ell}
With the above notations and assumptions, we have
\[
\nu_{\ell,*}[\UH_{g,\ddata}(\punt^{1/\ell}, \beta)]^{\red} = \ell^{-1} \cdot [\UH_{g,\bar{\ddata}}(\punt,\beta)]^{\red}.
\]
\end{proposition}

\begin{proof}
We follow the strategy of Proposition \ref{prop:can-vir-change-ell}. Write for simplicity $\UH = \UH_{g,\ddata}(\punt^{1/\ell}, \beta)$ and $\SH^{\prime,\curlywedge} = \UH_{g,\bar{\ddata}}(\punt,\beta)$, and denote by
\[
  \UH \to \UM \ \ \mbox{and} \ \  \SH^{\prime,\curlywedge} \to \fM^{\prime,\curlywedge}
\]
the strict morphisms  defined in \eqref{eq:take-log}, where $\UM$ and $\fM^{\prime,\curlywedge}$ parameterize  punctured maps to $\ainfty^{1/\ell}$ and $\ainfty$ respectively.

Consider the stack $\fM^a$ in the proof of Proposition
\ref{prop:can-vir-change-ell}.
Let $\fM^{a,\curlywedge} \to \fM^a$ be the subcategory parameterizing
commutative squares \eqref{diag:auxiliary-object-change-ell} in
$\fM^a$ with the $\curlywedge$-configurations.
Note that for an object \eqref{diag:auxiliary-object-change-ell} in
$\fM^{a,\curlywedge}$, both $\mathfrak{f}_{\ell}$ and
$\mathfrak{f}_{1}$ have the $\curlywedge$-configuration, thanks to
Corollary \ref{cor:change-ell-degeneracy}.
Thus diagram \eqref{diag:stack-change-ell} restricts to a commutative
diagram with the square Cartesian:
\[
 \xymatrix{
 \UH \ar[d]^{\nu_{\ell}} \ar[rr] \ar@/^1pc/[rrrr]&& \fM^{a,\curlywedge}    \ar[d]^{F_{t}} \ar[rr]_{F_{h}} &&  \UM \\
 \SH^{\prime,\curlywedge} \ar[rr]  && \fM^{\prime,\curlywedge}  &&
 }
\]

Similar to Proposition \ref{prop:can-vir-change-ell}, we may check that $F_h$ is strict, \'etale, and $F_t$ is proper, Deligne--Mumford type of pure degree $1/\ell$. Thus, the reduced perfect obstruction theory $\varphi^{\red}_{\UH/\UM}$ as in \eqref{diag:POT-factorization} defines a perfect obstruction theory for $\UH \to \fM^{a,\curlywedge}$. By Costello's virtual push-forward  \cite[Theorem 1.1]{HeWi21} and \cite[Theorem 5.0.1]{Co06}, it remains to verify
\begin{equation}\label{eq:red-pot-change-ell}
\varphi^{\red}_{\UH/\UM} \cong \varphi^{\red}_{\SH^{\prime,\curlywedge}/\fM^{\prime, \curlywedge}}\big|_{\UH}.
\end{equation}

By \eqref{eq:R-complex-cosection} and \eqref{eq:red-POT}, it suffices
to find a dashed isomorphism making the following square commutative
\begin{equation}\label{eq:compare-boundary-cplx-change-ell}
\xymatrix{
H^1(\EE_{\UH/\UM})[-1] \ar[rrr]^-{\sigma_{\UH/\UM}} \ar[d]^{\cong} &&&   \FF_{\UH/\UM} \ar@{-->}[d]^{\cong} \\
H^1\left(\EE_{\SH^{\prime,\curlywedge}/\fM^{\prime,\curlywedge}}\Big|_{\UH}\right)[-1] \ar[rrr]^-{\sigma_{\SH^{\prime,\curlywedge}/\fM^{\prime,\curlywedge}}\Big|_{\UH}} &&&   \FF_{\SH^{\prime,\curlywedge}/\fM^{\prime,\curlywedge}}\Big|_{\UH}\ ,
}
\end{equation}
where the left vertical arrow is an isomorphism due to the
identification of the canonical perfect obstruction theory
\eqref{eq:pullback-pot-change-ell}.
By Corollary~\ref{cor:change-ell-degeneracy}, we observe that
\[
\cO_{\UH}(\Delta_{\max}) \cong \cO_{\SH^{\prime,\curlywedge}}(\ell \cdot \Delta^{1/\ell}_{\max})\Big|_{\UH}.
\]
By \eqref{eq:boundary-complex} and
Corollary~\ref{lem:change-l-proper-cricial-locus} (1), we compute
\[
\FF_{\UH/\UM}[1] \cong \cO_{\UH}(\ttwist \Delta_{\max}) \cong \cO_{\SH^{\prime,\curlywedge}}(\ttwist \ell \cdot \Delta^{1/\ell}_{\max})\Big|_{\UH} \cong \FF_{\SH^{\prime,\curlywedge}/\fM^{\prime,\curlywedge}}\Big|_{\UH}.
\]
This defines the dashed isomorphism.
To see the commutativity of
\eqref{eq:compare-boundary-cplx-change-ell}, it remains to check that
\begin{equation}\label{eq;compare-cosection-change-ell}
 \sigma_{\UH/\UM} = \sigma_{\SH^{\prime,\curlywedge}/\fM^{\prime,\curlywedge}}\Big|_{\UH}.
\end{equation}

Let $f_{e,\UH} \colon \pC_{\UH} \to \punt_{e,\circ}$ be the lift as in \eqref{eq:bullet-universal-maps}, and $\pi_{\UH} \colon \pC_{\UH} \to \UH$ be the projection. Similarly, we have  $f_{e,\SH^{\prime,\curlywedge}} \colon \pC_{\SH^{\prime,\curlywedge}} \to \punt^{1/\ell}_{e,\circ}$ and $\pi_{\SH^{\prime,\curlywedge}} \colon \pC_{\SH^{\prime,\curlywedge}} \to \SH^{\prime,\curlywedge}$.  By \eqref{eq:relative-cosection}, we have
\[
\sigma_{\UH/\UM} = R^1\pi_{\UH,*} \big( f_{e,\UH}^* \diff \tW \big) \ \ \ \mbox{and} \ \ \ \sigma_{\SH^{\prime,\curlywedge}/\fM^{\prime,\curlywedge}} = R^1\pi_{\SH^{\prime,\curlywedge},*} \big( f_{e,\SH^{\prime,\curlywedge}}^* \diff \tW_{\ell} \big).
\]
Noting that the left arrow in \eqref{eq:induced-superpotential} is log
\'etale, we observe
\begin{equation*}
\diff \tW_{\ell} = \diff \tW|_{\punt^{1/\ell}_{e,\circ}}
\end{equation*}
using the identification
$\Omega^{\vee}_{\punt^{1/\ell}/\BC}|_{\punt^{1/\ell}_{e,\circ}} \cong
\Omega^{\vee}_{\punt/\BC}|_{\punt^{1/\ell}_{e,\circ}}$.
This implies the desired identity \eqref{eq;compare-cosection-change-ell}.
\end{proof}

\subsubsection{Comparison of reduced virtual cycles with \texorpdfstring{$\psi_{\min}$}{psi-min}-classes}
Since
\eqref{eq:canonical-stack-change-ell} preserves orders of
degeneracies by Corollary~\ref{cor:change-ell-degeneracy}, it induces a tautological morphism
\begin{equation}\label{eq:reduced-stack-change-ell-with-m}
\nu_{\ell} \colon \RMm_{g,\ddata}(\punt^{1/\ell}, \beta) \to \RMm_{g,\bar{\ddata}}(\punt,\beta).
\end{equation}
Let $\psi_{\min, \ddata}$ and $\psi_{\min, \bar{\ddata}}$ be the
$\psi_{\min}$-classes over
$\RMm_{g,\ddata}(\punt^{1/\ell}, \beta)$ and
$\RMm_{g,\bar{\ddata}}(\punt,\beta)$ respectively.

\begin{corollary}\label{cor:reduced-change-ell-with-min}
Notations and assumptions as above, we have
\[
\nu_{\ell,*}\left(\frac{t\cdot [\RMm_{g,\ddata}(\punt^{1/\ell}, \beta)]^{\red}}{-t-\psi_{\min, \ddata}}\right) = \ell^{-1} \cdot \frac{(t/\ell)\cdot [\RMm_{g,\bar{\ddata}}(\punt,\beta)]^{\red}}{-(t/\ell)-\psi_{\min, \bar{\ddata}}}
\]
\end{corollary}

\begin{proof}
By Corollary \ref{cor:change-ell-degeneracy}, we have
$
\nu_{\ell}^*\psi_{\min, \bar{\ddata}} = \ell \cdot \psi_{\min, \ddata},
$
hence
\[
\nu_{\ell}^*\left( \frac{(t/\ell)}{-(t/\ell)-\psi_{\min, \bar{\ddata}}} \right) = \frac{t}{-t-\psi_{\min, \ddata}}.
\]
By Proposition \ref{prop:reduced-change-ell}, we have the virtual push-forward
\[
\nu_{\ell,*}[\RMm_{g,\ddata}(\punt^{1/\ell}, \beta)]^{\red} = \ell^{-1} \cdot [\RMm_{g,\bar{\ddata}}(\punt,\beta)]^{\red}.
\]
Thus the statement follows from the projection formula.
\end{proof}

%%%%%%%%%%%%%
%% Targets %%
%%%%%%%%%%%%%

\section{Targets for punctured R-maps}\label{sec:logGLSM-targets}

In Section~\ref{sss:P-target}, we introduced a general class of
targets for punctured R-maps.
In this Section, we study a more concrete class of targets, already
discussed in the introduction, and discuss their relation to the
the hybrid model log GLSM targets of \cite[\S~1.3]{CJR21}.
We also consider superpotentials in this setting.

\subsection{Targets via root constructions}\label{ss:target}

\subsubsection{The input}\label{sss:input}
Recall the following data from \S \ref{intro:target}:
\begin{enumerate}
 \item A proper DM-stack ${\xinfty}$ with the projective coarse moduli;
 \item Two line bundles $\lbspin$ and $\lblog$ on $\xinfty$;
 \item Positive integers $r$, $d$ and $\ell$.
\end{enumerate}

This leads to a target $\punt^{1/\ell} \to \BC$ as in \eqref{eq:punt}.

\subsubsection{Independence of \texorpdfstring{$\ell$}{l}}
\label{sss:construct-target}

We next show that the target $\punt^{1/\ell} \to \BC$ with $\ell > 1$ can be reconstructed using appropriate input data with $\ell = 1$. Furthermore, as shown below, different input data can lead to the same target, which is a common phenomenon for (log) gauged linear sigma models, see \cite[Remark 1.6]{CJR21} and also Section~\ref{sss:LGCY} below.

Suppose we are given the input $(\infty_\cX, \lbspin, \lblog, r, d, \ell)$ leading to $\punt^{1/\ell}$ as above.  Let $\frac{d'}{r'} = \frac{d}{r\ell}$ be the reduced representation such that $\gcd(r', d')=1$ and $d', r' > 0$.
Write
\[
\widehat r = \lcm(r, r') \ \ \mbox{and} \ \  \widehat d = \frac{d' \widehat{r}}{r'} = \frac{\widehat{r} d}{r \ell}.
\]
Since $\gcd(\widehat r/r, d' \widehat r/r') = 1$, we obtain a (non-unique) pair of integers $a, b$ such that
  \begin{equation}
    \label{eq:bezout}
    \frac 1{\widehat r} = \frac ar + \frac{bd'}{r'} = \frac{a\ell + bd}{r\ell}.
  \end{equation}
  We then define $\widehat{\infty_\cX}$ via the Cartesian diagram
  \begin{equation}\label{eq:change-target}
    \xymatrix{
      \widehat{\infty_\cX} \ar[r] \ar[d]_{(L_1, L_2)} & \infty_\cX \ar[d]^{(\lblog^\vee, \cO)} \\
      (\BG_m)^2 \ar[r]_\alpha & (\BG_m)^2
    }
  \end{equation}
  where
$
    \alpha(\cL_1, \cL_2)
    = (\cL_1^{-d} \otimes \cL_2^\ell, \cL_1^a \otimes \cL_2^b).
$
Observe that $\det\begin{bmatrix} a & b \\ -d & \ell \end{bmatrix} \neq 0$, hence $\alpha$ is \'etale of DM-type.
Further setting
\[
 \widehat{\lbspin} = \lbspin \otimes L_1^{-r} \ \ \mbox{and} \ \ \widehat\lblog = (L_2)^\vee
\]
we arrive at new input data
$
(\widehat{\infty_\cX}, \widehat{\lbspin}, \widehat{\lblog},
  \widehat r, \widehat d, \widehat \ell = 1),
$
hence a new target $\widehat{\punt} \to \BC$.
Note that when $\ell = 1$, the new input data is identical to the
original input.

\begin{proposition}\label{prop:one-step-target}
  There is a canonical isomorphism of log stacks
  $\widehat{\punt} \to \punt^{1/\ell}$ over $\xinfty\times\BC$.
\end{proposition}
\begin{proof}

  Note that both $\ul\infty^{1/\ell}$ and $\ul{\widehat\infty}$ can be
  represented as (2-)equalizers.
  The morphism $\ul\punt^{1/\ell} \to \xinfty\times \BC $ parameterizes pairs
  $(\uspin, \tlog)$ with isomorphisms
  \[
  \begin{cases}
   \uspin^r \cong \lbspin\boxtimes \uomega \\
   \cO(\punt^{1/\ell})^{-\ell} \cong  \lblog^\vee \otimes  \uspin^d.
  \end{cases}
  \]
This is equivalent to the following equalizer
    \begin{equation*}
    \xymatrix{
      \ul\infty^{1/\ell} \ar[r] & \infty_\cX \times \BC \times (\BG_m)^2 \ar@<.5ex>[rrr]^-{(\uspin^r,  \cO(\punt^{1/\ell})^{-\ell})} \ar@<-.5ex>[rrr]_-{(\lbspin \boxtimes \uomega, \lblog^\vee \boxtimes \uspin^d)} &&& (\BG_m)^2
    },
  \end{equation*}
  Here, the first map is defined by the projection
  $\ul\infty^{1/\ell} \to \infty_\cX \times \BC$ and the line bundles $\uspin$ and
  $\cO(\punt^{1/\ell})^{\vee}$. By abuse of notations,
  $(\uspin, \cO(\punt^{1/\ell})^{\vee})$ denote the universal line bundles of the middle
  $(\BG_m)^2$.

  By \eqref{eq:underlying-P-target} and \eqref{eq:change-target},
  $\widehat{\ul\punt} \to \xinfty\times \BC $ parameterizes triples
  $(L_1, L_2,\widehat\uspin)$ satisfying
  \begin{equation}
    \label{eq:spin-tilde-eqs}
    \begin{cases}
      \widehat{\uspin}^{\widehat{r}} \cong \widehat{\lbspin}\boxtimes \uomega \\
      (L_1)^{-d}\otimes(L_2)^{\ell} \cong \lblog^\vee \\
      (L_1)^{a}\otimes(L_2)^{b} \cong \cO.
    \end{cases}
  \end{equation}
  This amounts to the equalizer
  \begin{equation*}
    \xymatrix{
      \ul{\widehat\infty} \ar[r] & \infty_\cX \times \BC \times (\BG_m)^3 \ar@<.5ex>[rrrr]^-{(\widehat{\uspin}^{\widehat r}, L_1^{-d} \boxtimes L_2^\ell, L_1^a \boxtimes L_2^b)} \ar@<-.5ex>[rrrr]_-{(\widehat{\lbspin} \boxtimes \uomega, \lblog^\vee, \cO)} &&&& (\BG_m)^3.
    }
  \end{equation*}
  Here the first map is defined by the composition
  $\widehat{\ul\infty} \to \widehat{\infty_\cX} \times \BC \to
  \infty_\cX \times \BC$ and the line bundles $L_1$, $L_2$ and
  $\widehat{\uspin}$. By abuse of notations,
  $(L_1, L_2, \widehat{\uspin})$ also denote the universal
  line bundles of the middle $(\BG_m)^3$.

  In view of the equalizer of $\widehat{\ul\punt}$, we rewrite the first equalizer as
  \begin{equation*}
    \xymatrix{
      \ul\infty^{1/\ell} \ar[r] & \infty_\cX \times \BC \times (\BG_m)^3 \ar@<.5ex>[rrrr]^-{(\uspin^r, \cO(\punt^{1/\ell})^{-\ell}, \uspin^a \otimes \cO(\punt^{1/\ell})^{-b})} \ar@<-.5ex>[rrrr]_-{(\lbspin \boxtimes \uomega, \lblog^\vee \boxtimes \uspin^d, \widehat{\uspin})} &&&& (\BG_m)^3
    },
  \end{equation*}
  where the map to the extra $\BG_m$ on the left is given by
  $\uspin^a \otimes \cO(\punt^{1/\ell})^b$, and by abuse of notations
  $(\cL_R, \cO(\punt^{1/\ell})^{\vee}, \widehat{\uspin})$ denote the universal line
  bundles of the middle $(\BG_m)^3$.
  Equivalently, the morphism $\ul\punt^{1/\ell} \to \xinfty\times \BC $ parameterizes the triple $(\uspin, \cO(\punt^{1/\ell})^{\vee}, \widehat\uspin)$ satisfying
  \begin{equation}
    \label{eq:spin-alternative-eqs}
    \begin{cases}
      \uspin^r \cong \lbspin\boxtimes \uomega \\
      \cO(\punt^{1/\ell})^{-\ell} \cong \uspin^d \otimes \lblog^\vee \\
      \uspin^a\otimes \cO(\punt^{1/\ell})^{-b} \cong \widehat\uspin.
    \end{cases}
  \end{equation}
  Note that the last equation is just a definition of
  $\widehat\uspin$.

  We now construct a morphism
  $\widehat{\ul\punt} \to \ul\punt^{1/\ell}$ over $\xinfty\times\BC$
  via the assignment
  \begin{equation}\label{eq:underlying-target-iso}
    (L_1 \otimes \widehat{\uspin}^{\frac{\widehat r}r}, L_2 \otimes \widehat{\uspin}^{\widehat{d}}, \widehat{\uspin}) \mapsto (\uspin, \cO(\punt^{1/\ell})^{\vee}, \widehat{\uspin}).
  \end{equation}
  To see this is well-defined, we rewrite \eqref{eq:spin-alternative-eqs} using the left hand side of \eqref{eq:underlying-target-iso}:
  \begin{equation*}
    \begin{cases}
      L_1^r  \otimes \widehat{\uspin}^{\widehat r} \cong (\widehat \lbspin \otimes L_1^r) \boxtimes \uomega \\
      L_2^{\ell} \otimes \widehat{\uspin}^{\frac{d\widehat r}r} \cong \lblog^\vee\otimes L_1^d \otimes \widehat{\uspin}^{\frac{d\widehat r}r}\\
      L_1^a\otimes L_2^{b} \otimes \widehat\uspin^{\frac{a\widehat r}r + \frac{bd\widehat r}{r\ell}} \cong \widehat\uspin,
    \end{cases}
  \end{equation*}
  which simplifies to \eqref{eq:spin-tilde-eqs} using
  \eqref{eq:bezout}.
  The morphism \eqref{eq:underlying-target-iso} is an
  isomorphism with the inverse
  \begin{equation*}
    (\uspin, \cO(\punt^{1/\ell})^{\vee}, \widehat{\uspin}) \mapsto (\uspin \otimes \widehat{\uspin}^{-\frac{\widehat r}r}, \cO(\punt^{1/\ell})^{\vee} \otimes \widehat{\uspin}^{-\widehat{d}}, \widehat{\uspin}) \cong (L_1, L_2, \widehat{\uspin}).
  \end{equation*}

  %To see the \eqref{eq:underlying-target-iso} defines an isomorphism, it suffices to observe that both $\widetilde{\ul\punt}$ and $\ul\punt$ are proper \'etale quasi-finite over $\xinfty\times\BC$ of degree $\frac{1}{r\ell}$. This can be computed directly via the constructions.

    Finally, we check that \eqref{eq:underlying-target-iso} induces
    \[
    \cO(\widehat{\punt})^{\vee} = \big(\widehat{\lblog}^\vee\otimes \widehat\uspin^{\widehat d} \big) \cong \big( L_2 \otimes \widehat{\uspin}^{\frac{d\widehat r}{r\ell}} \big) \mapsto \cO(\punt^{1/\ell})^{\vee},
    \]
hence induces an isomorphism of log stacks $\widehat \punt \to \punt^{1/\ell}$ as needed.
\end{proof}

\subsubsection{The unit sector}\label{sss:unit-sector}

Given a target $\punt^{1/\ell} \to \BC$ in \eqref{eq:punt}, we
introduce a special type of discrete data $(\ogamma, c)$ along
markings, called the {\em unit sector}.
It plays an important role in the next two sections.

We first assume $\ell = 1$ and set $c = - d$.
To describe $\ogamma$, consider the Cartesian squares

\[
\xymatrix{
 \ul{\punt}_{\bk} \ar[rr] \ar[d] && \xinfty \ar[rr] \ar[d] && \spec \bk \ar[d] \\
  \ul{\punt} \ar[rr] && \xinfty \times \BC \ar[rr] && \BC
}
\]
Then $\ogamma$ is defined to be the component with the universal gerbe
\[
\xymatrix{
\gamma :=  \ul{\punt}_{\bk} \ar[d] \ar[rr]^{=} && \ul{\punt}_{\bk} \\
\ogamma := \xinfty
}
\]
This defines the unit sector  $(\ogamma, c = -d)$ in the $\ell = 1$ case.  Further note that $\age\uspin|_{\gamma} = \frac{1}{r}$ in this case.

When $\ell > 1$, the unit sectors are defined by reducing to the
$\ell = 1$ case via Proposition \ref{prop:one-step-target}.
We have the compatibility of unit sectors with respect to changing $\ell$:
\begin{lemma}
  Let $(\ogamma,c)$ be the unit sector of $\infty^{1/\ell}$, and
  $(\ogamma',c')$ be the unit sector of $\infty$.
  Then $(\ogamma,c)$ and $(\ogamma',c')$ are compatible with respect to
  changing $\ell$ in the sense of Lemma \ref{lem:map-taking-ell-root}
  and Remark \ref{rem:marking-data-change-ell}.
\end{lemma}
\begin{proof}
  In this proof, we will freely use the notations of
  \S\ref{sss:construct-target}.

  The compatibility of contact orders follows from $c' = -d$ and
  \begin{equation*}
    \frac\ell{\widehat r/r} (-d) = -\widehat d = c,
  \end{equation*}
  see \eqref{eq:change-ell-discrete-data}, where $\rho = \widehat r/r$ in this case.

  Let $\ogamma''$ be the sector of $\punt^{1/\ell}$ corresponding to
  $\ogamma'$.
  It remains to prove $\ogamma'' = \ogamma$.
  Note we have a commutative diagram
  \begin{equation*}
    \xymatrix{
      \punt^{1/\ell}_{\bk} = \widehat\punt_\bk \ar[r] \ar[d] & \punt_\bk \ar[d] \\
      \widehat\xinfty \ar[r] & \xinfty,
    }
  \end{equation*}
  where the right vertical arrow is a $\mu_r$-gerbe, the top
  horizontal arrow is a $\mu_\ell$-gerbe, and the left vertical arrow
  is a $\mu_{\widehat r}$-gerbe.
  Then, according to Remark~\ref{rem:marking-data-change-ell},
  $\ogamma''$ is the unique sector mapping to the untwisted sector of
  $\xinfty$ such that $\age(\uspin)|_{\gamma''} = \frac 1r$ and
  $\age(\cO(\punt^{1/\ell}))|_{\gamma''} = \{-\frac d{r\ell}\}$. Thus it
  suffices to prove that $\ogamma$ satisfies the same conditions.
  Clearly, $\ogamma$ maps to the untwisted sector of $\infty_\cX$.
  Then, we compute
  \begin{equation*}
    \age(\cO(\punt^{1/\ell})|_{\gamma})
    = \{\age(\widehat\lblog|_{\gamma}) - \widehat d\age(\widehat\uspin|_{\gamma})\}
    = \{- \frac{\widehat d}{\widehat r}\}
    = \{-\frac d{r\ell}\},
  \end{equation*}
  and
  \begin{equation*}
    \age(\uspin|_{\gamma})
    = \{\age(L_1|_{\gamma}) + \widehat r/r \cdot \age(\widehat\uspin|_{\gamma})\}
    = \frac 1r,
  \end{equation*}
  using that $\widehat\lblog$ and $L_1$ are pulled back from
  $\widehat\xinfty$.
\end{proof}

\subsubsection{Log (co)tangent bundle of the target}
We now assume that $\xinfty$ is smooth.
By construction, the morphism
\[
\punt^{1/\ell} \to \BC \times \ainfty \to \BC
\]
is strict and smooth. In particular, $\ul{\punt}$ is smooth.

\begin{lemma}\label{lem:punt-tangent}
Suppose $\xinfty$ is smooth. There is a canonical exact sequence of vector bundles over $\punt^{1/\ell}$
\[
0 \to \cO_{\punt^{1/\ell}} \to \Omega^{\vee}_{\punt^{1/\ell}/\BC} \to \Omega^{\vee}_{\xinfty}|_{\punt^{1/\ell}} \to 0.
\]
\end{lemma}
\begin{proof}
Consider the distinguished triangle of cotangent complexes of underlying structures
\[
\LL_{\BC\times\ul{\ainfty}/\BC} \to \LL_{\ul{\punt}^{1/\ell}/\BC} \to \LL_{\ul{\punt}^{1/\ell}/\BC\times\ul{\ainfty}} \stackrel{[1]}{\to}
\]
Recall that
$\Omega_{\punt^{1/\ell}/\BC} =
\LL_{\ul{\punt}^{1/\ell}/\BC\times\ul{\ainfty}} $.
Rotating the triangle and taking duals, we have
\[
\LL_{\ul{\ainfty}}^{\vee}|_{\punt^{1/\ell}}[-1] \to \Omega^{\vee}_{\punt^{1/\ell}/\BC} \to \LL^{\vee}_{\ul{\punt}^{1/\ell}/\BC} \stackrel{[1]}{\to}
\]
Note that $\LL_{\ul{\ainfty}} = \cO_{\ainfty}[-1]$, hence $\LL_{\ul{\ainfty}}^{\vee}|_{\punt}[-1] = \cO_{\punt}$. By \eqref{eq:underlying-P-target} and \eqref{eq:g1-t-root}, we have a sequence of \'etale morphisms
\[
\ul{\punt^{1/\ell}} \to \ul{\punt} \to \xinfty\times\BC
\]
hence
$\LL_{\ul{\punt}^{1/\ell}/\BC} \cong
\Omega_{\xinfty}|_{\ul{\punt}^{1/\ell}}$ by the smoothness of
$\xinfty$.
Putting these together, we obtain the exact sequence
\[
0 \to \cO_{\punt^{1/\ell}} \to \Omega^{\vee}_{\punt^{1/\ell}/\BC} \to \Omega^{\vee}_{\xinfty}|_{\ul{\punt}^{1/\ell}} \to 0.
\]
\end{proof}

\subsubsection{Twisted superpotentials in the setup of \S \ref{ss:target}}\label{sss:potential-local-structure}

We now assume that $W$ is a nonzero superpotential for a target
$\punt^{1/\ell}$ as in \S \ref{ss:target}, and that $\tW$ is the
associated twisted superpotential.
The specific setup of this section puts extra constraints on $W$ and
$\tW$ that will be useful in calculations.

\begin{lemma}\label{lem:super-potential-weight}
  $W$ (and thus $\tW$) has order $\ttwist = \frac{r \cdot \ell}{d}$.
\end{lemma}
\begin{proof}
  Recall that $\CC_\omega$ is the $\CC^*_{\omega}$-representation of weight $1$.
  The existence of non-zero $W$ implies that the line bundle
  \begin{equation*}
    \CC_\omega \otimes \cO(\ttwist\infty_\bk)
  \end{equation*}
  on $\infty_\bk$ must have $\CC^*_{\omega}$-weight zero.
  We conclude the lemma by noticing that the
  $\CC^*_{\omega}$-weight of $\cO(\punt_{\bk}^{1/\ell})$ is
  $-\frac{d}{r\ell}$ by   \eqref{eq:underlying-P-target} and \eqref{eq:g1-t-root}.
\end{proof}
\begin{remark}\label{rem:choice-independent-weight}
  By Proposition \ref{prop:one-step-target}, the target
  $\punt^{1/\ell}$ may be constructed using different choices of
  $r, d, \ell$.
  However, the above lemma shows that in the presence of a non-zero
  $\tW$, the positive integer $\ttwist = \frac{r \cdot \ell}{d}$ is
  independent of such choices, and is intrinsic to
  $\punt^{1/\ell} \to \BC$.
\end{remark}
\begin{remark}
  Note that given Lemma~\ref{lem:super-potential-weight} and assuming $\ell = 1$, we may
  simplify the line bundle appearing in
  Definition~\ref{def:superpotential2}:
  %\Qile{Something seems not right: I calculated $W$ is a section of $\uomega \otimes \cO(\ttwist \infty) \cong \mathbf{L} \otimes \lblog^{r/d}$.}\Felix{I fixed the missing dual, and added the assumption $\ell = 1$. Otherwise, $r$ might not be divisible by $d$.}
  \begin{equation*}
    \uomega \otimes \cO(\ttwist \infty) \cong \uomega \otimes \lblog^{\ttwist} \otimes \uspin^{-d\ttwist}
    \cong \lblog^{\ttwist} \otimes \mathbf{L}
  \end{equation*}
  From this, it is not hard to see that any superpotential $W$ is the
  pullback of a section of $\lblog^{\ttwist} \otimes \mathbf{L}$ on
  $\infty_\cX$.
\end{remark}

\subsection{Targets from hybrid model log GLSM}
\label{ss:logGLSM-target}

We next show that the hybrid model log GLSM targets of
\cite[\S~1.3]{CJR21} fit into the construction of \S\ref{intro:target}.

\subsubsection{The input}\label{sss:input}
We first recall the targets of log R-maps in \cite{CJR21}. A {\em hybrid target} is determined by the following data:
\begin{enumerate}
 \item A proper Deligne--Mumford stack $\cX$ with a projective coarse moduli scheme $X$.
 \item A vector bundle $\bE$ over $\cX$ of the form
\begin{equation*}
  \bE = \bigoplus_{i \in \ZZ_{> 0}} \bE_i
\end{equation*}
where $\bE_i$ is a vector bundle with the positive grading $i$. Write $d := \gcd\big(i \ | \ \bE_i \neq 0 \big)$.
 \item A line bundle $\lbspin$ over $\cX$.
 \item A positive integer $r$.
\end{enumerate}

\subsubsection{The $r$-spin structure}
Consider the  Cartesian diagram
\begin{equation}\label{diag:spin-target}
\xymatrix{
\spint \ar[rrr]^\Sp \ar[d] &&& \BG_m \ar[d]^{\nu_r} \\
\BC\times\cX \ar[rrr]^{\LR\boxtimes \lbspin} &&& \BG_m
}
\end{equation}
where $\LR$ is the universal line bundle over $\BC$, $\nu_r$ is the
$r$th power map, the bottom arrow is defined by $\LR\boxtimes\lblog$,
and the top arrow is defined by the universal $r$-th root of
$\LR\boxtimes\lblog$, denoted by $\Sp$.

\subsubsection{The log GLSM targets}
Fix a \emph{twisting choice} $\etwist = \frac{\ell}{d} \in \frac{1}{d}\cdot\ZZ_{>0}$.
We form the weighted projective stack bundle over $\spint$:
\begin{equation}\label{equ:universal-proj}
\ul{\fP_{\ell}} := \ul{\PP}^{\bw}\left(\bigoplus_{i > 0}(\bE^{\vee}_{i,\spint}\otimes\cL_{\spint}^{\otimes i})\oplus \cO_{\spint} \right),
\end{equation}
where $\bw$ is the collection of the weights of the $\GG_m$-action such that the weight on the $i$-th factor is the positive integer
$\etwist\cdot i = \ell \cdot i/d$, while the weight of the last factor $\cO$ is 1.
Here for any vector bundle $V = \oplus_{i=1}^{r} V_i$ with a $\GG_m$-action of weight $\bw$, we use the notation
\begin{equation}\label{equ:def-proj-bundle}
\PP^\bw(V) = \left[\Big(\vb(V)\setminus \zero_V \Big) \Big/ \GG_m \right],
\end{equation}
where $\vb(V)$ is the total space of $V$, and $\zero_V$ is the zero section of $V$. Intuitively, $\ul{\fP_{\ell}}$ compactifies the GLSM which is independent of choices of $\ell$ :
\begin{equation}\label{eq:interior-R-target}
\fP^{\circ} := \bigoplus_{i > 0}(\bE^{\vee}_{i,\spint}\otimes\cL_{\spint}^{\otimes i}).
\end{equation}

The boundary $\infty_{\ell} = \ul{\fP_{\ell}} \setminus \fP^{\circ}$ is the Cartier divisor
defined by the vanishing of the coordinate corresponding to
$\cO_{\spint}$.
We make $\ul{\fP_{\ell}}$ into a log stack $\fP_{\ell}$ by equipping it with the log structure corresponding to the Cartier divisor $\infty_\fP$. The target $\fP_{\ell} \to \BC$ of log R-maps is given by the composition.
We arrive at the following commutative diagram
\begin{equation}\label{equ:hyb-target}
\xymatrix{
\fP_{\ell} \ar[r]^{\fp} & \spint \ar[r]^{\zeta} & \BC.
}
\end{equation}

For simplicity, we write $\fP := \fP_{1}$ and $\punt := \punt^{1/1}$.
By the above construction, we obtain a Cartesian diagram with strict
vertical arrows
\[
\xymatrix{
\fP_{\ell} \ar[rr] \ar[d] && \fP \ar[d]   \\
\cA^{1/\ell} \ar[rr]^{\ell-\mbox{th root}} && \cA
}
\]
which restricts to a Cartesian diagram with strict arrows
\begin{equation}\label{diag:R-target-from -GLSM}
\xymatrix{
 \punt^{1/\ell} \ar[rr] \ar[d]_{\cO(\punt^{1/\ell})} && \punt \ar[d]^{\cO(\punt)} \\
 \ainfty^{1/\ell} \ar[rr]^{\ell-\mbox{th root}} && \ainfty
}
\end{equation}
where we denote $\cO(\punt) := \cO_{\fP}(\punt)|_{\punt}$ and $\cO(\punt^{1/\ell}) := \cO_{\fP_{\ell}}(\punt^{1/\ell})|_{\punt^{1/\ell}}$. This defines
\begin{equation}\label{eq:GLSM-punctured-target}
\punt^{1/\ell} \to \BC
\end{equation}
given by the composition $\punt^{1/\ell} \to \fP_{\ell} \to \BC$.

\subsubsection{The boundary \texorpdfstring{$\infty^{1/\ell} \to \BC$}{} as a target of punctured R-maps}

We next check that \eqref{eq:GLSM-punctured-target} can be constructed as in \S \ref{ss:target}. In view of Proposition \ref{prop:one-step-target} and \eqref{diag:R-target-from -GLSM}, it suffices to consider the $\ell = 1$ case. Note that we have a morphism over $\cX$
\begin{equation}\label{eq:remove-spin}
\ul{\punt} = \PP^{\bw'}\left(\oplus_i(\bE^{\vee}_{i,\fX}\otimes\cL^{i}_{\fX}) \right) \longrightarrow \PP^{\bw'}\left(\oplus_i \bE^{\vee}_{i} \right) =: \xinfty
\end{equation}
where $\bw'$ assigns the weight $i/d$ for the $i$-th component.
By \eqref{diag:spin-target}, he morphism \eqref{eq:remove-spin} fits
into a Cartesian diagram
\[
\xymatrix{
\ul{\punt} \ar[rr] \ar[d]_{\uspin := \cL_{\cX}|_{\ul\punt}} && \xinfty\times\BC \ar[d]^{\lbspin\boxtimes\uomega} \\
\BG_m \ar[rr]^{\nu_r} && \BG_m
}
\]
This gives the desired diagram \eqref{eq:underlying-P-target} in the setup for punctured R-maps.

Denote by $\cO_{\xinfty}(-1)$ the tautological sub-line bundle of the weighted projective bundle $\xinfty \to \cX$. Then in the setup of punctured R-maps, \eqref{eq:l=1-log} is given by \eqref{diag:R-target-from -GLSM} and
\[
\cO(\punt) \cong \cO_{\xinfty}(1)\otimes\uspin^{-d}.
\]

To summarize:
\begin{proposition}
  \label{prop:target-comparison}
The target $\punt^{1/\ell} \to \BC$ as in \eqref{eq:GLSM-punctured-target} can be constructed as in \eqref{eq:punt} using the input data
$
(r, d, \ell, \xinfty, \lbspin, \lblog = \cO_{\xinfty}(1)).
$
\end{proposition}

\begin{remark}
Given a target from log GLSM \eqref{eq:GLSM-punctured-target}, it follows directly from  \cite[Proposition 2.7]{CJR21} that the stability of punctured R-maps in Definition \ref{def:P-stability} is compatible with the stability of log R-maps.
\end{remark}

\subsubsection{Superpotentials along the boundary}\label{ss:GLSM-superpotential-boundary}
For the log GLSM data  as in \S \ref{sss:input}, a {\em superpotential} is a morphism of stacks $W \colon \fP^{\circ} \to \uomega$ over $\BC$, \cite[\S 3.4.1]{CJR21}. The {\em critical locus} $\crit(W) \subset \fP^{\circ}$ is the closed substack along which the differential $\diff W$ vanishes. The superpotential $W$ is said to have {\em proper critical locus} if $\crit(W) \to \BC$ is proper. Next we observe that $W$ induces  a superpotential of the boundary $\infty$  as in Definition~\ref{def:superpotential2} as follows.

First,  we may view  $W$ as a section of the line bundle $\uomega|_{\fP^{\circ}}$. By
  \cite[Lemma~3.8]{CJR21}, $W$ extends to a rational section of $\uomega|_{\fP}$ with poles along $\punt$ of order precisely $\ttwist$. Thus the restriction $W|_{\punt}$ is a section of $\uomega \otimes \cO(\ttwist \infty)$, hence defines a superpotential over $\punt$ of order $\ttwist$.

\subsubsection{Twisted superpotentials from log GLSM}\label{ss:GLSM-potential}

The superpotential $W \colon \fP^{\circ} \to \uomega$ as above leads to a twisted superpotential $\tW_{\ell}$ as in \cite[\S 3.4.3]{CJR21} for constructing the reduced theory in log GLSM. We next show that $\tW_{\ell}$ restricts directly to a twisted superpotential $\tW_{\ell,\infty}$ of the target $\punt^{1/\ell} \to \BC$ of order $\ttwist = \ell r /d$. This compatibility is needed for computing reduced virtual cycles in \cite{CJR23P}.

First, we recall the twisted superpotential $\tW_{\ell}$ as in \cite[\S 3.4.3]{CJR21}.
By \eqref{diag:change-univ-target} and \eqref{diag:expand-R-target}, we have the Cartesian diagram of log stacks with strict vertical arrows
 \begin{equation}\label{eq:log-R-twisted-potential-source}
 \xymatrix{
 \punt^{1/\ell}_{e,\circ} \ar[rr] \ar[d] && \fP^{e,\circ}_{\ell} \ar[rr] \ar[d] &&  \fP_{\ell} \times \cA^{1/\ell}_{\max} \ar[d] \\
 \infty^{1/\ell}_{\cA^{e,\circ}}  \ar[rr] && \cA^{1/\ell, e,\circ} \ar[rr] && \cA^{1/\ell}\times\cA^{1/\ell}_{\max}.
 }
 \end{equation}

 On the other hand, set $\PP_{\omega}$ to be the log stack associated to the pair $(\PP(\uomega \oplus \cO), \infty_{\omega})$, whose the log structure is given by the smooth divisor $\infty_{\omega}$ defined by the vanishing of the coordinate of the factor $\cO$. Consider the Cartesian diagram
 \begin{equation}\label{eq:log-R-twisted-potential-target}
 \xymatrix{
 \PP_{\omega,\ell}^{e,\circ} \ar[rr] \ar[d] && \PP_{\omega}\times\cA_{\max}^{1/\ell} \ar[d] \\
 \cA^{e,0} \ar[rr] && \cA \times \cA
 }
 \end{equation}
 where the right vertical arrow is the product of the strict morphism $\PP_{\omega} \to \cA$ with the $\ttwist$-th root morphism $\cA_{\max}^{1/\ell} \to \cA$. The superpotential $W$ induces a commutative diagram with horizontal rational maps
\begin{equation}\label{eq:rational-potential}
\xymatrix{
\fP^{e,\circ}_{\ell} \ar@{-->}[rr] \ar[d] &&  \PP_{\omega}^{e,\circ} \ar[d] \\
 \fP_{\ell} \times \cA^{1/\ell}_{\max} \ar@{-->}[rr]^{W\times id} && \PP_{\omega}\times\cA_{\max}^{1/\ell}
}
\end{equation}

 By \cite[Lemma 3.9]{CJR21}, there is a morphism
 \begin{equation}\label{eq:twisted-potential-contraction}
 \mathfrak{c} \colon  \PP_{\omega}^{e,\circ} \longrightarrow \uomega\boxtimes\cO_{\cA^{1/\ell}_{\max}}(\ttwist \Delta_{\max}^{1/\ell})
 \end{equation}
 that contracts the proper transform of $\PP_{\omega}\times \Delta_{\max}^{1/\ell}$. The composition
 \begin{equation}\label{eq:log-R-twisted-potential}
 \xymatrix{
 \tW_{\ell} \colon \fP^{e,\circ}_{\ell} \ar@{-->}[r] & \PP_{\omega}^{e,\circ} \ar[r]^-{\mathfrak{c}} &  \uomega\boxtimes\cO_{\cA^{1/\ell}_{\max}}(\ttwist \Delta_{\max}^{1/\ell})
 }
 \end{equation}
 is a morphism over $\BC\times\cA^{1/\ell}_{\max}$ by \cite[Proposition 3.10]{CJR21}, and is called the {\em twisted superpotential}. The twisted superpotential $\tW_{\ell,\infty}$ of $\punt^{1/\ell}$ is defined to be the restriction
 \begin{equation}\label{eq:restricted-twisted-potential}
 \tW_{\ell,\infty} := \tW_{\ell}|_{\punt^{1/\ell}_{e,\circ}} \colon \punt^{1/\ell}_{e,\circ} \longrightarrow \uomega\boxtimes\cO_{\Delta_{\max}^{1/\ell}}(\ttwist \Delta_{\max}^{1/\ell}).
 \end{equation}

 \begin{proposition}\label{prop:GLSM-induced-superpotential}
Notations as above, we have
 \begin{enumerate}
   \item If $W$ has proper critical locus, then $\tW_{\ell,\infty}$ has proper critical locus.
  \item  $\tW_{\ell,\infty}$ is the twisted superpotential \eqref{eq:induced-superpotential} induced by $\tW_{1,\infty}$.
 \end{enumerate}
 \end{proposition}

 \begin{proof}
By \cite[Proposition 3.12]{CJR21}, $W$ has proper critical locus implies the critical locus of $\tW$ is proper over $\BC\times\cA^{1/\ell}_{\max}$. Thus (1) follows from the restriction construction \eqref{eq:restricted-twisted-potential}.

To verify (2), note that we have a commutative diagram
\[
\xymatrix{
 \infty^{1/\ell}_{\cA^{e,\circ}}  \ar[rr] \ar[d] && \cA^{1/\ell, e,\circ}  \ar[rr] \ar[d] && \cA^{1/\ell}\times\cA^{1/\ell}_{\max}  \ar[d] \\
 \infty_{\cA^{e,\circ}}  \ar[rr] && \cA^{e,\circ}  \ar[rr] && \cA \times \cA_{\max}
}
\]
where all vertical arrows are log \'etale. Pulling back along  \eqref{eq:log-R-twisted-potential-source}, we obtain a commutative diagram
\[
\xymatrix{
 \punt^{1/\ell}_{e,\circ} \ar[rr] \ar[d] && \fP^{e,\circ}_{\ell} \ar[rr] \ar[d] &&  \fP_{\ell} \times \cA^{1/\ell}_{\max} \ar[d] \\
  \punt_{e,\circ} \ar[rr] && \fP^{e,\circ} \ar[rr] &&  \fP \times \cA_{\max}  \\
}
\]
again with log \'etale vertical arrows, where the left vertical arrow is given by the corresponding arrow in \eqref{diag:universal-exp-change-ell}.

On the other hand, consider the commutative triangle
\[
\xymatrix{
\PP_{\omega}\times\cA^{1/\ell}_{\max} \ar[rr] \ar[rd] && \PP_{\omega}\times\cA_{\max} \ar[ld] \\
&\cA\times\cA&
}
\]
with the two sides as in \eqref{eq:log-R-twisted-potential-target}. Pulling back along the bottom arrow of \eqref{eq:log-R-twisted-potential-target}, we obtain the top arrow fitting in the following commutative diagram
\[
\xymatrix{
\PP^{e,\circ}_{\omega,\ell} \ar[rr] \ar[d] && \PP^{e,\circ}_{\omega,1} \ar[d] \\
\uomega\boxtimes\cO_{\cA^{1/\ell}_{\max}}(\ttwist \Delta_{\max}^{1/\ell}) \ar[rr] && \uomega\boxtimes\cO_{\cA_{\max}}(\frac{r}{d} \cdot \Delta_{\max}) \\
}
\]
where the two vertical arrows are given by \eqref{eq:twisted-potential-contraction}, and the bottom arrow is given by the isomorphism of line bundles
\[
\uomega\boxtimes\cO_{\cA^{1/\ell}_{\max}}(\ttwist \Delta_{\max}^{1/\ell}) \cong \uomega\boxtimes\cO_{\cA_{\max}}(\frac{r}{d} \cdot \Delta_{\max})\big|_{\BC\times\cA^{1/\ell}_{\max}}.
\]
We have arrived at the following commutative diagram
\begin{equation}\label{eq:GLSM-potential-change-ell}
\xymatrix{
\fP^{e,\circ}_{\ell} \ar[r] \ar@/^1.5pc/[rrr]^{\tW_{\ell}} \ar[d] & \fP_{\ell}\times\cA^{1/\ell}_{\max} \ar@{-->}[r] \ar[d] & \PP^{e,\circ}_{\omega,\ell} \ar[r] \ar[d] & \uomega\boxtimes\cO_{\cA^{1/\ell}_{\max}}(\ttwist \Delta_{\max}^{1/\ell}) \ar[d] \ar[r] & \BC\times\cA^{1/\ell}_{\max} \ar[d] \\
\fP^{e,\circ} \ar[r]  \ar@/_1.5pc/[rrr]_{\tW}  & \fP\times\cA^{1/\ell}_{\max} \ar@{-->}[r]  & \PP^{e,\circ}_{\omega} \ar[r] & \uomega\boxtimes\cO_{\cA_{\max}}(\frac{r}{d} \Delta_{\max}) \ar[r] &  \BC\times\cA_{\max}.
}
\end{equation}
where the two dashed arrows are the  rational maps induced by $W$ as in \eqref{eq:rational-potential}. Taking further restriction of \eqref{eq:GLSM-potential-change-ell}, we obtain
\[
\xymatrix{
\punt^{1/\ell}_{e,\circ} \ar[rr]^-{\tW_{\ell,\infty}} \ar[d] && \uomega\boxtimes\cO_{\Delta^{1/\ell}_{\max}}(\ttwist \Delta_{\max}^{1/\ell}) \ar[d] \ar[rr] && \BC\times\Delta^{1/\ell}_{\max} \ar[d] \\
\punt_{e,\circ} \ar[rr]^-{\tW_{\infty}} && \uomega\boxtimes\cO_{\Delta^{1/\ell}_{\max}}(\ttwist \Delta_{\max}^{1/\ell}) \ar[rr] && \BC\times\Delta_{\max}
}
\]
where the left square recovers \eqref{eq:induced-superpotential}. This proves (2).
\end{proof}

\begin{remark}
Note that the twisted superpotential $\tW_{\ell,\infty}$ is induced by the superpotential $W|_{\punt}$ in \S \ref{ss:GLSM-superpotential-boundary} via Lemma \ref{lem:super-potential-equivalence}.
Since we will directly work with  the twisted superpotential $\tW_{\ell,\infty}$, we leave the correspondence between $\tW_{\ell,\infty}$ and $W|_{\punt}$  to the reader. This can be checked directly following the above discussions.
\end{remark}

%%%%%%%%%%%%%%%%%%%%%%%%%%
%%%Forgetting unit sectors
%%%%%%%%%%%%%%%%%%%%%%%%%%

\section{Axioms of punctured R-maps}
\label{sec:remove-unit}

Starting from this section, we will establish several axioms analogous to the ones in usual
Gromov--Witten theory related to removing markings.

\subsection{The statements}\label{ss:statement-remove-unit}

Fix the discrete data $(g, \beta, \varsigma)$ for a punctured R-map to
the target $\infty$ as in \S \ref{intro:target}.
Consider the new discrete data $(g, \beta, \varsigma + \mathbf{1})$
obtained by adding a unit sector $\mathbf{1}$ as in \S
\ref{sss:unit-sector}.
Let $n$ be the number of markings in $\varsigma$.
We will use $p_{n+1}$ to denote the $(n+1)$-st unit marking. For
simplicity, we introduce notations
\begin{equation}\label{eq:short-hand-R-stacks}
\SH_{\varsigma + \mathbf 1} := \SH_{g, \varsigma + \mathbf 1}(\infty, \beta), \ \  \SH_{\varsigma} := \SH_{g, \varsigma}(\infty, \beta).
\end{equation}
Denote by $f_{\varsigma} \colon \pC_{\varsigma} \to \infty$ the
universal punctured map over $\SH_{\varsigma} $, and by
$\mathrm{S} \colon \widetilde{\cC}_{\varsigma} \to \pC_{\varsigma}$
the saturation morphism. Similarly, we introduce
\[
\UH_{\varsigma + \mathbf 1} := \UH_{g, \varsigma + \mathbf 1}(\infty, \beta), \ \  \UH_{\varsigma} := \UH_{g, \varsigma}(\infty, \beta), \ \ f^{\curlywedge}_{\varsigma} \colon \cC^{\curlywedge,\circ}_{\varsigma} \to \infty,
\]
 and the saturation
$\mathrm{S} \colon \widetilde{\cC}^{\curlywedge}_{\varsigma} \to \cC^{\curlywedge,\circ}_{\varsigma}$.

Note that the saturation morphism $\mathrm{S}$ is finite, and is an
isomorphism away from the punctured markings.
Furthermore, along the punctured markings, the source and target of
$\mathrm{S}$ differ only by nilpotents, see
\cite[Proposition~5.5]{ACGS20P}.

We will use $\SH^{\bullet}$ to denote either $\SH$ or $\UH$.
Similarly, in
\[
 \cC^{\bullet,\circ}, \ \ f^{\bullet}, \ \ \widetilde{\cC}^{\bullet}_{\ddata},
\]
the symbol $\bullet$ is either empty or $\curlywedge$.
We can now describe the universal curve:

\begin{proposition}\label{prop:moduli-remove-unit}
  There is a natural isomorphism
  $\SH_{\varsigma + \mathbf{1}} \cong \widetilde{\cC}_{\varsigma}$
  inducing
  a natural isomorphism
  $\UH_{\varsigma + \mathbf{1}} \cong
  \widetilde{\cC}^{\curlywedge}_{\varsigma}$, such that the
  composition
\[
\bF_{\mathbf{1}} \colon \SH^{\bullet}_{\varsigma + \mathbf{1}}  \cong \widetilde{\cC}^{\bullet}_{\varsigma} \stackrel{\mathrm{S}}{\longrightarrow} \cC^{\bullet,\circ}_{\varsigma} \stackrel{\pi_{\varsigma}}{\longrightarrow} \SH^{\bullet}_{\varsigma}.
\]
on the level of coarse stable maps of \eqref{eq:coarse-map}, is obtained by removing the marking $p_{n+1}$ and then taking the  stabilization of the usual stable map.
\end{proposition}

\begin{remark}
As categories over log schemes, $\UH$ can be viewed as a full subcategory of $\SH$ consisting of objects with uniform maximal degeneracies. The induced isomorphism $\UH_{\varsigma + \mathbf{1}} \cong \widetilde{\cC}^{\curlywedge}_{\varsigma}$ is in fact the restriction of the canonical one $\SH_{\varsigma + \mathbf{1}} \cong \widetilde{\cC}_{\varsigma}$ from the category point of view.
\end{remark}

This will be proved in Proposition~\ref{prop:stable-prmap-remove-unit}
and Proposition~\ref{prop:add-unit-sector-umd}.
Since the unit marking has negative contact order, adding or removing
it results in changes of both R- and log structures,
compared to the case of usual stable maps.
Consequently, the proof splits into two major steps.
First, in \S\ref{ss:twist-target}, \S\ref{ss:log-curve-add-marking}
and \S\ref{ss:universal-pmap-add-marking}, we focus on the logarithmic
side and consider adding markings in case of punctured maps to the
universal target $\ainfty$.
Readers are advised to skip these three technical subsections first,
and refer to them as needed.
Second, in \S\ref{ss:stable-PRmap-add-marking} and
\S\ref{ss:stable-PRmap-add-marking-max} we discuss adding unit
markings in the case of punctured R-maps.

Assuming that $2g - 2 + n > 0$, denote by $\oM_{g,n}$ the
Deligne--Mumford moduli stack of stable curves, and
$\pi \colon \oM_{g,n+1} \to \oM_{g,n}$ its universal curve.
We may view $\pi$ as a morphism of log stacks with the canonical log
structures given by the underlying stable curves.
Proposition~\eqref{prop:moduli-remove-unit} implies the following
commutative diagram
\begin{equation}\label{diag:forget-unit-stable-curve}
\xymatrix{
\SH^{\bullet}_{\ddata+\mathbf{1}} \ar[rr]^{\mathrm{p}_{\ddata+\mathbf{1}}} \ar[d]_{\bF_{\mathbf{1}}} && \oM_{g,n+1} \ar[d]^{\pi} \\
\SH^{\bullet}_{\ddata} \ar[rr]^{\mathrm{p}_{\ddata}} && \oM_{g,n}
}
\end{equation}
where the horizontal arrows take the stabilization of the coarse
source curves.
We also note that the target of $\ev_{n+1}$ is $\xinfty$ by the
definition of unit sectors in \S\ref{sss:unit-sector}.

On the cycle level, we show that both the canonical and reduced theory
of punctured R-maps satisfy similar axioms to the ones of usual
Gromov--Witten theory (\cite[(2.6)]{KoMa97}):

\begin{theorem}\label{thm:remove-unit}
With the above assumptions, we have
\begin{enumerate}
\item $\mathrm{S}_{*}[\SH^{\bullet}_{\varsigma + \mathbf{1}}]^{\star} = \pi_{\varsigma}^*[\SH^{\bullet}_{\varsigma}]^{\star}$.

\item $\bF_{\mathbf{1},*}\left(\ev_{n+1}^*D \cap [\SH^{\bullet}_{\varsigma + \mathbf{1}}]^{\star} \right) = \int_{\beta}D \cdot [\SH^{\bullet}_{\varsigma}]^{\star}$ for a divisor $D \in CH^1(\xinfty)$.

\item If further assuming $2g - 2 + n > 0$, then
  $\mathrm{p}_{\ddata+\mathbf{1},*}
  [\SH^{\bullet}_{\ddata+\mathbf{1}}]^{\star} =
  \pi^{*}\mathrm{p}_{\ddata,*} [\SH^{\bullet}_{\ddata}]^{\star}.$

\end{enumerate}
where $[-]^\star$ can be either $[-]^\vir$, or $[-]^\red$ when
$\bullet = \curlywedge$ and the reduced theory as in Theorem
\ref{thm:red-POT} is defined.
\end{theorem}

Statement (1) will be established in \S \ref{ss:fundamental-class-axiom-canonical} and \S\ref{ss:fundamental-class-axiom-reduced}, and (3)  will be proved in \S\ref{ss:unit-no-psi-min}.  Statement (2) follows from (1) and the projection formula as in the case of usual Gromov--Witten theory.

For later  use, denote by $\varsigma'$ and $\mathbf{1}'$ the reduced discrete data corresponding to $\varsigma$ and $\mathbf{1}$, and we write
\[
\fM_{\varsigma' + \mathbf{1}'} := \fM_{\varsigma' + \mathbf{1}'}(\ainfty), \ \ \fM_{\varsigma'} := \fM_{g, \varsigma'}(\ainfty).
\]
 By Proposition \ref{prop:one-step-target}, we may assume $\ell = 1$ for the rest of this section.

\subsection{Twists of targets}\label{ss:twist-target}

The construction of the forgetful morphism $\Fm$ requires modifying
both the log and R-structures.
Introducing an extra factor of $\cA$ to keep track of the Cartier
divisor of a marking, allows us to first perform a universal
construction, called \emph{twist}, on the level of targets.

Consider the sequence of morphisms
\begin{equation}
\xymatrix{
(\ainfty\times\cA)^{\circ} \ar[rr]^{\fp} && \ainfty\times\cA \ar[rr]^{id_{\ainfty}\times\nu_r} && \ainfty\times\cA  \ar[rr] &&  \ainfty
}
\end{equation}
where $\nu_r\colon \cA \to \cA$ is the $r$th root morphism, $\fp$ is the
puncturing (\cite[Definition~2.1]{ACGS20P}) such that
$\ocM_{(\ainfty \times \cA)^{\circ}} \subset \ocM_{(\ainfty \times
  \cA)}^{gp}$ is generated by $\ocM_{(\ainfty \times \cA)}$ and the
global section
$(1,-d) \in \ZZ\oplus\ZZ \cong \Gamma(\ocM_{(\ainfty \times
  \cA)^{\circ}}^{gp}) $, and the last arrow is the projection to
$\ainfty$. The newly constructed log stack $(\ainfty\times\cA)^{\circ}$ can be understood as follows.
\begin{lemma}
  \label{lemma:twisting-isomorphism}
  There is an isomorphism of log stacks
  \begin{equation*}
    \ft'\colon (\ainfty\times\cA)^{\circ} \to \ainfty \times \cA
  \end{equation*}
  induced by $\bar{\ft'}^\flat(1, 0) = (1, -d)$ and
  $\bar{\ft'}^\flat(0, 1) = (0, 1)$ on the level of characteristics.
\end{lemma}
\begin{proof}
  We check directly that there is an inverse $(\ft')^{-1}$ satisfying
  \[
  \overline{(\ft')^{-1}}^\flat(1, 0) = (1, d) \ \ \ \mbox{and} \ \ \ \overline{(\ft')^{-1} }^\flat(0, 1) = (0, 1).
  \]
\end{proof}
Composing $\ft'$ with the projection to the first factor, we obtain a
morphism
\begin{equation}\label{eq:ft}
  \ft \colon (\ainfty\times\cA)^{\circ} \to \ainfty
\end{equation}
induced by $\bar{\ft}^{\flat}(1) = (1, -d)$ on the level of characteristics. This is the {\em twist} on the level of universal target. Now consider the Cartesian diagram with strict vertical arrows
\begin{equation}\label{eq:target-add-marking}
\xymatrix{
(\infty\times\cA)^{\circ} \ar[rr]^{\mathrm{P}} \ar[d] && \infty\times\cA \ar[rr]^{id_{\infty}\times\nu_r}  \ar[d] && \infty\times\cA \ar[d] \ar[rr]  && \punt \ar[d] \\
(\infty_{\cA}\times\cA)^{\circ} \ar[rr]^{\fp} && \ainfty\times\cA \ar[rr]^{id_{\ainfty}\times\nu_r} && \ainfty\times\cA \ar[rr] && \ainfty.
}
\end{equation}
We then construct the twist of the R-map target by lifting $\ft$:
\begin{proposition}\label{prop:target-add-marking}
There is a dashed arrow $\mathrm{T}$ making the following square Cartesian
\begin{equation}\label{eq:twist-target}
\xymatrix{
\punt \ar[d]_{\cO(\punt)} && (\punt\times \cA)^{\circ} \ar@{-->}[ll]_{\mathrm{T}} \ar[d]  \\
\ainfty && (\ainfty\times\cA)^{\circ}  \ar[ll]_{\ft}
}
\end{equation}
such that furthermore
\begin{equation}\label{eq:twist-R}
\mathrm{T}^*\uomega = \uomega\otimes\cL_{\cA}^{r} \ \ \ \mbox{and} \ \ \ \mathrm{T}^*\uspin = \uspin\otimes\cL_{\cA}.
\end{equation}
\end{proposition}

\begin{proof}
We split the proof into several steps.

\smallskip

\noindent
{\bf Step 1. Construction of $\ul{\mathrm{T}}$.}
For clarity, we will use $(\uspin', \uomega')$ and $(\uspin,\uomega)$
to denote the universal line bundles of $\punt$ on the left and the
right columns of \eqref{eq:twist-target}, respectively.
Since we assume $\ell = 1$, we have the following diagram of solid
arrows:
\begin{equation}\label{eq:underlying-R-product}
\xymatrix{
\ul\punt\times\ul{\cA} \ar@/^1pc/[rrrd] \ar@/_1pc/[ddr]_{\uspin\otimes\cL_{\cA}} \ar@{-->}[rd]|-{\ \ul{\mathrm{T}} \ }&&& \\
& \ul\punt \ar[rr] \ar[d]^{\uspin'} && \xinfty\times\BC \ar[d]^{\lbspin\boxtimes \uomega'} \\
& \BG_m \ar[rr]^{\mbox{$r$th root}} && \BG_m
}
\end{equation}
where the square is \eqref{eq:underlying-P-target}, the top arrow is defined by the projection $\ul\punt \to \xinfty$ and $\ul\punt \times \ul\cA \to \BC$ with $\uomega\otimes\cL_{\cA}^r \mapsto \uomega'$. To define $\ul{\mathrm{T}}$, it suffices to verify the commutativity of the solid arrows, which follows from
\begin{equation}\label{eq:underlying-twist}
\lbspin\boxtimes \uomega' \cong \lbspin\otimes(\uomega\otimes\cL_{\cA}^r) \cong (\lbspin\otimes\uomega)\otimes\cL_{\cA}^r \cong  (\uspin)^r \otimes \cL_{\cA}^r \cong (\uspin \otimes \cL_{\cA})^r
\end{equation}
where all the line bundles are considered pulled back to $\ul\punt \times \ul\cA$. This implies that
\begin{equation}\label{eq:proof-twist-R}
\ul{\mathrm{T}}^* \uspin' = \uspin \otimes \cL_{\cA} \ \ \ \mbox{and} \ \ \ \ul{\mathrm{T}}^*\uomega' = \uomega\otimes\cL_{\cA}^{r}
\end{equation}
which is precisely \eqref{eq:twist-R}.

\smallskip
\noindent
{\bf Step 2. Construction of $\mathrm{T}$.} Since the two vertical arrows in \eqref{eq:twist-target} are strict, to show that $\ul{\mathrm{T}}$ lifts to $\mathrm{T}$, it suffices to check that the underlying structure $\ul{\mathrm{T}}$ is compatible with the morphism of log structures induced by $\ft$, leading to the commutativity of \eqref{eq:twist-target}.

Note  that $\cM_{\cA}$ is defined by the universal section $s$ of $\cL_{\cA}$, and $\cM_{\punt}$ is defined by \eqref{eq:l=1-log} since $\ell = 1$. Let $\cL'$ and $\cL$ be the universal line bundles defining the log structures of $\ainfty$ on the left and the right columns of \eqref{eq:twist-target}, respectively. Then the lift $\mathrm{T}$ over $\ul{\mathrm{T}}$ follows from the compatibility of line bundles over $(\punt\times \cA)^{\circ}$:
\[
\ul{\mathrm{T}}^*\cL' = \mathrm{T}^*(\lblog\otimes(\uspin')^{-d}) \cong \lblog\otimes(\uspin\otimes \cL_{\cA})^{-d} \cong (\lblog\otimes(\uspin)^{-d})\otimes \cL_{\cA}^{-d} = \cL\otimes \cL_{\cA}^{-d}.
\]

\smallskip
\noindent
{\bf Step 3. Verifying the Cartesian property.}
Denote by $\widetilde{\infty} = \infty \times_{\ainfty}(\ainfty\times\cA)^{\circ}$.
The commutativity of \eqref{eq:twist-target} induces a morphism $({\infty}\times{\cA})^{\circ} \to \widetilde{\infty}$.
We show this is an isomorphism by constructing its inverse explicitly. Since the two vertical arrows in  \eqref{eq:twist-target} are strict, it suffices to prove the  statement on the underlying level.

Unwinding the fiber product construction, the data of a morphism ${Y} \to \ul{\widetilde{\infty}}$ is equivalent to the following data
\[
\left( Y \to \xinfty\times\BC, \cL'_{R,Y}, \cL_{Y}, (\cL_{\cA,Y}, s_Y), (\cL'_{R,Y})^r \cong \lbspin|_{{Y}}\otimes \uomega'|_{Y},  \lblog|_{Y}\otimes\cL_{R,Y}^{d} \cong \cL'_Y \otimes\cL_{\cA,Y}^{-d}\right).
\]
Here $\cL'_{R,Y}$ is the line bundle on $Y$ corresponding to $Y \to \ul\punt$, and $\cL_{Y}, \cL_{\cA,Y}$ are line bundles corresponding to $Y \to \ainfty\times\cA$.

Set $\cL_{R,Y} = \cL'_{R,Y}\otimes\cL^{\vee}_{\cA,Y}$. Then we calculate
\[
\cL_{R,Y}^{r} \cong (\cL'_{R,Y}\otimes\cL^{\vee}_{\cA,Y})^r \cong \lbspin_{Y}\otimes (\uomega'|_{Y} \otimes \cL_{\cA,Y}^{-r}).
\]
In view of \eqref{eq:proof-twist-R}, the data $(Y \to \xinfty\times\BC, \cL_{R,Y})$ is equivalent to a commutative triangle
\[
\xymatrix{
Y \ar[rr]^{\cL'_{R,Y}} \ar@/_1pc/[rrd]_{\uomega'|_{Y} \otimes \cL_{\cA,Y}^{-r}} && \ul{\infty} \ar[d] \\
&&\BC
}
\]
Together with $Y \to \ul{\cA}$ induced by $(\cL_{\cA,Y}, s_Y)$, we
obtain
$Y \to \ul\ainfty \times \ul\cA = \ul{(\ainfty\times\cA)^{\circ}}$,
hence $\ul{\widetilde{\infty}} \to \ul{(\ainfty\times\cA)^{\circ}}$.
Tracing through the above construction, we see that this is the
desired inverse.
\end{proof}

While \eqref{eq:twist-R} implies that $\mathrm{T}$ is not a morphism
over $\BC$, the proof of this proposition shows the following
compatibility of R-structures:

\begin{corollary}\label{cor:target-log-R}
There is a commutative diagram
\[
\xymatrix{
\punt \ar[d] && (\punt\times \cA)^{\circ} \ar[ll]_{\mathrm{T}} \ar[d]  \ar[rr] && \punt \ar[d]  \\
\ainfty\times\BC \ar[d] && (\ainfty\times\cA)^{\circ}\times\BC  \ar[ll]_{\ft\times id_{\BC}} \ar[d] \ar[rr] && \ainfty\times\BC \ar[d] \\
\BC && \ul{\cA}\times\BC \ar[ll]^{\uomega \mapsfrom \cL_{\cA}^{r}\otimes\uomega}  \ar[rr]_{(\cL_{\cA}, \uomega) \mapsto \uomega} && \BC
}
\]
where the two squares on the top are Cartesian with strict vertical
arrows, and the top and middle horizontal arrows on the right are
given by the composition of the corresponding horizontal arrows in
\eqref{eq:target-add-marking}.
\end{corollary}

\subsection{Adding markings on log curves}\label{ss:log-curve-add-marking}

Denote by $\fMtw_{n}$ the stack of twisted pre-stable curves with
$n$ markings and let $\fC_{n} \to \fMtw_{n}$ be its universal curve
equipped with the canonical log structures.
Similarly, we have the universal family $\fC_{n+1} \to \fMtw_{n+1}$
of pre-stable twisted curves with $(n+1)$ markings, again
equipped with the canonical log structures.
Let $\fC_{n,\fC} \to \fC_{n}$ be the pull-back of
$\fC_{n} \to \fMtw_{n}$.

The underlying universal curve $\ul{\fC}_{n}$ carries a family of
pre-stable twisted curve $\ul{\fC}_{n+1^{\rig}} \to \ul{\fC}_n$ with
an extra untwisted $(n+1)$-st marking, denoted $p^{\rig}_{n+1}$, and a
morphism
$\ul{\mathrm{St}} \colon \ul{\fC}_{n+1^{\rig}} \to \ul{\fC}_{n, \fC}$
of marked curves over $\ul{\fC}_{n}$ which removes $p^{\rig}_{n+1}$
from the set of markings, and contracts rational components containing
$p^{\rig}$ such that the image of
$p^{\rig}_{n+1} \to \ul{\fC}_{n+1^{\rig}} \to \ul{\fC}_{n, \fC}$ hits
a special point of $\ul{\fC}_{n, \fC} \to \ul{\fC}_n$.
Similarly as in \cite[Lemma~7]{Be97} and \cite[Section 8.1]{AGV08},
the tautological morphism $ \ul{\fC}_{n} \to \ul{\fMtw}_{n+1} $
induced by the family $\ul{\fC}_{n+1^{\rig}} \to \ul{\fC}_n$ with
$(n+1)$ markings is \'etale.

\begin{lemma}\label{lem:universal-log-curve-add-marking}
There is a canonical strict and \'etale morphism  $\fC_{n} \to \fMtw_{n+1}$ lifting the tautological morphism  $
\ul{\fC}_{n} \to \ul{\fMtw}_{n+1}$.
\end{lemma}
\begin{proof}
  By \cite[\S 3.1]{Ol07}, the log structure of $\fMtw_{n+1}$ is the
  divisorial log structure given by the normal crossings boundary
  $\partial \fMtw_{n+1}$, which consists of the singular fibers.
  Similarly, the log structure of $\fC_{n}$ is the divisorial log
  structure given by the normal crossings boundary $\partial \fC_{n}$
  consisting of the $n$ markings and the pullback of
  $\partial\fMtw_{n}$.
  The tautological morphism $\ul{\fC}_{n} \to \ul{\fMtw}_{n+1}$ is a
  morphism of the toroidal pairs
  $(\fC_{n},\partial \fC_{n}) \to (\fMtw_{n+1}, \partial \fMtw_{n+1})$
  preserving boundary divisors, hence induces a log morphism
  $\fC_{n} \to \fMtw_{n+1}$.
  Furthermore, the \'etaleness of the tautological morphism restricts
  to an \'etale morphism of the boundaries
  $\partial \fC_{n} \to \partial \fMtw_{n+1}$.
  This implies the strictness in the statement.
\end{proof}

Consider the canonical log curve
\begin{equation}\label{eq:canonical-log-curve-extra-marking}
{\fC}_{n+1^{\rig}} \to {\fC}_n
\end{equation}
with the underlying $\ul{\fC}_{n+1^{\rig}} \to \ul{\fC}_n$.
Let ${\fC}_{n+1^{\rig}} \to {\fC}_{n+(1^{\rig})}$ be the log morphism
removing $p^{\rig}_{n+1}$ from the set of markings.
Since we do not change log structures away from $p_{n+1}^{\rig}$, ${\fC}_{n+(1^{\rig})} \to {\fC}_n$ is a family of log curves with the canonical log structure.

\begin{lemma}\label{lem:universal-log-curve-contraction}
There  is a canonical log morphism $\mathrm{St} \colon {\fC}_{n+(1^{\rig})} \to \fC_{n,\fC}$ over ${\fC}_n$ lifting the underlying morphism $\ul{\mathrm{St}}$. Furthermore, $\mathrm{St}$ is log \'etale with $\mathrm{St}_{*}\cM_{{\fC}_{n+1^{\rig}}} = \cM_{\fC_{n,\fC}}$ where $\mathrm{St}_{*}\cM_{{\fC}_{n+1^{\rig}}}$ is the push-forward log structure on $\ul{\fC}_{n,\fC}$ as in \cite[(1.4)]{Ka88}.
\end{lemma}
\begin{proof}
  We use the description of log structures as in the proof of Lemma
  \ref{lem:universal-log-curve-add-marking}.
  By Lemma \ref{lem:universal-log-curve-add-marking}, the log
  structure $\cM_{{\fC}_{n+(1^{\rig})}}$ is the divisorial one given
  by the normal crossings boundary $\partial {\fC}_{n+(1^{\rig})}$
  consisting of the pull-back of $\partial \fMtw_{n+1}$ and the $n$
  markings.
  The log structure $\cM_{ \fC_{n,\fC}}$ is the divisorial one given
  by the toroidal boundary
  $\partial \fC_{n,\fC} = \mathrm{pr}^{-1}_{1}\partial \fC_{n} \cup
  \mathrm{pr}^{-1}_{2}\partial\fC_{n}$
  where
  $\mathrm{pr}_i \colon \fC_{n,\fC} = \fC_{n} \times_{\fM_{n}}\fC_{n}
  \to \fC_n$ are the projections to the first and second factors.

  Observe that
  $\ul{\mathrm{St}}(\partial {\fC}_{n+(1^{\rig})}) = \partial
  \fC_{n,\fC}$, and hence $\ul{\mathrm{St}}$ induces a morphism of
  toroidal pairs
  \begin{equation*}
    (\ul{\fC}_{n+(1^{\rig})}, \partial {\fC}_{n+(1^{\rig})}) \to (\ul{\fC}_{n,\fC}, \partial \fC_{n,\fC})
  \end{equation*}
  surjective along the boundary.
  Since the log structures involved are all divisorial, this leads to
  the log \'etale morphism $\mathrm{St}$.

  Finally, the divisorial log structure $\cM_{{\fC}_{n+(1^{\rig})}}$
  is the subsheaf of $\cO_{{\fC}_{n+(1^{\rig})}}$ consisting of
  regular functions vanishing only along
  $\partial {\fC}_{n+(1^{\rig})}$ of order $ \geq 0$.
  Since $\mathrm{St}$ induces an isomorphism outside the boundary, and
  only contracts rational bridges to some strata of the boundary
  $\partial \fC_{n,\fC}$, the push-forward
  $\mathrm{St}_{*}\cM_{{\fC}_{n+1^{\rig}}}$ consists of regular
  functions vanishing only along $\partial \fC_{n,\fC}$ of order
  $\geq 0$, which is the divisorial log structure
  $ \cM_{\fC_{n,\fC}}$.
\end{proof}

\subsection{Adding markings on universal punctured maps}\label{ss:universal-pmap-add-marking}

Consider the universal punctured map and the universal puncturing over
$\fM_{\varsigma'} := \fM_{g, \varsigma'}(\ainfty)$, respectively:
\[
\mathfrak{f}_{\varsigma'} \colon \fC^{\circ}_{\varsigma'} \to \ainfty \ \ \ \mbox{and} \ \ \ \fC^{\circ}_{\varsigma'} \to \fC_{\varsigma'}.
\]
Consider the saturation
$\mathrm{S} \colon \widetilde{\fC}_{\varsigma'} \to
\fC^{\circ}_{\varsigma'}$.
Our next goal is to construct a natural family of punctured maps over
$\widetilde{\fC}_{\varsigma'}$ with the discrete data
$\varsigma' + \bf{1}'$.

\subsubsection{Punctured maps over \texorpdfstring{$\widetilde{\fC}_{\varsigma'}$}{C_sigma}}

Pulling back the universal families along
$\widetilde{\fC}_{\varsigma'} \to \fC^{\circ}_{\varsigma'} \to
\fM_{\varsigma'}$, we obtain the punctured map
${\mathfrak{f}}_{\varsigma',\widetilde{\fC}_{\varsigma'}} \colon
{\fC}^{\circ}_{\varsigma',\widetilde{\fC}_{\varsigma'}} \to \punt$ and
the puncturing
${\fC}^{\circ}_{\varsigma',\widetilde{\fC}_{\varsigma'}} \to
{\fC}_{\varsigma',\widetilde{\fC}_{\varsigma'}}$ over
$\widetilde{\fC}_{\varsigma'}$.
Consider the tautological morphism $\fM_{\varsigma'} \to \fMtw_n$
induced by the log curve $\fC_{\varsigma'} \to \fM_{\varsigma'}$.
Pulling back $\mathrm{St}$ as in
Lemma~\ref{lem:universal-log-curve-contraction} along the composition
$\widetilde{\fC}_{\varsigma'} \to \fM_{\varsigma'} \to \fMtw_{n}$, we
obtain a log \'etale morphism of log curves over
$\widetilde{\fC}_{\varsigma'}$:
\begin{equation}\label{eq:log-contraction}
\widetilde{\mathrm{St}} \colon  \widetilde{\fC}_{\varsigma'+(1^{\rig})} \longrightarrow {\fC}_{\varsigma',\widetilde{\fC}_{\varsigma'}}
\end{equation}

\begin{lemma}\label{lem:punctured-curve-add-marking}
  There is a canonical punctured curve
  $\widetilde{\fC}^{\circ}_{\varsigma'+(1^{\rig})} \to
  \widetilde{\fC}_{\varsigma'}$ fitting into the following commutative
  diagram over $\widetilde{\fC}_{\varsigma'}$:
\[
\xymatrix{
\widetilde{\fC}^{\circ}_{\varsigma'+(1^{\rig})} \ar[rr]^{\widetilde{\mathrm{St}}^{\circ}} \ar[d] && {\fC}^{\circ}_{\varsigma',\widetilde{\fC}_{\varsigma'}} \ar[d] \\
\widetilde{\fC}_{\varsigma'+(1^{\rig})} \ar[rr]^{\widetilde{\mathrm{St}}} && {\fC}_{\varsigma',\widetilde{\fC}_{\varsigma'}}
}
\]
such that the vertical arrows are puncturings along punctured markings labeled by $\varsigma'$, and $\widetilde{\mathrm{St}}^{\circ}$ is pre-stable in the sense of \cite[Definition 2.5]{ACGS20P}.
\end{lemma}

\begin{proof}
  Denote by $\cZ \subset \widetilde{\ul \fC}_{\varsigma'+(1^{\rig})}$
  the collection of rational bridges contracted by
  $\widetilde{\mathrm{St}}$.
  Away from $\cZ$, we define $\widetilde{\mathrm{St}}^{\circ}$ to be
  the identity.
  Let $Z \subset \cZ$ be an irreducible component over a geometric
  point of $\widetilde{\fC}_{\varsigma'}$.
  Then $Z$ is rational with precisely two special points.
  It suffices to construct $\widetilde{\mathrm{St}}^{\circ}$ around
  $Z$.

Suppose $Z$ contracts to a node or a marking with a positive contact
order.
In this case, since punctured structures are not involved, we define
the left vertical morphism to be an isomorphism along $Z$, and hence
set
$\widetilde{\mathrm{St}}^{\circ}|_{Z} = \widetilde{\mathrm{St}}|_{Z}$.

Next assume that $Z$ contracts to a marking $p_2 \in {\fC}^{\circ}_{\varsigma',\widetilde{\fC}_{\varsigma'}}$ with contact
order $c < 0$.
Denote by $p_1 \in Z$ the marking on $Z$, and
$p_0 \in \mathfrak{C}_{\varsigma'}$ the image of
$p_1$.
Let $q \in Z$ be the node joining $Z$ with another component
$Z' \subset \widetilde{\ul \fC}_{\varsigma'+(1^{\rig})}$
To proceed, we fix the following notations.

Noting that $p_0$ is a marking with contact order $c$, let
$\sigma_{p_0} \in \cM_{\fC_{\varsigma'}}$ be a local section given by
the generator of the divisorial log structure associated to the
marking $p_{0}$.
Similarly, let
$\sigma_{p_1} \in \cM_{\widetilde{\fC}_{\varsigma'+(1^{\rig})}}$ and
$\sigma_{p_2} \in
\cM_{{\fC}_{\varsigma',\widetilde{\fC}_{\varsigma'}}}$ be local
sections corresponding to the markings $p_1$ and $p_2$ respectively.
We will reserve $\delta \in \cM_{\punt}$ for a local generator of the
log structure.
We will use the same notations
$\sigma_{p_0}, \sigma_{p_1}, \sigma_{p_2}$ and $\delta$ for local
sections of the corresponding log structures pulled back to various
charts appearing in the following discussions when there is no danger
of confusion.
For simplicity, write
\[
\cM_{0} := \cM_{\widetilde{\fC}_{\varsigma'}}|_{\widetilde{\fC}_{\varsigma'+(1^{\rig})}}, \ \ \cM_1 := \cM_{\widetilde{\fC}_{\varsigma'+(1^{\rig})}}, \ \ \cM_{1}^{\circ} = \cM_{\widetilde{\fC}^{\circ}_{\varsigma'+(1^{\rig})}},
\]
\[
\cM_{2} := \widetilde{\mathrm{St}}^*\cM_{{\fC}_{\varsigma',\widetilde{\fC}_{\varsigma'}}}, \ \  \cM^{\circ}_2 := \widetilde{\mathrm{St}}^*\cM_{{\fC}^{\circ}_{\varsigma',\widetilde{\fC}_{\varsigma'}}}, \ \ \cN :=  \widetilde{\mathrm{St}}^* \circ {\mathfrak{f}}_{\varsigma',\widetilde{\fC}_{\varsigma'}}^*\cM_{\punt},
\]
and
$\mathfrak{f}^{\flat} :=
\widetilde{\mathrm{St}}^*({f}_{\varsigma',\widetilde{\fC}_{\varsigma'}}^\flat)
\colon \cN \to \cM^{\circ}_2$.
Since $\cM_{i}$ is a sublog structure of $\cM_{i}^{\circ}$, we
naturally identify elements of $\cM_{i}$ with the corresponding
elements in $\cM_{i}^{\circ}$ for $i=1,2$.

Since $\cM_{1}^{\circ}$ and $\cM_{1}$ are identical away from punctures, our next goal is to construct $\cM_{1}^{\circ}$ at the marking $p_1$ and to construct the  morphism $\widetilde{\mathrm{St}}^{\circ}$ along $Z$.  Since the underlying of $\widetilde{\mathrm{St}}^{\circ}$ is given by $\widetilde{\ul{\mathrm{St}}}$, it suffices to construct the dashed arrow of log structures making the following square commutative
\[
\xymatrix{
 \cM_{1}^{\circ} &&& \cM_{2}^{\circ} \ar@{-->}[lll]_-{(\widetilde{\mathrm{St}}^{\circ})^{\flat}} \\
\cM_{1} \ar[u] &&& \cM_{2} \ar[lll]_{\widetilde{\mathrm{St}}^{\flat}} \ar[u]
}
\]

\bigskip

\noindent
{\bf Case 1: Away from $p_1$.}

We first construct $(\widetilde{\mathrm{St}}^{\circ})^{\flat}$ around
the node $q$.
Let $U \to \ul{\widetilde{\fC}}^{\circ}_{\varsigma'+(1^{\rig})}$ be a
smooth neighborhood of $q$ containing no other special points, hence
$\cM_{1}|_{U} = \cM_{1}^{\circ}|_{U}$.
Let $\sigma_{q,1}, \sigma_{q,2} \in \cM_{1}|_{U}$ be the local
sections corresponding to the local coordinates of $Z'$ and $Z$ around
the node $q$ respectively.
Choosing coordinates appropriately, we may assume that in $\cM_{1}$
\begin{equation}\label{equ:add-marking-node1}
\sigma_{p_0} = \sigma_{q,1} + \sigma_{q,2}, \ \ \ \  \widetilde{\mathrm{St}}^{\flat}(\sigma_{p_2}) = \sigma_{q,1}
\end{equation}
where $\sigma_{p_0}$ is identified with its image in $\cM_{1}|_{U}$ via $\cM_0 \to \cM_1$.

On the other hand, we observe that
\begin{equation}\label{equ:add-marking-node2}
\mathfrak{f}^{\flat}(\delta) = e - c \sigma_{p_2} + u \in \cM_2^{\circ}|_{U}
\end{equation}
for some unit $u \in \cO^*_{U} \subset \cM_2$, and a local section $e \in \cM_{0}|_{U}$ identified with its image via $\cM_{0} \to \cM_{2} \subset \cM_{2}^{\circ}$.
Note that $\cM^{\circ}_{2}|_{U}$ is the fine log structure generated by $\cM_{2}|_{U}$ and $\mathfrak{f}^{\flat}(\delta)$. Since $(\widetilde{\mathrm{St}}^{\circ})^{\flat}|_{\cM_2} = \widetilde{\mathrm{St}}^{\flat}$, to construct $\widetilde{\mathrm{St}}^{\circ}|_{U}$ it remains to identify $(\widetilde{\mathrm{St}}^{\circ})^{\flat}\big(\mathfrak{f}^{\flat}(\delta)\big) \in \cM_{1}^{\circ}$.

Combining \eqref{equ:add-marking-node1} and \eqref{equ:add-marking-node2} with the above discussion, we calculate  in $\cM_{1}^{gp}|_{U}$  that
\begin{equation}\label{eq:around-node}
\begin{split}
(\widetilde{\mathrm{St}}^{\circ})^{\flat}\big(\mathfrak{f}^{\flat}(\delta)\big) &= e - c\cdot \widetilde{\mathrm{St}}^{\flat}(\sigma_{p_2}) + u = e - c \sigma_{q,1} + u \\
&= e - c\cdot (\sigma_{p_0} - \sigma_{q,2}) + u = (e - c \sigma_{p_0}) + c \sigma_{q,2} + u
\end{split}
\end{equation}
To show that
$(\widetilde{\mathrm{St}}^{\circ})^{\flat}\big(\mathfrak{f}^{\flat}(\delta)\big)
\in \cM_{2}|_{U}$, it suffices to show that
$(e - c \sigma_{p_0}) \in \cM_{0}|_{U}$, which is equivalent to
$(\bar{e} - c \bar{\sigma}_{p_0}) \in \ocM_{0}|_{U}$.
Here $\bar{e}$ and $\bar{\sigma}_{p_0}$ are the local sections in the
sheaf of characteristic monoids.

To continue, let $Z_0 \subset \fC^{\circ}_{\varsigma'}$ be the
component containing $p_0$.
Then locally at $p_0$ we have
$\bar{\mathfrak{f}}_{\varsigma'}(\bar{\delta}) = \bar{e}_0 - c
\bar{\sigma}_{p_0}$ where $\bar{\delta}$ is the local section in the
characteristic sheaf corresponding to $\delta$.
Hence $\bar{e}_0$ as a local section in $\ocM_{\fM_{\varsigma'}}$ is
the degeneracy of $Z_0$.
On the other hand, the local section $\bar{e}$ is the degeneracy of
$Z'$ by \eqref{equ:add-marking-node2}.
Note that away from $q$, the morphism
$\ocM_{\ainfty} \to \ocM_{1}|_{U}$ induced by
$\mathfrak{f}^{\flat}_{U}$ and the following composition
\[
U \longrightarrow {\fC}^{\circ}_{\varsigma',\widetilde{\fC}_{\varsigma'}} \longrightarrow \widetilde{\fC}_{\varsigma'} \stackrel{ \mathrm{Sat}}{\longrightarrow} \fC^{\circ}_{\varsigma'}  \stackrel{\mathfrak{f}_{\varsigma'}}{\longrightarrow} \ainfty.
\]
agree. This  implies that $\bar{e} = \bar{e}_0$ with $\bar{e}_0$ identified with its image in $\ocM_0$.
Hence we obtain $\bar{e} - c \bar{\sigma}_{p_0} = \bar{e}_0 - c  \bar{\sigma}_{p_0} \in \ocM_{0}|_{U}$ as needed.  Indeed, we may choose $\sigma_{q,2}$ be the coordinate on $Z \setminus p_1$. Hence \eqref{eq:around-node} defines the morphism $(\widetilde{\mathrm{St}}^{\circ})^{\flat}$ away from the marking $p_1$.

\bigskip

\noindent
{\bf Case 2: Around $p_1$.}
It remains to construct $\cM_{1}^{\circ}$ and
$(\widetilde{\mathrm{St}}^{\circ})^{\flat}$ around $p_1$.
Let $V \to \ul{\widetilde{\fC}}^{\circ}_{\varsigma'+(1^{\rig})}$ be a
smooth neighborhood of $p_1$ containing no other special points.
Choose a local section $\sigma_{p_1} \in \cM_{1}|_{V}$ corresponding
to the local coordinate of $Z$ around $p_1$.
By a careful choice, we may assume that
$\sigma_{p_1}|_{Z \setminus \{p_1, q\}} =(\sigma_{q,2})^{-1}|_{Z
  \setminus \{p_1, q\}}$ in $\cO^*_{Z \setminus \{p_1, q\}}$.
Thus by \eqref{eq:around-node}, we have
\begin{equation}\label{eq:expand-puncture}
(\widetilde{\mathrm{St}}^{\circ})^{\flat}\big(\mathfrak{f}^{\flat}(\delta)\big) = (e - c \sigma_{p_0}) + c \sigma_{q,2} + u = (e - c \sigma_{p_0}) - c \sigma_{p_1} + u, \ \ \ \mbox{ in $\cM_{1}^{gp}|_{V}$.}
\end{equation}
Note that $\bar{e} = \bar{e}_0$ is the degeneracy of $Z_0$, hence its image in the structure sheaf vanishing constantly along $Z$.
By the pre-stability in \cite[Definition 2.5]{ACGS20P}, we define $\cM_1^{\circ}$ around $p_1$ to be the fine log structure generated by $\cM_1$ and $(\widetilde{\mathrm{St}}^{\circ})^{\flat}\big(\mathfrak{f}^{\flat}(\delta)\big)$ in \eqref{eq:expand-puncture}.
Thus, \eqref{eq:expand-puncture} then defines the unique $(\widetilde{\mathrm{St}}^{\circ})^{\flat}|_{V}$ compatible with \eqref{eq:around-node}.

Finally, one checks that the construction does not depend on the local choices of charts. This finishes the proof.
\end{proof}

Composing with ${\mathfrak{f}}_{\varsigma',\widetilde{\fC}_{\varsigma'}} $, the above lemma defines a punctured map  over $\widetilde{\fC}_{\varsigma'}$
\begin{equation}\label{eq:universal-add-0marking}
\tilde{\mathfrak{f}}_{\varsigma' + (1^{\rig})} \colon \widetilde{\fC}^{\circ}_{\varsigma'+(1^{\rig})} \to  \ainfty
\end{equation}
with the discrete data $\varsigma'$ and an extra unmarked section
$p^{\rig}_{n+1}$ avoiding all special points.
For later use, we consider the following situation:

\begin{lemma}\label{lem:max-degeneracy-add-marking}
  Consider a morphism $T \to \fM_{\varsigma'}$ induced by a punctured
  map $f \colon \pC \to \ainfty$ over $T$.
  Let $\widetilde{\cC} \to \pC$ be the saturation, and
  $\tilde{f} \colon \widetilde{\cC}^{\circ}_{(1^{\rig})} \to \ainfty$
  be the punctured map over $\widetilde{\cC}$ obtained by pulling back
  $\tilde{\mathfrak{f}}_{\varsigma' + (1^{\rig})}$.
  Suppose all contact orders in $\varsigma'$ are negative.
  Then $f$ has uniform maximal degeneracy iff $\tilde{f}$ has uniform
  maximal degeneracy.
  In particular, the maximal degeneracies of $\tilde{f}$ and  $f$ coincide.
\end{lemma}
\begin{proof}
Note that $\tilde{f}$ is given by the following composition
\[
\xymatrix{
\widetilde{\cC}^{\circ}_{(1^{\rig})} \ar[rr]_{\widetilde{\mathrm{St}}^{\circ}_T} \ar@/^1pc/[rrrr]^{\tilde{f}} && \pC_{\widetilde{\cC}} \ar[rr]_{f_{\widetilde{\cC}}} && \ainfty
}
\]
where $\widetilde{\mathrm{St}}^{\circ}_T$ is the pull-back of
$\widetilde{\mathrm{St}}^{\circ}$ from Lemma
\ref{lem:punctured-curve-add-marking} and $f_{\widetilde{\cC}}$ is the
pull-back of $f$ via $\widetilde{\cC} \to \pC \to T$.
Thus $\tilde{f}$ and $f_{\widetilde{C}}$ are isomorphic away from the
contracted loci of $\widetilde{\mathrm{St}}^{\circ}_T$.
Let $Z \subset \widetilde{\cC}^{\circ}_{(1^{\rig})}$ be the rational
bridge over a geometric point $s \in \widetilde{\cC}$ contracted by
$\widetilde{\mathrm{St}}^{\circ}_T$.
Thus, by degree considerations, the two special points on $Z$, denoted
$p$ and $q$, have contact orders $c$ and $-c$ respectively.

Suppose $p,q$ are both nodes joining $Z$ with two irreducible components $Z_1$ and $Z_2$ respectively. Denote by $e, e_1, e_2 \in \ocM_s$ the degeneracies of $Z, Z_1$ and $Z_2$ under $\tilde{f}$ respectively.  For $i=1,2$, since $Z_i$ is not contracted in $\pC_{\widetilde{\cC}}$, the degeneracies of $Z_i$ with respect to $\tilde{f}$ and $f_{\widetilde{\cC}}$ coincide.
If $c = -c = 0$, then $e = e_1  = e_2$. If $c > 0$, then $e_1 > e > e_2$.

It remains to consider the case that $p$  is a marking and $q$ is a node joining $Z$ with $Z_2$. In this case, the assumption $c<0$ implies that $q$ has contact order $-c >0$, hence $e_2 > e$.

Thus the maximal degeneracy of $\tilde{f}$ is the same as the maximal degeneracy of $f$.
\end{proof}

Consider the  morphisms of punctured curves over $\widetilde{\fC}_{\varsigma'}$
\[
\widetilde{\fC}^{\circ}_{\varsigma'+1'} \to \widetilde{\fC}^{\circ}_{\varsigma'+1^{\rig}} \to \widetilde{\fC}^{\circ}_{\varsigma'+(1^{\rig})}
\]
where the right arrow is obtained by adding an extra marking given by
$p^{\rig}_{n+1}$, and the left arrow is obtained by taking the $r$th
root stack along $p^{\rig}_{n+1}$ with the resulting marking denoted
by $p_{n+1}$.
We put the trivial puncturing along both markings $p^{\rig}_{n+1}$  and $p_{n+1}$.

Composing with $\tilde{\mathfrak{f}}_{\varsigma' + (1^{\rig})}$, we
obtain two punctured maps
$\tilde{\mathfrak{f}}_{\varsigma' + 1^{\rig}} \colon
\widetilde{\fC}^{\circ}_{\varsigma'+1^{\rig}} \to \ainfty$ and
$\tilde{\mathfrak{f}}_{\varsigma' + 1'} \colon
\widetilde{\fC}^{\circ}_{\varsigma'+1'} \to \ainfty$ over
$\widetilde{\fC}_{\varsigma'}$ with zero contact order along both
$p_{n+1}^{\rig}$ and $p_{n+1}$.
Denote by $\varsigma'+1^{\rig}$ and $\varsigma'+1'$ the corresponding
discrete data.
We have an isomorphism induced by taking the $r$th root along the
marking labeled by $1^{\rig}$:
\begin{equation}\label{eq:change-twist}
\fM_{\varsigma'+1'} := \fM_{\varsigma'+1'}(\ainfty) \cong \fM_{\varsigma'+1^{\rig}} := \fM_{\varsigma'+1^{\rig}}(\ainfty)
\end{equation}
The following proposition plays a crucial role in the proof of
Theorem~\ref{thm:remove-unit} (2) and (3):

\begin{proposition}\label{prop:universal-adding-marking}
The tautological morphism $\widetilde{\fC}_{\varsigma'} \to \fM_{\varsigma'+1^{\rig}}$ induced by $\tilde{\mathfrak{f}}_{\varsigma' + 1^{\rig}}$ is strict and \'etale.
\end{proposition}
\begin{proof}
We divide the proof into several steps.

\smallskip

\noindent
{\bf Step 1: Strictness.}

First, we show that the morphism in question is strict.
Since the characteristic sheaf of monoids of
$\fM_{\varsigma'+1^{\rig}}$ is basic \cite[Definition~2.38]{ACGS20P}, by \cite[Proposition~2.40]{ACGS20P} it suffices to show that the
punctured map \eqref{eq:universal-add-0marking} over
$\widetilde{\fC}_{\varsigma'}$ is basic.
In particular, we need to show that the degeneracies of irreducible
components of the source curve over each geometric fiber satisfy
precisely the edge relations of \cite[(2.16)]{ACGS20P}.
By the basicness of $\fM_{\varsigma'}$, it suffices to consider the
components contracted by \eqref{eq:log-contraction}.

Let $Z$ be a such contracted rational bridge.
Suppose $Z$ contracts to a punctured point, and consider the situation
as in {\bf Case 1} of the proof of
Lemma~\ref{lem:punctured-curve-add-marking}.
By \eqref{equ:add-marking-node2} and \eqref{eq:around-node}, the two
components $Z$ and $Z'$ have degeneracies
$\bar{e} - c\bar{\sigma}_{p_0} = \bar{e}_0 - c\bar{\sigma}_{p_0}$ and
$\bar{e} = \bar{e}_0$ respectively.
The first equation in \eqref{equ:add-marking-node1} then implies that
the degeneracies of $Z$ and $Z'$ satisfy precisely the edge relation
\cite[(2.16)]{ACGS20P} for basicness.
The other cases when $Z$ contracts to a marking or a node are similar
and simpler, hence left to the reader.
This implies the strictness as in the statement.

\smallskip

\noindent
{\bf Step 2: Unwinding the \'etaleness.}

Since the morphism in question is locally of finite type between
locally noetherian stacks, it suffices to check that it is formally
\'etale.
Suppose we have the commutative diagram of solid arrows
\begin{equation}\label{diag:formal-etale}
\xymatrix{
T_0 \ar[rr]^{[f_{T_0}]} \ar[d] &&  \widetilde{\fC}_{\varsigma'} \ar[rr] \ar[d] && \fC_{n} \ar[d] \\
T \ar[rr]_{[f_{T}]} \ar@{-->}[urr]  \ar@{-->}[urrrr] && \fM_{\varsigma'+1^{\rig}}  \ar[rr] && \fMtw_{n+1},
}
\end{equation}
where the left vertical arrow is a square-zero extension, the two right horizontal arrows are the tautological morphisms, and the two left horizontal arrows are strict. It remains to show that there exists a unique lift $T \to \widetilde{\fC}_{\varsigma'}$ making the above diagram commutative. Since the question is local on $T$, we may assume that $T$ is of finite type.

Denote by $f_T \colon \pC_T \to \ainfty$ and
$f_{T_0} \colon \pC_{T_0} \to \ainfty$ the punctured maps over $T$ and
$T_0$ obtained by pulling back the universal punctured maps over
$\fM_{\varsigma'+1^{\rig}}$ and $\widetilde{\fC}_{\varsigma'}$
respectively.
By abuse of notations, we use $f_T \colon \pC_T \to \ainfty$ for the
same punctured map over $T$ except that the $(n+1)$-st marking is
removed from the set of markings.
The commutativity of the left square is equivalent to the equality of
objects $[f_{T_0}] = [f_T]|_{T_0}$.
By Lemma~\ref{lem:punctured-curve-add-marking}, $f_{T_0}$ factors
through a contraction
$\widetilde{\mathrm{St}}^{\circ}_{T_0} \colon \pC_{T_0} \to
\overline{\pC_{T_0}}$ with the resulting punctured map over $T_0$
denoted by $\bar{f}_{T_0} \colon \overline{\pC_{T_0}} \to \ainfty$.
Equivalently, we need to exhibit the existence and uniqueness of a
punctured map $\bar{f}_{T} \colon \overline{\pC_{T}} \to \ainfty$ over
$T$ making the following diagram commutative

\[
\xymatrix{
\pC_{T} \ar@/^1.5pc/[rrrr]^{f_{T}} \ar@{-->}[rr]_{\widetilde{\mathrm{St}}^{\circ}_{T}} && \overline{\pC_{T}} \ar@{-->}[rr]_{\bar{f}_{T} \ \ \ } && \ainfty \\
\pC_{T_0} \ar[rr]_{\widetilde{\mathrm{St}}^{\circ}_{T_0}} \ar[u]  && \overline{\pC_{T_0}} \ar[rru]_{\bar{f}_{T_0}} \ar@{-->}[u] &&
}
\]

\smallskip

\noindent
{\bf Step 3: Construction of the punctured curve $\overline{\pC_{T}} \to T$.}
By the \'etaleness of Lemma \ref{lem:universal-log-curve-add-marking}, we obtain a unique lift $T \to \fC_{n}$ making \eqref{diag:formal-etale} commutative. Pulling  back \eqref{eq:canonical-log-curve-extra-marking}, we obtain the log curve $\overline{\cC}_T \to T$. Further pulling back $\mathrm{St}$ as in Lemma \ref{lem:universal-log-curve-contraction}, we obtain $\widetilde{\mathrm{St}}_{T} \colon \cC_T \to \overline{\cC}_T$ where $\cC_T \to T$ is the log curve associated to $\pC_T \to T$.  We then define
\[
\overline{\cC}^{\circ}_{T} = (\overline{\uC}_{T}, \cM_{\overline{\cC}^{\circ}_{T}} :=  \widetilde{\mathrm{St}}_{T,*}\cM_{\pC_T}).
\]
Next, we will show that $\overline{\cC}^{\circ}_{T}  \to T$ is a punctured curve making the following commutative
\[
\xymatrix{
\pC_{T} \ar@{-->}[d]_{\widetilde{\mathrm{St}}^{\circ}_{T}}  \ar[rr]^{\mathrm{P}} && \cC_T \ar[d]^{\widetilde{\mathrm{St}}_{T}}  \\
\overline{\cC}^{\circ}_{T}  \ar@{-->}[rr]^{\overline{\mathrm{P}}}  &&  \overline{\cC}_{T}
}
\]

First, the adjunction \cite[Definition 1.1.5, 3]{Og18} implies a natural morphism $\widetilde{\mathrm{St}}_{T}^*\cM_{\overline{\cC}^{\circ}_{T}} = \widetilde{\mathrm{St}}_{T}^*(\widetilde{\mathrm{St}}_{T,*}\cM_{\pC_T}) \to \cM_{\pC_T}$ defining the log map $\widetilde{\mathrm{St}}^{\circ}_{T}$.
Observe that $\widetilde{\mathrm{St}}_{T,*}\cM_{\cC_T} = \cM_{\overline{\cC}_{T}}$ by Lemma \ref{lem:universal-log-curve-contraction}.
Thus the puncturing $\mathrm{P}^{\flat} \colon \cM_{\cC_T} \hookrightarrow \cM_{\pC_T}$ pushes forward to a morphism of log structures $\overline{\mathrm{P}}^{\flat} \colon \cM_{\overline{\cC}_{T}} \hookrightarrow \cM_{\overline{\cC}^{\circ}_{T}}$ defining $\overline{\mathrm{P}} \colon \overline{\cC}^{\circ}_{T} \to \overline{\cC}_{T}$.
We will show that $\overline{\mathrm{P}} $ is a puncturing, hence making $\overline{\cC}^{\circ}_{T} \to T$  a punctured curve.
By \cite[Definition 2.1]{ACGS20P}, this amounts to show that $\overline{\mathrm{P}}^{\flat}$ defines an inclusion $\cM^{gp}_{\overline{\cC}^{\circ}_{T}} \hookrightarrow \cM^{gp}_{\overline{\cC}_{T}}$,
which can be checked on the level of characteristics $\ocM^{gp}_{\overline{\cC}^{\circ}_{T}} \hookrightarrow \ocM^{gp}_{\overline{\cC}_{T}}$.

Second, note that away from markings with negative contact orders,
$\overline{\mathrm{P}}^{\flat}$ is an isomorphism of log structures.
Now let $\bar{p} \in \overline{\uC}_{T} $ be a marking with contact
order $c < 0$.
Let $V \to \overline{\uC}_{T}$ be a neighborhood of $\bar{p}$, and
denote $U := V\times_{\overline{\uC}_{T}}\uC_{T}$.
If $\widetilde{\ul{\mathrm{St}}}_{T}|_{U}$ is an isomorphism, then
$\overline{\mathrm{P}}^{\flat}|_{V} = {\mathrm{P}}^{\flat}|_{U}$ is a
puncturing along $\bar{p}$.
We thus assume there is a rational bridge $Z \subset \uC_T$ contracted
by $\widetilde{\ul{\mathrm{St}}}_{T}$ to $\bar{p}$.
Denote by $p \in Z$ the marking corresponding to $\bar{p}$.
Since the question is on the level of characteristics, it suffices to
consider the fibers
$\ocM^{gp}_{\overline{\cC}^{\circ}_{T}}|_{\bar{p}} \hookrightarrow
\ocM^{gp}_{\overline{\cC}_{T}}|_{\bar{p}}$ at $\bar{p}$.
This can be checked by further pulling back to $p$:
\[
\widetilde{\mathrm{St}}_{T}^*(\widetilde{\mathrm{St}}_{T,*}\ocM_{\pC_T})|_{p} \subset \ocM^{gp}_{\pC_T}|_{p} = \ocM^{gp}_{\overline{\cC}_T}|_{p}
\]
This verifies that $\overline{\mathrm{P}}$ is a puncturing, as needed.

\smallskip

\noindent
{\bf Step 4: Construction of the punctured map ${\bar{f}}_T$.}
We first construct the underlying map $\ul{\bar{f}}_T  \colon \overline{\uC}_T \to \ul{\infty}_{\cA}$. Let $k$ be a positive integer divisible by the orders of inertia groups of all markings and nodes of ${\uC_T} \to \ul{T}$, the underlying of the log curve  ${\cC_T} \to T$. Take the coarse curve ${\uC_T} \to {\ul{C}_T}$. Since $\ul{\infty}_{\cA} \cong \BG_m$, let $L$ be the corresponding line bundle over ${\ul{\cC}_T}$ defining $\ul{\bar{f}}_T$. Then $L^{k}$ descends to a line bundle over ${\ul{C}_T}$, hence a morphism
$\ul{{f}}^c_T \colon {\ul{C}_T} \to \ul{\infty}_{\cA}$ making the following diagram commutative
\[
\xymatrix{
{\ul{\cC}_T} \ar[rr] \ar[d]_{\ul{{f}}_T} && {\ul{C}_T} \ar[d]^{\ul{{f}}^c_T} \\
\ul{\infty}_{\cA} \ar[rr]_{\mbox{$k$th power}} && \ul{\infty}_{\cA}
}
\]

Let $\cZ$ be a rational bridge contracted by
$\widetilde{\mathrm{St}}^{\circ}_{T}$, and $p, q \in \cZ$ be the two
special points.
Observe that $p$ and $q$ have the same inertia groups
with the contact orders $c$ and $-c$ respectively.
This implies that
$L|_{\cZ} \cong \cO_{\cZ}(cp - cq)$ is torsion,
hence $L^{k}|_{\cZ}$
descends to the the trivial line bundle over $Z$, where
$\cZ \to Z \subset {\ul{C}_T}$ is the coarse morphism. Consider the
coarse morphism $\overline{\uC}_T \to \overline{\ul{C}}_{T}$ over
$T$. Thus, $\ul{{f}}^c_T$ factors through
$\ul{\bar{f}}^c_T \colon \overline{\ul{C}}_T \to \ul{\infty}_{\cA}$.
We have arrived at the following commutative diagram of solid arrows
\[
\xymatrix{
{\ul{\cC}_T} \ar[rr] \ar[d]^{\widetilde{\ul{\mathrm{St}}}_{T}} \ar@/_1.5pc/[dd]_{\ul{{f}}_T} && \sqrt{\ul{C}_T} \ar[rr] \ar[d] && \ul{C}_T \ar[d]  \\
\overline{\uC}_T \ar@{-->}[rr] \ar@{-->}[d]^{\ul{\bar{f}}_T} && \sqrt{\overline{\ul{C}}_{T}} \ar[d] \ar[rr] && \overline{\ul{C}}_T \ar[d]^{\ul{\bar{f}}^c_T} \\
\ul{\infty}_{\cA} \ar@{=}[rr] && \ul{\infty}_{\cA} \ar[rr]_{\mbox{$k$th power}} && \ul{\infty}_{\cA}
}
\]
where the two squares on the right are both Cartesian. Thus to construct $\ul{\bar{f}}_T$, it suffices to construct the horizontal  dashed arrow in the above diagram.

By \cite[Proposition 9.1.1]{AbVi02}, the composition
${\ul{\cC}_T} \to \sqrt{\ul{C}_T} \to \sqrt{\overline{\ul{C}}_{T}}$
factors through a twisted stable map
${\ul{\cC}^{st}_T} \to \sqrt{\overline{\ul{C}}_{T}}$ over $\ul{T}$.
By construction, the morphism ${\ul{\cC}_T} \to {\ul{\cC}^{st}_T}$
contracts precisely those rational bridges that are contracted by
$\widetilde{\ul{\mathrm{St}}}_{T}$, possibly followed by further
rigidifications along some special points. More explicitly, the underlying morphism  ${\ul{\cC}_T} \to {\ul{\cC}^{st}_T}$
factors through $\widetilde{\ul{\mathrm{St}}}_{T} \colon {\ul{\cC}_T} \to \overline{\uC}_T$.
The composition $\overline{\uC}_T \to {\ul{\cC}^{st}_T} \to  \sqrt{\overline{\ul{C}}_{T}} \to \ul{\infty}_{\cA}$ defines the unique $\ul{\bar{f}}_T$ as needed.

We have arrived  at the following commutative diagram of solid arrows:
\[
\xymatrix{
\pC_{T} \ar@/^1.5pc/[rrrr]^{f_{T}} \ar[rr]_{\widetilde{\mathrm{St}}^{\circ}_{T}} \ar[d]_{\mathrm{P}} && \overline{\cC}^{\circ}_{T} \ar@{-->}[rr]_{\bar{f}_{T}} \ar[d]^{\overline{\mathrm{P}}} && \ainfty \ar[d] \\
\cC_{T} \ar[rr]^{\widetilde{\mathrm{St}}_{T}} && \overline{\cC}_{T}  \ar[rr]^{\ul{\bar{f}}_T} && \ul{\infty}_{\cA}
}
\]
By adjunction \cite[Definition 1.1.5, 3]{Og18} again, the morphism
$f^{\flat}_T \colon f^*_{T}\cM_{\ainfty} \cong
\widetilde{\mathrm{St}}_{T}^*\ul{\bar{f}}_T^*\cM_{\ainfty} \to
\cM_{\pC_T}$ defines a morphism
$\bar{f}_T^{\flat} \colon \ul{\bar{f}}_T^*\cM_{\ainfty} \to
\cM_{\overline{\cC}^{\circ}_{T}}$, hence defines the punctured map
$\bar{f}_T$.
The uniqueness of $\bar{f}_T$ follows directly by construction.

\smallskip

\noindent
{\bf Step 5: Compatibility over $T_0$.}
Finally, we observe that the above constructed punctured map
$\bar{f}_T$ restricts to $\bar{f}_{T_0}$ over $T_0$.
Indeed, by prestability, the sheaf of monoids
$\cM_{\overline{\cC}^{\circ}_{T_0}} \subset \cM^{gp}_{\cC_{T_0}}$ is
locally generated by $\cM_{\overline{\cC}_{T_0}}$ and
$\bar{f}^{\flat}_{T_0}(\delta)$ for a local generator of
$f^*_{T_0}\cM_{\ainfty}$.
Pushing forward along $\widetilde{\mathrm{St}}^{\circ}_{T}|_{T_0}$ we
observe that
$\cM_{\overline{\cC}^{\circ}_{T}}|_{\overline{\cC}_{T_0}} \subset
\cM^{gp}_{\overline{\cC}_{T_0}}$ is locally generated by
$\cM_{\overline{\cC}_{T_0}}$ and
$\bar{f}^{\flat}_{T}|_{T_0}(\delta)$.
Note that the pull-backs agree in $\cM^{gp}_{\cC_T}$:
\[
\widetilde{\mathrm{St}}_{T}^*\big(\bar{f}^{\flat}_{T}|_{T_0}(\delta) \big) = f^{\flat}_{T}(\delta)|_{T_0} =  \widetilde{\mathrm{St}}_{T}^*\big( \bar{f}^{\flat}_{T_0}(\delta) \big).
\]
Since $\bar{f}^{\flat}_{T}|_{T_0}(\delta) = \bar{f}^{\flat}_{T_0}(\delta)$ in $\cM^{gp}_{\cC_{T_0}}$, we obtain the desired compatibility $\bar{f}_T|_{T_0} = \bar{f}_{T_0}$.
\end{proof}

\subsubsection{Twisting R-maps along a marking}

Consider $\fM_{\varsigma'+{\bf 1}'} := \fM_{\varsigma'+{\bf 1}'}(\ainfty)$ with the universal punctured map $\mathfrak{f}_{\varsigma' + {\bf 1}'} \colon \fC^{\circ}_{\varsigma'+{\bf 1}'} \to \ainfty$. Note that the $(n+1)$-st marking  $p_{n+1} \subset \fC^{\circ}_{\varsigma'+{\bf 1}'}$ labeled by ${\bf 1}'$ has contact order $-d < 0$.

\begin{lemma}\label{lem:split-marking}
There is a canonical dashed arrow lifting $\mathfrak{f}_{\varsigma' + {\bf 1}'}$
\[
\xymatrix{
 && (\ainfty\times\cA)^{\circ} \ar[d]^{\ft} \\
  \fC^{\circ}_{\varsigma'+{\bf 1}'} \ar@{-->}[urr]^{\mathfrak{f}} \ar[rr]_{\mathfrak{f}_{\varsigma' + {\bf 1}'}} && \ainfty
}
\]
\end{lemma}
\begin{proof}
  The section $p_{n+1} \subset \fC^{\circ}_{\varsigma'+{\bf 1}'}$
  defines a log map $\fC^{\circ}_{\varsigma'+{\bf 1}'} \to \cA$.
  Combining this with $\mathfrak{f}_{\varsigma' + {\bf 1}'}$ yields a
  morphism $\fC^{\circ}_{\varsigma'+{\bf 1}'} \to \ainfty \times \cA$.
  Composing this morphism with the inverse $(\ft')^{-1}$ from
  Lemma~\ref{lemma:twisting-isomorphism} yields the desired dashed arrow
  $\ft$.
\end{proof}

Since the contact order of $f_{\varsigma'+{\bf 1}'}$ along the $(n+1)$-st marking is $-d$, by the definition of $\ft$ in \eqref{eq:ft}, the composition
\begin{equation}\label{eq:remove-contact-order}
\fC^{\circ}_{\varsigma'+{\bf 1}'} \longrightarrow (\ainfty\times\cA)^{\circ} \stackrel{\fp}{\longrightarrow} \ainfty\times\cA \longrightarrow  \ainfty
\end{equation}
with the third arrow the projection, has zero contact order along $p_{n+1}$.
Thus we have a commutative diagram
\begin{equation}\label{diag:remove-contact}
\xymatrix{
&& \fC^{\circ}_{\varsigma'+{\bf 1}'} \ar[lld]_{\mathfrak{f}_{\varsigma' + {\bf 1}'}} \ar[d]_{\mathfrak{f}} \ar[rr] && \fC^{\circ}_{\varsigma'+1'} \ar[d] \ar[rrd]^{ \mathfrak{f}_{\varsigma'+{1}'}  } \ar[rr] &&  \fC^{\circ}_{\varsigma'+(1^{\rig})} \ar[d]^{\mathfrak{f}_{\varsigma'+(1^{\rig})}} \\
\ainfty && (\ainfty\times\cA)^{\circ} \ar[rr]_{\fp} \ar[ll]^{\ft}&& \ainfty\times\cA \ar[rr] && \ainfty
}
\end{equation}
where the top left horizontal arrow is the puncturing along the
$(n+1)$-marking labeled by ${\bf 1}'$, and the top right horizontal
arrow rigidifies and removes the $(n+1)$-st marking from the set of
markings.
The middle vertical arrow is the product of
$\mathfrak{f}_{\varsigma'+{1}'}$ and the morphism
$\fC^{\circ}_{\varsigma',1'} \to \cA$ defined by the divisorial log
structure associated to the $(n+1)$-st marking.
The morphism $\mathfrak{f}_{\varsigma' +1'}$ is a punctured map over
$\fM_{\varsigma'+{1}'}$ by removing puncturing structure along the
$(n+1)$-st marking.
This defines a canonical morphism
\[
F_{{\bf 1}' \mapsto 1'} \colon \fM_{\varsigma'+{\bf 1}'}  \to  \fM_{\varsigma' + 1'}.
\]

\begin{proposition}\label{prop:universal-remove-marking}
The morphism $F_{{\bf 1}' \mapsto 1'}$ is an isomorphism.
\end{proposition}
\begin{proof}
We construct  the  inverse $F := F_{{1}' \mapsto {\bf 1}'}$ of $F_{{\bf 1}' \mapsto 1'}$  as follows.  Let $\mathfrak{f}_{\varsigma'+1'} \colon \fC^{\circ}_{\varsigma'+1'}  \to \ainfty$ be the universal punctured map over $\fM_{\varsigma' + 1'}$. Consider  the Cartesian diagram in the fine category
\[
\xymatrix{
\fC^{\circ}_{\varsigma'+1^{\circ}} \ar[rr] \ar[d] && \fC^{\circ}_{\varsigma'+1'} \ar[d]^{(\mathfrak{f}_{\varsigma'+1'}, p_{n+1})} \\
(\ainfty\times\cA)^{\circ} \ar[rr]^{\fp} && \ainfty\times\cA
}
\]
where the arrow $p_{n+1}$ is defined as in the proof of Lemma
\ref{lem:split-marking}.
Observe that $\ft|_{\ainfty\times(\cA\setminus\ainfty)}$ is an
isomorphism onto its image.
Hence the top arrow is isomorphic away from the $(n+1)$-st marking
$p_{n+1}$.
By the definition of Cartesian product in the category of fine log
structures, the top arrow is a puncturing along $p_{n+1}$.
The composition
$\fC^{\circ}_{\varsigma'+1^{\circ}} \to (\ainfty\times\cA)^{\circ}
\stackrel{\ft}{\to} \ainfty$ defines a punctured map over
$\fM_{\varsigma' + 1'}$ with contact order $-d$ along $p_{n+1}$.
This defines the morphism
$F \colon \fM_{\varsigma' + 1'} \to \fM_{\varsigma'+{\bf 1}'}$.
Tracing through the construction of $F_{{\bf 1}' \mapsto 1'}$ in the
proof of Lemma \ref{lem:split-marking} and $F_{{1}' \mapsto {\bf 1}'}$
above, it is straight forward to check that the two morphisms are
inverse to each other.
\end{proof}

Note that
\[
\fC^{\circ}_{\varsigma'+{\bf 1}'}\setminus p_{n+1} \cong \fC^{\circ}_{\varsigma'+1'} \setminus p_{n+1} \cong \fC^{\circ}_{\varsigma'+(1^{\rig})} \setminus p^{\rig}_{n+1}.
\]
By \eqref{eq:ft}, one checks that the restriction $\ft|_{(\ainfty\times\spec \bk)^{\circ}}$ is an isomorphism of $\ainfty \cong  (\ainfty\times\spec \bk)^{\circ}$. Following the proof of Lemma \ref{lem:split-marking} and Proposition \ref{prop:universal-remove-marking},  we observe that the following punctured maps in \eqref{diag:remove-contact} are only differ along the $(n+1)$-st marking:
\begin{equation}\label{eq:modify-unit-sector-pt}
\mathfrak{f}_{\varsigma' + {\bf 1}'}|_{\fC^{\circ}_{\varsigma'+{\bf 1}'}\setminus p_{n+1}} = \mathfrak{f}_{\varsigma'+{1}'}|_{\fC^{\circ}_{\varsigma'+1'} \setminus p_{n+1}} = \mathfrak{f}_{\varsigma'+(1^{\rig})}|_{\fC^{\circ}_{\varsigma'+(1^{\rig})} \setminus p^{\rig}_{n+1}}.
\end{equation}

Combining \eqref{eq:change-twist}, Proposition
\ref{prop:universal-adding-marking} and Proposition
\ref{prop:universal-remove-marking}, we obtain an analog of
\cite[Lemma~7]{Be97}:
\begin{corollary}
  The composition
  \begin{equation}\label{eq:universal-add-marking}
    F_{+} \colon \widetilde{\fC}_{\varsigma'} \longrightarrow \fM_{\varsigma' + 1^{\rig}} \stackrel{\cong}{\longrightarrow}  \fM_{\varsigma' + 1'} \stackrel{\cong}{\longrightarrow}  \fM_{\varsigma'+{\bf 1}'}
  \end{equation}
  is a strict \'etale morphism.
\end{corollary}

\subsection{Adding a unit sector to a stable punctured R-map}\label{ss:stable-PRmap-add-marking}
Denote by $f_{\varsigma} \colon \pC_{\varsigma} \to \infty$ the universal punctured R-map over $\SH_{\varsigma} := \SH_{g, \varsigma}(\infty, \beta)$, and $\mathrm{S}  \colon  \widetilde{\cC}_{\varsigma} \to \pC_{\varsigma}$ the saturation morphism. Observe that  $\widetilde{\cC}_{\varsigma} \cong \widetilde{\fC}_{\varsigma'} \times_{\fM_{\varsigma'}}\SH_{\varsigma}$. We write $ \SH_{\varsigma+ {\bf 1}} := \SH_{g, \varsigma + \bf{1}}(\infty, \beta)$ for simplicity.

\begin{proposition}\label{prop:stable-prmap-remove-unit}
There is a canonical isomorphism
$
\widetilde{\cC}_{\varsigma} \to \SH_{\varsigma + \bf{1}}(\infty, \beta)
$.
\end{proposition}

\begin{proof}

The proof of the proposition is divided into several steps.

\smallskip

\noindent
{\bf  Step 1: Construction of $\widetilde{\cC}_{\varsigma} \to \SH_{\varsigma + \bf{1}}$.}
We will construct a stable punctured R-map
\begin{equation}\label{equ:stable-pr-map-with-additional-unit}
\tilde{f}_{\varsigma+{\bf 1}} \colon \widetilde{\cC}^{\circ}_{\varsigma + {\bf 1}} \to \punt
\end{equation}
over $\widetilde{\cC}_{\varsigma}$ inducing $\widetilde{\cC}_{\varsigma} \to \SH_{\varsigma + \bf{1}}$ as follows.

Consider the punctured map
$\tilde{\mathfrak{f}}_{\varsigma'} \colon
\widetilde{\fC}^{\circ}_{\varsigma'+{\bf 1}'} \to \ainfty$ over
$ \widetilde{\fC}_{\varsigma'}$ obtained by pulling back the universal
punctured map over $ \fM_{\varsigma'+{\bf 1}'}$ along
\eqref{eq:universal-add-marking}.
Denote by
$\tilde{\mathfrak{f}} \colon \widetilde{\fC}^{\circ}_{\varsigma'+{\bf
    1}'} \to (\ainfty\times\cA)^{\circ}$ the lifting of
$\tilde{\mathfrak{f}}_{\varsigma'}$ given by Lemma
\ref{lem:split-marking}, and let
$\tilde{f} \colon \widetilde{\cC}^{\circ}_{\varsigma + {\bf 1}} \to
(\ainfty\times\cA)^{\circ}$ be the pull-back of $\tilde{\mathfrak{f}}$
via the projection
$\widetilde{\cC}_{\varsigma} \cong \widetilde{\fC}_{\varsigma'}
\times_{\fM_{\varsigma'}}\SH_{\varsigma} \to
\widetilde{\fC}_{\varsigma'}$.
Denote by
$p_{n+1} \subset \widetilde{\cC}^{\circ}_{\varsigma + {\bf 1}}$ the
$(n+1)$-st marking.
There is a commutative diagram
\begin{equation}\label{diag:add-R-marking}
\xymatrix{
\widetilde{\cC}^{\circ}_{\varsigma + {\bf 1}} \ar[r]  \ar[d]_{\omega^{\log}\otimes \cO(r\cdot p_{n+1})}& \widetilde{\cC}^{\circ}_{\varsigma + 1} \ar[r] & \widetilde{\cC}^{\circ}_{\varsigma + (1^{\rig})} \ar[r] & \pC_{\varsigma, \widetilde{\cC}_{\varsigma}} \ar[r]^-{f_{\varsigma}|_{\widetilde{\cC}_{\varsigma}}}  & \punt \ar[d] \\
\BG_m \ar[rrrr]^{\cong} &&&& \BC
}
\end{equation}
where on the top, the left arrow is the puncturing along the
$(n+1)$-st marking, the second arrow rigidifies and removes $p_{n+1}$
from the set of markings, and the third arrow is the pull-back of
$\widetilde{\mathrm{St}}^{\circ}$ in Lemma
\ref{lem:punctured-curve-add-marking}.
The arrows on the top of \eqref{diag:add-R-marking} define a morphism
$\widetilde{\cC}^{\circ}_{\varsigma + {\bf 1}} \to \punt\times\cA$
with the map to $\cA$ defined by the divisorial log structure
associated to the $(n+1)$-st marking.
This arrow fits in a commutative diagram
\[
\xymatrix{
\widetilde{\cC}^{\circ}_{\varsigma + {\bf 1}} \ar[rr] \ar[d]_{\tilde{f}} && \punt\times\cA  \ar[d] \\
(\ainfty\times\cA)^{\circ} \ar[rr]^-{\fp} && \ainfty\times\cA.
}
\]

By the Cartesian diagram \eqref{eq:target-add-marking}, we obtain
$ \widetilde{\cC}^{\circ}_{\varsigma + {\bf 1}} \to
(\punt\times\cA)^{\circ}.
$ Further composing with $\mathrm{T}$ in Proposition
\ref{prop:target-add-marking}, we obtain a punctured map
$\tilde{f}_{\varsigma+{\bf 1}} \colon
\widetilde{\cC}^{\circ}_{\varsigma + {\bf 1}} \to \punt$ over
$\widetilde{\cC}_{\varsigma}$.
The fact that $\tilde{f}_{\varsigma+{\bf 1}}$ is an R-map follows from
\eqref{diag:add-R-marking} and Corollary \ref{cor:target-log-R}.
The $(n+1)$-st marking of $\tilde{f}_{\varsigma+{\bf 1}}$ is a unit
sector by \S \ref{sss:unit-sector}.
This defines the morphism
$\widetilde{\cC}_{\varsigma} \to \SH_{g, \varsigma + \bf{1}}$.

\smallskip
\noindent
{\bf Step 2: Removing the $(n+1)$-st marking.}
We begin to construct the inverse of $\widetilde{\cC}_{\varsigma} \to \SH_{g, \varsigma + \bf{1}}$.
Denote by $f_{\varsigma+{\bf 1}} \colon \pC_{\varsigma + {\bf 1}} \to \punt$ the universal map over $\SH_{\varsigma + \bf{1}}$.  We next remove $\bf{1}$ from the set of markings as follows.

Note that we  have  a commutative diagram of solid arrows
\[
\xymatrix{
&&& \pC_{\varsigma + {\bf 1}} \ar@/_1.5pc/[llld]_-{f_{\varsigma+{\bf 1}}} \ar@/^1.5pc/[ldd]  \ar@{-->}[ld]|-{\  f \ } \\
\punt \ar[d] && (\punt\times \cA)^{\circ} \ar[ll]_{\mathrm{T}} \ar[d]  & \\
\ainfty && (\ainfty\times\cA)^{\circ}  \ar[ll]_{\ft} &
}
\]
where the square is Cartesian by Proposition \ref{prop:target-add-marking}, and the curved arrow on the right is produced by Lemma \ref{lem:split-marking}. Hence we obtain the dashed arrow $f$. Since the composition
\[
\pC_{\varsigma + {\bf 1}} \stackrel{f}{\longrightarrow} (\infty\times\cA)^{\circ}  \stackrel{\mathrm{P}}{\longrightarrow}  \punt\times\cA \longrightarrow \punt
\]
has zero contact order along the $(n+1)$-st marking, there is a commutative diagram similar to \eqref{diag:remove-contact}
\begin{equation}\label{diag:R-map-remove-marking}
\xymatrix{
&& \pC_{\varsigma + {\bf 1}} \ar[rr] \ar[d]_{f} \ar[lld]_-{f_{\varsigma+{\bf 1}}} &&  \pC_{\varsigma + {1}} \ar[rr] \ar[d]^{} && \pC_{\varsigma + (1^{\rig})} \ar[d]^{f_{\varsigma+(1^{\rig})}}\\
\punt && (\punt\times\cA)^{\circ} \ar[rr]^{\mathrm{P}} \ar[ll]_{\mathrm{T}} && \punt\times\cA \ar[rr] && \punt
}
\end{equation}
where on the top, the left horizontal arrow is the puncturing along the $(n+1)$-st marking, and the right horizontal arrow is rigidifying and removing this section from the set of markings.  Denote by $p_{n+1} \subset   \pC_{\varsigma + {1}}$ the $(n+1)$-st marking, and $p^{\rig}_{n+1} \subset \pC_{\varsigma + (1)}$ the resulting non-marked section. Note that the log map $\pC_{\varsigma + {1}} \to \cA$ given by the middle vertical arrow is induced by the  log structure associated to $p_{n+1}$.

We check that $f_{\varsigma+(1^{\rig})}$ is an R-map. Indeed, by Corollary \ref{cor:target-log-R}, we have
\[
\big(f_{\varsigma+(1^{\rig})}^*\uomega\big)|_{ \cC_{\varsigma + {\bf 1}}} \otimes \cO_{ \cC_{\varsigma + {\bf 1}}}(r\cdot p_{n+1}) \cong f_{\varsigma+{\bf 1}}^*\uomega \cong \omega^{\log}_{ \cC_{\varsigma + {\bf 1}}/\SH_{\varsigma+{\bf 1}}}.
\]
Thus the R-structure $f_{\varsigma+(1^{\rig})}^*\uomega \cong \omega^{\log}_{ \cC_{\varsigma + (1)^{\rig}}/\SH_{\varsigma+{\bf 1}}}$ follows from
\[
\big(f_{\varsigma+(1^{\rig})}^*\uomega\big)|_{ \cC_{\varsigma + {\bf 1}}} = \omega^{\log}_{ \cC_{\varsigma + {\bf 1}}/\SH_{\varsigma+{\bf 1}}} \otimes \cO_{ \cC_{\varsigma + {\bf 1}}}(r\cdot p_{n+1}) \cong \omega^{\log}_{ \cC_{\varsigma + (1)^{\rig}}/\SH_{\varsigma+{\bf 1}}}|_{\cC_{\varsigma + {\bf 1}}}.
\]

\smallskip
\noindent
{\bf Step 3: Stabilization of $f_{\varsigma+(1^{\rig})}$.}
By \eqref{diag:R-map-remove-marking}, observe that
\begin{equation}\label{eq:identify-away-from-extra-marking}
\pC_{\varsigma + {\bf 1}}\setminus p_{n+1} \cong \pC_{\varsigma + (1^{\rig})}\setminus p^{\rig}_{n+1} \ \  \mbox{and} \ \
f_{\varsigma + {\bf 1}}|_{\pC_{\varsigma + {\bf 1}}\setminus p_{n+1}} = f_{\varsigma + (1^{\rig})}|_{\pC_{\varsigma + (1^{\rig})}\setminus p^{\rig}_{n+1}}.
\end{equation}
Since $f_{\varsigma + {\bf 1}}$ is stable, the unstable components of
$\pC_{\varsigma + (1^{\rig})}$ with respect to
$f_{\varsigma + (1^{\rig})}$ are rational bridges containing the
section $p^{\rig}_{n+1}$.
Denote by $\ul{C}_{\varsigma+(1^{\rig})}$ the underlying coarse curve
of $\pC_{\varsigma + (1^{\rig})}$.
Let
$\ul{C}_{\varsigma+(1^{\rig})} \to \overline{\ul{C}}_{\varsigma}$ be
the contraction of unstable components of
$f_{\varsigma + (1^{\rig})}$.
We then have a commutative diagram
\[
\xymatrix{
\ul{\cC}_{\varsigma + (1^{\rig})} \ar@/^1pc/[rrrrrd]^{\ul{f}_{\varsigma + (1^{\rig})}} \ar[rd] \ar@/_1.5pc/[rdd] &  &&  && \\
&  \ul{\punt}_{\ul{C}_{\varsigma+(1^{\rig})}} \ar[rr] \ar[d] && \ul{\punt}_{\overline{\ul{C}}_{\varsigma}} \ar[rr] \ar[d] && \ul{\punt} \ar[d] \\
&  \ul{C}_{\varsigma+(1^{\rig})} \ar[rr] \ar@/_1pc/[rrrr]_{\omega^{\log}}&& \overline{\ul{C}}_{\varsigma} \ar[rr]^{\omega^{\log}} && \BC
}
\]
with both squares Cartesian.
The composition
$\ul{\cC}_{\varsigma + (1^{\rig})} \to
\ul{\punt}_{\ul{C}_{\varsigma+(1^{\rig})}} \to
\ul{\punt}_{\overline{\ul{C}}_{\varsigma}}$ contracts precisely the
unstable rational bridges.
Hence there is a contraction of twisted curves
$\mathrm{St} \colon \ul{\cC}_{\varsigma + (1^{\rig})} \to
\overline{\ul{\cC}}_{\varsigma}$ such that the above composition
factors through a stable map
$\overline{\ul{\cC}}_{\varsigma} \to
\ul{\punt}_{\overline{\ul{C}}_{\varsigma}}$.
The composition
$\overline{\ul{\cC}}_{\varsigma} \to
\ul{\punt}_{\overline{\ul{C}}_{\varsigma}} \to \ul{\punt}$ is a stable
underlying R-map, denoted by $\overline{\ul f}_{\varsigma}$.

Consider
$\overline{{\cC}}^{\circ}_{\varsigma} =
(\overline{\ul{\cC}}_{\varsigma}, \mathrm{St}_{*} \cM_{\pC_{\varsigma
    + (1^{\rig})}})$.
By {\bf Step 3} in the proof of Proposition
\ref{prop:universal-adding-marking}, the projection
$\overline{{\cC}}^{\circ}_{\varsigma} \to \SH_{\varsigma + \bf{1}}$
defines a family of punctured curves together with a log map
$\overline{f}_{\varsigma} \colon \overline{{\cC}}^{\circ}_{\varsigma}
\to \punt$ with the underlying $\overline{\ul f}_{\varsigma}$.
Thus, we obtain $\overline{f}_{\varsigma}$ as the stabilization of
$f_{\varsigma+(1^{\rig})}$.
This defines the morphism
$\SH_{g, \varsigma + \bf{1}} \to \widetilde{\cC}_{\varsigma}$.
Tracing through the construction, we observe that this provides the
inverse to the morphism constructed in {\bf Step 1}, as needed.
\end{proof}

\subsection{Adding a unit sector with uniform maximal degeneracy}\label{ss:stable-PRmap-add-marking-max}
Recall the notations
$\SH^{\curlywedge}_{\varsigma}, \SH^{\curlywedge}_{\varsigma+ \mathbf{1}}, \fM^{\curlywedge}_{\varsigma'},$ and $\fM^{\curlywedge}_{\varsigma'+\mathbf{1}'}$
for the stacks of punctured maps with uniform maximal degeneracy in
$\SH_{\varsigma}, \SH_{\varsigma+ \mathbf{1}}, \fM_{\varsigma'},$ and  $\fM_{\varsigma'+\mathbf{1}'}$
respectively. The universal punctured  maps over $\SH^{\curlywedge}_{\varsigma}$ and $\fM^{\curlywedge}_{\varsigma'}$ are denoted by
$f^{\curlywedge}_{\varsigma} \colon \cC^{\circ,\curlywedge}_{\varsigma} \to \punt$ and $\mathfrak{f}^{\curlywedge}_{\varsigma'} \colon \fC^{\circ,\curlywedge}_{\varsigma'} \to \ainfty$ respectively. They are pull-backs of the corresponding universal maps over $\SH_{\varsigma}$ and $\fM_{\varsigma'}$.

The saturation morphisms $\widetilde{\cC}^{\curlywedge}_{\varsigma} \to \cC^{\circ,\curlywedge}_{\varsigma}$ and $\widetilde{\fC}^{\curlywedge}_{\varsigma'} \to \fC^{\circ,\curlywedge}_{\varsigma'}$ are again denoted by $\mathrm{S}$.  By the functoriality of saturation, the compositions
\[
\widetilde{\cC}^{\curlywedge}_{\varsigma} \to \cC^{\circ,\curlywedge}_{\varsigma} \to \cC^{\circ}_{\varsigma} \ \  \mbox{and} \ \  \widetilde{\fC}^{\curlywedge}_{\varsigma'} \to \fC^{\circ,\curlywedge}_{\varsigma'} \to \fC^{\circ}_{\varsigma'}
\]
factor through the corresponding saturations
\[
\widetilde{\cC}_{\varsigma} \to \cC^{\circ}_{\varsigma} \ \  \mbox{and} \  \  \widetilde{\fC}_{\varsigma'} \to \fC^{\circ}_{\varsigma'}.
\]
Thus, we may pull back the universal  maps over $\SH_{\varsigma +  \mathbf{1}}$ and $\fM_{\varsigma' + \mathbf{1}'}$ along
\[
\widetilde{\cC}^{\curlywedge}_{\varsigma} \to  \widetilde{\cC}_{\varsigma} \cong \SH_{\varsigma +  \mathbf{1}} \ \ \ \mbox{and} \ \ \  \widetilde{\fC}^{\curlywedge}_{\varsigma'} \to  \widetilde{\fC}_{\varsigma'} \to \fM_{\varsigma' + \mathbf{1}'},
\]
and obtain punctured maps
\[
\tilde{f}^{\curlywedge}_{\varsigma+{\bf 1}} \colon \widetilde{\cC}^{\circ,\curlywedge}_{\varsigma + {\bf 1}} \to \punt
 \  \ \ \mbox{and} \ \ \   \tilde{\mathfrak{f}}^{\curlywedge}_{\varsigma' + \mathbf{1}} \colon \widetilde{\fC}^{\circ,\curlywedge}_{\varsigma'+ \mathbf{1}} \to  \ainfty.
\]
over $\widetilde{\cC}^{\curlywedge}_{\varsigma}$ and $\widetilde{\fC}^{\curlywedge}_{\varsigma'}$ respectively. By Lemma \ref{lem:max-degeneracy-add-marking}, the punctured maps $\tilde{f}^{\curlywedge}_{\varsigma+{\bf 1}}$ and $\tilde{\mathfrak{f}}^{\curlywedge}_{\varsigma' + \mathbf{1}}$ have uniform maximal degeneracy, hence inducing tautological morphisms
\begin{equation}\label{equ:forgetful-with-umd}
\widetilde{\cC}^{\curlywedge}_{\varsigma} \to  \SH^{\curlywedge}_{\varsigma+ \mathbf{1}} \ \ \ \mbox{and} \ \ \  F^{\curlywedge}_+ \colon \widetilde{\fC}^{\curlywedge}_{\varsigma'} \to \fM^{\curlywedge}_{\varsigma'+\mathbf{1}'}
\end{equation}
fitting  in the Cartesian diagram with strict vertical arrows
\begin{equation}\label{diag:forgetful-with-umd}
\xymatrix{
\widetilde{\cC}^{\curlywedge}_{\varsigma} \ar[rr] \ar[d] &&  \widetilde{\fC}^{\curlywedge}_{\varsigma'} \ar[d]^{F^{\curlywedge}_{+}} \\
  \SH^{\curlywedge}_{\varsigma+ \mathbf{1}} \ar[rr] && \fM^{\curlywedge}_{\varsigma'+\mathbf{1}'}.
}
\end{equation}

\begin{proposition}\label{prop:add-unit-sector-umd}
Suppose all contact orders in $\varsigma$ are negative. Then

\begin{enumerate}
\item the following square is Cartesian with strict and \'etale vertical arrows
 \[
 \xymatrix{
 \widetilde{\fC}^{\curlywedge}_{\varsigma'} \ar[rr] \ar[d]_-{F^{\curlywedge}_{+}} &&  \widetilde{\fC}^{}_{\varsigma'} \ar[d]^-{F^{}_{+}}  \\
 \fM^{\curlywedge}_{\varsigma' + \mathbf{1}'}  \ar[rr] &&   \fM^{}_{\varsigma' + \mathbf{1}'}
 }
 \]
 \item The left vertical arrow in \eqref{diag:forgetful-with-umd} is an isomorphism.
\end{enumerate}
 \end{proposition}
\begin{proof}
By construction we have an isomorphism
\[
 \widetilde{\fC}^{\curlywedge}_{\varsigma'} \cong \fM^{\curlywedge}_{\varsigma'}\times_{\fM^{}_{\varsigma'}}\widetilde{\fC}^{}_{\varsigma'}
\]
where the fiber product is taken in the fs category. Then (1) translates to the fact that the natural morphism of fiber products in the fs category
\[
\fM^{\curlywedge}_{\varsigma'}\times_{\fM^{}_{\varsigma'}}\widetilde{\fC}^{}_{\varsigma'} \to  \fM^{\curlywedge}_{\varsigma' + \mathbf{1}'}\times_{\fM^{}_{\varsigma' + \mathbf{1}'}}\widetilde{\fC}^{}_{\varsigma'}
\]
is an isomorphism. This follows from the if and only if statement of Lemma \ref{lem:max-degeneracy-add-marking}. Note that the above two isomorphisms induce isomorphisms on the level of stacks of R-maps
\[
\widetilde{\cC}^{\curlywedge}_{\varsigma}  \cong \SH^{\curlywedge}_{\varsigma}\times_{\SH^{}_{\varsigma}}\widetilde{\cC}_{\varsigma} \cong  \SH^{\curlywedge}_{\varsigma + \mathbf{1}}\times_{\SH_{\varsigma + \mathbf{1}}}\widetilde{\cC}_{\varsigma}.
\]
By Proposition \ref{prop:stable-prmap-remove-unit}, the right hand side is $\SH^{\curlywedge}_{\varsigma + \mathbf{1}}$. This proves (2).
\end{proof}

\subsection{The fundamental class axiom of the canonical theory}
\label{ss:fundamental-class-axiom-canonical}

By Proposition \ref{prop:stable-prmap-remove-unit}, we have arrived at the following commutative diagram with both squares Cartesian
\begin{equation}\label{diag:add-unit-marking}
\xymatrix{
\SH_{\varsigma+ {\bf 1}} \ar[rr]^{\mathrm{S}} \ar[d] \ar@/_2pc/[dd] &&   \pC_{\varsigma} \ar[rr]^{\pi_{\varsigma}}  \ar[d] && \SH_{\varsigma} \ar[d] \\
  \widetilde{\fC}_{\varsigma'} \ar[rr]^{\mathrm{S}} \ar[d]^{F_{+}} && \fC^{\circ}_{\varsigma'} \ar[rr]^{\pi_{\varsigma'}} && \fM_{\varsigma'} \\
\fM_{\varsigma'+{\bf 1}'}
}
\end{equation}

\begin{proposition}\label{prop:canonical-fundamental-class-axiom}
$\mathrm{S}_*[\SH_{\varsigma+ {\bf 1}}]^{\vir} = \pi_{\varsigma}^*[\SH_{\varsigma}]^{\vir}$
\end{proposition}

\begin{proof}
Let $\EE_{\SH_{\varsigma} / \fM_{\varsigma'}}$ be the canonical perfect  obstruction theory of $\SH_{\varsigma} \to \fM_{\varsigma'}$ as in \eqref{eq:tan-POT}. It pulls back to a perfect obstruction theory of  $\pC_{\varsigma} \to \fC^{\circ}_{\varsigma'}$, inducing a virtual cycle $[\pC_{\varsigma}]^{\vir}$. Then we have $\pi_{\varsigma}^*[\SH_{\varsigma}]^{\vir} = [\pC_{\varsigma}]^{\vir}$. It remains  to verify that $\mathrm{S}_*[\SH_{\varsigma+ {\bf 1}}]^{\vir} = [\pC_{\varsigma}]^{\vir}$.

On the other hand, since $F_{+}$ is strict and \'etale, see \eqref{eq:universal-add-marking}, the canonical perfect obstruction theory $\EE_{\SH_{\varsigma+ {\bf 1}} / \fM_{\varsigma'+{\bf 1}'}}$ of $\SH_{\varsigma+ {\bf 1}} \to \fM_{\varsigma'+{\bf 1}'}$ defines a perfect obstruction theory of $\SH_{\varsigma+ {\bf 1}} \to \widetilde{\fC}_{\varsigma'}$, inducing the same virtual cycle $[\SH_{\varsigma+ {\bf 1}}]^{\vir}$. In what follows, we will show
\begin{equation}\label{equ:add-unit-canonical-POT}
\EE_{\SH_{\varsigma+ {\bf 1}} / \fM_{\varsigma'+{\bf 1}'}} \cong \EE_{\SH_{\varsigma} / \fM_{\varsigma'}}|_{\SH_{\varsigma+ {\bf 1}}}.
\end{equation}
Hence $\mathrm{S}_*[\SH_{\varsigma+ {\bf 1}}]^{\vir} = [\pC_{\varsigma}]^{\vir}$ follows from the virtual push-forward of \cite{Co06, Ma12}.

Note that  the isomorphism $\SH_{\varsigma+{\bf 1}} \cong \widetilde{\cC}_{\varsigma}$ in Proposition \ref{prop:stable-prmap-remove-unit} identifies the universal stable punctured R-map $f_{\varsigma + \mathbf{1}}  \colon \pC_{\varsigma + \bf{1}} \to  \punt$ over $\SH_{\varsigma+{\bf 1}}$ with the stable punctured R-map $\tilde{f}_{\varsigma+{\bf 1}} \colon \widetilde{\cC}^{\circ}_{\varsigma + {\bf 1}} \to \punt$ over $\widetilde{\cC}_{\varsigma}$. We thus obtain the following commutative  diagram
\begin{equation}\label{diag:relating-R-maps}
\xymatrix{
&& \pC_{\varsigma + \bf{1}}  \ar[lld]_{f_{\varsigma + \mathbf{1}}} \ar[d]^{{f}} \ar[rr] && \pC_{\varsigma, \widetilde{\cC}_{\varsigma}} \ar[d]^{f_{\varsigma,\widetilde{\cC}_{\varsigma}}} \\
\punt \ar[d] && (\punt\times\cA)^{\circ} \ar[d] \ar[ll]^{\mathrm{T}}  \ar[rr] && \punt \ar[d] \\
\ainfty\times\BC && (\ainfty \times \cA)^{\circ}\times \BC \ar[ll]  \ar[rr] && \ainfty\times\BC
}
\end{equation}
where the two bottom squares  are Cartesian with strict vertical arrows by Corollary \ref{cor:target-log-R}, and the middle and the top of the diagram  is given by \eqref{diag:R-map-remove-marking}. Thus, we have
\begin{equation}\label{eq:compare-pull-back-log-tangent}
f_{\varsigma+\mathbf{1}}^*\Omega^{\vee}_{\punt/\ainfty\times\BC} \cong f^*\Omega^{\vee}_{(\punt\times\cA)^{\circ}/(\ainfty\times\cA)^{\circ}\times\BC} \cong \left(f_{\varsigma,\widetilde{\cC}_{\varsigma}}^*\Omega^{\vee}_{\punt/\ainfty\times\BC}\right)|_{\pC_{\varsigma + \bf{1}}}.
\end{equation}
Pushing forward \eqref{eq:compare-pull-back-log-tangent} and applying the base change, we obtain
\[
\pi_{\varsigma + \mathbf{1},*} f_{\varsigma+\mathbf{1}}^*\Omega^{\vee}_{\punt/\ainfty\times\BC} \cong \left(\pi_{\varsigma,*}f_{\varsigma,\widetilde{\cC}_{\varsigma}}^*\Omega^{\vee}_{\punt/\ainfty\times\BC}\right)|_{\SH_{\varsigma + \bf{1}}}
\]
which is \eqref{equ:add-unit-canonical-POT} by \eqref{eq:tan-POT}. This completes the proof.
\end{proof}

\subsection{The fundamental class axiom of the reduced theory}
\label{ss:fundamental-class-axiom-reduced}

We now assume the assumptions needed for the reduced theory in Theorem
\ref{thm:red-POT}.
In particular, all contact orders in $\varsigma$ are negative.
By \eqref{diag:forgetful-with-umd} and Proposition
\ref{prop:add-unit-sector-umd}, Diagram \eqref{diag:add-unit-marking}
induces the following commutative diagram with both squares Cartesian
\begin{equation}\label{diag:add-unit-marking-max-deg}
\xymatrix{
\SH^{\curlywedge}_{\varsigma+ {\bf 1}} \ar[rr]^{\mathrm{S}} \ar[d] \ar@/_2pc/[dd] &&   \cC^{\circ,\curlywedge}_{\varsigma} \ar[rr]^{\pi_{\varsigma}}  \ar[d] && \SH^{\curlywedge}_{\varsigma} \ar[d] \\
  \widetilde{\fC}^{\curlywedge}_{\varsigma'} \ar[rr]^{\mathrm{S}} \ar[d]^{F^{\curlywedge}_{+}} && \fC^{\circ,\curlywedge}_{\varsigma'} \ar[rr]^{\pi_{\varsigma'}} && \fM^{\curlywedge}_{\varsigma'} \\
\fM^{\curlywedge}_{\varsigma'+{\bf 1}'},
}
\end{equation}
where we follow the notations in \S \ref{ss:stable-PRmap-add-marking-max}.

\begin{proposition}\label{prop:red-fundamental-class-axiom}
$\mathrm{S}_{*}[\SH^{\curlywedge}_{\varsigma+ {\bf 1}}]^{\red} = \pi_{\varsigma}^*[\SH^{\curlywedge}_{\varsigma}]^{\red}$.
\end{proposition}

\begin{proof}
Let $\EE^{\red}_{\SH^{\curlywedge}_{\varsigma} / \fM^{\curlywedge}_{\varsigma'}}$ (resp. $\EE^{\red}_{\SH^{\curlywedge}_{\varsigma + \mathbf{1}} / \fM^{\curlywedge}_{\varsigma' + \mathbf{1}'}}$) be the reduced perfect  obstruction theory of $\SH^{\curlywedge}_{\varsigma} \to \fM^{\curlywedge}_{\varsigma'}$ (resp. $\SH^{\curlywedge}_{\varsigma + \mathbf{1}} \to \fM^{\curlywedge}_{\varsigma' + \mathbf{1}'}$) as in Theorem \ref{thm:red-POT}. It pulls back to a perfect obstruction theory of  $\cC^{\circ,\curlywedge}_{\varsigma} \to \fC^{\circ,\curlywedge}_{\varsigma'}$, inducing a virtual cycle $[\cC^{\circ,\curlywedge}_{\varsigma}]^{\red}$ such that $\pi_{\varsigma}^*[\SH^{\curlywedge}_{\varsigma}]^{\red} = [\cC^{\circ,\curlywedge}_{\varsigma}]^{\red}$. It remains to verify that $\mathrm{S}_*[\SH^{\curlywedge}_{\varsigma+ {\bf 1}}]^{\red} = [\cC^{\circ,\curlywedge}_{\varsigma}]^{\red}$.

Since $F^{\curlywedge}_{+}$ is strict and \'etale, similar to Proposition \ref{prop:canonical-fundamental-class-axiom}, it suffices to show
\begin{equation}\label{equ:add-unit-red-POT}
\EE^{\red}_{\SH^{\curlywedge}_{\varsigma + \mathbf{1}} / \fC^{\circ,\curlywedge}_{\varsigma'}} := \EE^{\red}_{\SH^{\curlywedge}_{\varsigma + \mathbf{1}} / \fM^{\curlywedge}_{\varsigma' + \mathbf{1}'}} \cong \EE^{\red}_{\SH^{\curlywedge}_{\varsigma} / \fM^{\curlywedge}_{\varsigma'}}|_{\SH^{\curlywedge}_{\varsigma + \mathbf{1}}}.
\end{equation}
By \eqref{eq:red-POT}, to prove \eqref{equ:add-unit-red-POT} it suffices to verify
\begin{equation}\label{equ:add-unit-POT-boundary}
  \left(\EE_{\SH^{\curlywedge}_{\varsigma} / \fM^{\curlywedge}_{\varsigma'}} \to \FF_{\SH^{\curlywedge}_{\varsigma} / \fM^{\curlywedge}_{\varsigma'}} \right)|_{\SH^{\curlywedge}_{\varsigma + \mathbf{1}}}
  \cong \left(\EE_{\SH^{\curlywedge}_{\varsigma + \mathbf{1}} / \fM^{\curlywedge}_{\varsigma'+ \mathbf{1}'}} \to \FF_{\SH^{\curlywedge}_{\varsigma+ \mathbf{1}} / \fM^{\curlywedge}_{\varsigma'+ \mathbf{1}'}}\right).
\end{equation}
By \eqref{equ:add-unit-canonical-POT}, we have an isomorphism of the canonical perfect obstruction theories
\[
\EE_{\SH^{\curlywedge}_{\varsigma} / \fM^{\curlywedge}_{\varsigma'}}|_{\SH^{\curlywedge}_{\varsigma + \mathbf{1}}} \cong \EE_{\SH^{\curlywedge}_{\varsigma + \mathbf{1}} / \fM^{\curlywedge}_{\varsigma'+ \mathbf{1}'}}.
\]
By \eqref{eq:boundary-complex}, the boundary complexes $\FF_{\SH^{\curlywedge}_{\varsigma} / \fM^{\curlywedge}_{\varsigma'}}$ and $\FF_{\SH^{\curlywedge}_{\varsigma+ \mathbf{1}} / \fM^{\curlywedge}_{\varsigma'+ \mathbf{1}'}}$ are defined by the corresponding maximal degeneracies, which are coincide by Lemma \ref{lem:max-degeneracy-add-marking}.  Hence we have
\[
\FF_{\SH^{\curlywedge}_{\varsigma} / \fM^{\curlywedge}_{\varsigma'}}|_{\SH^{\curlywedge}_{\varsigma + \mathbf{1}}} \cong \FF_{\SH^{\curlywedge}_{\varsigma+ \mathbf{1}} / \fM^{\curlywedge}_{\varsigma'+ \mathbf{1}'}}.
\]
It remains to identify morphisms on both sides of  \eqref{equ:add-unit-POT-boundary}.  Consider the two R-maps as in \eqref{eq:bullet-universal-maps}, over $\UH_{\varsigma+ {\bf 1}}$ and $\UH_{\varsigma}$ respectively:
\[
{f}^{\curlywedge}_{\varsigma+{\bf 1},e} \colon {\cC}^{\circ,\curlywedge}_{\varsigma + {\bf 1}} \to \infty_{e,\circ} \ \ \ \mbox{and} \  \  \  {f}^{\curlywedge}_{\varsigma, e} \colon  \cC^{\circ,\curlywedge}_{\varsigma} \to \infty_{e,\circ}.
\]
where $e$ denotes the maximal degeneracy of both punctured maps by Lemma \ref{lem:max-degeneracy-add-marking}.
 Denote by ${f}^{\curlywedge}_{\varsigma, e, \widetilde{\cC}_{\varsigma}} \colon  \cC^{\circ,\curlywedge}_{\varsigma,\widetilde{\cC}_{\varsigma}} \to \infty_{e,\circ}$ the  pull-back of  ${f}^{\curlywedge}_{\varsigma, e}$ over $\widetilde{\cC}^{\curlywedge}_{\varsigma} \cong \UH_{\varsigma+ {\bf 1}}$. Consider the following diagram with horizontal  arrows given by  \eqref{eq:dW-over-curve}:
\begin{equation}\label{diag:compare-dW-forget-marking}
\xymatrix{
{f}^{\curlywedge,*}_{\varsigma+{\bf 1},e}\Omega^{\vee}_{\punt/\BC} \ar[rrr]^-{{f}^{\curlywedge,*}_{\varsigma+{\bf 1},e}\diff \tW_{\bullet}} \ar@{=}[d] &&&  \omega^{\log}_{{\cC}^{\circ,\curlywedge}_{\varsigma + {\bf 1}}/\SH^{\curlywedge}_{\varsigma + {\bf 1}}}\otimes {f}^{\curlywedge,*}_{\varsigma+{\bf 1},e}\cO(\ttwist\Delta_{\max,\bullet}) \\
\left( {f}^{\curlywedge,*}_{\varsigma,e}\Omega^{\vee}_{\punt/\BC}\right)|_{\widetilde{\cC}^{\circ,\curlywedge}_{\varsigma + {\bf 1}}} \ar[rrr]^-{\left({f}^{\curlywedge,*}_{\varsigma,e}\diff \tW_{\bullet}\right)|_{{\cC}^{\circ,\curlywedge}_{\varsigma + {\bf 1}}}} &&& \left(\omega^{\log}_{\cC^{\circ,\curlywedge}_{\varsigma,\widetilde{\cC}_{\varsigma}}/\widetilde{\cC}^{\curlywedge}_{\varsigma}}\otimes {f}^{\curlywedge,*}_{\varsigma,e}\cO(\ttwist\Delta_{\max,\bullet}) \right)|_{{\cC}^{\circ,\curlywedge}_{\varsigma + {\bf 1}}} \ar@{^{(}->}[u]
}
\end{equation}
where the left vertical ``$=$'' is \eqref{eq:compare-pull-back-log-tangent}, and the right vertical arrow is given by tensoring
\[
{f}^{\curlywedge,*}_{\varsigma+{\bf 1},e}\cO(\ttwist\Delta_{\max,\bullet}) \cong \left({f}^{\curlywedge,*}_{\varsigma,e}\cO(\ttwist\Delta_{\max,\bullet}) \right)|_{{\cC}^{\circ,\curlywedge}_{\varsigma + {\bf 1}}}
\]
from Lemma \ref{lem:max-degeneracy-add-marking} with the inclusion
\begin{equation}\label{eq:compare-omega-log-add-marking}
\omega^{\log}_{\cC^{\circ,\curlywedge}_{\varsigma,\widetilde{\cC}_{\varsigma}}/\widetilde{\cC}^{\curlywedge}_{\varsigma}}|_{{\cC}^{\circ,\curlywedge}_{\varsigma + {\bf 1}}} \cong \omega^{\log}_{{\cC}^{\circ,\curlywedge}_{\varsigma + {\bf 1}}/\SH^{\curlywedge}_{\varsigma + {\bf 1}}}(-r p_{n+1}) \hookrightarrow \omega^{\log}_{{\cC}^{\circ,\curlywedge}_{\varsigma + {\bf 1}}/\SH^{\curlywedge}_{\varsigma + {\bf 1}}}
\end{equation}
Since both sides of  \eqref{equ:add-unit-POT-boundary} is given  by \eqref{eq:relative-cosection}, it suffices to show  the commutativity of \eqref{diag:compare-dW-forget-marking}.

Indeed, the top arrow and the composition of the other three sides in \eqref{diag:compare-dW-forget-marking} yield two morphisms in
\[
\Hom_{{\cC}^{\circ,\curlywedge}_{\varsigma + {\bf 1}}}\left( {f}^{\curlywedge,*}_{\varsigma+{\bf 1},e}\Omega^{\vee}_{\punt/\BC}, \omega^{\log}_{{\cC}^{\circ,\curlywedge}_{\varsigma + {\bf 1}}/\SH^{\curlywedge}_{\varsigma + {\bf 1}}}\otimes {f}^{\curlywedge,*}_{\varsigma+{\bf 1},e}\cO(\ttwist\Delta_{\max,\bullet}) \right).
\]
It amounts to show that these two morphisms are identical. By \eqref{eq:identify-away-from-extra-marking} and
\eqref{diag:relating-R-maps}, away from the $(n+1)$-st marking, we
have a factorization on the left that induces a factorization on the
right by the construction of \eqref{eq:bullet-universal-maps}:
\[
\xymatrix{
{\cC}^{\circ,\curlywedge}_{\varsigma + {\bf 1}} \setminus p_{n+1} \ar[rr]^-{{f}^{\curlywedge}_{\varsigma+{\bf 1}}} \ar[d] && \punt && {\cC}^{\circ,\curlywedge}_{\varsigma + {\bf 1}} \setminus p_{n+1} \ar[rr]^-{{f}^{\curlywedge}_{\varsigma+{\bf 1},e}} \ar[d] && \infty_{e,\circ} \\
\cC^{\circ,\curlywedge}_{\varsigma,\widetilde{\cC}_{\varsigma}} \ar[rru]_{{f}^{\curlywedge}_{\varsigma}} &&&&  \cC^{\circ,\curlywedge}_{\varsigma,\widetilde{\cC}_{\varsigma}} \ar[rru]_{{f}^{\curlywedge}_{\varsigma,e}} &&
}
\]
Thus these two morphisms are identical over the open dense
$\cC^{\circ,\curlywedge}_{\varsigma,\widetilde{\cC}_{\varsigma}}
\setminus p_{n+1}$, hence are identical.
Indeed, note that along $p_{n+1}$ the morphism
${f}^{\curlywedge,*}_{\varsigma+{\bf 1},e}\diff \tW_{\bullet}$
vanishes along all markings including $p_{n+1}$, by Lemma
\ref{lem:factor-through-omega}.
Furthermore, the inclusion \eqref{eq:compare-omega-log-add-marking},
hence the composition of the left, lower and right arrows of
\eqref{diag:compare-dW-forget-marking} also vanishes along $p_{n+1}$.
This completes the proof.
\end{proof}

\subsection{Proof of the unit axiom}\label{ss:unit-no-psi-min}

We next prove Theorem~\ref{thm:remove-unit}
(3) for reduced virtual cycles.
The case of canonical virtual cycles is similar and is left to the
reader.
We will assume $2g - 2 + n > 0$.

First, consider the following commutative diagram:
\[
\xymatrix{
\SH^{\curlywedge}_{\varsigma+ {\bf 1}} \ar[d]_{\mathrm{S}}  \ar[rr] &&   \widetilde{\fC}^{\curlywedge}_{\varsigma'} \ar[d]^{\mathrm{S}} && \\
\cC^{\circ,\curlywedge}_{\varsigma} \ar[rr] \ar[d]_{\stab} && \fC^{\circ,\curlywedge}_{\varsigma'} \ar[d]^{\stab} && \\
\cZ \ar[rr] \ar[d]_{\pi_{\cZ}} \ar@{-->}@/_1pc/[rrrr]_{\ \ \ \ \ \mathrm{p}_{\cZ}} && \mathfrak{Z} \ar[rr] \ar[d] && \oM_{g,n+1} \ar[d]^{\pi} \\
\SH^{\curlywedge}_{\varsigma} \ar[rr] \ar@{-->}@/_1pc/[rrrr]_{\ \ \ \ \mathrm{p}_{\ddata}}&& \fM^{\curlywedge}_{\varsigma'} \ar[rr] &&  \oM_{g,n}
}
\]
where the two bottom squares are Cartesian squares both on the
logarithmic and underlying level, and the two dashed arrows are the
obvious compositions.
Thus the two arrows labeled by $\stab$ are obtained by stabilization
to the usual stable curves.
Furthermore, the four horizontal solid arrows on the left are all
strict.

Denote by
$\EE^{\red}_{\cC^{\circ,\curlywedge}_{\varsigma}/\fC^{\circ,\curlywedge}_{\varsigma'}
}$ and $\EE^{\red}_{\cZ/\mathfrak{Z}}$ the perfect obstruction
theories obtained by pulling back
$\EE^{\red}_{\SH^{\curlywedge}_{\varsigma} /
  \fM^{\curlywedge}_{\varsigma'}}$, leading to the corresponding
virtual cycles $[\cC^{\circ,\curlywedge}_{\varsigma}]^{\red}$ and
$[\cZ]^{\red}$.
By \eqref{equ:add-unit-red-POT}, they further pull back to the perfect
obstruction theory
$\EE^{\red}_{\SH^{\curlywedge}_{\varsigma + \mathbf{1}} /
  \fC^{\circ,\curlywedge}_{\varsigma'}}$ defining
$[\SH^{\curlywedge}_{\varsigma + \mathbf{1}}]^{\red}$.
Since the two arrows
\[
\xymatrix{
\widetilde{\fC}^{\curlywedge}_{\varsigma'} \ar[rr]^{\mathrm{S}} && \fC^{\circ,\curlywedge}_{\varsigma'} \ar[rr]^{\stab} && \mathfrak{Z}
}
\]
in the above diagram are proper and birational, the virtual
push-forward of \cite[Theorem 1.1]{HeWi21} and \cite[Theorem
5.0.1]{Co06} yields
\[
\stab_* \mathrm{S}_*  [\SH^{\curlywedge}_{\varsigma + \mathbf{1}}]^{\red} = [\cZ]^{\red} = \pi_{\cZ}^*[\SH^{\curlywedge}_{\varsigma}]^{\red}.
\]
Finally, we compute
\[
 \mathrm{p}_{\ddata+\mathbf{1},*} [\SH^{\curlywedge}_{\ddata+\mathbf{1}}]^{\red}
= \mathrm{p}_{\cZ,*}\stab_* \mathrm{S}_* [\SH^{\curlywedge}_{\varsigma+ {\bf 1}}]^{\red}
= \mathrm{p}_{\cZ,*}  \pi_{\cZ}^*[\SH^{\curlywedge}_{\varsigma}]^{\red}
= \pi^{*}\mathrm{p}_{\ddata,*}  [\SH^{\curlywedge}_{\ddata}]^{\red},
\]
as stated in Theorem \ref{thm:remove-unit} (3).

%%%%Forget unit sectors with uniform minimal degeneracy

\section{Axioms with \texorpdfstring{$\psi_{\min}$}{psi-min}-classes}
\label{sec:psimin-axioms}

In this section, we prove analogs of the axioms of punctured R-maps
from Section~\ref{sec:remove-unit} that also incorporate the
$\psi_{\min}$-classes from Section~\ref{ss:psi-min}.
Due to its relevance to the applications and to simplify notations, we
will restrict ourselves to the more subtle reduced theory, and will
leave the case of canonical theory to a future work.

\subsection{The statement}

\subsubsection{Setup of notations}\label{sss:nota-string-divisor}

Using the notation and assumptions from \S
\ref{ss:statement-remove-unit}, recall that $\RMm_{\bullet}$ and
$\SH^{\curlywedge}_{\bullet}$ are the stacks of stable log R-maps in
$\SH_{\bullet}$ of uniform extremal degeneracy and uniform maximal
degeneracy, respectively, where $\bullet$ denotes either $\varsigma$ or
$\varsigma + \mathbf{1}$.
We recall and introduce some classes involved in the equations in this
section below.

Consider the minimal
degeneracy $e_{\min,\bullet} \in \Gamma(\RMm_{\bullet}, \ocM_{\RMm_{\bullet}})$ and its associated line bundle $\cO(e_{\min,\bullet})$, leading to the tautological $\psi_{\min}$-classes as in \eqref{eq:psi-min-class}:
\[
\psi_{\min,\bullet} := c_{1}(\cO(-e_{\min,\bullet})).
\]

On the other hand, consider the universal $\mu_\rho$-gerbe
\[
\xymatrix{
\mathcal{I}_{\mu_\rho}\punt_{\bk} \ar[rr]^{\ev_{\cI}} \ar[d]_{\pi_{\cI}} && \punt_{\bk} \\
\overline{\mathcal{I}}_{\mu_\rho}\punt_{\bk}
}
\]
We define the {\em tautological DF-class}
\begin{equation}\label{eq:tautological-log-class}
\psi_{\DF,\rho} := \rho \cdot \pi_{\cI,*}\ev_{\cI}^*c_1\big( \cO(-\punt) \big),
\end{equation}
where $\cO(\punt)$ is the line bundle of the target log structure as
in \eqref{eq:l=1-log}.
By \cite[\S 2.1]{AGV08}, $\psi_{\DF,\rho}$ is a bivariant Chow
class. For convenience, we write
\[
\psi_{\DF} := \sum_\rho \psi_{\DF,\rho}
\]
as a bivariant Chow class over
$\ocI \punt := \sqcup_\rho \overline{\mathcal{I}}_{\mu_\rho}\punt_{\bk}$.

Note that the Chow (resp. cohomology) ring $CH^*(\punt_{\bk}) \cong CH^*(\xinfty)$
(resp. $H^*(\punt_{\bk}) \cong H^*(\xinfty)$) can be identified via
pull-back along the \'etale morphism $\punt_{\bk} \to \xinfty$.
Thus, for a class $\alpha \in CH^*(\xinfty)$ (resp.
$\alpha \in  H^*(\xinfty)$) we define the bivariant Chow
(resp. cohomology) classes
\begin{equation}\label{eq:evaluation-class}
  \alpha_{\overline{\mathcal{I}}_{\mu_\rho}} := \rho \cdot \pi_{\cI,*}\ev_{\cI}^*\alpha \ \ \ \mbox{and} \ \ \ \alpha_{\ocI} := \sum_\rho  \alpha_{\overline{\mathcal{I}}_{\mu_\rho}}.
\end{equation}

For the $j$th marking, consider the underlying evaluation morphism
\[
\ev_j \colon \SH_{\bullet} \to \ocI\punt_{\bk}.
\]
Abusing notation, we will use the same notation from the evaluation
maps from $\UH_\bullet$ and $\RMm_\bullet$.
By \S~\ref{sss:unit-sector}, as a unit sector, the $(n+1)$-st
evaluation $\ev_{n+1}$ factors through the component
$\xinfty \subset \ocI \punt_{\bk}$ parameterizing
$\mu_r$-gerbes. Furthermore, recall the tautological morphisms
\[
F^{\curlywedge} \colon \RMm_{\bullet} \to \SH^{\curlywedge}_{\bullet} \ \ \ \mbox{and} \ \ \  F_{\mathbf{1}} \colon \SH^{\curlywedge}_{\varsigma + \mathbf{1}} \cong \widetilde{\cC}^{\curlywedge}_{\varsigma} \to \SH^{\curlywedge}_{\varsigma}.
\]
from \eqref{diag:main-player} and Proposition \ref{prop:moduli-remove-unit}, and the the virtual push-forward from Proposition \ref{prop:birational-red-vcycle}
\[
F^{\curlywedge}_* [\RMm_{\bullet}]^{\red} = [\SH^{\curlywedge}_{\bullet}]^{\red}.
\]

%------------------------------------------------The unit axiom
\subsubsection{The string and divisor type equations with \texorpdfstring{$\psi_{\min}$}{psi-min}}

\begin{theorem}\label{thm:string-divisor}
  For any integer $k \geq 0$ and $D \in CH^1(\xinfty)$ we have
\begin{equation}\label{eq:string}
F_{\mathbf{1},*} \circ F^{\curlywedge}_* \left( \psi_{\min,\varsigma+\mathbf{1}}^{k} \cap [\RMm_{\varsigma+\mathbf{1}}]^{\red} \right) =  \sum_{j=1}^n \frac{|c_j|}{r_j}  \sum_{k' = 0}^{k-1} (\ev_j^*\psi_{\DF}^{k'}) \cap F^{\curlywedge}_{*} \left( \psi_{\min,\varsigma}^{k-1-k'} \cap [\RMm_{\varsigma}]^{\red}\right) ,
\end{equation}

\begin{multline}\label{eq:divisor}
F_{\mathbf{1},*} \circ F^{\curlywedge}_* \left( \psi_{\min,\varsigma+\mathbf{1}}^{k} \cdot \ev_{n+1}^*\big(D\big) \cap [\RMm_{\varsigma+\mathbf{1}}]^{\red} \right) =   \big(\int_{\beta}D \big) \cdot F^{\curlywedge}_* \left( \psi_{\min,\varsigma}^k \cap [\RMm_{\varsigma}]^{\red} \right) \\
 + \sum_{j=1}^n \frac{|c_j|}{r_j} \sum_{k' = 0}^{k-1}\ev_j^*\big(D_{\ocI} \cdot \psi_{\DF}^{k'} \big) \cap F^{\curlywedge}_* \left( \psi_{\min,\varsigma}^{k-1-k'} \cap [\RMm_{\varsigma}]^{\red} \right).
\end{multline}
Note that when $k=0$, the $k'$-sum in each of the above equations is
empty, hence zero.
\end{theorem}

Comparing to Gromov--Witten theory, Equation \eqref{eq:string} may be
viewed as an analog of the string equation and \eqref{eq:divisor} may
be viewed as an analog of the divisor equation.
These two equations may be conveniently organized as formal power
series in $t^{-1}$ for a formal parameter $t$ as follows, which will
be useful in the localization calculation of \cite{CJR23P}:

\begin{corollary}\label{cor:string-divisor}
Notations as in Theorem \ref{thm:string-divisor}, we have
  \begin{equation}\label{eq:string-series}
    F_{\mathbf{1},*} \circ F^{\curlywedge}_* \left( \frac{[\RMm_{\varsigma+\mathbf{1}}]^{\red}}{t - \psi_{\min,\varsigma+\mathbf{1}}} \right) =
    \sum_{j=1}^n \frac{|c_j|}{r_j} F^{\curlywedge}_{*} \left( \frac{[\RMm_{\varsigma}]^{\red}}{(t - \ev_j^*\psi_{\DF})(t - \psi_{\min,\varsigma})} \right) ,
  \end{equation}
  \begin{multline}\label{eq:divisor-series}
    F_{\mathbf{1},*} \circ F^{\curlywedge}_* \left( \ev_{n+1}^*\big(D\big) \cap \frac{[\RMm_{\varsigma+\mathbf{1}}]^{\red}}{t - \psi_{\min,\varsigma+\mathbf{1}}} \right) = \\
    \big(\int_{\beta}D \big) F^{\curlywedge}_{*} \left( \frac{[\RMm_{\varsigma}]^{\red}}{t - \psi_{\min,\varsigma}} \right)
    + \sum_{j=1}^n \frac{|c_j|}{r_j} F^{\curlywedge}_{*} \left( \ev_j^*\big(D_{\ocI}\big) \cap \frac{[\RMm_{\varsigma}]^{\red}}{(t - \ev_j^*\psi_{\DF})(t - \psi_{\min,\varsigma})} \right) ,
  \end{multline}
\end{corollary}

\begin{proof}
Observe that
\begin{equation}\label{eq:psi-min-log-series}
\begin{split}
& \frac{1}{(t - \psi_{\min,\varsigma})(t - \ev_j^*\psi_{\DF})} = \frac{1}{t^2} \cdot \frac{1}{(1 - \ev_j^*\psi_{\DF}/t) (1 - \psi_{\min,\varsigma}/t)} \\
=&  \frac{1}{t^2} \cdot \left( \sum_{k_1 = 0}^{\infty} \big(\ev_j^*\psi_{\DF}/t  \big)^{k_1} \right) \left(\sum_{k_2 = 0}^{\infty} \big(  \psi_{\min,\varsigma}/t \big)^{k_2} \right) \\
=& \frac{1}{t} \cdot  \sum_{k=1}^{\infty} \frac{1}{t^k} \sum_{k' = 0}^{k-1} \ev_j^*\psi_{\DF}^{k'} \cdot \psi_{\min,\varsigma}^{k-1-k'}.
\end{split}
\end{equation}
On the other hand, we compute
\[
 \frac{[\RMm_{\varsigma+\mathbf{1}}]^{\red}}{t - \psi_{\min,\varsigma+\mathbf{1}}} = \frac{1}{t} \cdot \frac{[\RMm_{\varsigma+\mathbf{1}}]^{\red}}{1 - \psi_{\min,\varsigma+\mathbf{1}}/t}
= \sum_{k=0}^{\infty} \frac{1}{t^{k+1}} \psi_{\min,\varsigma+\mathbf{1}}^{k} \cap [\RMm_{\varsigma+\mathbf{1}}]^{\red}.
\]
Thus \eqref{eq:string-series} follows from \eqref{eq:string} and \eqref{eq:psi-min-log-series}  by comparing the coefficients of $\frac{1}{t^{k+1}}$.

Equation \eqref{eq:divisor-series} is proved similarly using \eqref{eq:divisor}, \eqref{eq:psi-min-log-series} and the following
\[
 \ev_{n+1}^*\big(D\big) \cap\frac{[\RMm_{\varsigma+\mathbf{1}}]^{\red}}{t - \psi_{\min,\varsigma+\mathbf{1}}}
= \sum_{k=0}^{\infty} \frac{1}{t^{k+1}} \psi_{\min,\varsigma+\mathbf{1}}^{k} \cdot \ev_{n+1}^*\big(D\big) \cap [\RMm_{\varsigma+\mathbf{1}}]^{\red}.
\]
\end{proof}

We finally observe a particularly simple combination of the string and
divisor type equations, which looks surprisingly similar to the
dilaton equation in Gromov--Witten theory:
\begin{corollary}
  We have the identity
  \begin{equation*}
    F_{\mathbf{1},*} \circ F^{\curlywedge}_* \left( \frac{(t - \ev_{n + 1}^* \psi_{\DF}) \cap [\RMm_{\varsigma+\mathbf{1}}]^{\red}}{t - \psi_{\min,\varsigma+\mathbf{1}}} \right) = \frac dr (2g - 2 + n) F^{\curlywedge}_{*} \left( \frac{[\RMm_{\varsigma}]^{\red}}{t - \psi_{\min,\varsigma}} \right)
  \end{equation*}
\end{corollary}
\begin{proof}
  We compute using \eqref{eq:divisor-series}:
  \begin{align*}
    &F_{\mathbf{1},*} \circ F^{\curlywedge}_* \left( \frac{\ev_{n + 1}^* \psi_{\DF} \cap [\RMm_{\varsigma+\mathbf{1}}]^{\red}}{t - \psi_{\min,\varsigma+\mathbf{1}}} \right) \\
    =&  F_{\mathbf{1},*} \circ F^{\curlywedge}_* \left( \frac{\ev_{n + 1}^* \psi_{\DF}|_{\xinfty} \cap [\RMm_{\varsigma+\mathbf{1}}]^{\red}}{t - \psi_{\min,\varsigma+\mathbf{1}}} \right) \\
    =& \big(\int_{\beta}\psi_{\DF}|_{\xinfty} \big) F^{\curlywedge}_{*} \left( \frac{[\RMm_{\varsigma}]^{\red}}{t - \psi_{\min,\varsigma}} \right)
       + \sum_{j=1}^n \frac{|c_j|}{r_j} F^{\curlywedge}_{*} \left(\frac{\ev_j^* (\psi_{\DF}|_{\xinfty})_{\ocI} \cap [\RMm_{\varsigma}]^{\red}}{(t - \ev_j^*\psi_{\DF})(t - \psi_{\min,\varsigma})} \right) \\
    =& \big(\int_{\beta}\psi_{\DF}|_{\xinfty} \big) F^{\curlywedge}_{*} \left( \frac{[\RMm_{\varsigma}]^{\red}}{t - \psi_{\min,\varsigma}} \right)
    + \sum_{j=1}^n \frac{|c_j|}{r_j} F^{\curlywedge}_{*} \left(\frac{\ev_j^* \psi_{\DF} \cap [\RMm_{\varsigma}]^{\red}}{(t - \ev_j^*\psi_{\DF})(t - \psi_{\min,\varsigma})} \right)
  \end{align*}
  Then, we have
  \begin{align*}
    &F_{\mathbf{1},*} \circ F^{\curlywedge}_* \left( \frac{(t - \ev_{n + 1}^* \psi_{\DF}) \cap [\RMm_{\varsigma+\mathbf{1}}]^{\red}}{t - \psi_{\min,\varsigma+\mathbf{1}}} \right) \\
    = &\big(-\int_{\beta}\psi_{\DF}|_{\xinfty} \big) F^{\curlywedge}_{*} \left( \frac{[\RMm_{\varsigma}]^{\red}}{t - \psi_{\min,\varsigma}} \right)
    + \sum_{j=1}^n \frac{|c_j|}{r_j} F^{\curlywedge}_{*} \left(\frac{(t - \ev_j^* \psi_{\DF}) \cap [\RMm_{\varsigma}]^{\red}}{(t - \ev_j^*\psi_{\DF})(t - \psi_{\min,\varsigma})} \right) \\
    = &\left(-\int_{\beta}\psi_{\DF}|_{\xinfty} + \sum_{j = 1}^n \frac{|c_j|}{r_j} \right) F^{\curlywedge}_{*} \left( \frac{[\RMm_{\varsigma}]^{\red}}{t - \psi_{\min,\varsigma}} \right) \\
    =& \frac dr (2g - 2 + n) F^{\curlywedge}_{*} \left( \frac{[\RMm_{\varsigma}]^{\red}}{t - \psi_{\min,\varsigma}} \right),
  \end{align*}
  where we in addition apply \eqref{eq:string-series} in the first step, and Lemma~\ref{lem:balancing} and our target set-up in the last step.
\end{proof}

%------------------------------------------------The unit axiom
\subsubsection{The unit axiom with \texorpdfstring{$\psi_{\min}$}{psi-min}}

When $2g - 2 + n > 0$, the two horizontal arrows in
\eqref{diag:forget-unit-stable-curve} induce
\[
\mathrm{p}_{\ddata+\mathbf{1}} \colon \RMm_{\ddata + {\bf 1}} \stackrel{}{\longrightarrow} \oM_{g,n+1} \ \ \ \mbox{and} \ \ \  \mathrm{p}_{\ddata} \colon \RMm_{\ddata} \stackrel{}{\longrightarrow} \oM_{g,n}.
\]
Let $\delta_{j,n+1} \subset \oM_{g,n+1}$ be the boundary divisor where the $j$th and $(n+1)$-st marking collide.

\begin{theorem}\label{thm:unit-axiom-with-min}
  For any integer $k \geq 0$ and $D \in CH^1(\xinfty)$ we have
\begin{multline}\label{eq:unit-string}
\mathrm{p}_{\ddata+\mathbf{1},*} \left( \psi_{\min,\varsigma+\mathbf{1}}^{k} \cap [\RMm_{\varsigma+\mathbf{1}}]^{\red} \right) =   \pi^* \mathrm{p}_{\ddata,*} \left(\psi_{\min,\varsigma}^k \cap [\RMm_{\ddata}]^{\red} \right)  \\
 +  \sum_{j=1}^n \frac{|c_j|}{r_j} \delta_{j,n+1} \cap \pi^* \mathrm{p}_{\ddata,*} \left( \sum_{k' = 0}^{k-1}  \ev_j^*\big( \psi_{\DF}^{k'} \big) \cdot   \psi_{\min,\varsigma}^{k-1-k'} \cap [\RMm_{\ddata}]^{\red} \right)
\end{multline}
When $k=0$, the $k'$-sum in each of the above equations is
empty, hence zero.
\end{theorem}

Theorem~\ref{thm:remove-unit} can be viewed as the special case of
Theorem~\ref{thm:string-divisor} and \ref{thm:unit-axiom-with-min} for
$k = 0$.
The rest of this section is devoted to proving
Theorem~\ref{thm:string-divisor} and \ref{thm:unit-axiom-with-min}.

\subsection{The auxiliary stack \texorpdfstring{$\RMm_{\varsigma, \varsigma + \mathbf{1}}$}{}}

Compared to \S \ref{sec:remove-unit}, a main difficulty in the proof of Theorem \ref{thm:string-divisor} is
that there is no morphism from $\RMm_{\varsigma + \mathbf{1}}$ to
$\RMm_{\varsigma}$ as $F_{\mathbf{1}}$ does not preserve the uniform
minimal degeneracy.
To overcome this issue, we introduce an auxiliary stack
$\RMm_{\varsigma, \varsigma + \mathbf{1}}$ given by the following diagram of Cartesian squares in the fs category
\begin{equation}\label{diag:forget-unit+min}
\xymatrix{
\RMm_{\varsigma, \varsigma + \mathbf{1}} \ar[rr] \ar[dd]_{F^{\diamondsuit}_{\bf 1}} && \widetilde{\fC}^{\diamondsuit}_{\varsigma',\varsigma'+\mathbf{1}'} \ar[rr] \ar[d]_{\widetilde{\Psi}} && \widetilde{\fC}^{\curlywedge,\diamondsuit}_{\varsigma',\varsigma'+\mathbf{1}'} \ar[d] \ar[rr] && \UMm_{\varsigma' + \mathbf{1}'} \ar[d] \\
&&  \widetilde{\fC}^{\diamondsuit}_{\varsigma'} \ar[d] \ar[rr] && \widetilde{\fC}^{\curlywedge}_{\varsigma'} \ar[rr]^{F^{\curlywedge}_{+}} \ar[d]  && \fM^{\curlywedge}_{\varsigma' + \mathbf{1}'} \\
\RMm_{\varsigma} \ar[rr] &&  \UMm_{\varsigma'} \ar[rr] && \fM^{\curlywedge}_{\varsigma'} &&
}
\end{equation}
where $\widetilde{\fC}^{\diamondsuit}_{\varsigma'}$ is the saturation
of the universal punctured curve
${\fC}^{\circ,\diamondsuit}_{\varsigma'} \to \UMm_{\varsigma'}$.
Thus, viewed as a subcategory,
$\widetilde{\fC}^{\diamondsuit}_{\varsigma',\varsigma'+\mathbf{1}'}$
parameterizes punctured maps in $\UMm_{\varsigma' + \mathbf{1}'}$
whose images under $F_{\mathbf{1}}$ are in $\UMm_{\varsigma'}$.

On the other hand, we have a factorization
\begin{equation}\label{eq:auxiliary-factorization}
\xymatrix{
\RMm_{\varsigma+\mathbf{1}} \ar[rr] \ar@/_1pc/[rrrr] && \widetilde{\fC}^{\curlywedge,\diamondsuit}_{\varsigma',\varsigma'+\mathbf{1}'} \ar[rr] &&  \UMm_{\varsigma' + \mathbf{1}'}
}
\end{equation}
induced by the top right Cartesian square  of \eqref{diag:forget-unit+min} and
\[
\RMm_{\varsigma+\mathbf{1}} \to \SH^{\curlywedge}_{\varsigma+\mathbf{1}} \cong \widetilde{\cC}^{\curlywedge}_{\varsigma} \to \widetilde{\fC}^{\curlywedge}_{\varsigma'}.
\]
By Proposition \ref{prop:add-unit-sector-umd} and \eqref{diag:forget-unit+min}, the right horizontal arrow in \eqref{eq:auxiliary-factorization} is strict and \'etale, hence all  three arrows in \eqref{eq:auxiliary-factorization} are strict.
Furthermore, the  reduced perfect obstruction theory of $\RMm_{\varsigma+\mathbf{1}}$ defines a reduced perfect obstruction theory of the left arrow in \eqref{eq:auxiliary-factorization}, inducing the same reduced virtual cycle $[\RMm_{\varsigma+\mathbf{1}}]^{\red}$.

Combining \eqref{diag:forget-unit+min} and \eqref{eq:auxiliary-factorization}, we obtain a commutative diagram
\begin{equation}\label{diag:cartesian-auxiliary}
\xymatrix{
\RMm_{\varsigma, \varsigma + \mathbf{1}} \ar[rr]^{\Phi} \ar[d]  && \RMm_{\varsigma+\mathbf{1}} \ar[d] \ar[rr]  && \SH^{\curlywedge}_{\varsigma + \mathbf{1}} \ar[rr]^{F_{\mathbf{1}}} \ar[d] && \SH^{\curlywedge}_{\varsigma} \ar[d] \\
\widetilde{\fC}^{\diamondsuit}_{\varsigma',\varsigma'+\mathbf{1}'} \ar[rr] && \widetilde{\fC}^{\curlywedge,\diamondsuit}_{\varsigma',\varsigma'+\mathbf{1}'} \ar[d] \ar[rr] && \widetilde{\fC}^{\curlywedge}_{\varsigma'}  \ar[rr]  \ar[d] && \fM^{\curlywedge}_{\varsigma'} \\
&& \UMm_{\varsigma' + \mathbf{1}'}  \ar[rr] &&  \fM^{\curlywedge}_{\varsigma' + \mathbf{1}'}  &&
}
\end{equation}
where all squares are Cartesian  and  all vertical arrows are strict.
To summarize, we obtain
\begin{equation}\label{diag:compare-auxiliary}
\xymatrix{
\RMm_{\varsigma} \ar[d] && \RMm_{\varsigma, \varsigma + \mathbf{1}} \ar[ll]_{F^{\diamondsuit}_{\mathbf{1}}} \ar[rr]^{\Phi} \ar[d]&& \RMm_{\varsigma+\mathbf{1}} \ar[d] \\
\UMm_{\varsigma'} && \widetilde{\fC}^{\diamondsuit}_{\varsigma',\varsigma'+\mathbf{1}'} \ar[ll]_{{F}^{\diamondsuit}_{\mathbf{1}'}} \ar[rr]^{\tilde{\Phi}} && \widetilde{\fC}^{\curlywedge,\diamondsuit}_{\varsigma',\varsigma'+\mathbf{1}'}
}
\end{equation}
with both squares Cartesian and  all vertical arrows strict.
Thus viewed as a subcategory, $\RMm_{\varsigma, \varsigma + \mathbf{1}}$ parameterizes punctured R-maps in $\RMm_{\varsigma+\mathbf{1}}$, whose images under $F^{\diamondsuit}_{\mathbf{1}}$ are in $\RMm_{\varsigma}$. We obtain a commutative diagram with tautological arrows
\begin{equation}\label{eq:joint-push-forward}
\xymatrix{
\RMm_{\varsigma, \varsigma + \mathbf{1}} \ar[rr]^{\Phi} \ar[d]_{F^{\diamondsuit}_{\mathbf{1}}} && \RMm_{\varsigma+\mathbf{1}}   \ar[rr]^{F^{\curlywedge}} &&  \SH^{\curlywedge}_{\varsigma+\mathbf{1}} \ar[d]^{F_{\mathbf{1}}}\\
\RMm_{\varsigma} \ar[rrrr]^{F^{\curlywedge}} && && \SH^{\curlywedge}_{\varsigma}
}
\end{equation}

By \eqref{equ:add-unit-red-POT} and \eqref{diag:cartesian-auxiliary},
both the left and right vertical arrows of
\eqref{diag:compare-auxiliary} have the reduced perfect obstruction
theories pulled back from that of
$\SH^{\curlywedge}_{\varsigma} \to \fM^{\curlywedge}_{\varsigma'}$.
Hence they further pull back to the same reduced perfect obstruction
theory of the middle vertical arrow, and define the reduced virtual
cycle $[\RMm_{\varsigma, \varsigma + \mathbf{1}}]^{\red}$.
By \eqref{diag:forget-unit+min}, the arrow
$\widetilde{\fC}^{\diamondsuit}_{\varsigma',\varsigma'+\mathbf{1}'}
\to
\widetilde{\fC}^{\curlywedge,\diamondsuit}_{\varsigma',\varsigma'+\mathbf{1}'}$
is proper and birational, hence the virtual push-forward
\begin{equation}\label{eq:Mm-reduced-circle-push-forward}
\Phi_{*}[\RMm_{\varsigma, \varsigma + \mathbf{1}}]^{\red} = [\RMm_{\varsigma+\mathbf{1}}]^{\red}.
\end{equation}

Consider the two universal punctured maps over
$ \widetilde{\fC}^{\diamondsuit}_{\varsigma',\varsigma'+\mathbf{1}'}$
\[
\mathfrak{f}^a_{\varsigma'} \colon \fC^{\circ,a}_{\varsigma'} \to \ainfty \ \ \mbox{and} \ \  \mathfrak{f}^a_{\varsigma'+\mathbf{1}'} \colon \fC^{\circ,a}_{\varsigma'+\mathbf{1}'} \to \ainfty
\]
obtained by pulling back from $\UMm_{\varsigma'}$ and $\UMm_{\varsigma' + \mathbf{1}'}$ via the bottom arrows in \eqref{diag:compare-auxiliary}.  Similarly,  pulling back via $\Phi$ and $F^{\diamondsuit}_{\mathbf{1}}$,  we obtain two punctured R-maps over $\RMm_{\varsigma, \varsigma + \mathbf{1}}$ respectively
\[
f^a_{\varsigma + \mathbf{1}} \colon \cC^{\circ,a}_{\varsigma + \mathbf{1}} \to \infty \  \ \ \mbox{and} \ \ \ f^a_{\varsigma} \colon \cC^{\circ,a}_{\varsigma} \to \infty.
\]

Denote by $f^a_{\varsigma + \mathbf{1}} \colon \cC^{\circ,a}_{\varsigma + \mathbf{1}} \to \infty$ and $f^a_{\varsigma} \colon \cC^{\circ,a}_{\varsigma} \to \infty$ the universal punctured R-maps over $\RMm_{\varsigma, \varsigma + \mathbf{1}}$ obtained by pulling back via $\Phi$ and $F^{\diamondsuit}_{\mathbf{1}}$ respectively. The minimal degeneracies
\[
e_{\min,\varsigma+\mathbf{1}} \in \Gamma\left(\fM_{\varsigma' + \mathbf{1}}, \ocM_{\fM_{\varsigma' + \mathbf{1}'}}\right) \ \ \mbox{and} \ \ \ e_{\min,\varsigma} \in \Gamma\left(\fM_{\varsigma'}, \ocM_{\fM_{\varsigma'}}\right)
\]
pull back to the minimal degeneracies of $\mathfrak{f}^a_{\varsigma'+\mathbf{1}'}$ and $\mathfrak{f}^a_{\varsigma'}$
\[
e^a_{\min,\varsigma+\mathbf{1}} := \overline{\tilde{\Phi}}^{\flat}(e_{\min,\varsigma+\mathbf{1}}) \ \ \ \mbox{and} \ \ \ e^a_{\min,\varsigma} := \overline{\tilde{F}}^{\diamondsuit, \flat}_{\mathbf{1}}(e_{\min,\varsigma})
\]
respectively, with the corresponding line bundles
\[
\cO(e^a_{\min,\varsigma+\mathbf{1}}) = \tilde{\Phi}^*\cO(e_{\min,\varsigma+\mathbf{1}}) \ \ \ \mbox{and} \ \  \ \cO(e^a_{\min,\varsigma}) = \tilde{F}^{\diamondsuit,*}_{\mathbf{1}} \cO(e_{\min,\varsigma})
\]
As usual, by abuse of notations, we will use
\[
e^a_{\min,\varsigma+\mathbf{1}}, \ \ e^a_{\min,\varsigma}, \ \ \cO(e^a_{\min,\varsigma+\mathbf{1}}), \ \ \cO(e^a_{\min,\varsigma})
\]
to denote the corresponding pull-backs over $\RMm_{\varsigma,\varsigma+\mathbf{1}}$. By abuse of notations, denote by
\[
\psi_{\min, \varsigma} = c_1\left(\cO(-e^a_{\min,\varsigma}) \right) \ \ \ \mbox{and} \ \ \ \psi_{\min, \varsigma + \mathbf{1}} = c_1\left( \cO(-e^a_{\min,\varsigma+\mathbf{1}})\right)
\]

Applying \eqref{eq:Mm-reduced-circle-push-forward} and the projection
formula, the left hand sides of \eqref{eq:string} and
\eqref{eq:divisor} are push-forwards from
$\RMm_{\varsigma, \varsigma + \mathbf{1}}$:
\begin{equation}\label{eq:auxiliary-string}
\Phi_{*}\left( \psi_{\min,\varsigma+\mathbf{1}}^{k} \cap [\RMm_{\varsigma, \varsigma + \mathbf{1}}]^{\red} \right) = \psi_{\min,\varsigma+\mathbf{1}}^{k} \cap [\RMm_{\varsigma+\mathbf{1}}]^{\red}
\end{equation}
\begin{equation}\label{eq:auxiliary-divisor}
\Phi_{*}\left( \psi_{\min,\varsigma+\mathbf{1}}^{k}  \cdot \ev_{n+1}^*\big(D\big) \cap [\RMm_{\varsigma, \varsigma + \mathbf{1}}]^{\red} \right)  = \psi_{\min,\varsigma+\mathbf{1}}^{k} \cdot \ev_{n+1}^*\big(D\big) \cap [\RMm_{\varsigma+\mathbf{1}}]^{\red}.
\end{equation}

Thus, we can approach the computationt of \eqref{eq:string} and
\eqref{eq:divisor} by studying the push-forward along
$F^{\curlywedge}\circ F^{\diamondsuit}_{\mathbf{1}}$ as in
\eqref{eq:joint-push-forward}.

\subsection{Comparing minimal degeneracies\texorpdfstring{ over $\widetilde{\fC}^{\diamondsuit}_{\varsigma',\varsigma'+\mathbf{1}'}$}{}}
Recall the morphism
\[
\widetilde{\mathrm{St}}^{a} \colon \fC^{\circ,a}_{\varsigma'+\mathbf{1}'} \to \fC^{\circ,a}_{\varsigma'}
\]
given by Lemma \ref{lem:punctured-curve-add-marking} that removes the
$(n+1)$-st marking, and contracts the corresponding rational
components.
Away from the $(n+1)$-st marking the punctured map
$\mathfrak{f}^a_{\varsigma'+\mathbf{1}}$ is given by the composition
$\mathfrak{f}^a_{\varsigma'} \circ \widetilde{\mathrm{St}}^{a}$.
Thus, $e^a_{\min, \varsigma}$ is a (not necessarily minimal)
degeneracy of $\mathfrak{f}^a_{\varsigma'+\mathbf{1}'}$.
We then have
$e^a_{\min,\varsigma+\mathbf{1}} \poleq e^a_{\min, \varsigma}$, hence the global section
\begin{equation}\label{eq:min-difference}
{\delta}_{\min} := e^a_{\min, \varsigma} - e^a_{\min,\varsigma+\mathbf{1}} \in \Gamma\left(\widetilde{\fC}^{\diamondsuit}_{\varsigma',\varsigma'+\mathbf{1}'}, \ocM_{\widetilde{\fC}^{\diamondsuit}_{\varsigma',\varsigma'+\mathbf{1}'}} \right).
\end{equation}

\begin{lemma}\label{lem:non-zero-min-difference}
The fiber $\delta_{\min}|_{s}$ over a closed point  $s  \to \widetilde{\fC}^{\diamondsuit}_{\varsigma',\varsigma'+\mathbf{1}'}$ is nonzero if and only if the following conditions hold
\begin{enumerate}
 \item $\widetilde{\mathrm{St}}^{a}$ contracts a component $Z$ over $s$ with $p_{j,s}, p_{n+1,s} \in Z$ for some $j$th marking $p_j$ with $j \neq n+1$.

 \item $Z$ is the unique minimal component of $\mathfrak{f}^a_{\varsigma'+\mathbf{1}'}$ over $s$.
\end{enumerate}
\end{lemma}
\begin{proof}
  If $\widetilde{\mathrm{St}}^{a}$ contracts no components over $s$,
  then the above discussion implies that $\delta_{\min} = 0$.
  Suppose $Z$ is the contracted component over $s$.
  Assume that the minimal degeneracy of
  $\mathfrak{f}^a_{\varsigma'+\mathbf{1}'}$ over $s$ is achieved at
  another component $Z'$.
  Since $\widetilde{\mathrm{St}}^{a}$ is isomorphic on the generic
  point of $Z'$, the degeneracy of $Z'$ with respect to
  $\mathfrak{f}^a_{\varsigma'+\mathbf{1}'}$ is equal to the degeneracy
  of the image $\widetilde{\mathrm{St}}^{a}(Z')$ with respect to
  $\mathfrak{f}^a_{\varsigma'}$.
  This implies
  $e^a_{\min, \varsigma}|_s = e^a_{\min,\varsigma+\mathbf{1}}|_s$.
  Thus, only the $Z$ described in the statement are possible.
\end{proof}

Let $\Div({\delta}_{\min}) \subset \widetilde{\fC}^{\diamondsuit}_{\varsigma',\varsigma'+\mathbf{1}'}$ be the Cartier divisor associated to ${\delta}_{\min}$ as in \eqref{eq:global-log-divisor}.
By \eqref{eq:min-difference}, its class satisfies
\begin{equation}\label{eq:div-min-difference}
\Div({\delta}_{\min}) = - \psi_{\min,\varsigma} + \psi_{\min,\varsigma + \mathbf{1}}.
\end{equation}

Let $\Delta_j \subset \Div({\delta}_{\min})$ be the component whose
closed points satisfy the two conditions in
Lemma~\ref{lem:non-zero-min-difference} for a fixed $j$.
Observe that $\Delta_j \cap \Delta_{j'} = \emptyset$ for $j \neq j'$.
Extending $\delta_{\min}|_{\Delta_j}$ by zero, we obtain a global
section
$\delta_{j} \in
\Gamma\left(\widetilde{\fC}^{\diamondsuit}_{\varsigma',\varsigma'+\mathbf{1}'},
  \ocM_{\widetilde{\fC}^{\diamondsuit}_{\varsigma',\varsigma'+\mathbf{1}'}}
\right)$.
We can similarly define $\Div(\delta_j)$, and noting that
$\Delta_j = \Div(\delta_j)$.
This gives decompositions
\[
\delta_{\min} = \sum_{j} \delta_j \  \  \  \mbox{and} \ \ \ \Div(\delta_{\min}) = \sum_{j} \Delta_j.
\]
Combining this with \eqref{eq:div-min-difference}, we obtain
\begin{equation}\label{eq:compare-Lmin}
 \psi_{\min,\varsigma + \mathbf{1}} = \psi_{\min,\varsigma} + \sum_{j} \Delta_j.
\end{equation}

Let
$p_{j, \Delta_j} \subset
\fC^{\circ,a}_{\varsigma'+\mathbf{1}'}\times_{\widetilde{\fC}^{\diamondsuit}_{\varsigma',\varsigma'+\mathbf{1}'}}\Delta_j$
be the $j$th marking of the source curve over $\Delta_j$.
On the characteristic level, along $p_{j,\Delta_j}$ we have
\[
\bar{\mathfrak{f}}^a_{\varsigma'+\mathbf{1}'}(\delta) = e^a_{\min,\varsigma+\mathbf{1}} + c_j  \sigma_j.
\]
By \eqref{eq:compare-torsor}, we have
\[
(\bar{\mathfrak{f}}^a_{\varsigma'+\mathbf{1}'})|_{p_{j, \Delta_j}}^*\big(\cO(-\ainfty) \big)  = \cO_{p_{j, \Delta_j}}(-e^a_{\min,\varsigma+\mathbf{1}}|_{p_{j,\Delta_j}} - c_j p_{j,\Delta_j}).
\]
Observe that $p_{j,\Delta_j}$ is contained in a rational component
with $3$ special points so that $\cO_{p_{j,\Delta_j}}(p_{j,\Delta_j})$ becomes
  trivial.
Pushing forward along
$\pi_j \colon p_j \to
\widetilde{\fC}^{\diamondsuit}_{\varsigma',\varsigma'+\mathbf{1}'}$,
we obtain
\begin{equation}\label{eq:min-over-delta}
\mathfrak{\psi}_{\DF,r_j} \cdot \Delta_j = \psi_{\min,\varsigma+\mathbf{1}} \cdot \Delta_j
\end{equation}
where we introduce a universal analogue of \eqref{eq:tautological-log-class}:
\[
\mathfrak{\psi}_{\DF,r_j} := r_j\cdot \big(\pi_{j, *}(\bar{\mathfrak{f}}^a_{\varsigma'+\mathbf{1}'})|_{p_{j}}^*c_1(\cO(-\ainfty))\big).
\]

Repeatedly applying \eqref{eq:compare-Lmin} and \eqref{eq:min-over-delta}, we compute over $\widetilde{\fC}^{\diamondsuit}_{\varsigma',\varsigma'+\mathbf{1}'}$ that
\begin{equation}\label{eq:powers-of-min}
\begin{split}
\psi_{\min,\varsigma+\mathbf{1}}^{k} &=  \psi_{\min,\varsigma+\mathbf{1}}^{k-1} \cdot \Big( \psi_{\min,\varsigma} + \sum_{j} \Delta_j\Big) \\
&= \psi_{\min,\varsigma} \cdot \psi_{\min,\varsigma+\mathbf{1}}^{k-1} + \sum_{j} \psi_{\min,\varsigma+\mathbf{1}}^{k-1}\cdot\Delta_j  \\
&= \psi_{\min,\varsigma} \cdot \psi_{\min,\varsigma+\mathbf{1}}^{k-1} + \sum_{j} \mathfrak{\psi}_{\DF,r_j}^{k-1}\cdot\Delta_j \\
& \cdots \\
&= \psi_{\min,\varsigma}^k + \sum_{j=1}^n \sum_{k' = 0}^{k-1} \mathfrak{\psi}_{\DF,r_j}^{k'}\cdot   \psi_{\min,\varsigma}^{k-1-k'}\cdot \Delta_j.
\end{split}
\end{equation}

To compute the push-forward of $\Delta_j$ along $F^{\diamondsuit}_{\mathbf{1}'}$, recall from \eqref{diag:forget-unit+min} and \eqref{diag:compare-auxiliary} that  $F^{\diamondsuit}_{\mathbf{1}'}$ is given by the composition
\begin{equation}\label{eq:decompose-F1'}
 \widetilde{\fC}^{\diamondsuit}_{\varsigma',\varsigma'+\mathbf{1}'} \stackrel{\tilde{\Psi}}{\longrightarrow}  \widetilde{\fC}^{\diamondsuit}_{\varsigma'} \stackrel{\mathrm{S}}{\longrightarrow} \fC^{\circ,\diamondsuit}_{\varsigma'} \longrightarrow \UMm_{\varsigma'}.
\end{equation}
Denote by
$p^{\diamondsuit}_{\varsigma',j} \subset
\fC^{\circ,\diamondsuit}_{\varsigma'}$ the $j$th marking.
We compute the cycle-theoretic push-forward:

\begin{lemma}\label{lem:push-forward-Deltaj}
$\mathrm{S}_{*} \tilde{\Psi}_* \Delta_j = |c_j|\cdot p^{\diamondsuit}_{\varsigma',j}$.
\end{lemma}

\begin{proof}
Note that both $\tilde{\Psi}$ and $\mathrm{S}$ are projective and birational, hence the push-forwards are  well-defined.

We first show that the set-theoretical image
$\mathrm{S}\circ\tilde{\Psi}(\Delta_j)$ is precisely
$p^{\diamondsuit}_{\varsigma',j}$.
By Lemma~\ref{lem:non-zero-min-difference}, the set-theoretical image
$\mathrm{S}\circ\tilde{\Psi}(\Delta_j)$ is contained in
$p^{\diamondsuit}_{\varsigma',j}$.
On the other hand, let
$e_j \in \Gamma(\UMm_{\varsigma'}, \ocM_{\UMm_{\varsigma'}})$ be the
degeneracy of the component containing
$p^{\diamondsuit}_{\varsigma',j}$.
Let
$\sigma_j \in \Gamma(\fC^{\circ,\diamondsuit}_{\varsigma'},
\ocM_{\fC^{\circ,\diamondsuit}_{\varsigma'}})$ be the generator of the
characteristic sheaf of the log structure given by the marking
$p^{\diamondsuit}_{\varsigma',j}$.
Note that for a general point $s \in p^{\diamondsuit}_{\varsigma',j}$,
we have $e_j|_{s} = e_{\min,\varsigma'}|_{s}$.
Let $Z_s \subset \fC^{\circ,a}_{\varsigma'+\mathbf{1}'}$ be the
contracted component over $s$, then it has degeneracy
$e_{Z_s} = e_j|_{s} + c_j \cdot \sigma_j|_{s} \poleq
e_{\min,\varsigma}|_{s}$ since $c_j < 0$.
Thus $Z_s$ is the unique minimal component.
Therefore, $\Delta_j$ dominates, hence surjects onto
$p^{\diamondsuit}_{\varsigma',j}$.
Furthermore, denote by $\tilde{s} \in \Delta_j$ the generic point over
$s$.
The above discussion also implies that
\[
  \Delta_j|_{\tilde s} = \delta_{\min}|_{\tilde{s}} = e_{\min,\varsigma}|_{\tilde{s}} - e_{\min,\varsigma+\mathbf{1}}|_{\tilde{s}} = - c_{j}\cdot \sigma_j|_{\tilde{s}}.
\]
The statement then follows from $p^{\diamondsuit}_{\varsigma',j} = \Div(\sigma_j)$.
\end{proof}

\subsection{Proof of Theorem \ref{thm:string-divisor}}
Consider  the Cartesian diagram with strict vertical arrows
\[
\xymatrix{
\RMm_{\varsigma, \varsigma+\mathbf{1}} \ar[rr]^{\Psi} \ar[d]_{\widetilde\varphi} &&  \widetilde{\cC}^{\diamondsuit}_{\varsigma}  \ar[rr]^{\mathrm{S}} \ar[d] &&  \cC^{\circ,\diamondsuit}_{\varsigma} \ar[rr]^{\pi} \ar[d]_{\varphi} && \RMm_{\varsigma} \ar[d] \\
 \widetilde{\fC}^{\diamondsuit}_{\varsigma',\varsigma'+\mathbf{1}'} \ar[rr]^{\tilde{\Psi}} &&  \widetilde{\fC}^{\diamondsuit}_{\varsigma'} \ar[rr]^{\mathrm{S}} && \fC^{\circ,\diamondsuit}_{\varsigma'} \ar[rr] && \UMm_{\varsigma'}.
}
\]
where the bottom is \eqref{eq:decompose-F1'}, hence $F^{\diamondsuit}_{\mathbf{1}} =  \pi\circ \mathrm{S} \circ \Psi$.
We equip all the vertical arrows with the reduced perfect obstruction theories pulled back from $\RMm_{\varsigma} \to \UMm_{\varsigma'}$, hence the reduced virtual cycles $[\RMm_{\varsigma, \varsigma+\mathbf{1}}]^{\red}$, $[\widetilde{\cC}^{\diamondsuit}_{\varsigma} ]^{\red}$ and $[\cC^{\circ,\diamondsuit}_{\varsigma} ]^{\red}$, satisfying the virtual push-forward
\begin{equation}\label{eq:push-forward-Mm-red-cycle}
\mathrm{S}_* \Psi_* [\RMm_{\varsigma, \varsigma+\mathbf{1}}]^{\red}  = [\cC^{\circ,\diamondsuit}_{\varsigma} ]^{\red} = \pi^*[\RMm_{\varsigma}]^{\red}.
\end{equation}
In particular, we have
\begin{equation}\label{eq:push-red=0}
F^{\diamondsuit}_{\mathbf{1},*}[\RMm_{\varsigma, \varsigma+\mathbf{1}}]^{\red}  = 0.
\end{equation}

Further consider the inclusion of Cartier divisors
\[
\iota_{\Delta_j} \colon \Delta_j \hookrightarrow \widetilde{\fC}^{\diamondsuit}_{\varsigma',\varsigma'+\mathbf{1}'} \ \ \ \mbox{and} \ \ \ \iota_{p^{\diamondsuit}_{\varsigma',j}} \colon p^{\diamondsuit}_{\varsigma',j} \hookrightarrow \fC^{\circ,\diamondsuit}_{\varsigma'}.
\]
By the functoriality of virtual circles \cite[Proposition 7.5]{BeFa97}, we have
\[
\iota_{\Delta_j}^! [\RMm_{\varsigma, \varsigma + \mathbf{1}}]^{\red} = \widetilde{\varphi}^{!}_{\bE^{\red}}[\Delta_j] \ \ \mbox{and} \ \ \ \iota_{p^{\diamondsuit}_{\varsigma',j}}^! [\cC^{\circ,\diamondsuit}_{\varsigma} ]^{\red} = \varphi^!_{\bE^{\red}} [p^{\diamondsuit}_{\varsigma',j}]
\]
where $\widetilde{\varphi}^{!}_{\bE^{\red}}$ and
$\varphi^!_{\bE^{\red}}$ are the virtual pull-backs as in \cite{Ma12}.
By virtual push-forward \cite[Proposition 5.29]{Ma12},
Lemma~\ref{lem:push-forward-Deltaj} implies that
\begin{equation}\label{eq:push-forward-Mm-bd-red-cycle}
\mathrm{S}_* \Psi_* \left(\iota_{\Delta_j}^! [\RMm_{\varsigma, \varsigma + \mathbf{1}}]^{\red}\right)  = |c_j| \cdot \iota_{p^{\diamondsuit}_{\varsigma',j}}^! [\cC^{\circ,\diamondsuit}_{\varsigma} ]^{\red}.
\end{equation}

We now compute
\begin{equation}\label{eq:first-push-string}
\begin{split}
& F^{\diamondsuit}_{\mathbf{1},*}\Big(\psi_{\min,\varsigma+\mathbf{1}}^{k} \cap [\RMm_{\varsigma, \varsigma + \mathbf{1}}]^{\red}\Big) \\
&=  F^{\diamondsuit}_{\mathbf{1},*}\left(\psi_{\min,\varsigma}^k\cap [\RMm_{\varsigma, \varsigma + \mathbf{1}}]^{\red}\right) +  F^{\diamondsuit}_{\mathbf{1},*}\left(\sum_{j=1}^n \sum_{k' = 0}^{k-1} \ev_j^*\psi_{\DF}^{k'}\cdot \psi_{\min,\varsigma}^{k-1-k'}\cap  \iota_{\Delta_j}^! [\RMm_{\varsigma, \varsigma + \mathbf{1}}]^{\red}\right) \\
&=  F^{\diamondsuit}_{\mathbf{1},*}\left(\sum_{j=1}^n \sum_{k' = 0}^{k-1} \ev_j^*\psi_{\DF}^{k'}\cdot \psi_{\min,\varsigma}^{k-1-k'}\cap  \iota_{\Delta_j}^! [\RMm_{\varsigma, \varsigma + \mathbf{1}}]^{\red}\right) \\
&=  \sum_{j=1}^n \sum_{k' = 0}^{k-1} \ev_j^*\psi_{\DF}^{k'}\cdot \psi_{\min,\varsigma}^{k-1-k'}\cdot  F^{\diamondsuit}_{\mathbf{1},*} \Big(\iota_{\Delta_j}^! [\RMm_{\varsigma, \varsigma + \mathbf{1}}]^{\red}\Big) \\
&=  \sum_{j=1}^n \sum_{k' = 0}^{k-1} \ev_j^*\psi_{\DF}^{k'}\cdot \psi_{\min,\varsigma}^{k-1-k'}\cdot  \pi_*\Big(|c_j| \cdot  \iota_{p^{\diamondsuit}_{\varsigma',j}}^! [\cC^{\circ,\diamondsuit}_{\varsigma} ]^{\red}  \Big) \\
&=  \sum_{j=1}^n \frac{|c_j|}{r_j} \sum_{k' = 0}^{k-1} \ev_j^*\psi_{\DF}^{k'} \cdot \psi_{\min,\varsigma}^{k-1-k'}   \cap [\RMm_{\varsigma}]^{\red}
\end{split}
\end{equation}
where the first equality follows from \eqref{eq:powers-of-min}, the
second equality follows from \eqref{eq:push-red=0}, and the last two
follow from \eqref{eq:push-forward-Mm-bd-red-cycle}, and the fact that
$p^{\diamondsuit}_{\varsigma',j}$ is a $\mu_{r_j}$-gerbe.

By \eqref{eq:auxiliary-string} and the commutativity of \eqref{eq:joint-push-forward}, Equation \eqref{eq:string} follows from further pushing forward \eqref{eq:first-push-string} along $F^{\curlywedge}_*$ and applying the projection formula.

\bigskip

Now we turn to prove  Equation \eqref{eq:divisor}. By the definition of  unit sectors in \S \ref{sss:unit-sector}, the $(n+1)$-st evaluation $\ev_{n+1}$ factors through the component $\xinfty \subset \ocI \punt_{\bk}$ parameterizing $\mu_r$-gerbes. Indeed, we have the factorization
\[
\xymatrix{
\SH_{\varsigma+\mathbf{1}} \ar[r]^{\cong} \ar@/_2pc/[rrrrr]^{\ev_{n+1}} & \widetilde{\cC}_{\varsigma} \ar[rr]^{\mathrm{S}} && \pC_{\varsigma} \ar[rr]^{f_{\varsigma}} && \xinfty
}
\]
By the projection formula and \eqref{eq:push-forward-Mm-red-cycle}, we have
\begin{equation}\label{eq:push-red-cap-divisor}
F^{\diamondsuit}_{\mathbf{1},*}\left( \ev_{n+1}^*D \cap [\RMm_{\varsigma, \varsigma+\mathbf{1}}]^{\red} \right) = \left(\int_{\beta}D \right) \cdot [\RMm_{\varsigma}]^{\red}.
\end{equation}
By \cite[\S 2.1]{AGV08}, we observe
\begin{equation}\label{eq:pullback-D-change-marking}
\ev_{n+1}|_{\Delta_j \cap \RMm_{\varsigma, \varsigma + \mathbf{1}}}^*D = \ev_j|_{\Delta_j \cap \RMm_{\varsigma, \varsigma + \mathbf{1}}}^* D_{\ocI}.
\end{equation}
Similarly to \eqref{eq:first-push-string}, using \eqref{eq:push-red-cap-divisor} and \eqref{eq:pullback-D-change-marking} we compute
\begin{multline}\label{eq:first-push-divisor}
 F^{\diamondsuit}_{\mathbf{1},*}\Big(\psi_{\min,\varsigma+\mathbf{1}}^{k} \cdot \ev_{n+1}^*\big(D\big)  \cap [\RMm_{\varsigma, \varsigma + \mathbf{1}}]^{\red}\Big)
= \left(\int_{\beta}D \right) \cdot \psi_{\min,\varsigma}^k \cap [\RMm_{\varsigma}]^{\red} \\
 + \sum_{j=1}^n \frac{|c_j|}{r_j}  \sum_{k' = 0}^{k-1}  \ev_j^*\big(D_{\ocI} \cdot \psi_{\DF}^{k'} \big) \cdot  \psi_{\min,\varsigma}^{k-1-k'} \cap [\RMm_{\varsigma}]^{\red}
\end{multline}
By \eqref{eq:auxiliary-divisor} and the commutativity of \eqref{eq:joint-push-forward},  the divisor equation \eqref{eq:divisor}  follows from further pushing forward along $F^{\curlywedge}_*$ and applying the projection formula.

\subsection{Proof of Theorem \ref{thm:unit-axiom-with-min}}
Consider the following commutative diagram
\begin{equation}\label{diag:proof-unit-axiom}
\xymatrix{
\RMm_{\varsigma, \varsigma + \mathbf{1}} \ar@/^1.5pc/@{-->}[rrrr]^{\Phi_{\ddata}} \ar[r]_{\tilde{\Psi}} \ar[d]_{\widetilde\varphi} & \widetilde{\cC}^{\diamondsuit}_{\ddata} \ar[r]_{\mathrm{S}} \ar[d] & \cC^{\circ,\diamondsuit}_{\ddata} \ar[r]_{\stab} \ar[d]_{\varphi} & \cZ \ar@/^1pc/@{-->}[dd]^{\mathrm{p}_{\mathcal{Z}}}  \ar[d]_{\varphi_{\cZ}} \ar[r] & \RMm_{\ddata} \ar[d] \ar@/^1pc/@{-->}[dd]^{\mathrm{p}_{\ddata}}\\
\widetilde{\fC}^{\diamondsuit}_{\varsigma',\varsigma'+\mathbf{1}'} \ar[r]_{{\Psi}} & \widetilde{\fC}^{\diamondsuit}_{\varsigma'} \ar[r]_{\mathrm{S}} & \mathfrak{C}^{\circ,\diamondsuit}_{\ddata'} \ar[r]_{\stab} & \mathfrak{Z} \ar[r] \ar[d]_{\mathrm{p}_{\mathfrak{Z}}} & \UMm_{\ddata'} \ar[d] \\
&&& \oM_{g,n+1} \ar[r]_{\pi} &  \oM_{g,n}
}
\end{equation}
where all squares are Cartesian, the vertical arrows on the top are all strict,  the dashed arrows denote the corresponding compositions, and $\stab$ is the stabilization of the source curves as usual.

Denote by
$\EE^{\red}_{\cC^{\circ,\diamondsuit}_{\varsigma}/\fC^{\circ,\diamondsuit}_{\varsigma'}
}$ and $\EE^{\red}_{\cZ/\mathfrak{Z}}$ the perfect obstruction
theories obtained by pulling back
$\EE^{\red}_{\RMm_{\varsigma} /
  \UMm_{\varsigma'}}$, leading to the corresponding
virtual cycles $[\cC^{\circ,\diamondsuit}_{\varsigma}]^{\red}$ and
$[\cZ]^{\red}$. Since the arrows $\tilde{\Psi}$, $\mathrm{S}$ and $\stab$ in the middle row
of the above diagram are proper and birational, the virtual
push-forward of \cite[Theorem 1.1]{HeWi21} and \cite[Theorem~5.0.1]{Co06} yields
\begin{equation}\label{eq:unit-axiom-push-vc}
 \stab_* \mathrm{S}_*\widetilde\Psi_*  [\RMm_{\varsigma, \varsigma + \mathbf{1}}]^{\red} = [\cZ]^{\red} = \pi^*[\RMm_{\varsigma}]^{\red}.
\end{equation}

Let $\delta_{j,n+1} \subset \oM_{g,n+1}$ be the boundary divisor where the $j$th and $(n+1)$-st marking collide, which pulls back to a divisor $\iota_{\delta_{j,n+1, \mathfrak{Z}}} \colon \delta_{j,n+1, \mathfrak{Z}} := \mathrm{p}_{\mathfrak{Z}}^{-1}\delta_{j,n+1} \hookrightarrow \mathfrak{Z}$. Note  that
\[
\stab_*[p^{\diamondsuit}_{\varsigma',j}] = \frac{1}{r_j} \cdot [ \delta_{j,n+1, \mathfrak{Z}}].
\]
Further applying virtual push-forward to \eqref{eq:push-forward-Mm-bd-red-cycle} along $\stab$, we have
\begin{equation}\label{eq:push-forward-Mm-red-cycle-to-Z}
\begin{split}
\stab_*\Psi_* \mathrm{S}_*\left(\iota_{\Delta_j}^! [\RMm_{\varsigma, \varsigma + \mathbf{1}}]^{\red}\right) &= \stab_* \left(|c_j| \cdot \iota_{p^{\diamondsuit}_{\varsigma',j}}^! [\cC^{\circ,\diamondsuit}_{\varsigma} ]^{\red}. \right) \\
&= |c_j| \cdot  \stab_* \left( \varphi^!_{\bE^{\red}} [p^{\diamondsuit}_{\varsigma',j}] \right) \\
&=\frac{|c_j|}{r_j} \cdot \varphi_{\cZ,\bE^{\red}}^![\delta_{j,n+1, \mathfrak{Z}}] \\
&= \frac{|c_j|}{r_j} \cdot \iota_{\delta_{j,n+1, \mathfrak{Z}}}^![\cZ]^{\red}.
\end{split}
\end{equation}

Consider the following commutative diagram
\[
\xymatrix{
\RMm_{\varsigma, \varsigma + \mathbf{1}} \ar[rr]^{\Phi} \ar[d]_{F^{\diamondsuit}_{\mathbf{1}}}  && \RMm_{\varsigma+\mathbf{1}} \ar[rr]^{\mathrm{p}_{\ddata+\mathbf{1}}}   && \oM_{g,n+1} \ar[d]^{\pi} \\
\RMm_{\varsigma} \ar[rrrr]_{\mathrm{p}_{\ddata}}  &&&& \oM_{g,n}.
}
\]
such that $\mathrm{p}_{\ddata + \mathbf{1}}\circ \Phi = \mathrm{p}_{\cZ}\circ\stab\circ\mathrm{S}\circ\Psi$.
Thus \eqref{eq:unit-string} follows from \eqref{eq:auxiliary-string} and the following calculation similar to \eqref{eq:first-push-string}:

\[
\begin{split}
& \mathrm{p}_{\mathcal{Z},*}\stab_*\mathrm{S}_*\Psi_*\Big(\psi_{\min,\varsigma+\mathbf{1}}^{k}  \cap [\RMm_{\varsigma, \varsigma + \mathbf{1}}]^{\red}\Big) \\
=& \mathrm{p}_{\mathcal{Z},*} \stab_*\left(\psi_{\min,\varsigma}^k \cap [\cC^{\circ,\diamondsuit}_{\ddata}]^{\red}
 + \sum_{j=1}^n  \sum_{k' = 0}^{k-1}  \ev_j^*\big( \psi_{\DF}^{k'} \big) \cdot   \psi_{\min,\varsigma}^{k-1-k'} \cap |c_j| \cdot  \varphi^!_{\bE^{\red}} [p^{\diamondsuit}_{\varsigma',j}] \right)\\
=& \mathrm{p}_{\mathcal{Z},*} \left( \psi_{\min,\varsigma}^k \cap [\cZ]^{\red}\right) + \mathrm{p}_{\mathcal{Z},*} \left(\sum_{j=1}^n \frac{|c_j|}{r_j}  \sum_{k' = 0}^{k-1}  \ev_j^*\big( \psi_{\DF}^{k'} \big) \cdot   \psi_{\min,\varsigma}^{k-1-k'} \cap  \iota_{\delta_{j,n+1, \mathfrak{Z}}}^![\cZ]^{\red}  \right) \\
=& \pi^* \mathrm{p}_{\ddata,*} \left(\psi_{\min,\varsigma}^k \cap [\RMm_{\ddata}]^{\red} \right)
 +  \sum_{j=1}^n \frac{|c_j|}{r_j} \delta_{j,n+1} \cap \pi^* \mathrm{p}_{\ddata,*} \left( \sum_{k' = 0}^{k-1}  \ev_j^*\big( \psi_{\DF}^{k'} \big) \cdot   \psi_{\min,\varsigma}^{k-1-k'} \cap [\RMm_{\ddata}]^{\red} \right)
 \end{split}
\]
where \eqref{eq:unit-axiom-push-vc} and \eqref{eq:push-forward-Mm-red-cycle-to-Z} is used to compute the first and second equalities respectively.

This finishes the proof of Theorem \ref{thm:unit-axiom-with-min}.

%%%%%%%%%%%%%%%
% Effectiveness
%%%%%%%%%%%%%%%

\section{Effective invariants and cycles}
\label{sec:effective}

In this section, we will study the invariants of the reduced theory,
and introduce \emph{effective invariants}, which are an important part
of the various applications of log GLSM.

Let $\punt$ be a target as in \S \ref{ss:target}, where we assume in
addition that $\infty_\cX$ is smooth.
We fix the discrete data $(g, \beta, \ddata = \{c_i\}_{i=1}^n)$ as in
\eqref{eq:P-data} for punctured R-maps to $\punt$, and write as usual
that $\SH_{\varsigma} := \SH_{g, \varsigma}(\infty, \beta)$.
Since we only consider the reduced theory, in this section, we will
assume $c_i < 0$ for all $i$.
Recall the moduli $\RMm_{\varsigma}$ and its reduced virtual cycle $[\RMm_{\varsigma}]^{\red}$ from \S\ref{sss:diamond-PR-maps}
together with the class $\psi_{\min}$ from \S\ref{ss:psi-min}.

Cycles of the form
\begin{equation}\label{eq:interesting-cycle}
  \psi_{\min}^k \cdot \prod_{i = 1}^n \ev_i^*\alpha_i \cap [\RMm_{\varsigma}]^{\red},
\end{equation}
where $\alpha_i \in H^*(\infty_\cX)$ are cohomology or Chow cohomology classes, play an
important role in the localization formula for log GLSM \cite{CJR23P}.
In case \eqref{eq:interesting-cycle} is of degree zero, we obtain the
corresponding \emph{reduced punctured log GLSM invariant} or
\emph{effective invariant}
\begin{equation}\label{eq:interesting-inv}
\int_{[\RMm_{\varsigma}]^{\red}} \psi_{\min}^k \cdot \prod \ev_i^*\alpha_i := \deg \left(\psi_{\min}^k \cdot \prod \ev_i^*\alpha_i \cap [\RMm_{\varsigma}]^{\red}\right).
\end{equation}
More generally, we may consider the \emph{effective cycles (in
  cohomology)}
\begin{equation} \label{eq:effective-cycle}
  p_{\varsigma, *} \left( \psi_{\min}^k \cdot \prod \ev_i^*\alpha_i \cap [\RMm_{\varsigma}]^{\red} \right)\in H^*(\oM_{g, n}),
\end{equation}
where $p_{\varsigma, *}\colon \RMm_\varsigma \to \oM_{g, n}$ denotes
the forgetful morphism.
Alternatively, if instead $\alpha_i \in A^*(\xinfty)$ are Chow classes, we may also study the
analogous \emph{effective cycles in Chow}:
\begin{equation} \label{eq:effective-cycle-Chow}
  p_{\varsigma, *} \left( \psi_{\min}^k \cdot \prod \ev_i^*\alpha_i \cap [\RMm_{\varsigma}]^{\red} \right)\in A^*(\oM_{g, n})
\end{equation}
By multilinearity, we will assume that the $\alpha_i$ are homogeneous
for the rest of this section.

We will first further specialize our setup, and focus our attention on
the effective invariants and cycles for the Gromov--Witten theory of a
complete intersection in a smooth projective variety.
In particular, we will show that in the case of a quintic threefold,
the effective invariants are determined by
$1 + \lfloor (2g - 2)/5\rfloor$ ``basic effective invariants''.
Later, in Section~\ref{ss:setup-higher-spin}, we return to the general
setup, discuss applications to the LG/CY-correspondence and to the
study of Witten's $r$-spin class.
In Section~\ref{ss:tautological}, we conclude with a relation to the
tautological ring.

\begin{remark}
  Effective invariants and cycles concern integrals over the reduced
  virtual cycle.
  We defer the discussion of integrals over the canonical virtual
  cycle to a future research.
  They are expected to be related to (twisted) double ramification
  cycles (see \cite{BHPSS20P}), and to be easier to compute.
\end{remark}

\subsection{The setup for complete intersections}
\label{ss:effective-setup}

We now consider the set-up in \S \ref{sss:input} for a hybrid log GLSM
target $\fP$ where we further assume that
\begin{enumerate}
\item $\cX$ is a smooth projective variety.
\item $\bE = \bE_1$
\item $r =d = \ell = 1$
\item $\lbspin \cong \cO$ is the trivial line bundle.
\item A section $s_\cX \in H^0(\bE)$ defines a
  smooth complete intersection $\cZ = (s_\cX = 0) \subset \cX$ of
  codimension $\rk \bE$.
\end{enumerate}

By Proposition~\ref{prop:target-comparison}, the above set-up leads to
\[
\infty_\cX \cong \PP(\bE^\vee) \ \ \ \mbox{and} \ \ \ \lblog = \cO_{\xinfty}(1).
\]
Furthermore, the line bundle $\lblog$ on $\xinfty$ is ample
(resp. nef) if and only if $\bE$ is ample (resp. nef).
By \S \ref{ss:GLSM-superpotential-boundary}, the superpotential $W$ is just a section of $\lblog$.
Explicitly, it is given by pairing the section
$s_\cX|_{\infty_\cX} \in H^0(\bE)$ with the tautological section
$p \in H^0(\bE^\vee \otimes \lblog)$.
Finally, let $\beta_{\cX}$ be the push-forward of the curve class
$\beta \in H_2(\xinfty)$ along $\xinfty \to \cX$.

Note that the assumptions imply that the source curves in the relevant
moduli spaces of stable log R-maps to $\fP$ and punctured R-maps to
$\infty$ have no orbifold structure.

\begin{remark}
As explained in \cite[Section~4.1]{CJR21}, log GLSM for this set-up
(with a superpotential determined from $s_\cX$) recovers the
Gromov--Witten invariants of $\cZ$ with insertions from $\cX$.
It will be shown in the subsequent paper \cite{CJR23P} that these Gromov--Witten invariants
of the complete intersection $\cZ$ are determined by the (twisted)
Gromov--Witten invariants of $\cX$ as in \cite{CoGi07}, as well as the
effective invariants \eqref{eq:interesting-inv} in the above setting with the
superpotential induced by $s_{\cX}$ as in \S \ref{ss:GLSM-potential}.%
\footnote{This determination can be further lifted to the level of Gromov--Witten
classes of $\cZ$.}

In applications of log GLSM, (twisted) Gromov--Witten invariants of
the ambient space $\cX$ are usually well-understood.
Thus, in such a situation, the computation of the effective invariants
is a main obstacle to computing the Gromov--Witten invariants of the
complete intersection, and any method of computing them will have important
applications.
\end{remark}

In the following sections, we content ourselves by using the string
\eqref{eq:string} and divisor \eqref{eq:divisor} equations, as well as
dimensional constraints (see \S \ref{ss:dimension}) and a balancing
condition (see \S \ref{ss:balancing}) to reduce the computation of the
effective invariants to a smaller set of invariants.
For instance, for the application to the Gromov--Witten theory of
quintic threefolds, in each genus $g$, there are
$1 + \lfloor (2g - 2)/5\rfloor$ such \emph{basic effective
  invariants}, see Example~\ref{ex:quintic-effective-inv}.
This reduction of the number of invariants is sufficient for
computations in low genus \cite{GJR18P} and to prove structural
results, such as in \cite{GJR18P}.

In another direction, for applications to Gromov--Witten classes,
cohomological field theories and the tautological ring, we will also
study effective cycles \eqref{eq:effective-cycle},
\eqref{eq:effective-cycle-Chow}.
We will also reduce the number of cycles that need to be considered in
this cycle-valued situation.

\subsection{Numerical constraints}

\subsubsection{The balancing condition}
\label{ss:balancing}
Suppose $\SH_{\varsigma} \neq \emptyset$. Then we have a punctured
R-map $f\colon \pC \to \punt$ in $\SH_{\varsigma}$.
By Lemma~\ref{lem:balancing} and \eqref{eq:l=1-log}, we have
\[
\sum_{i=1}^n c_i = \deg f^*\cO(\punt) = \int_{\beta} c_1(\lblog) - (2g-2 + n).
\]
Thus $\SH_{\varsigma}$ is empty unless  the \emph{balancing condition} holds:
\begin{equation}\label{eq:balancing-condition}
0 \geq \sum_{i=1}^n (c_i + 1) = \int_{\beta}c_1(\lblog) - (2g-2).
\end{equation}
where equality holds iff $c_i = -1$ for all $i$.

\subsubsection{The reduced virtual dimension}
\label{ss:dimension}
Note that $\punt = \xinfty \times \BC$ by the setup in \S \ref{ss:effective-setup}.
By \eqref{eq:base-dimension} and Lemma \ref{lem:punt-tangent}, we compute the reduced virtual dimension as follows
\begin{equation}\label{eq:red-vir-dim}
\begin{split}
\red \dim \SH_{\varsigma} &= \chi(f^*\Omega^{\vee}_{\punt/\BC}) + 1  +  \dim \fM_{g,\ddata'}(\infty_{\cA}) \\
&= \int_{\beta} c_1(T_{\xinfty}) + (\dim \xinfty + 1) (1-g) + 1 + 3g-4 + n \\
&= (2 - \dim \xinfty) (g-1) - \int_{\beta} c_1(K_{\xinfty}) + n,
\end{split}
\end{equation}
where $K_{\bullet}$ stands for the canonical bundle.
The Euler sequence of the projective bundle $\xinfty \to \cX$ implies that
\[
K_{\xinfty} = \cO_{\xinfty}(-\rk \bE) \otimes K_{\cX} \otimes \det \bE.
\]
So, we compute
\begin{equation}\label{eq:deg-K-xinfty}
\begin{split}
\int_{\beta} c_1(K_{\xinfty}) &= \int_{\beta} c_1\Big(\cO_{\xinfty}(-\rk \bE) \otimes K_{\cX} \otimes \det \bE \Big) \\
&= -\rk \bE  \cdot \int_{\beta}c_1\Big(\cO_{\xinfty}(1)\Big) + \int_{\beta_{\cX}} c_1\Big(K_{\cX} \otimes \det \bE \Big) \\
&= - \rk \bE \cdot \Big(2g-2 + \sum_i (c_i + 1) \Big) + \int_{\beta_{\cX}} c_1\Big(K_{\cX} \otimes \det \bE \Big)
\end{split}
\end{equation}
Combining \eqref{eq:red-vir-dim} and \eqref{eq:deg-K-xinfty}, we have
\begin{equation}\label{eq:red-vir-dim-original}
\red \dim \SH_{\varsigma}  = (3- \dim \cX + \rk \bE)(g-1) - \int_{\beta_{\cX}} c_1(K_{\cX}\otimes \det \bE) + n + \rk \bE \cdot \sum_i (c_i + 1).
\end{equation}
Note that if we assume that there is a curve class $\beta_\cZ$ on
$\cZ$ pushing forward to $\beta_\cX$, then the moduli of stable maps
to $\cZ$ has virtual dimension
\[
\vir\dim \oM_{g,n}\Big(\cZ, \beta_\cZ\Big) = (3- \dim \cX + \rk \bE)(g-1) - \int_{\beta_{\cX}} c_1(K_{\cX}\otimes \det \bE) + n.
\]
The above calculation thus yields the more memorable formula
\begin{equation}\label{eq:red-vir-dim-X}
\red \dim \SH_{\varsigma}  = \vir\dim \oM_{g,n}\Big(\cZ, \beta_\cX\Big) + \rk \bE \cdot \sum_i (c_i + 1).
\end{equation}
In particular, since $c_i \le -1$ for all $i$, we observe that
\begin{corollary}
  \label{cor:general-type-vanishing}
  If $\vir\dim \oM_{g,n}\Big(\cZ, \beta_\cX\Big) < 0$, then
  $[\RMm_{\varsigma}]^{\red}  = 0$.
\end{corollary}

\subsection{Determining \texorpdfstring{$g=0,1$}{g=0,1} cases}\label{ss:determine-g=0-1}
The following is an immediate consequence of the balancing condition \eqref{eq:balancing-condition}:
\begin{lemma}\label{lem:g=1-effectiveness}
  In the above situation, we have:
  \begin{enumerate}
  \item If $g=0$ and $\lblog$ is nef, then $\SH_{\varsigma} = \emptyset$.
  \item If $g=1$ and $\lblog$ is ample, then $\SH_{\varsigma} = \emptyset$ unless $\beta = 0$ and $c_i = -1$ for all $i$.
  \end{enumerate}
\end{lemma}

\begin{remark}
  If $\lblog$ is nef but not ample, while
  $\SH_{\varsigma} = \emptyset$ when $g = 0$, even for $g = 1$, there
  may be infinitely many $\beta$ such that
  $\SH_\varsigma \neq \emptyset$.
  If $\lblog$ is not nef, even for $g = 0$, there may be infinitely
  many $\beta$ such that $\SH_\varsigma \neq \emptyset$.
\end{remark}

In the case of the above lemma, there are thus no nonzero effective
invariants/cycles in genus zero.

Next, we consider the genus one case assuming that $\lblog$ is ample
and $\SH_\varsigma \neq \emptyset$.
In this case, the reduced virtual dimension
\eqref{eq:red-vir-dim-original} drastically simplifies to
\begin{equation*}
  \red \dim \SH_{\varsigma}  = n.
\end{equation*}
Since $\beta = 0$ and $\infty_\cX$ is a variety, the evaluation maps are equal to each
other, and hence we have an equality of effective cycles in either homology or in Chow:
\begin{equation}\label{eq:genus-one-effective-cycle}
  \psi_{\min}^k \cdot \prod \ev_i^*\alpha_i \cap [\RMm_{\varsigma}]^{\red}
  = \psi_{\min}^k \cdot \ev_1^*(\prod \alpha_i) \cap [\RMm_{\varsigma}]^{\red}
\end{equation}
Repeatedly applying the string equation \eqref{eq:string}, the
genus-one effective invariants are determined explicitly by the
$g=n=1$ case.
Furthermore, for effective Chow cycles \eqref{eq:effective-cycle-Chow}
in the genus one case, we use the string equation \eqref{eq:string}
and the unit axiom Theorem~\ref{thm:remove-unit} (3) to reduce to the
case $g=n = 1$.
Next, we find in this special case an explicit computation of the
classes $[\RMm_{\varsigma}]^{\red}$ in
Proposition~\ref{prop:g1-reduced-virtual-cycle} and $\psi_{\min}$ in
\eqref{eq:g1-psimin} below, hence completely determine the genus-one
theory.

\subsection{Calculations in the \texorpdfstring{$g=n=1$, $\beta = 0$}{g=n=1, beta=0} case}\label{ss:g=1-calculation}

In this subsection, we consider the important special case $g = n = 1$
and $\beta = 0$.
Note that in the set-up of Section~\ref{ss:effective-setup}, we have
$\ul{\punt} = \xinfty\times\BC$, and the source curves under
consideration have no orbifold structures.
Thus, the unique marking is a trivial gerbe with contact order $-1$.
The stack of genus one stable punctured R-maps under consideration is
then denoted by $\SH_{1,1}(\punt,0)$.

\subsubsection{The construction of \texorpdfstring{$\SH_{1,1}(\punt,0)$}{R_1,1}}\label{sss:genus-one-moduli}
Let $\overline{\scrM}_{1,1}$ be the stack of genus one, one-pointed
stable curves, and $\overline\scrM_{1,1}(\xinfty,0)$
be the corresponding stack of stable maps to $\xinfty$.
There is a tautological morphism that takes the corresponding stable
maps to $\xinfty$:
\begin{equation}\label{eq:P-to-stable-map}
\SH_{1,1}(\punt,0) \to \overline\scrM_{1,1}(\xinfty,0) \cong \overline{\scrM}_{1,1}\times \xinfty.
\end{equation}
By the stability condition, this morphism is defined by composing stable punctured $R$-maps with the projection $\punt \to \xinfty$, and removing all log structures. In particular, source curves of stable punctured maps in $\SH_{1,\ddata}(\punt,0)$ are automatically stable, hence irreducible.

Let $\hodge$ be the Hodge bundle over $\overline{\scrM}_{1,1}$.
The line bundle $\hodge^\vee\boxtimes \lblog$ and its zero section defines
a Deligne--Faltings log structure of rank one on
$\overline\scrM_{1,1}(\xinfty,0)$ as in \eqref{eq:generic-rank1-DF}, denoted by
$\cM^d$. Denote by $\cM_{\overline{\scrM}_{1,1}}$ the canonical log
structure of stable curves over $\overline{\scrM}_{1,1}$.

\begin{proposition}\label{prop:genus-1-stack}
The tautological morphism \eqref{eq:P-to-stable-map} induces an isomorphism of the underlying stacks. Furthermore, on  the level of log structures we have
\begin{equation}\label{eq:genus-1-moduli-log}
\cM_{\SH_{1,1}(\punt,0)} = \cM_{\overline{\scrM}_{1,1}}|_{\SH_{1,1}(\punt,0)} \oplus_{\cO^*} \cM^d,
\end{equation}
where $\ocM^d = \NN_{\SH_{1,1}(\punt,0)}$ is identified with the
constant sheaf generated by the degeneracy of the unique irreducible
component of the source curve.
In particular, the following configurations coincide
\begin{equation}\label{eq:genus-1-moduli-configuration}
\SH_{1,1}(\punt,0) = \UH_{1,1}(\punt,0) =  \RMm_{1,1}(\punt,0).
\end{equation}
\end{proposition}

\begin{proof}
  First, since all source curves in question are
  irreducible, by the definition of basic monoids \cite[\S
  2.3.1]{ACGS20P}, we obtain a splitting on the level of characteristics
  \begin{equation}\label{eq:genus-1-characteristics-splitting}
  \ocM_{\SH_{1,1}(\punt,0)} =
  \ocM_{\overline{\scrM}_{1,1}}|_{\SH_{1,1}(\punt,0)} \oplus
  \NN_{\SH_{1,1}(\punt,0)}.
  \end{equation}
  In particular, nodes of the source curves have zero contact
  orders. This proves \eqref{eq:genus-1-moduli-configuration}.

  Since $\ul{\punt} = \xinfty\times \BC$, by the stability condition
  $\scrM := \overline\scrM_{1,1}(\xinfty,0)$ is the stack
  parameterizing underlying stable R-maps.
  Indeed, let $\ul{f}_{\scrM}\colon \ul{C}_{\scrM} \to \xinfty$ the
  universal stable map over $\scrM$. Then the unversal underlying
  R-map over $\scrM$ is
  \[
  \xymatrix{
  \ul{f}_{\SH} \colon \ul{C}_{\scrM} \ar[rrr]^-{(\ul{f}_{\scrM}, \ \omega^{\log}_{\ul{C}_{\scrM}/\scrM})} &&& \ul{\punt} = \xinfty \times \BC
  }
  \]
  We next show that any underlying stable R-map
  $\ul{f} \colon C \to \ul{\punt}$ given by $\ul{T} \to \scrM$ admits
  a unique lift $f \colon C^{\circ} \to \punt$ to a stable punctured
  R-map over $T$.

  Suppose $f$ is a lift of $\ul{f}$.
  Then \eqref{eq:genus-1-characteristics-splitting} implies a splitting
  $\cM_{T} = \cM^{\sharp}_{T}\oplus_{\cO^*}\cM^d_T$ with
  $\cM^{\sharp}_{T}$ given by the pull-back of
  $\cM_{\overline{\scrM}_{1,1}}$ and $\cM^d_T$ corresponds to $\NN_{\SH_{1,1}(\punt,0)}$.
  Denote by $\cT$ the line bundle over $T$ corresponding to
  $\cM^d_T$ as in \eqref{eq:generic-rank1-DF}.
  Let $p \subset \ul{C}$ be the unique marking.
  By \eqref{eq:l=1-log}, the morphism
  $f^{\flat} \colon f^*\cM_{\punt} \to \cM_{C^{\circ}}$ of log
  structures is equivalent to an isomorphism of line bundles
\begin{equation}\label{eq:g=1-map-lift}
f^*\left(\uomega^\vee\boxtimes\lblog\right) \stackrel{\cong}{\longrightarrow} \cT|_{C^{\circ}}(-p),
\end{equation}
which is equivalent to
\begin{equation}\label{eq:g=1-lb-iso}
\cT \stackrel{\cong}{\longrightarrow} \omega_C^\vee \otimes f^* \lblog \stackrel{\cong}\longrightarrow (\pi_T^*\pi_{T,*}\omega_C)^\vee \otimes f^* \lblog =   \left(\hodge^\vee\boxtimes\lblog\right)|_{T},
\end{equation}
where the second isomorphism follows from the canonical isomorphism $\omega_C \cong \pi_T^*\pi_{T,*}\omega_C$ since $C$ is of genus $1$.
Note that this isomorphism is uniquely determined by the underlying structure, leading to the uniqueness of lifting. This implies that $\cM^d_T$ is indeed the pull-back of $\cM^d$, hence proves the splitting \eqref{eq:genus-1-moduli-log} on   the log structure level.

Finally, to show the existence of a lifting of a given $\ul{f}$, we first construct the log structure $\cM^d_{T}$ corresponding to $\cT$ using \eqref{eq:g=1-lb-iso}. Since the source curves are irreducible with at most one node of contact  order zero, and a marking of contact order $-1$, the punctured curve $C^{\circ} \to T$ is uniquely determined by its underlying and  $\cM^d_{T}$ by \cite[Remark 2.2]{ACGS20P}. The morphism of log structures $f^{\flat}$ is then uniquely determined by \eqref{eq:g=1-map-lift} hence \eqref{eq:g=1-lb-iso}. This finishes the proof.
\end{proof}

The proposition also determines the divisor classes $\psi_{\min}$ and
$\Delta_{\max}$ over $\SH_{1,1}(\punt,0)$:
\begin{equation}
  \label{eq:g1-psimin}
 \psi_{\min} = -\Delta_{\max} = -c_1(\hodge^\vee\boxtimes \lblog).
\end{equation}

\subsubsection{The canonical virtual cycle}

We first consider the canonical obstruction theory and virtual cycle,
before turning to the reduced theory in the following section.
By Proposition \ref{prop:genus-1-stack}, the underlying stack of
$\SH:= \SH_{1,1}(\punt,0)$ is smooth.
Next we compute the obstruction sheaf for the canonical obstruction
theory.

\begin{lemma}\label{lem:g1-obstruction}
  There is a canonical exact sequence
\[
0 \to \hodge^{\vee}|_{\SH} \to H^1(\EE_{\SH/\fM}) \to \big(\hodge^{\vee}\boxtimes\Omega^{\vee}_{\xinfty}\big)|_{\SH} \to 0
\]
where $\EE_{\SH/\fM}$ is the canonical perfect obstruction theory
defined in \eqref{eq:tan-POT}.
\end{lemma}

\begin{proof}
Consider the universal punctured R-map $f \colon \pC \to \punt$ over $\SH$ with the projection $\pi\colon \pC \to \SH$. By Lemma \ref{lem:punt-tangent} and \eqref{eq:tan-POT}, we obtain a distinguished triangle
\begin{equation}\label{eq:g1-POT}
\rd\pi_*\cO_{\pC} \longrightarrow \EE_{\SH/\fM} \longrightarrow \rd\pi_*\cO_{\pC}\boxtimes\Omega^{\vee}_{\xinfty} \stackrel{[1]}{\longrightarrow}
\end{equation}
Taking part of the long exact sequence, we obtain
\[
0 \longrightarrow \cO_{\SH} \longrightarrow H^0(\EE_{\SH/\fM}) \stackrel{\phi}{\longrightarrow} \Omega^{\vee}_{\xinfty}|_{\SH} \longrightarrow \cdots
\]
We will show that $\phi$ is surjective so that \eqref{eq:g1-POT} leads
to an exact sequence
\begin{equation}\label{eq:g=1-obstruction}
0 \longrightarrow R^1\pi_*\cO_{\pC} \longrightarrow H^1(\EE_{\SH/\fM}) \longrightarrow  R^1\pi_*\cO_{\pC} \boxtimes \Omega^{\vee}_{\xinfty} \longrightarrow 0.
\end{equation}

Indeed, the isomorphism $\ul{\SH} \cong \overline{\scrM}_{1,1}\times\xinfty$ implies that $\Omega^{\vee}_{\ul{\SH} / \overline{\scrM}_{1,1}} \cong \Omega^{\vee}_{\xinfty}$.
On the other hand, the sequence
$\ul{\SH} \to \ul{\fM} \to \overline{\scrM}_{1,1}$ yields a triangle
\[
\LL^{\vee}_{\ul{\SH}/\ul{\fM}} \to \LL^{\vee}_{\ul{\SH}/\overline{\scrM}_{1,1}} \to \LL^{\vee}_{\ul{\fM}/\overline{\scrM}_{1,1}} \stackrel{[1]}{\to}
\]
We may take $\fM \subset \fM_{1, (-1)}(\ainfty)$ to be the open
substack with stable source curves, hence containing the image of
$\SH \to \fM_{1, (-1)}(\ainfty)$.
The morphism $\ul{\fM} \to \overline{\scrM}_{1,1}$ is smooth of
dimension $-1$, which implies
$H^0(\LL^{\vee}_{\ul{\fM}/\overline{\scrM}_{1,1}}) = 0$, and hence the
surjection
$H^0(\LL^{\vee}_{\ul{\SH}/\ul{\fM}}) \twoheadrightarrow
H^0(\LL^{\vee}_{\ul{\SH}/\overline{\scrM}_{1,1}})$.
Since both $\ul{\SH} \to \ul{\fM}$ and
$\ul{\SH} \to \overline{\scrM}_{1,1}$ are smooth, the morphism $\phi$
is given by the following composition
\[
H^0(\EE_{\SH/\fM}) \stackrel{\cong}{\longrightarrow} H^0(\LL^{\vee}_{\ul{\SH}/\ul{\fM}}) \twoheadrightarrow  H^0(\LL^{\vee}_{\ul{\SH}/\overline{\scrM}_{1,1}}) \cong \Omega^{\vee}_{\ul{\SH} / \overline{\scrM}_{1,1}} \cong \Omega^{\vee}_{\xinfty}
\]
hence is surjective.

Finally, the statement follows from \eqref{eq:g=1-obstruction} and $R^1\pi_*\cO_{\pC} =  \hodge^{\vee}|_{\SH}$ by Serre duality.
\end{proof}

\begin{proposition}\label{prop:g1-virtual-cycle}
The canonical virtual cycle $[\SH]^\vir$ is given by
  \begin{equation}
    \label{eq:g1-virtual-cycle}
    [\SH]^\vir = c_1(\hodge^\vee|_\SH) \cdot e(\hodge^\vee\boxtimes\Omega_{\infty_\cX}^\vee) \in A_0(\SH).
  \end{equation}
\end{proposition}
\begin{proof}
  The formula \eqref{eq:g1-virtual-cycle} follows from Lemma~\ref{lem:g1-obstruction} using
  the fact that the underlying stacks of $\fM$ and $\SH$ are smooth, and that $\hodge^\vee$ is
  rank-one locally free.

  By Proposition \ref{prop:genus-1-stack}, we have
  $\dim \SH = \dim (\overline{\scrM}_{1, 1} \times \infty_\cX) = \dim \infty_\cX + 1$. Thus
  \eqref{eq:g1-virtual-cycle} implies that $[\SH]^\vir$ is of
  dimension zero.
\end{proof}

\subsubsection{The reduced virtual cycle}

Recall the notation $\SH:= \SH_{1,1}(\punt,0)$, and consider the
reduced virtual cycle $[\SH]^{\red}$ defined due to
\eqref{eq:genus-1-moduli-configuration}.

\begin{proposition}\label{prop:g1-reduced-virtual-cycle}
The reduced virtual cycle $[\SH]^\red$ is of dimension one,
  and is given by
  \begin{equation}
    \label{eq:g1-reduced-virtual-cycle}
    [\SH]^\red = c_{\dim \infty_\cX}\left(\hodge^\vee \boxtimes (\Omega_{\infty_\cX}^\vee + \cO_{\xinfty} - \lblog)\right) \in A_1(\SH),
  \end{equation}
  where the right hand side is the Chern class of a class in K-theory,
  and can be computed by taking the dimension-one part of the mixed
  degree class
  \begin{equation*}
    \frac{c(\hodge^\vee\boxtimes\Omega_{\infty_\cX}^\vee) \cdot c(\hodge^\vee|_{\SH})}{c(\hodge^\vee \boxtimes \lblog)} \in A_*(\SH).
  \end{equation*}
\end{proposition}
\begin{proof}
  This follows from the bottom exact sequence of
  \eqref{diag:lift-red-obstruction}, noting that
\[\FF_{\SH/\fM}[1] = \cO(\Delta_{\max}) \cong \hodge^\vee \boxtimes \lblog\]
by Proposition \ref{prop:genus-1-stack}.
\end{proof}

\subsection{Vanishing in higher genus}
\label{ss:vanishing}

In this subsection, we study several situations in which all $g \ge 2$
effective invariants and cycles vanish.
We will restrict ourselves mostly to the situation
$\dim \cZ = \dim \cX - \rk \bE > 3$, since otherwise the balancing
condition and virtual dimension constraint are automatically satisfied
when $\beta = 0$.
We begin by stating a general vanishing criterion.

\begin{lemma}\label{lem:general-non-vanishing}
  Suppose that the following holds
  \begin{equation}
    \label{eq:general-bound-0}
    (3- \dim \cX + \rk \bE)(g-1) - \int_{\beta_{\cX}} c_1(K_{\cX}\otimes \det \bE) < 0.
  \end{equation}
  Then the cycle \eqref{eq:interesting-cycle} vanishes if for each
  $i$, one of the following holds:
  \begin{enumerate}
  \item $c_i \leq -2$,
  \item $c_i = -1$ and $\alpha_i \in H^{\geq 2}(\xinfty)$.
  \end{enumerate}
  The effective invariant and cycles \eqref{eq:interesting-inv},
  \eqref{eq:effective-cycle} and \eqref{eq:effective-cycle-Chow}
  vanish unless there is an $i$ such that $c_i = -1$ and
  $\alpha_i \in H^1(\xinfty)$.
\end{lemma}
\begin{proof}
  Applying \eqref{eq:general-bound-0} to \eqref{eq:red-vir-dim-original}, we have
  \[
    \red \dim \SH_{\varsigma} < \sum_{i=1}^n \left( 1 + \rk \bE \cdot (c_i + 1)\right) \leq n.
  \]
  Thus, by a dimension count, in order for
  \eqref{eq:interesting-cycle} to be non-vanishing, there must exist
  some $i$ such that $\alpha_i \in H^0(\xinfty)$ and $c_i = -1$.
  Applying the string equation \eqref{eq:string} or the unit axiom
  \eqref{eq:unit-string} repeatedly, we may assume that either
  $\alpha_i \in H^{\geq 2}(\xinfty)$ for any $c_i = -1$, or
  $c_i \leq -2$ for all $i$.
  In both cases, the vanishing of \eqref{eq:interesting-inv},
  \eqref{eq:effective-cycle} and \eqref{eq:effective-cycle-Chow}
  follow from the vanishing of \eqref{eq:interesting-cycle}.
\end{proof}

\begin{proposition}\label{prop:vanishing-no-markings}
  Assume that $g \geq 2$ and $\dim \cZ = \dim \cX - \rk \bE > 3$.

  \begin{enumerate}
  \item The inequality \eqref{eq:general-bound-0} holds if
    \begin{equation}\label{eq:general-bound}
      (3- \dim \cX + \rk \bE)\int_\beta c_1(\lblog) - 2\int_{\beta_{\cX}} c_1(K_{\cX}\otimes \det \bE) < 0.
    \end{equation}
  \item The inequality \eqref{eq:general-bound-0} holds for all
    effective $\beta$ (including $\beta = 0$) if
    \eqref{eq:general-bound} holds for a set of generators of
    $\overline{NE}(\cX)$, the closure of the cone of effective curves
    (also known as the pseudo-effective cone).
  \end{enumerate}
\end{proposition}
\begin{remark}
  Note that (2) surprisingly claims that the strict bound
  \eqref{eq:general-bound-0} holds for $\beta = 0$ despite the fact
  that \eqref{eq:general-bound} never holds for $\beta = 0$.
\end{remark}
\begin{proof}
  Applying the balancing condition \eqref{eq:balancing-condition}, we obtain
  \[
    (3- \dim \cX + \rk \bE)(g-1) - 2\int_{\beta_{\cX}} c_1(K_{\cX}\otimes \det \bE) \leq 1/2 \cdot (3- \dim \cX + \rk \bE)\int_\beta c_1(\lblog) - \int_{\beta_{\cX}} c_1(K_{\cX}\otimes \det \bE).
  \]
  This implies (1).

  Statement (2) for $\beta \neq 0$ follows from the convexity of
  $\overline{NE}(\cX)$.
  While the inequality \eqref{eq:general-bound} does not hold for
  $\beta = 0$, the inequality \eqref{eq:general-bound-0} is
  automatically satisfied when $\dim \cZ > 3$ and $\beta = 0$.
\end{proof}

\subsubsection{Hypersurfaces}

We now consider the case of a hypersurface $\cZ \subset \cX$ in
more detail -- this is the case of $\rk \bE = 1$ in the setup of \S
\ref{ss:effective-setup}.
In particular, we have
\[
\xinfty \cong \cX, \ \  \lblog = \bE, \ \  \beta = \beta_\cX.
\]

\begin{example}
  \label{ex:projective-hypersurface}
  Take $\cX = \PP^N$ and $\bE = \cO_{\PP^N}(d)$ for $d > 0$.
  We will assume in the following that $N \ge 5$, so that $\cZ$ is of
  dimension at least $4$.

  Then, all effective invariants and cycles vanish for $g \geq 2$ if
  the following is satisfied
  \begin{equation}\label{eq:P^N-hypersurface-threshold-d}
    d > \frac{2N + 2}{N - 2},
  \end{equation}
  or equivalently
  \begin{equation}\label{eq:P^N-hypersurface-threshold-N}
    N > \frac{2d + 2}{d - 2} \ \ \mbox{and} \ \ d > 2.
  \end{equation}
  This is a consequence of
  Proposition~\ref{prop:vanishing-no-markings}, since
  $\overline{NE}(\PP^N)$ is spanned by the class of a line,
  \eqref{eq:general-bound} holds if
  \begin{equation}\label{eq:P^N-hypersurface-threshold}
    (4 - N) d - 2 (d - N - 1) < 0,
  \end{equation}
  which is equivalent to \eqref{eq:P^N-hypersurface-threshold-d} since
  $N \ge 5$.
  This implies the vanishing in Lemma~\ref{lem:general-non-vanishing}
  since $H^1(\PP^N) = 0$.

  The inequality \eqref{eq:P^N-hypersurface-threshold-d} is
  automatically satisfied in the Calabi--Yau and general type cases
  $d \ge N + 1$, which is compatible with Corollary~\ref{cor:general-type-vanishing}.
  In the Fano case $d \leq N$, the following is a complete
  list of $(d, N)$ satisfying \eqref{eq:P^N-hypersurface-threshold-N}:
  \begin{equation}\label{eq:explicit-hyper-surface-in-Pn}
    (d = 3, N \geq 9), \ (d = 4, N \geq 6), \ (d \ge 5, N \geq d)
  \end{equation}
\end{example}

\begin{example}\label{ex:projective-spaces-product}
  Example \ref{ex:projective-hypersurface} immediately generalizes to the case
\[
  \cX = \prod_{i = 1}^k \PP^{N_i} \ \ \ \mbox{and} \ \ \ \bE = \cO_{\cX}(d_1, d_2, \cdots, d_k)
\]
  with $\sum_i N_i \geq 5$. In this case all effective invariants and cycles vanish for $g \geq 2$ assuming
  \begin{equation}\label{eq:projective-space-product}
    d_j > \frac{2N_j + 2}{\sum_i N_i - 2}, \ \ \ \ \mbox{for all } j.
  \end{equation}

For example, assuming $k \geq 2$, $N_1 = \cdots = N_k = N \geq 1$ and $kN \geq 5$, then \eqref{eq:projective-space-product} holds iff one of the following is satisfied
  \begin{enumerate}
   \item $d_i \geq 1$ for all $i$ if $(k, N)$ belongs to the list
   \[
   (k = 3, N \geq 5), (k=4, N \geq 3), (k = 5, 6, N \geq 2), (k \geq7, N \geq 1).
   \]
   \item $d_i \geq 2$ for all $i$ if $(K, N)$ belongs to the list
   \[
   (k = 2, N \geq 3), (k = 3, 4, N \geq 2),  (k \geq 5, N \geq 1).
   \]
  \end{enumerate}
\end{example}

\begin{example}\label{ex:Grass-hypersurface}
  We may also consider a degree $d$ hypersurface $\cZ$ in a
  Pl\"ucker-embedded Grassmannian
  $\cX = \mathrm{Gr}(K, N) \subset \PP^{{N \choose K} - 1}$.
  In this case, we have $K_\cX \otimes \bE \cong \cO(d - N)$.
  We will assume that $K(N - K) \ge 5$ so that $\cZ$ is of dimension
  at least $4$.
  By Proposition \ref{prop:vanishing-no-markings} and a similar
  discussion as in Example~\ref{ex:projective-hypersurface}, all
  cycles \eqref{eq:interesting-cycle} vanish for $g \geq 2$ if the
  following is satisfied:
  \begin{equation}
    \label{eq:Grassmann-threshold-d}
    d > \frac{2N}{K(N - K) - 2}
  \end{equation}
  In this case, the Fano condition is $d < N$.

  For example, when $K = 2$, the pair $(d, N)$ satisfies the Fano
  condition and the inequality \eqref{eq:Grassmann-threshold-d} iff
  \[
    (d = 2, N \ge 7), \ (d = 3, N \ge 5), \ (d \ge 4, N \ge d + 1).
  \]
\end{example}

\subsubsection{Fano complete intersections}

We now discuss a criterion for the vanishing of effective invariants and
cycles which is especially useful in the case of $\rk\bE \geq 2$.

Assume $\bE$ is ample. We fix an ample line bundle $A$ over $\cX$, and consider the following
\[
m_A := \sup\{m \in \QQ \ | \ \bE \otimes A^{-m} \ \mbox{is nef} \} \ \ \mbox{and} \ \  m'_A := \inf\{m \in \QQ \ | \ K_{\cX}\otimes \det \bE \otimes A^{m} \ \mbox{is nef} \}.
\]
Note that $m_A > 0$ is called the {\em Barton invariant of $\bE$}
\cite[Example 6.2.14]{LazII}.
Furthermore $m'_A > 0$ iff $-K_{\cX}\otimes \det \bE$ is ample, or
equivalently $\cZ$ is Fano.

\begin{proposition}\label{prop:vanishing-reduced-cycle}
  Assume that $\bE$ is ample, that $\dim \cX - \rk \bE \ge 4$, and
  that $g\ge 2$.
  Furthermore, suppose that there is an ample line bundle $A$ such that
  $m'_A \ge 0$ and
  \[
    3- \dim \cX + \rk \bE + \frac{2m'_A}{m_A} < 0.
  \]
  Then, the effective invariant \eqref{eq:interesting-inv} and cycles \eqref{eq:effective-cycle} and
  \eqref{eq:effective-cycle-Chow} vanish assuming
  $\alpha_i \not\in H^1(\xinfty)$ for any $c_i = -1$.
\end{proposition}
\begin{proof}
  We establish this proposition by establishing the conditions in
  Lemma~\ref{lem:general-non-vanishing}.

   By the following lemma, \eqref{eq:general-bound} holds if
  \begin{equation*}
    (3 - \dim \cX + \rk \bE) \cdot m_A \int_{\beta_\cX} c_1(A) + 2m'_A \int_{\beta_\cX} c_1(A) < 0.
  \end{equation*}
  Assuming that $\beta \neq 0$, we have
  $ \int_{\beta_\cX} c_1(A) \neq 0$. Then the above inequality is equivalent
  to
  \[
    3- \dim \cX + \rk \bE + \frac{2m'_A}{m_A} < 0.
  \]
  The conditions in Lemma~\ref{lem:general-non-vanishing} (including
  the case $\beta = 0$) then follow from
  Proposition~\ref{prop:vanishing-no-markings}.
\end{proof}
\begin{lemma}
Suppose $\bE$ is ample. Let $A$ be an ample line bundle over $\cX$ such that $\bE\otimes A^{\vee}$ is nef. Then we have
\[
\int_{\beta} c_1(\cO_{\xinfty}(1)) \geq \int_{\beta_{\cX}} c_1(A).
\]
\end{lemma}
\begin{proof}
  Note the isomorphisms $\xinfty \cong \PP(\bE^{\vee}\otimes A)$ and
  $\cO_{\xinfty}(1) \cong \cO_{\PP(\bE^{\vee}\otimes A)}(1)\otimes A$,
  and that $\cO_{\PP(\bE^{\vee}\otimes A)}(1)$ is nef.
  We compute
\[
\int_{\beta} c_1(\cO_{\xinfty}(1)) = \int_{\beta} c_1(\cO_{\PP(\bE^{\vee}\otimes A)}(1)\otimes A) \geq   \int_{\beta_{\cX}} c_1(A).
\]
\end{proof}

\begin{example}
  \label{ex:projective-complete-intersection}
  Consider the case where $\cX = \PP^{N}$ and
  $\bE = \cO_{\PP^N}(d_1)\oplus \cdots \cO_{\PP^N}(d_R)$ with
  $2 \leq d_1 \leq d_2 \leq \cdots \leq d_R$.
  In this case, we have
  $K_{\cX}\otimes \det \bE = \cO(-N-1 + \sum d_j)$.
  Choose $A = \cO_{\PP^N}(1)$. Then we have $m'_A = N+1 - \sum d_j$
  and $m_A = d_1$.
  In order to apply Proposition~\ref{prop:vanishing-reduced-cycle}, we
  assume that $N - R \ge 4$, and the Fano case $\sum d_j \leq N$,
  which is equivalent to the condition $m'_A > 0$.

  Since $H^1(\xinfty) = 0$, the condition in
  Proposition~\ref{prop:vanishing-reduced-cycle} that
  $\alpha_i \not\in H^1(\xinfty)$ is vacuous.
  Applying Proposition~\ref{prop:vanishing-reduced-cycle}, we
  observe that all cycles  \eqref{eq:interesting-cycle}, \eqref{eq:effective-cycle} and \eqref{eq:effective-cycle-Chow} and the effective invariants \eqref{eq:interesting-inv}
  vanish for any $g \geq 2$ and any $\beta \in H_2(\xinfty)$ assuming
\begin{equation}\label{eq:Pn-vanishing-threshold}
3 - N + R + \frac{2}{d_1} \cdot (N + 1 - \sum d_j) < 0 .
\end{equation}
This threshold \eqref{eq:Pn-vanishing-threshold} in the Fano range $\sum d_j \leq N$ turns out to be  very efficient:

\smallskip

\noindent
{\bf The $R = 1$ case.}
In this case
\eqref{eq:Pn-vanishing-threshold} is equivalent to
\eqref{eq:P^N-hypersurface-threshold}, see \eqref{eq:explicit-hyper-surface-in-Pn} for an explicit  list.

\smallskip
\noindent
{\bf The $R \geq 2, d_1 \geq 3$ case.} We observe that in the Fano range, \eqref{eq:Pn-vanishing-threshold} is automatically satisfied:
\begin{align*}
 3 - N + R + \frac{2}{d_1} \cdot (N + 1 - \sum d_j)
 =&  (3 + \frac{2}{d_1}) - (1-\frac{2}{d_1})N - \sum\left(\frac{2d_j}{d_1} - 1 \right) \\
 \leq & (3 + \frac{2}{3}) - (1-\frac{2}{d_1})R d_1 - \sum (2-1) \\
 =& 3 + \frac{2}{3} - R (d_1 -1) \\
 <& 3 + \frac{2}{3} - 2 (3-1) = - \frac{1}{3}.
\end{align*}
where the first inequality uses $d_1 \leq d_j$ and $R d_1 \leq \sum d_j \leq N$, and the last inequality uses $2 \leq R$ and $3 \leq d_1$. Thus the genus $g\geq 2$ reduced theory vanishes completely for all Fano complete intersections in this case!

\smallskip
\noindent
{\bf The $R \geq 2, d_1 = 2$ case.}
In this situation, the threshold \eqref{eq:Pn-vanishing-threshold} simplifies to
\[
  \sum_i (d_i - 1) > 4.
\]

Therefore, when $R \ge 2$, the threshold
\eqref{eq:Pn-vanishing-threshold} holds for Fano complete
intersections $\cZ$, unless $(d_1, \cdots, d_R)$ is one of the
following
\[
(2,2), \ (2,3), \ (2,4), \ (2,2,2), \ (2,2,3), \ (2,2,2,2).
\]
\end{example}

\subsection{Basic effective invariants}
\label{ss:basic-effective}

Next, we investigate situations in which the effective
invariants, while in general non-vanishing, are determined by a small
set of \emph{basic effective invariants}.
The most interesting examples are Calabi--Yau threefold complete
intersections, including quintic threefolds.

\begin{definition}
  \label{def:basic-effective}
  Assume that $g \ge 2$ and $\beta$ satisfy
  \begin{equation}
    \label{eq:vdim-zero}
    (3 - \dim \cX + \rk \bE)(g - 1) - \int_{\beta_\cX} c_1(K_\cX \otimes \det \bE) = 0,
  \end{equation}
  as well as the balancing condition
  $\int_\beta c_1(\lblog) \le 2g - 2$ of
  \eqref{eq:balancing-condition} with \emph{equality} when
  $\rk \bE \ge 2$.
  Then, the \emph{basic effective invariant} of genus $g$ and degree
  $\beta$ is the effective invariant
  \begin{equation*}
    \deg [\RMm_{\varsigma}]^{\red},
  \end{equation*}
  for the special discrete data
  $\varsigma = (g, \beta, \{-2\}_{i = 1}^n)$ where
  $n = 2g - 2 - \int_\beta c_1(\lblog)$.
\end{definition}
\begin{remark}
  The assumptions in the definition ensure that the balancing
  condition holds and $ \red \dim \SH_{\varsigma} = 0$ by \eqref{eq:red-vir-dim-original}.
  %the virtual dimension constraints are satisfied.
\end{remark}
\begin{proposition}
  \label{prop:reduce-to-basic-effective}
  Under the assumptions of Definition~\ref{def:basic-effective}, and
  further assuming that $H^1(\xinfty) = 0$, any genus $g$, degree
  $\beta$ effective invariant \eqref{eq:interesting-inv} is determined
  by the corresponding basic effective invariants.
\end{proposition}
\begin{proof}
  By \eqref{eq:string}, \eqref{eq:divisor} and $H^1(\xinfty) = 0$, we
  may assume that in \eqref{eq:interesting-inv}, if $c_i = -1$, then
  $\alpha_i \in H^{\ge 3}(\xinfty)$.
  By \eqref{eq:red-vir-dim-original} and \eqref{eq:vdim-zero}, we have
  the reduced virtual dimension
  \begin{equation*}
    \red \dim \SH_{\varsigma} = \sum_{i = 1}^n (1 + \rk\bE (c_i + 1)).
  \end{equation*}
  Hence, when $\rk\bE \ge 2$, the integral vanishes for dimension
  reasons unless $k = 0$ and $n = 0$ in \eqref{eq:interesting-inv}.
  When $\rk\bE = 1$, we instead conclude that $k = 0$, $c_i = -2$ and
  $\alpha_1 \in H^0(\xinfty)$ for all $i$.
\end{proof}
\begin{assumption}
  \label{ass:all-basic-effective}
  Let $\xinfty$ be such that $H^1(\xinfty) = 0$.
  Assume that for any $(g, \beta)$ with $g \ge 2$ satisfying the balancing condition
  \eqref{eq:balancing-condition}, we have
  \begin{equation}
    \label{eq:non-strict-dim-threshold}
    (3 - \dim \cX + \rk \bE)(g - 1) - \int_{\beta_\cX} c_1(K_\cX \otimes \det \bE) \le 0.
  \end{equation}
\end{assumption}
\begin{example}
  \label{ex:effective-CY3}
  Assumption~\ref{ass:all-basic-effective} is satisfied when
  $H^1(\xinfty) = 0$, and $\cZ$ is a Calabi--Yau threefold,
  since then
  \begin{equation*}
    3- \dim \cX + \rk \bE = 0 \ \ \mbox{and} \ \  c_1(K_{\cX}\otimes \det \bE) = 0.
  \end{equation*}
  Note that we do allow $H^1(\cZ) \neq 0$.
  This is the most interesting situation for applications.
\end{example}
\begin{corollary}
  Under Assumption~\ref{ass:all-basic-effective}, any $g \ge 2$
  effective invariant \eqref{eq:interesting-inv} of either vanishes or
  is determined by the corresponding basic effective invariant.

  The set of genus $g$ basic effective invariants is indexed by the
  set of curve classes
  \begin{equation}\label{eq:CY3-rk1-curve-class}
    \left\{\beta \ \Big|  \ \int_{\beta} c_1(\lblog) \leq  2g-2 \mbox{ and (\ref{eq:vdim-zero}) holds }\right\}
  \end{equation}
  in the case $\rk \bE = 1$, and
  \begin{equation}\label{eq:CY3-higher-rk-curve-class}
    \left\{\beta \ \Big|  \ \int_{\beta} c_1(\lblog) =  2g-2 \mbox{ and (\ref{eq:vdim-zero}) holds }\right\}
  \end{equation}
  when $\rk \bE \ge 2$.
  Furthermore, when $\bE$ is ample, these sets are finite for any $g$.
\end{corollary}
\begin{proof}
  This follows directly from Lemma~\ref{lem:general-non-vanishing} and
  Proposition~\ref{prop:reduce-to-basic-effective}.
\end{proof}
\begin{remark}
  Thus, combining with the explicit formulas for effective invariants
  in genus zero and one, in the situation of
  Assumption~\ref{ass:all-basic-effective} and if $\bE$ is ample, the
  localization formula of \cite{CJR21} leads to a computation of the
  Gromov--Witten invariants of the complete intersection $\cZ$ up to
  a finite number of basic effective invariants.
\end{remark}
\begin{remark}
  If $H^1(\cX) = 0$, the basic effective invariants also determine all
  effective cycles in cohomology.
  Indeed, since $\deg\colon H_0(\overline\scrM_{g, n}) \to \QQ$ is an
  isomorphism, all cycles in the situation $\red\dim\SH_\varsigma = 0$
  are determined from the basic effective invariants.
  From this, we can use \eqref{eq:unit-string} inductively to
  determine the remaining effective cycles in cohomology.

  It would be very interesting to study what can be said about
  effective cycles in Chow in this situation.
  If Question~\ref{q:effective-tautological} below has an affirmative
  answer in this case, it would imply that the effective cycles in
  Chow are also determined from the basic effective invariants.
\end{remark}

\subsubsection{Calabi--Yau threefold examples}

The most interesting examples for basic effective invariants are in
the Calabi--Yau case (see Example~\ref{ex:effective-CY3}).
We discuss a few exmaples, and point out connections to physics.
\begin{example}\label{ex:quintic-effective-inv}
  Suppose $\cX = \PP^4$ and $\bE = \lblog = \cO_{\PP^4}(5)$.
  Then $X_5 := \cZ = (s_\cX = 0)$ is a smooth quintic threefold.
  For each $g \geq 1$, $\beta$ in \eqref{eq:CY3-rk1-curve-class}
  satisfies $\beta \leq \lfloor \frac{2g-2}{5}\rfloor$.
  Thus there are $\lfloor \frac{2g-2}{5}\rfloor + 1$ many basic
  effective invariants.

  The number of basic effective invariants agrees with the number
  \cite[(3.80)]{HKQ09} of unknown coefficients at genus $g$ that
  physicists find after using the holomorphic anomaly equations, the
  conifold gap and the orbifold regularity.%
  \footnote{Using the formula for the degree zero Gromov--Witten
    invariants of $X_5$, they further reduce the number of
    coefficients to $\lfloor \frac{2g-2}{5}\rfloor$.
    In the same way, the localization formula of \cite{CJR21} can be
    used to determine the degree zero effective invariant in terms of
    the degree zero Gromov--Witten invariant.}
  This suggests that these constraints on the Gromov--Witten theory of
  $\cX$ can be approached via the localization formula of
  \cite{CJR23P} even without knowing the values of the basic effective
  invariants.
  This is investigated in \cite{GJR18P}.
\end{example}
\begin{example}\label{ex:CY3-complete-intersection}
  Let $\cZ$ be a complete intersection Calabi--Yau threefold in
  projective space.
  Beside quintic threefolds in Example \ref{ex:quintic-effective-inv},
  there are four other examples of such $\cZ$:
\[
X_{3,3} \subset \mathbb{P}^5, \ \ X_{2,2,2,2} \subset \mathbb{P}^7, \ \ X_{2,4} \subset \mathbb{P}^5, \ \ X_{2,2,3} \subset \mathbb{P}^6,
\]
where $X_{d_1, \cdots, d_k}$ means a complete intersection of $k$
hypersurfaces of degree $d_1, \cdots, d_k$ respectively.
We consider the number of basic effective invariants in these
examples, and compare them with the unknown coefficients appearing in the
physics work \cite{HKQ09}.

\smallskip
\noindent
{\bf The $X_{3,3}$ case.}
  In this case, we have
  \[
  \cX = \PP^5, \ \ \bE = \cO_{\PP^5}(3)^{\oplus 2}, \ \ \infty_\cX \cong \PP(\cO_{\PP^5}(-3)^{\oplus 2}) \cong \PP^5 \times
  \PP^1, \ \ \lblog \cong \cO_{\PP^5}(3) \boxtimes \cO_{\PP^1}(1).
  \]
  Letting $\beta_1 = \int_\beta c_1\left(\cO_{\PP^5}(1) \right)$ and
  $\beta_2 = \int_\beta c_1\left(\cO_{\PP^1}(1)\right)$, and using
  $\int_{\beta} c_1(\lblog) = 2g-2$, we observe
  \[
  \beta_1 \le \frac{2g - 2}3, \ \ \beta_2 = 2g - 2 - 3\beta_1.
  \]
  Thus there are $\lfloor \frac{2g-2}{3}\rfloor + 1$ many basic
  effective invariants.
  As in Example~\ref{ex:quintic-effective-inv}, this agrees with the
  count of unknown coefficients in physics (up to a ``$+1$'')
  \cite[(4.92)]{HKQ09}.

\smallskip
\noindent
{\bf The $X_{2,2,2,2} $ case.}    In this case, we have
  \[
  \cX = \PP^7, \ \ \bE = \cO_{\PP^7}(2)^{\oplus 4}, \ \ \infty_\cX  \cong \PP^7 \times
  \PP^3, \ \ \lblog \cong \cO_{\PP^7}(2) \boxtimes \cO_{\PP^1}(1).
  \]
    Letting $\beta_1 = \int_\beta c_1\left(\cO_{\PP^7}(1) \right)$ and
  $\beta_2 = \int_\beta c_1\left(\cO_{\PP^3}(1)\right)$, and using
  $\int_{\beta} c_1(\lblog) = 2g-2$, we observe that
  \[
  \beta_1 \le g-1, \ \ \beta_2 = 2g - 2 - 2\beta_1.
  \]
  Thus there are $g$ many basic effective invariants, again agreeing
  with \cite[(4.92)]{HKQ09}.

\smallskip
\noindent
{\bf The $X_{2,4} $ case.}   In this case, we have
  \[
  \cX = \PP^5, \ \ \bE = \cO_{\PP^5}(2)\oplus  \cO_{\PP^5}(4), \ \ \infty_\cX \cong \PP(\cO_{\PP^5}(-2)\oplus\cO_{\PP^5}(-4)),
  \]
  and $\lblog$ is the dual of the tautological sub-bundle.  Letting $\beta_1 = \int_\beta c_1\left(\cO_{\PP^5}(1) \right)$, we have
  \[
   0 \leq \beta_1 \leq \frac{1}{2} \cdot \int_\beta c_1(\lblog) = g-1.
  \]
  Thus there are $g$ basic effective invariants in this case.
  This \emph{only} agrees with the number of unknown coefficients in
  \cite{HKQ09} if only the holomorphic anomaly equations and the
  conifold gap condition are used, but \emph{not} the orbifold gap
  condition.
  This suggests that there may be further constraints on the basic
  effective invariants for $\frac{g-1}2 < \beta_1 < g - 1$.
  Geometrically, a stable map $C \to \infty_\cX$ with curve class in
  this range necessarily has component(s) dominating fiber(s) of the
  $\PP^1$-bundle $\xinfty \to \cX$.

\smallskip
\noindent
{\bf The $X_{2,2,3} $ case.} In this case, we have
  \[
  \cX = \PP^6, \ \ \bE = \cO_{\PP^6}(2)\oplus \cO_{\PP^6}(2)\oplus  \cO_{\PP^6}(3), \ \ \infty_\cX \cong \PP(\cO_{\PP^6}(-2)\oplus \cO_{\PP^6}(-2)\oplus  \cO_{\PP^6}(-3)),
  \]
    and $\lblog$ is the dual of the tautological sub-bundle.  Repeating the same analysis as above,
    there are  $g$ basic effective invariants in this case, agreeing with the number in \cite{HKQ09}.
\end{example}

\begin{example}
  As an example of a Calabi--Yau threefold complete intersection in a
  Grassmannian, consider a complete intersection of seven planes with
  a Pl\"ucker-embedded $\mathrm{Gr}(2, 7) \subset \PP^{20}$.
  This appears on one side of the \emph{Pfaffian--Grassmannian}
  correspondence (see for example \cite{Ro00, Ku06P, ADS15} and
  \cite[Section~7.4.1]{FJR18}).

  In this case, we have
  \begin{align*}
    \cX &= \mathrm{Gr}(2, 7),& \bE &= \cO_{\PP^{20}}(1)^{\oplus 7}|_\cX, \\
    \infty_\cX &= \PP(\cO_{\PP^20}(-1)^{\oplus 7}|_\cX) \cong \mathrm{Gr}(2, 7) \times \PP^6,&
    \lblog &\cong \cO_{\PP^{20}}(1)|_\cX \boxtimes \cO_{\PP^6}(1)
  \end{align*}
  Letting $\beta_1 = \int_\beta \cO_{\PP^{20}}(1)|_\cX$ and
  $\beta_2 = \int_\beta \cO_{\PP^6}(1)$, we get the condition
  \begin{equation*}
    \beta_1 + \beta_2 = 2g - 2.
  \end{equation*}
  Hence, there are $(2g - 2) + 1$ basic effective invariants.

  We do not see any direct relation to the work \cite{HoKo09} in
  physics, but it would be very interesting to further investigate the
  computation of higher genus Gromov--Witten invariants in this
  setting.
\end{example}

\subsubsection{Other examples}

We point out some other situations when
Assumption~\ref{ass:all-basic-effective} holds.
\begin{example}
  Assume that
  \begin{equation*}
    3 - \dim \cX + \rk \bE = 0 \ \ \mbox{and} \ \  c_1(K_{\cX}\otimes \det \bE) < 0.
  \end{equation*}
  Thus $\cZ$ is a threefold of general type.
  In this example, $\beta = 0$ is the only possible curve class with
  non-negative reduced virtual dimension, and thus there is exactly
  one basic effective invariant in each genus.
  On the other hand, threefolds of general type have exactly one
  vanishing Gromov--Witten invariant in genus $\ge 2$, and so the
  localization formula of \cite{CJR21} will not lead to new
  information in this example.
\end{example}
\begin{example}
  In the context of Example~\ref{ex:projective-hypersurface}, the
  cases $(d = 3, N = 8)$ and $(d = 4, N = 5)$ just outside of the list
  \eqref{eq:explicit-hyper-surface-in-Pn} fall under
  Assumption~\ref{ass:all-basic-effective}.

  For instance, for $d = 3$, $N = 8$,
  \eqref{eq:non-strict-dim-threshold} becomes $4(1 - g) + 6d \le 0$,
  which is a consequence of the balancing condition $3d \le 2g - 2$.
  In that case, there is exactly one basic effective invariant for
  $\beta = \frac{2g - 2}3$ only defined when $g \equiv 1 \pmod{3}$.

  In the case $d = 4$, $N = 5$, there is one basic effective invariant
  for $\beta = \frac{g - 1}2$ only defined when $g \equiv 1 \pmod{2}$.
\end{example}
\begin{example}
  For a hypersurface of bidegree $(3, 1)$ in
  $\cX = \PP^5 \times \PP^1$, note that $\xinfty$ agrees with the
  $X_{3, 3}$ case in the Calabi--Yau threefold
  Example~\ref{ex:CY3-complete-intersection}.
  In particular, both of these examples satisfy
  Assumption~\ref{ass:all-basic-effective}, and they have the same
  $\lfloor\frac{2g - 2}3\rfloor$ basic effective invariants in genus
  $g$.

  There are also other examples of hypersurfaces in products of
  projective spaces satisfying
  Assumption~\ref{ass:all-basic-effective}.
  For instance, for a hypersurface of bidegree $(2, 2)$ in
  $\cX = \PP^3 \times \PP^3$, there are $g$ basic effective
  invariants.
\end{example}

\subsection{Changing \texorpdfstring{$\ell$}{l} for effective invariants and cycles}
\label{ss:setup-higher-spin}

We consider here effective invariants under the change of $\ell$ of
Section~\ref{sec:log-root}.
Let $\infty$ be the target from \S \ref{ss:effective-setup}, and let
$\infty^{1/\ell}$ be the $\ell$-th root as in \eqref{eq:ell-root-target}.
We will observe that there is a direct correspondence between effective
invariants/cycles of $\infty$ and of $\infty^{1/\ell}$.

First, \eqref{eq:change-ell-discrete-data} and
Lemma~\ref{lem:map-taking-ell-root} establish a bijection between
discrete data $(g, \beta, \bar\varsigma)$ for punctured R-maps to
$\infty$ and discrete data $(g, \beta, \varsigma)$ for punctured
R-maps to $\infty^{1/\ell}$. It then follows from
Corollary~\ref{cor:reduced-change-ell-with-min}:

\begin{corollary}
  \label{cor:effective-change-ell}
  We have the identity   of effective invariants
  \begin{equation*}
    \ell^{1 + k} \int_{[\RMm_{g, \varsigma}(\infty^{1/\ell}, \beta)]^{\red}} \psi_{\min}^k \cdot \prod \ev_i^*\alpha_i
    = \int_{[\RMm_{g, \bar\varsigma}(\infty, \beta)]^{\red}} \psi_{\min}^k \cdot \prod \ev_i^*\alpha_i ,
  \end{equation*}
  and the identity   of effective cycles
  \begin{equation*}
    \ell^{1 + k} p_{\varsigma, *} \left( \psi_{\min}^k \cdot \prod \ev_i^*\alpha_i \cap [\RMm_{g, \varsigma}(\infty^{1/\ell}, \beta)]^{\red} \right)
    = p_{\bar\varsigma, *} \left( \psi_{\min}^k \cdot \prod \ev_i^*\alpha_i \cap [\RMm_{g, \bar\varsigma}(\infty, \beta)]^{\red} \right).
  \end{equation*}
\end{corollary}

\subsubsection{LG/CY correspondence for effective invariants}
\label{sss:LGCY}

Corollary~\ref{cor:effective-change-ell} has a major application to
the forthcoming work \cite{GJR23P} on the all genus LG/CY
correspondence for quintic threefolds, which is an explicit
relationship between the Gromov--Witten (CY) theory of quintic
threefolds and the FJRW (LG) theory of quintic threefold
singularities.
In the log GLSM approach to these theories, the effective invariants
of each theory form the most difficult part.
What enables the forthcoming work \cite{GJR23P}, is a close
relationship between the $\infty$-divisors of the log GLSM targets of
both theories.

More explicitly, the log GLSM target for Gromov--Witten theory is
isomorphic to
\[\ul\fP^{\mathrm{GW}} = \PP_{\BC \times \PP^4}((\uomega \boxtimes
\cO(-5)) \oplus \cO)\]
over $\BC$.
The log GLSM target for FJRW theory is isomorphic to
\[
\ul\fP^{\mathrm{FJRW}} = \PP_{\fX}(\cL_\fX^{\oplus 5} \oplus \cO),
\]
where $\fX \cong \BG_m \to \BC$ is induced by taking the $5$th power,
and where $\cL_\fX$ is the universal $5$th root line bundle.
The corresponding $\infty$-divisors are
$\ul\infty^{\mathrm{GW}} \cong \BC \times \PP^4$ with logarithmic
structure defined from the line bundle
$\cO(\infty^{\mathrm{GW}}) = \uomega^\vee \boxtimes \cO(5)$, and
$\ul\infty^{\mathrm{FJRW}} = \PP_{\fX}(\cL_\fX^{\oplus 5}) \cong \fX
\times \PP^4$ with logarithmic structure defined from the line bundle
$\cO(\infty^{\mathrm{FJRW}}) = \cL_\fX^\vee \boxtimes \cO(1)$.

\begin{lemma}
  \label{lem:LGCY}
  There is an isomorphism
  $\infty^{\mathrm{GW}, 1/5} \cong \infty^{\mathrm{FJRW}}$ over $\BC$,
  where $\infty^{\mathrm{GW}, 1/5}$ is the root construction of
  Section~\ref{sec:log-root} for $\ell = 5$ applied to
  $\infty^{\mathrm{GW}}$.
\end{lemma}
\begin{proof}
  The isomorphism
  $\ul\infty^{\mathrm{FJRW}} \to \ul\infty^{\mathrm{GW}, 1/5}$ is
  induced by the natural morphism
  $\fX \times \PP^4 \to \BC \times \PP^4$, as well as the $5$th root
  $\cL_\fX^\vee \boxtimes \cO(1)$ of $\cO(\infty^{\mathrm{GW}})$.
  Under this correspondence, clearly, $\cO(\infty^{\mathrm{FJRW}})$ is
  identified with $\cO(\infty^{\mathrm{GW}, 1/5})$, and so the
  isomorphism extends to the level of logarithmic stacks over $\BC$.
\end{proof}

Hence, (up to powers of $5$) the Gromov--Witten and FJRW theory of
quintics have the same effective invariants.
We call this the \emph{LG/CY-correspondence for effective invariants}.
This together with the localization formula of \cite{CJR23P} applied to
the log R-map moduli space for both the Gromov--Witten and FJRW theory
of quintics are the geometric input for the work \cite{GJR23P}.

\ChFelix{
\begin{remark}
  There is an LG/CY-correspodence for effective invariants not just
  for the Gromov--Witten and FJRW theory of quintic threefolds, but
  more generally for degree $d$ hypersurfaces in $\PP^N$, which are
  not necessarily Calabi--Yau.
  In this case, the relevant targets of punctured R-maps are related
  via a logarithmic root construction for $\ell = d$.
\end{remark}
}

\subsubsection{Witten's \texorpdfstring{$r$}{r}-spin class, effective spin structures and holomorphic differentials}
\label{sss:r-spin}

One of the initial motivations toward the development \cite{CJRS18P}
of the technique of log GLSM was to solve the conjecture
\cite[Conjecture~A.1]{PPZ19}, which provides a formula for the class
of a stratum of holomorphic differentials in terms of Witten's
$r$-spin class.
In future work \cite{CJPPRSSZ22P}, we will address the conjecture (and
generalizations) using the techniques developed in \cite{CJR23P},
\cite{CJRS18P} and this paper.
Understanding the dependence of Witten's $r$-spin class under changing
$r$ plays a crucial role for this.
We briefly outline this connection here.

The log GLSM target corresponding to Witten's $r$-spin class is
$\ul\fP_r = \PP_{\fX_r}(\cL_{\fX_r} \oplus \cO) \to \BC$, where
$\fX_r \cong \BG_m \to \BC$ is defined via the $r$th power map, and
where $\cL_{\fX_r}$ is the universal $r$th root.
Clearly, the corresponding $\infty$-divisor is given by
$\ul\infty_r = \fX_r$ with $\cO(\infty_r) = \cL_{\fX_r}^\vee$, and the
target $\infty_r$ can be obtained from $\infty_1$ via an $r$th
logarithmic root construction.

Combining \cite[Theorem~1.6]{CJRS18P} and the localization formula of
\cite{CJR23P}, we can express Witten's $r$-spin class using a
combination of effective cycles, and Chiodo's class \cite{Ch08}, which
is an analog of twisted Gromov--Witten theory.
The dependence on $r$ of Chiodo's class on the moduli of curves is
well-understood (see e.g. \cite{JPPZ17, ClJa18}), and
Corollary~\ref{cor:effective-change-ell} tells us the dependence of
the effective cycles on $r$.

It will be proven in \cite{CJPPRSSZ22P} that any effective cycle
\eqref{eq:effective-cycle} for $k = 0$ agrees with the class of the
closure of the corresponding locus of curves with effective $r$-spin
structures of \cite{Po06}.
For $r = 1$, we obtain the classes of strata of holomorphic
differentials.

\subsection{Tautological effective cycles}
\label{ss:tautological}

For every log GLSM target $\punt$ as in Section~\ref{sec:effective},
the effective cycles of \eqref{eq:effective-cycle} and
\eqref{eq:effective-cycle-Chow} are cycles in the cohomology or Chow
ring of $\oM_{g, n}$.
It is natural to ask whether these cycles lie in the tautological ring
of \cite{FaPa05}:
\begin{question}
  \label{q:effective-tautological}
  Under what conditions (on the target $\punt$, the discrete
  invariants, and the insertions) does an effective cycle
  \eqref{eq:effective-cycle} lie in the tautological ring
  $RH^*(\oM_{g, n}) \subset H^*(\oM_{g, n})$?
  What can be said about the analogous question for the effective
  cycles lie in the tautological Chow ring
  $R^*(\oM_{g, n}) \subset A^*(\oM_{g, n})$?
\end{question}
This question is an analog of the corresponding question for
Gromov--Witten classes discussed in \cite[Section~0.7]{ABPZ21P} and
\cite{Ts21P}, and which we state below in
Question~\ref{q:gw-tautological}.
In \cite{CJR23P}, we will discuss some connections to these
questions.

Before continuing, let us state a couple of immediate results towards
Question~\ref{q:effective-tautological}.
The vanishing results in Section~\ref{ss:vanishing} imply that many
$g \ge 2$ effective cycles are tautological since they vanish.
The Chow version of Question~\ref{q:effective-tautological} implies
its cohomology version, in the case that all insertions lie in the
image of the cycle class map.
The future work \cite{CJPPRSSZ22P} discussed in
Section~\ref{sss:r-spin} will imply that
Question~\ref{q:effective-tautological} has an affirmative answer in
cohomology for $k = 0$ for the $r$-spin target $\fX_r \to \BC$, and we
plan to also study the case of general $k$.
The discussion in Section~\ref{ss:setup-higher-spin} implies that
studying Question~\ref{q:effective-tautological} for a target $\infty$
is equivalent to studying the same question for its $r$th logarithmic
root $\infty^{1/r}$.

We now recall the definition of Gromov--Witten classes.
Let $\cX$ be a smooth Deligne--Mumford stack $\cX$ with projective
coarse moduli, a genus $g$, a curve class $\beta \in H_2(\cX; \ZZ)$,
and let $\ul\alpha_1, \dotsc, \ul\alpha_n \in H^*(\ocI\cX)$ be
(Chen--Ruan) cohomology classes.
Then, we may define the Gromov--Witten class in cohomology
\begin{equation*}
  p_* \left( \prod \ev_i^*\alpha_i \cap [\oM_{g, n}(\cX, \beta)]^{\vir} \right)\in H^*(\oM_{g, n}),
\end{equation*}
where $p\colon \oM_{g, n}(\cX, \beta) \to \oM_{g, n}$ denotes the
forgetful map from the moduli of stable maps.
If $\ul\alpha_1, \dotsc, \ul\alpha_n \in A^*(\ocI\cX)$ are Chow
classes, we may also define the corresponding Gromov--Witten class in
Chow:
\begin{equation*}
  p_* \left( \prod \ev_i^*\alpha_i \cap [\oM_{g, n}(\cX, \beta)]^{\vir} \right)\in A^*(\oM_{g, n}),
\end{equation*}

\begin{question}
  \label{q:gw-tautological}
  Under what conditions do Gromov--Witten classes in cohomology lie in
  the tautological ring?
  What about the analogous question for Gromov--Witten classes in Chow?
\end{question}
While all currently known examples of Gromov--Witten classes in
cohomology are tautological, the same is not true in Chow.
Hence, we should expect that Question~\ref{q:effective-tautological}
does not always have an affirmative answer in Chow.
We may still expect the following strong connection between
Question~\ref{q:gw-tautological} and
Question~\ref{q:effective-tautological}.
\begin{conjecture}
  \label{conj:tautological}
  Suppose $\infty_\cX$ has tautological
  Gromov--Witten classes in cohomology (resp. Chow).
  Then effective classes of $\punt$ are tautological in
  cohomology (resp. Chow).
\end{conjecture}
The explicit formulas in the genus one case \eqref{eq:g1-reduced-virtual-cycle}
 together with the
Grothendieck--Riemann--Roch calculation of \cite{Ts10}, yields the
following small evidence for the conjecture:
\begin{proposition}
  Conjecture~\ref{conj:tautological} is true in the genus one
  situation of \S~\ref{sss:genus-one-moduli}.
  In this special case, it is enough to assume that Gromov--Witten
  classes of $\infty_\cX$ are tautological in genus zero and one.
\end{proposition}

\bibliographystyle{abbrv}
\bibliography{hybrid}

\end{document}